\documentclass[leqno]{article}

\usepackage[frenchb,english]{babel}
\usepackage[utf8]{inputenc}
\usepackage{amsmath}
\usepackage{amssymb}
\usepackage{amsfonts}
\usepackage{enumerate}
\usepackage{vmargin}
\usepackage[all]{xy}
\usepackage{mathrsfs}
\usepackage{mathtools}
\usepackage{tikz}
\usepackage{lmodern}
\usepackage{verbatim}
\usepackage[colorlinks=true,pagebackref=true,linkcolor=blue]{hyperref}
\usepackage{comment}
\setmarginsrb{3cm}{3cm}{3.5cm}{3cm}{0cm}{0cm}{1.5cm}{3cm}

\usepackage{tikz}

\tikzstyle{shaded}=[fill=red!10!blue!20!gray!30!white]
\tikzstyle{shaded line}=[double=red!10!blue!20!gray!30!white, double distance=1.5mm, draw=black]
\tikzstyle{unshaded}=[fill=white]
\tikzstyle{unshaded line}=[double=white, double distance=1.5mm, draw=black]
\tikzstyle{Tbox}=[circle, draw, thick, fill=white, opaque,]
\tikzstyle{empty box}=[circle, draw, thick, fill=white, opaque, inner sep=2mm]
\tikzstyle{background rectangle}= [fill=red!10!blue!20!gray!40!white,rounded corners=2mm] 
\tikzstyle{on}=[very thick, red!50!blue!50!black]
\tikzstyle{off}=[gray]

\tikzstyle{traces}=[scale=.2, inner sep=1mm]
\tikzstyle{quadratic}=[scale=.4, inner sep=1mm, baseline]
\tikzstyle{annular}=[scale=.7, inner sep=1mm, baseline]
\tikzstyle{make triple edge size}= [scale=.4, inner sep=1mm,baseline] 
\tikzstyle{icosahedron network}=[scale=.3, inner sep=1mm, baseline]
\tikzstyle{ATLsix}=[scale=.25, baseline]
\tikzstyle{TL12}=[scale=.15,baseline]
\tikzstyle{PAdefn}=[scale=.7,baseline]
\tikzstyle{TLEG}=[scale=.5,baseline]

\newcommand{\E}{\ensuremath{\mathbb{E}}}
\newcommand{\N}{\ensuremath{\mathbb{N}}}
\newcommand{\B}{\mathrm{B}} 
\newcommand{\F}{\mathcal{F}} 
\let\H\relax 
\newcommand{\H}{\mathrm{H}}
\newcommand{\I}{\mathrm{I}} 
\newcommand{\II}{\mathrm{II}} 

\newcommand{\KP}{\ensuremath{\mathbb{KP}}}
\let\L\relax 
\newcommand{\L}{\mathrm{L}}
\let\O\relax 
\newcommand{\O}{\mathrm{O}}
\newcommand{\Q}{\mathrm{Q}}
\newcommand{\scr}{\mathscr}
\newcommand{\SU}{\mathrm{SU}}
\newcommand{\Aut}{\mathrm{Aut}}
\newcommand{\QG}{\ensuremath{\mathbb{G}}} 
\newcommand{\QH}{\ensuremath{\mathbb{H}}} 
\newcommand{\M}{\mathrm{M}}
\newcommand{\CB}{\mathrm{CB}} 
\newcommand{\HS}{\mathrm{HS}}
\newcommand{\can}{\mathrm{can}} 

\let\P\relax 
\newcommand{\P}{\mathcal{P}} 
\let\cal\relax
\newcommand{\cal}{\mathcal}
\newcommand{\Z}{\ensuremath{\mathbb{Z}}}
\newcommand{\R}{\ensuremath{\mathbb{R}}}
\newcommand{\C}{\mathrm{C}}
\newcommand{\T}{\ensuremath{\mathbb{T}}}

\newcommand{\ZZ}{\mathbb{Z}}
\newcommand{\W}{\mathrm{W}}
\newcommand{\vect}{\ensuremath{\mathop{\mathrm{span}\,}\nolimits}}

\newcommand{\Id}{\mathrm{Id}}
\newcommand{\Mult}{\mathrm{Mult}}
\newcommand{\Mod}{\mathrm{Mod}} 
\newcommand{\VN}{\mathrm{VN}} 
\newcommand{\la}{\langle}
\newcommand{\ra}{\rangle}

\newcommand{\PPT}{\mathrm{PPT}}
\renewcommand{\leq}{\ensuremath{\leqslant}}
\renewcommand{\geq}{\ensuremath{\geqslant}}

\newcommand{\qed}{\hfill \vrule height6pt  width6pt depth0pt}
\newcommand{\bnorm}[1]{ \big\| #1  \big\|}

\newcommand{\bgnorm}[1]{ \bigg\| #1  \bigg\|}

\newcommand{\norm}[1]{\left\Vert#1\right\Vert}
\newcommand{\xra}{\xrightarrow}
\newcommand{\co}{\colon}
\newcommand{\otp}{\widehat{\ot}}
\newcommand{\otpb}{\hat{\ot}}
\newcommand{\ot}{\otimes}
\newcommand{\ovl}{\overline}
\newcommand{\otvn}{\ovl\ot}

\newcommand{\D}{\mathrm{D}}
\newcommand{\dec}{\mathrm{dec}}

\newcommand{\cb}{\mathrm{cb}}
\newcommand{\ccb}{\mathrm{ccb}}

\newcommand{\SO}{\mathrm{SO}}

\newcommand{\QWEP}{\mathrm{QWEP}}
\newcommand{\dsp}{\displaystyle}
\newcommand{\op}{\mathrm{op}} 
\let\i\relax 
\newcommand{\i}{\mathrm{i}}

\newcommand{\ov}{\overset}
\newcommand{\sa}{\mathrm{sa}}
\newcommand{\JW}{\mathrm{JW}}
\newcommand{\Fix}{\mathrm{Fix}}
\newcommand{\EA}{\mathrm{EA}}
\newcommand{\WH}{\mathrm{WH}} 

\newcommand{\MD}{\mathrm{MD}} 
\newcommand{\LMD}{\mathrm{LMD}} 
\newcommand{\RMD}{\mathrm{RMD}} 
\newcommand{\w}{\mathrm{w}} 
\newcommand{\flip}{\mathrm{flip}}

\newcommand{\epsi}{\varepsilon}
\renewcommand{\d}{\mathop{}\mathopen{}\mathrm{d}} 
\newcommand{\e}{\mathrm{e}} 
\renewcommand{\d}{\mathop{}\mathopen{}\mathrm{d}}

\newcommand{\Pol}{\mathrm{Pol}} 

\DeclareMathOperator{\tr}{Tr} 
\DeclareMathOperator{\Ran}{Ran} 
\DeclareMathOperator{\dom}{dom} 
\let\Re\relax 
\DeclareMathOperator{\Re}{Re} 
\DeclareMathOperator{\Irr}{Irr} 
\DeclareMathOperator{\Ind}{Ind} 
\DeclareMathOperator{\Sp}{Sp} 
\newtheorem{thm}{Theorem}[section]
\newtheorem{defi}[thm]{Definition}
\newtheorem{prop}[thm]{Proposition}

\newtheorem{cor}[thm]{Corollary}
\newtheorem{lemma}[thm]{Lemma}

\newtheorem{remark}[thm]{Remark}
\newtheorem{example}[thm]{Example}

\newenvironment{proof}[1][]{\noindent {\it Proof #1} : }{\hbox{~}\qed
\smallskip
}

\usepackage{tocloft}
\numberwithin{equation}{section}
\usepackage[nottoc,notlot,notlof]{tocbibind}

\let\OLDthebibliography\thebibliography
\renewcommand\thebibliography[1]{
  \OLDthebibliography{#1}
  \setlength{\parskip}{0pt}
  \setlength{\itemsep}{0pt plus 0.3ex}
}

\usepackage{scalerel}

\newcommand\reallywidehat[1]{\arraycolsep=0pt\relax%
\begin{array}{c}
\stretchto{
  \scaleto{
    \scalerel*[\widthof{\ensuremath{#1}}]{\kern-.5pt\bigwedge\kern-.5pt}
    {\rule[-\textheight/2]{1ex}{\textheight}} 
  }{\textheight} %
}{0.5ex}\\           
#1\\                 
\rule{-1ex}{0ex}
\end{array}
}

\begin{document}
\selectlanguage{english}
\title{\bfseries{Quantum information theory via Fourier multipliers on quantum groups}}
\date{}
\author{\bfseries{C\'edric Arhancet}}

\maketitle


\begin{abstract}
In this paper, we compute the exact values of the minimum output entropy and the completely bounded minimal entropy of very large classes of quantum channels acting on matrix algebras $\mathrm{M}_n$. Our new and simple approach relies on the theory of locally compact quantum groups and our results use a new and precise description of bounded Fourier multipliers from $\mathrm{L}^1(\mathbb{G})$ into $\mathrm{L}^p(\mathbb{G})$ for $1 < p \leq \infty$ where $\mathbb{G}$ is a co-amenable locally compact quantum group and on the automatic completely boundedness of these multipliers that this description entails. Indeed, our approach even allows to use convolution operators on quantum hypergroups. This enable us to connect equally the topic of computation of entropies and capacities to subfactor planar algebras. We also give a upper bound of the classical capacity of each considered quantum channel which is already sharp in the commutative case. Quite surprisingly, we  observe by direct computations that some Fourier multipliers identifies to direct sums of classical examples of quantum channels (as dephasing channel or depolarizing channels). Indeed, we show that the study of unital qubit channels can be seen as a part of the theory of Fourier multipliers on the von Neumann algebra of the quaternion group $\mathbb{Q}_8$. Unexpectedly, we also connect ergodic actions of (quantum) groups to this topic of computation, allowing some transference to other channels. We also connect the Quantum Harmonic analysis of Werner. Finally, we investigate entangling breaking and $\mathrm{PPT}$ Fourier multipliers and we characterize conditional expectations which are entangling breaking. 
\end{abstract}




\makeatletter
 \renewcommand{\@makefntext}[1]{#1}
 \makeatother

\tableofcontents

\section{Introduction}

\hspace{0.4cm} One of the most fundamental questions in quantum information explicitly stated in \cite{Sho95} concerns with the amount of information that can be transmitted reliably through a quantum channel. For that, many capacities was introduced for describing the capability of the channel for delivering information from the sender to the receiver. Determine these quantities is central to characterizing the limits of the quantum channel's ability to transmit information reliably. 

In the Sch\"odinger picture, a quantum channel is a trace preserving completely positive map $T \co S^1_n\to S^1_n$ acting on the Schatten trace class space $S^1_n$. The minimum output entropy of a quantum channel is defined by
\begin{equation}
\label{Def-minimum-output-entropy-ini}
\H_{\min}(T)
\ov{\mathrm{def}}{=} \min_{x \geq 0, \tr x =1} \H(T(x))
\end{equation}
where $\H(\rho) \ov{\mathrm{def}}{=} -\tr(\rho\log \rho)$ denotes the von Neumann entropy. Recall that $\H$ is the quantum analogue of the Shannon entropy and can be seen as a uncertainty measure of $\rho$. See for example \cite{Pet01} for a short survey of this notion. 
Note that $\H(\rho)$ is always positive, takes its maximum values when $\rho$ is maximally mixed and its minimum value zero if $\rho$ is pure. So the minimum output entropy can be seen as a measure of the deterioration of the purity of states. The minimum output entropy of a specific channel is difficult to calculate in general.

In this work, we investigate more generally quantum channels acting on a finite-dimensional noncommutative $\L^1$-space $\L^1(\cal{M})$ associated with a finite-dimensional von Neumann algebra\footnote{\thefootnote. Such a algebra $\cal{M}$ is isomorphic to a direct
sum of matrix algebras $\M_{n_1} \oplus \cdots \oplus \M_{n_K}$.} $\cal{M}$ equipped with a faithful trace $\tau$. The first aim of this paper is to compute the quantity
\begin{equation}
\label{Def-minimum-output-entropy}
\H_{\min}(T)
\ov{\mathrm{def}}{=} \min_{x \geq 0, \tau(x)=1} \H(T(x))
\end{equation} 
for any $T \co \L^1(\mathcal{M})\to \L^1(\cal{M})$ belonging to a large class of explicit quantum channels. Here $\H(\rho) \ov{\mathrm{def}}{=} -\tau(\rho \log \rho)$ denotes the Segal entropy introduced in \cite{Seg60}. Note that if the trace $\tau$ is \textit{normalized} and if $\rho$ is a positive element of $\L^1(\cal{M})$ with $\tau(\rho)=1$ then $\H(\rho)$ belongs to $[-\infty,0]$ with $\H(\rho)=0$ if and only if $\rho=1$. We refer \cite{LuP17}, \cite{LPS17}, \cite{LuP19}, \cite{LoW22}, \cite{NaU61}, \cite{OcS78} and \cite{Rus73} for more information on the Segal entropy.


The considered channels are the (trace preserving completely positive) Fourier multipliers on finite quantum groups, see e.g. \cite{Daw10} and \cite{Daw11} for more information. Such a channel $T \co \L^1(\QG) \to \L^1(\QG)$ acts on a finite-dimensional noncommutative $\L^1$-space $\L^1(\QG)$ and is a noncommutative variant of a classical Fourier multiplier $M_{\varphi} \co \L^1(\T) \to \L^1(\T)$, $\sum_k a_k \e^{2\pi \i k \cdot} \mapsto \sum_k \varphi_k a_k \e^{2\pi \i k \cdot}$ where $\varphi=(\varphi_k)$ is a suitable sequence of $\ell^\infty_\Z$. Here $\T \ov{\mathrm{def}}{=} \{z \in \C : |z|=1\}$ is the one-dimensional torus and we refer to \cite{Gra14} for the classical theory of multipliers on $\T$. Loosely speaking, a finite quantum group is a finite-dimensional $\mathrm{C}^*$-algebra $A$ equipped with a suitable ``coproduct'' map $\Delta \co A \to A \ot A$ that is a noncommutative variant of the commutative $\mathrm{C}^*$-algebra $\mathrm{C}(G)$ of continuous functions on a compact group $G$ equipped with the map $\Delta \co \mathrm{C}(G)\to\mathrm{C}(G) \ot \mathrm{C}(G)$  where $\Delta(f)(s,t)=f(st)$ for $f \in \mathrm{C}(G)$ and $s,t \in G$. It is possible to do harmonic analysis and representation theory in this setting. See \cite{FSS17}, \cite{KuV00}, \cite{MaD98} and \cite{Maj06} for short presentations.

Note that we can see our channels as acting on matrix algebras by using a suitable $*$-homomorphism $J$ and a conditional expectation $\E$, see Section \ref{Sec-Examples}.

Our approach relies on an observation of the authors of \cite{ACN00}. In the case $\mathcal{M}=\M_n$, they observed that we can express the minimum output entropy as a derivative of suitable operators norms:
\begin{equation}
\label{Smin-as-derivative-intro}
\H_{\min}(T)
=-\frac{\d}{\d p} \norm{T}_{\L^1(\mathcal{M}) \to \L^p(\mathcal{M})}^p|_{p=1}. 
\end{equation}
See Section \ref{Appendix} for a proof of this formula for the slightly more general case of finite-dimensional von Neumann algebras. So it suffices to compute the operator norm $\norm{\cdot}_{\L^1(\mathcal{M}) \to \L^p(\mathcal{M})}$ in order to compute the minimum output entropy. Unfortunately, in general, the calculation of these operator norms is not much easier than the direct computation of the minimum output entropy. However, the first purpose of this paper is to show that the computation of $\norm{T}_{\L^1(\QG) \to \L^p(\QG)}$ is possible in the case of an arbitrary Fourier multiplier $T$ on a finite quantum group $\QG$ and more generally on a co-amenable compact quantum group of Kac type. 

Moreover, we prove that for these channels the minimum output entropy is equal to the completely bounded minimal entropy (reverse coherent information). This notion was introduced for $\cal{M}=\M_n$ in \cite{DJKRB06} and rediscovered in \cite{GPLS09}, see also \cite{YHW19}. We can use the following definition which will be justified in Theorem \ref{thm-entropy-concrete}
\begin{equation}
\label{Def-intro-Scb-min}
\H_{\cb,\min}(T)
\ov{\mathrm{def}}{=} -\frac{\d}{\d p} \norm{T}_{\cb,\L^1(\mathcal{M}) \to \L^p(\mathcal{M})}^p|_{p=1} 
\end{equation} 
where the subscript $\cb$ means completely bounded and is natural in operator space theory, see the books \cite{BLM04}, \cite{EfR00}, \cite{Pau02}, \cite{Pis98} and \cite{Pis03}. We refer again to Theorem \ref{thm-entropy-concrete} for a concrete expression of $\H_{\cb,\min}(T)$. Indeed, we will show that $\norm{T}_{\cb,\L^1(\QG) \to \L^p(\QG)}=\norm{T}_{\L^1(\QG) \to \L^p(\QG)}$ for any bounded Fourier multiplier $T$ on a co-amenable compact quantum group of Kac type $\QG$ (hence in particular on finite quantum groups) and any $1 \leq p \leq \infty$ (see Theorem \ref{Thm-description-multipliers-2}) and we will immediately deduce the equality $\H_{\min}(T)=\H_{\cb,\min}(T)$. Note that if $\cal{M}$ is a \textit{non-abelian} von Neumann algebra, we will prove in Section \ref{Bounded-vs-cb} that the completely bounded norm $\norm{\cdot}_{\cb,\L^1(\mathcal{M}) \to \L^p(\mathcal{M})}$ and the classical operator norm $\norm{\cdot}_{\L^1(\mathcal{M}) \to \L^p(\mathcal{M})}$ are not equal. So the previous equality is specific to the Fourier multipliers on co-amenable compact quantum groups of Kac type. Finally, we refer to \cite{GuW15} for another application of these completely bounded norms.

Note that the completely bounded minimal entropy $\H_{\cb,\min}$ is additive (see Theorem \ref{Additivity-Scb-min}) in sharp contrast with the minimum output entropy \cite{Has09}. 

Fundamental characteristics of a quantum channel are the different capacities. We refer to the survey \cite{GIN18} for more information on the large diversity of quantum channel capacities. We refer to the survey \cite{GIN18} for more information on the large diversity of quantum channel capacities and to \cite{FuW07}, and \cite{WiY16} for some related papers. Despite of the significant progress in recent years, the computation of most capacities remains largely open. The Holevo capacity of a quantum channel $T \co \L^1(\cal{M}) \to  \L^1(\cal{M})$ can be defined by the formula
\begin{equation}
\label{Def-Holevo-capacity}
\chi(T)
\ov{\mathrm{def}}{=} \sup \bigg\{ \H\bigg(\sum_{i=1}^{N} \lambda_iT(\rho_i)\bigg)-\sum_{i=1}^{N} \lambda_i \H(T(\rho_i))\bigg\}
\end{equation}
where the supremum is taken over all $N \in \mathbb{N}$, all probability distributions $(\lambda_1,\ldots,\lambda_N)$ and all states $\rho_1,\ldots,\rho_N$ of $\L^1(\cal{M})$, see e.g. \cite[p.~21]{GIN18} and \cite[p.~82]{Key02} for the case $\cal{M}=\M_n$. It is known that the classical capacity $\C(T)$ can be expressed by the asymptotic formula
\begin{equation}
\label{Classical-capacity}
\C(T)
=\lim_{n \to +\infty} \frac{\chi(T^{\ot n})}{n}.
\end{equation}
This quantity is the rate at which one can reliably send \textit{classical} information through $T$, see Definition \ref{Def-capacities}. Using our computation of $\H_{\min}(T)$, 
we are able to determine a \textit{sharp} upper bound of $\C(T)$ for any Fourier multiplier $T$ acting on a finite quantum group $\QG$ (see Theorem \ref{thm-capacity}). It is the second aim of this paper. 
Recall that the classical capacity is always bigger 
than the quantum capacity $\Q(T)$
which the the rate at which one can reliably send \textit{quantum} information through $T$, see Definition \ref{Def-capacities}.

We refer to \cite{GJL18b}, \cite{GJL18a}, \cite{JuP15} and \cite{JuP16} for papers using operator space theory for estimating quantum capacities and to the books \cite{GMS15}, \cite{Hol19}, \cite{KhW20}, \cite{NiC10}, \cite{Pet08}, \cite{Wat18} and \cite{Wil17} for more information on quantum information theory. 

To conclude, we also hope that this connection between quantum information theory and the well-established theory of Fourier multipliers on operator algebras will be the starting point of other progress on the problems of quantum information theory but also in noncommutative harmonic analysis. We leave open an important number of implicit questions. We equally refer to \cite{BCLY20}, \cite{CrN13} and \cite{LeY22} for papers on quantum information theory in link with noncommutative harmonic analysis. See also \cite{ACN20} and \cite{NSSS21} for very recent paper on Fourier multipliers on quantum groups.

\paragraph{Structure of the paper} The paper is organized as follows. Section \ref{Sec-preliminaries} gives background on noncommutative $\L^p$-spaces, compact quantum groups and operator space theory. These notions are fundamental for us. Section \ref{sec-Frobenius} introduce Frobenius von Neumann algebras which are some generalizations of quantum groups, equipped with a convolution. These structures appear naturally in theory of subfactor planar algebras.

In Section \ref{Sec-convolution}, we prove that the convolution induces a \textit{completely} bounded map on some noncommutative $\L^p$-spaces associated to quantum hypergroups. In Section \ref{Sec-bounded}, we describe bounded multipliers from $\L^1(\cal{M})$ into $\L^p(\cal{M})$ where $\cal{M}$ is a co-amenable quantum hypergroup (e.g. a co-amenable compact quantum group of Kac type). We deduce from Section \ref{Sec-convolution} that these multipliers are necessarily completely bounded with the same completely bounded norm. We also describe by duality the (completely) bounded multipliers $\L^p(\cal{M}) \to \L^\infty(\cal{M})$ which will be useful for obtaining the result on the classical capacities. In Section \ref{completely-bounded-description}, we describe some completely bounded multipliers on compact quantum groups of Kac type, beyond the coamenable case. In Section \ref{Bounded-vs-cb}, we compare the spaces of bounded maps and of completely bounded maps from $\L^1(\cal{M})$ into $\L^\infty(\cal{M})$ in the case of a von Neumann algebra $\cal{M}$. We show that in the case of a nonabelian von Neumann algebra $\mathcal{M}$, these norms are different. We equally examine the completely bounded Schur multiplies from $S^1$ into the space $\B(\ell^2)$ of bounded operators acting on the Hilbert space $\ell^2$. 
In Section \ref{summing-multipliers}, we examine the $\ell^p(S^p_d)$-summing (left or right) Fourier multipliers acting from $\L^\infty(\cal{M})$ into $\L^\infty(\cal{M})$ where $\cal{M}$ is a co-amenable Frobenius von Neumann algebra (e.g. a co-amenable compact quantum group of Kac type). 

In Section \ref{Entropy-capacity-I}, we investigate the relations between the capacities of a channel $T \co \L^1(\cal{M}) \to \L^1(\cal{M})$ and its extension $JT\E \co \L^1(\cal{N}) \to \L^1(\cal{N})$ where $J \co \cal{M} \to \cal{N}$ is a trace preserving unital normal $*$-injective homomorphism and where $\E$ is the associated conditional expectation. In Section \ref{Entropy-capacity}, we compute the minimal output entropy and we obtain the equality with the completely bounded minimal entropy. Finally, we obtain our sharp upper bound on the classical capacity. 

In Section \ref{Sec-Kac-Paljukin} and Section \ref{Sec-Dihedral}, we illustrate our results with the Kac-Paljutkin quantum group and dihedral groups.

In Section \ref{Sec-actions-Transference} and the next sections, we show how our results on quantum channels on operator algebras allows us to recover by transference some results of \cite{GJL18b} and \cite{GJL18a} for quantum channels $T \co S^1_n \to S^1_n$ acting on the Schatten trace class space $S^1_n$.



In Section \ref{Sec-Multiplicative}, we give a description of multiplicative domains of Markov maps (=quantum channels) and describe multiplicative domains of Fourier multipliers. In section \ref{Sec-PPT}, we characterize PPT conditional expectations and entangling breaking PPT conditional expectations (Corollary \ref{Cor-esp-cond}). We also give a necessary condition on PPT Fourier multipliers.

In Section \ref{Appendix}, we give a proof of the formula \eqref{Smin-as-derivative-intro} which describes the minimum output entropy $\H_{\min}(T)$ as a derivative of operator norms. In Section \ref{Appendix2}, we give a  proof of the multiplicativity of the completely bounded norms $\norm{\cdot}_{\cb,\L^q \to \L^p}$ if $1 \leq q \leq p \leq \infty$. We deduce the additivity of $\H_{\cb,\min}(T)$. 


\section{Preliminaries}
\label{Sec-preliminaries}

\paragraph{Noncommutative $\L^p$-spaces}
Here we recall some facts on noncommutative $\L^p$-spaces associated with semifinite von Neumann algebras. We refer to \cite{Dix81} \cite{Tak02} \cite{Sak98} \cite{Str81} for the theory of von Neumann algebras and to \cite{PiX03} for more information on noncommutative $\L^p$-spaces. Let $\cal{M}$ be a semifinite von Neumann algebra equipped with a normal semifinite faithful trace $\tau$. We denote by $\cal{M}_+$ the positive part of $\cal{M}$. Let $\cal{S}_+$ be the set of all $x \in \cal{M}_+$ whose support projection have a finite trace. Then any $x \in \cal{S}_+$ has a finite trace. Let $\cal{S} \subset \cal{M}$ be the complex linear span of $\cal{S}_+$, then $\cal{S}$ is a weak*-dense $*$-subalgebra of $\cal{M}$.

Let $1 \leq p <\infty$. For any $x \in \cal{S}$, the operator $\vert x\vert^p$
belongs to $\cal{S}_+$ and we set
\begin{equation}
\label{Def-norm-Lp}
\norm{x}_{\L^p(\cal{M})}
\ov{\mathrm{def}}{=} \bigl(\tau(\vert x\vert^p)\bigr)^{\frac{1}{p}}.
\end{equation}
Here $\vert x \vert \ov{\mathrm{def}}{=}(x^*x)^{\frac{1}{2}}$ denotes the modulus of $x$. It turns out that $\norm{\cdot}_{\L^p(\cal{M})}$ is a norm on $\cal{S}$. By definition, the noncommutative $\L^p$-space $\L^p(\cal{M})$ associated with $(\cal{M},\tau)$ is the completion of $(\cal{S},\norm{\cdot}_{\L^p(\cal{M})})$. For convenience, we also set $\L^{\infty}(\cal{M}) \ov{\mathrm{def}}{=} \cal{M}$ equipped with its operator norm. Note that by definition, $\L^p(\cal{M}) \cap \cal{M}$ is dense in $\L^p(\cal{M})$ for any $1 \leq p < \infty$.


Furthermore, $\tau$ uniquely extends to a bounded linear functional on $\L^1(\cal{M})$, still denoted by $\tau$. Indeed, we have 
\begin{equation}
\label{trace-continuity}
\vert\tau(x)\vert 
\leq \norm{x}_{\L^1(\cal{M})}, 
\quad x \in \L^1(\cal{M}).
\end{equation}
We recall the noncommutative H\"older's inequality. If $1 \leq p,q,r \leq \infty$ satisfy $\frac{1}{r}=\frac{1}{p}+\frac{1}{q}$ then
\begin{equation}
\label{Holder}
\norm{xy} _{\L^r(\cal{M})}
\leq \norm{x}_{\L^p(\cal{M})} \norm{y}_{\L^q(\cal{M})},\qquad x\in \L^p(\cal{M}), y \in \L^q(\cal{M}).
\end{equation}
Conversely for any $z \in \L^r(\cal{M})$, by using the same ideas than the proof of \cite[Theorem 4.20]{Ray03} there exist $x \in \L^p(\cal{M})$ and $y \in \L^q(\cal{M})$ such that 
\begin{equation}
\label{inverse-Holder}
z
=xy
\quad \text{with} \quad 
\norm{z} _{\L^r(\cal{M})}
=\norm{x}_{\L^p(\cal{M})} \norm{y}_{\L^q(\cal{M})}.
\end{equation}
For any $1 \leq p < \infty$, let $p^* \ov{\mathrm{def}}{=} \frac{p}{p-1}$ be the conjugate number of $p$. Applying \eqref{Holder} with $q=p^*$ and $r=1$ together with \eqref{trace-continuity}, we obtain a map $\L^{p^*}(\cal{M}) \to (\L^p(\cal{M}))^*$, $y \mapsto \tau(xy)$ which induces an isometric isomorphism
\begin{equation}
(\L^p(\cal{M}))^* 
=\L^{p^*}(\cal{M}),\qquad 1 \leq p <\infty,\quad \frac{1}{p}
+\frac{1}{p^*}
=1.
\end{equation}
In particular, we may identify $\L^1(\cal{M})$ with the unique predual $\cal{M}_*$ of $\cal{M}$. Finally, for any positive element in $\L^p(\cal{M})$, recall that there exists a positive element $y \in \L^{p^*}(\cal{M})$ such that
\begin{equation}
\label{dual-by-positive}
\norm{y}_{\L^{p^*}(\cal{M})}=1
\quad \text{and} \quad
\norm{x}_{\L^p(\cal{M})}
=\tau(xy).
\end{equation}

The following lemma is folklore.

\begin{lemma}
\label{lemma-trace-preserving} 
Let $\cal{M}$ and $\cal{N}$ be von Neumann algebras equipped with normal semifinite faithful traces. Let $T \co \cal{M} \to \cal{N}$  be a trace preserving unital weak* continuous positive map. Suppose that $1 \leq p < \infty$. Then $T$ induces a contraction $T \co \L^p(\cal{M}) \to \L^p(\cal{N})$. Moreover, if $T \co \cal{M} \to \cal{N}$ is, in addition, an injective $*$-homomorphism, $T$ induces an isometry $T \co \L^p(\cal{M}) \to \L^p(\cal{N})$.
\end{lemma}

\paragraph{Locally compact quantum groups}
Let us recall some well-known definitions and properties concerning compact quantum groups. We refer to \cite{Cas17}, \cite{EnS92}, \cite{FSS17}, \cite{Kus05}, \cite{MaD98}, \cite{NeT13}, \cite{Tus22}, \cite{VDa14} and references therein for more details.

A locally compact quantum group is a pair $\QG=(\L^\infty(\QG),\Delta)$ with the following properties:
\begin{enumerate}
\item $\L^\infty(\QG)$ is a von Neumann algebra,

\item $\Delta \co \L^\infty(\QG) \to \L^\infty(\QG) \otvn \L^\infty(\QG)$ is a co-multiplication,
that is an injective faithful normal unital $*$-homomorphism which is co-associative: 
\begin{equation}
\label{Co-associativity}
(\Delta \ot \Id)\Delta
=(\Id \ot \Delta)\Delta,
\end{equation}

\item there exist normal semifinite faithful weights $\varphi,\psi$ on $\L^\infty(\QG)$, called called the left and right Haar weights satisfying
\[
\begin{split} 
& \varphi((\omega \ot \Id)\Delta(x))=\omega(1)\varphi(x)\text{ for all } \omega \in \L^\infty(\QG)_{*}^{+},x \in \L^\infty(\QG)_{+}
\text{ such that }\varphi(x) < \infty,\\
 & \psi((\Id \ot \omega)\Delta(x))=\omega(1)\psi(x)
\text{ for all }\omega \in \L^\infty(\QG)_{*}^{+},x \in \L^\infty(\QG)_+ 
\text{ such that }\psi(x) <\infty.
\end{split}
\]
 \end{enumerate}

Let $\QG$ be a locally compact quantum group. By \cite[Theorem 5.2]{Kus05}, the left and right Haar weights, whose existence is assumed, are unique up to scaling. If the left and right Haar weights $\varphi$ and $\psi$ of $\QG$ coincide then we say that $\QG$ is unimodular.

Let $\mathfrak{n}_\varphi \ov{\mathrm{def}}{=} \{x \in \L^\infty(\QG) : \varphi(x^*x) <\infty \} $. By \cite[Theorem 5.4]{Kus05}, there exists a unique $\sigma$-strongly* closed linear operator $S \co \dom S \subset \L^\infty(\QG) \to \L^\infty(\QG)$ such that the linear space
\begin{equation}
\label{}
\big\langle (\Id \ot \varphi)(\Delta(z)(1 \ot y)) : y,z \in \frak{n}_\varphi \big\rangle
\end{equation}
is a core for $S$ with respect to the $\sigma$-strong* topology and satisfying
\begin{equation}
\label{Antipode-Magic-relation-4}
S\big((\Id \ot \varphi)(\Delta(z)(1 \ot y))\big)
=(\Id \ot \varphi)\big((1 \ot z)\Delta(y)\big), \quad y,z \in \frak{n}_\varphi.
\end{equation}
We say that $S$ is the antipode of $\L^\infty(\QG)$. By \cite[Proposition 5.5]{Kus05} there exists a unique $*$-antiautomorphism $R \co \L^\infty(\QG) \to \L^\infty(\QG)$ and a unique weak* continuous group $\tau=(\tau_t)_{t \in \R}$ of $*$-automorphisms such that\footnote{\thefootnote. Note that the equation $S=R\tau_{-\frac{\i}{2}}$ has to be understood in the following sense: when $x \in \L^\infty(\QG)$ is \textit{analyti}c with respect to $\tau$, then $x$ belongs to the domain of $S$ and $S(x)=R\tau_{-\frac{\i}{2}}(x)$.}
\begin{equation}
\label{relation-SR}
S=R\tau_{-\frac{\i}{2}}, \quad R^2=\Id 
\quad \text{and} \quad \tau_{t}R=R\tau_{t}, \quad \text{for any }\quad t \in \R.
\end{equation}
We say that $R$ is the unitary antipode and that $\tau$ is the scaling group of $\QG$. If $\flip \co \L^\infty(\QG) \otvn \L^\infty(\QG) \to \L^\infty(\QG) \otvn \L^\infty(\QG)$, $x \ot y \mapsto y \ot x$ is the flip, by \cite[Proposition 5.6]{Kus05} we have
\begin{equation}
\label{R-and-Delta}
\Delta \circ R
=\flip \circ (R \ot R) \circ \Delta.
\end{equation}
Since $\varphi \circ R$ is a right invariant weight, we can suppose that 
\begin{equation}
\label{varphi-R}
\psi
=\varphi \circ R.
\end{equation}
If $(\sigma_t)_{t \in \R}$ is the modular automorphism group associated to the left Haar weight $\varphi$, the locally compact quantum group $\QG$ is said to be of Kac type (or a Kac algebra) if $\QG$ has trivial scaling group and $R\sigma_t = \sigma_{-t}R$ for any $t \in \R$. In this case, we have
\begin{equation}
\label{equality-antipodes}
S=R.
\end{equation}
We refer to \cite[p.~7]{VaV03} for other characterizations of Kac algebras.

Following \cite[Proposition 3.1]{Run08}, we say that a locally compact quantum group $\QG$ is compact if the left Haar weight $\varphi$ of $\QG$ is finite. Such group is always unimodular and the corresponding Haar weight can always be chosen to be a state $h \co \L^\infty(\QG) \to \mathbb{C}$ (called the Haar state). In this case, by \eqref{varphi-R} $R$ preserves the finite trace $h$, i.e.~$h \circ R=h$. Moreover, we have by essentially \cite[p.~14]{KuV03}
\begin{equation}
\label{Def-haar-state}
(h \ot \Id) \circ \Delta(x)
=h(x)1
=(\Id \ot h) \circ \Delta(x), \quad x \in \L^\infty(\QG).
\end{equation}
For a compact quantum group $\QG$ being of Kac type is equivalent to the traciality of the Haar state $h$, often denoted by $\tau$. This is equivalent to the fact that its dual $\hat{\QG}$ defined in \cite[Theorem 5.4]{Kus05} is unimodular.

We say that $\QG$ is finite if the algebra $\L^\infty(\QG)$ is finite-dimensional. Such a compact quantum group is of Kac type by \cite{VDa97}.


\paragraph{Representations} 
Suppose that $\QG$ is a compact quantum group. A unitary matrix $u=[u_{ij}]_{1 \leq i,j \leq n}$ element of $\M_n(\L^\infty(\QG))$ with entries $u_{ij} \in \L^\infty(\QG)$ is called a $n$-dimensional unitary representation of $\QG$ if for any $i,j=1,\ldots,n$ we have 
\begin{equation}
\label{Def-rep}
\Delta(u_{ij})
=\sum_{k=1}^{n} u_{ik} \ot u_{kj}.
\end{equation}

We say that $u$ is irreducible if $\{x \in \M_n : (x \ot 1)u = u(x \ot 1)\} = \mathbb{C}1$. 
We denote by $\Irr(\QG)$ the set of unitary equivalence \textit{classes} of irreducible finite-dimensional unitary representations of $\QG$. For each $\pi \in \Irr(\QG)$, we fix a representative $u^{(\pi)}\in \L^\infty(\QG) \ot \M_{n_\pi}$ of the class $\pi$.



We let 
\begin{equation}
\label{}
\Pol(\QG)
\ov{\mathrm{def}}{=} \mathrm{span}\, \big\{u_{ij}^{(\pi)} : u^{(\pi)}=[u_{ij}^{(\pi)}]_{i,j=1}^{n_{\pi}},\pi \in \Irr(\QG)\big\}.
\end{equation}
This is a weak* dense subalgebra of $\L^\infty(\QG)$.

\paragraph{Antipodes on compact quantum groups of Kac type}

If $\QG$ is a compact quantum group of Kac type, it is known that the antipode of $\QG$ is an $*$-antiautomorphism $R \co \L^\infty(\QG) \to \L^\infty(\QG)$ determined by 
$$
R\big(u_{ij}^{(\pi)}\big)
=\big(u_{ji}^{(\pi)}\big)^*
,\quad \pi \in \Irr(\QG), 1\leq i,j \leq n_\pi.
$$



\paragraph{Fourier transforms}



Suppose that $\QG$ is a compact quantum group of Kac type. For a linear functional $\varphi$ on $\Pol(\QG)$, we define the Fourier transform $\hat{\varphi}=(\hat{\varphi}(\pi))_{\pi \in \Irr(\QG)}$, which is an element of $\ell^\infty(\hat{\QG}) \ov{\mathrm{def}}{=} \oplus_{\pi} \M_{n_\pi}$, by 
$$
\hat{\varphi}(\pi)
\ov{\mathrm{def}}{=} (\varphi \ot \Id)\big((u^{(\pi)})^{*}\big),\quad \pi \in \Irr(\QG).
$$
In particular, any $x \in \L^{\infty}(\QG)$ induces a linear functional $h(\cdot\,x) \co \Pol(\QG) \to \C$, $y \mapsto h(yx)$ on $\Pol(\QG)$ and the Fourier transform $\hat{x}=(\hat{x}(\pi))_{\pi \in \Irr(\QG)}$ of $x$ is defined by 
$$
\hat{x}(\pi)
\ov{\mathrm{def}}{=} (h(\cdot x) \ot \Id)\big((u^{(\pi)})^{*}\big)
,\quad\pi \in \Irr(\QG).
$$
Let $\varphi_{1},\varphi_{2}$ be linear functionals on $\Pol(\QG)$. Consider their convolution product $\varphi_{1} * \varphi_{2} \ov{\mathrm{def}}{=} (\varphi_{1} \ot \varphi_{2}) \circ \Delta$. 
We recall that
\begin{equation}
\label{conv-L1-et-transfo-fourier}
\widehat{\varphi_{1} * \varphi_{2}}(\pi)
=\hat{\varphi}_{2}(\pi)\hat{\varphi}_{1}(\pi).
\end{equation}

The following result is elementary and well-known.

\begin{prop}
\label{prop-L1-into-c0}
The Fourier transform $\mathcal{F} \co \L^{\infty}(\QG)^{*} \to \ell^{\infty}(\hat{\QG})$, $f \mapsto \hat{f}$ is a contraction, and moreover $\mathcal{F}$ sends $\L^1(\QG)$ injectively into $\ell^{\infty}(\hat{\QG})$.
\end{prop}


\begin{example} \normalfont
\label{Ex-quantum-group-classical}
Let $G$ be a locally compact group and define
\begin{equation}
\label{Delta_commutatif}
\Delta(f)(s,t)
\ov{\mathrm{def}}{=} f(st),\quad f \in \mathrm{C}(G),\quad s,t\in G.
\end{equation} 
Then $\Delta$ admits a weak* continuous extension $\Delta \co \L^\infty(G) \to \L^\infty(G) \otvn \L^\infty(G)$. If $\mu_G$ is a left Haar measure on $G$, then a left Haar weight $\varphi$ on the abelian von Neumann algebra $\L^\infty(G)$ is given by $\varphi(f)=\int_G f \d \mu_G$ and similarly we obtain a right Haar weight $\psi$ with a right Haar measure. In this situation, $\L^\infty(G)$ equipped with $\Delta$, $\varphi$ and $\psi$ is a locally compact quantum group which is compact if and only if the group $G$ is compact. The antipode $S$ is given by $(Sf)(s)=f(s^{-1})$ where $s \in G$. If $G$ is compact, for any $f \in \L^\infty(G)$, we can show that 
$$
\hat{f}(\pi)
=\int_G \pi(s)^*f(s) \d\mu_G(s),\quad \pi \in \Irr(G).
$$ 
We direct the reader to \cite[Section 13.2]{Tus22} for a more complete presentation.
\end{example}

\begin{example} \normalfont
\label{example-QG}
Let $G$ be a locally compact group with neutral element $e$ and $\VN(G)$ its associated von Neumann algebra generated by the left translation operators $\{ \lambda_s : s \in G\}$ on $\L^2(G)$. The ``dual'' $\QG=\hat{G}$ of $G$ is a locally compact quantum group defined by $\VN(G)$ equipped with the coproduct $\Delta \co \VN(G) \to \VN(G) \otvn \VN(G)$ defined by
\begin{equation}
\label{coproduct-VNG}
\Delta(\lambda_s)
=\lambda_s \ot \lambda_s, \quad s \in G.
\end{equation}
The left and right Haar weights of $\hat{G}$ are both equal to the Plancherel
weight of $G$ \cite[Section VII.3]{Tak03}.

In the case where $G$ is discrete, this quantum group is compact and the Haar state is defined as the unique trace $\tau$ on $\VN(G)$ such that $\tau(1)=1$ and $\tau(\lambda_s)=0$ for $s \in G \setminus \{e\}$. For any $f \in \VN(G)$, we have
$$
\hat f(s)
= \tau(f\lambda(s)^*),\quad s \in G.
$$
The antipode $R \co \VN(G) \to \VN(G)$ is defined by $R(\lambda_s)=\lambda_{s^{-1}}$, where $s \in G$. See \cite[Section 13.2]{Tus22} for more information.
\end{example}

\begin{example} \normalfont
Consider the case of a compact abelian group $G$. For any character $\chi$ of $G$, we have
\begin{equation}
\label{delta-epsi}
\Delta(\chi)
=\chi \ot \chi, \quad \chi \in  \hat{G}.
\end{equation}
\end{example}

\begin{example} \normalfont
For $\SU_q(2)$, we refer to the paper \cite{KrS16}. Recall that the compact quantum group $\SU_q(2)$ is Kac if and only if $q=-1$ or $q=1$ and coamenable if $q \not=0$. See also \cite{VaV03} for other examples of Kac algebras.
\end{example}

\paragraph{Operator spaces} 
We refer to the books \cite{BLM04}, \cite{EfR00}, \cite{Pau02} and \cite{Pis03} for background on operators spaces. We denote by $\ot_{\min}$ and $\widehat{\ot}$ the injective and the projective tensor product of operator spaces. If $E$ and $F$ are operator spaces, we let $\CB(E,F)$ for the space of all completely bounded maps endowed with the norm
$$
\norm{T}_{\cb, E \to F}
=\sup_{n \geq 1} \bnorm{\Id_{\M_n} \ot T}_{\M_n(E) \to \M_n(F)}.
$$
If $E$ and $F$ are operator spaces, \cite[(1.32)]{BLM04} we have a complete isometry
\begin{equation}
\label{belle-injection-2}
E \ot_{\min} F 
\subset \CB(F^*,E)
\end{equation}
and each element of $x \ot y$ of $E \ot_{\min} F$ identifies to a weak*-weak continuous map (see \cite[p.~48]{DeF93}. Moreover, recall the complete isometry of \cite[p.~27]{BLM04}
\begin{equation}
\label{belle-injection}
E^* \ot_{\min} F 
\subset \CB(E,F).
\end{equation}
We have
\begin{equation}
\label{otp-duality}
(E \otp F)^*
=\CB(E,F^*).
\end{equation}
Let $\cal{M}$ and $\cal{N}$ be von Neumann algebras. Recall that by \cite[Theorem 2.5.2]{Pis03} \cite[p.~40]{BLM04} the map
\begin{equation}
\label{Ident-magic}
\begin{array}{cccc}
    \Psi  \co &  \cal{M} \otvn \cal{N} &  \longrightarrow   & \CB(\cal{M}_*,\cal{N})   \\
           &   z  &  \longmapsto       &  f \mapsto (f \ot \Id)(z) \\
\end{array}
\end{equation}
is a completely isometric isomorphism. Let $X$ be a Banach space and $E$ be an operator space. For any bounded linear map $T \co X \to E$, we have by \cite[(1.12)]{BLM04}
\begin{equation}
\label{max-et-cb}
\norm{T}_{\cb,\max(X) \to E}
=\norm{T}_{X \to E}. 
\end{equation}
For any bounded linear map $T \co E \to X$, we have by \cite[(1.10)]{BLM04}
\begin{equation}
\label{min-et-cb}
\norm{T}_{\cb,E \to \min(X)}
=\norm{T}_{E \to X}. 
\end{equation}
Recall that by \cite[Exercise 3.10 (ii)]{Pau02} if $T \co E \to F$ is a bounded map between operator spaces we have 
\begin{equation}
\label{estimate-in-d}
\norm{\Id_{\M_d} \ot T}_{\M_d(E) \to \M_d(F)} 
\leq d\norm{T}_{E \to F}.
\end{equation}

\paragraph{Stein's interpolation theorem}
We recall an <<abstract version>> of Stein's interpolation theorem for compatible couples of Banach spaces. See \cite[Theorem 1]{CwJ84}, \cite[Theorem 2.7 p.~52]{Lun18} and \cite[Theorem 2.1]{Voi92} for the proof. Here, we use the open strip $S \ov{\mathrm{def}}{=} \{z \in \mathbb{C} : 0 < \Re z < 1 \}$.

\begin{thm}
\label{thm-Stein}
Let $(X_0,X_1)$ and $(Y_0,Y_1)$ be two compatible couples of Banach spaces such that $X_0$ is separable. Let $(T_z)_{z \in \ovl{S}}$ be a family of bounded operators from $X_0 \cap X_1$ into $Y_0+Y_1$ such that 
\begin{enumerate}
	
\item for any $x \in X_0 \cap X_1$ and any $y \in (Y_0+Y_1)^*$ the function $z \mapsto \la T_z(x),y\ra_{Y_0+Y_1,(Y_0+Y_1)^*}$ is continuous and bounded on $\ovl{S}$ and analytic on $S$,
	
\item for any $x \in X_0 \cap X_1$, the functions $t \to T_{\i t}(x)$ and $t \mapsto T_{1+\i t}(x)$ take values in $Y_0$ and $Y_1$,
	
\item there exist positive constants $M_0$ and $M_1$ such that for any $x \in X_0 \cap X_1$
$$
\sup_{t \in \R} \norm{T_{\i t}(x)}_{Y_0}
\leq M_0 \norm{x}_{X_0}
\quad \text{and} \quad
\sup_{t \in \R} \norm{T_{1+\i t}(x)}_{Y_1}
\leq M_0 \norm{x}_{X_1}.
$$
\end{enumerate}
Then for any $ 0< \theta < 1$, $T_\theta$ admits a unique extension as a bounded linear map from the space $(X_0,X_1)_\theta$ into the space $(Y_0,Y_1)_\theta$ and
$$
\norm{T_\theta}_{(X_0,X_1)_\theta \to (Y_0,Y_1)_\theta}
\leq M_0^{1-\theta}M_1^{\theta}.
$$
\end{thm}

\paragraph{Vector-valued noncommutative $\L^p$-spaces} 
The theory of vector-valued noncommutative $\L^p$-spaces was initiated by Pisier \cite{Pis98} for the case where the underlying von Neumann algebra $\cal{M}$ is \textit{hyperfinite} and equipped with a normal semifinite faithful trace. Under these assumptions, according to \cite[pp.~37-38]{Pis98}, for any operator space $E$, the operator spaces $\cal{M} \ot_{\min} E$ and $\cal{M}_*^{\op} \widehat{\ot} E$ can be injected into a common topological vector space. This compatibility in the sense of interpolation theory, explained in \cite[p.~37]{Pis98} and \cite[p.~139]{Pis03} relies heavily on the fact that the von Neumann algebra $\cal{M}$ is \textit{hyperfinite}. Suppose $1 \leq p \leq \infty$. Then we can define by complex interpolation the operator space
\begin{equation}
\label{Def-vector-valued-Lp-non-com}
\L^p(\cal{M},E)
\ov{\mathrm{def}}{=} \big(\cal{M} \ot_{\min} E, \cal{M}_*^{\op} \widehat{\ot} E\big)_{\frac{1}{p}}.
\end{equation}
When $E=\mathbb{C}$, we get the noncommutative $\L^p$-space $\L^p(\cal{M})$ defined by \eqref{Def-norm-Lp}.

Let $\cal{M}$ be a hyperfinite von Neumann algebra. Let $E$ and $F$ be operator spaces. If a map $T \co E \to F$ is completely bounded then by \cite[(3.1) p.~39]{Pis98} it induces a completely bounded map $\Id_{\L^p(\cal{M})} \ot T \co \L^p(\cal{M},E) \to \L^p(\cal{M},F)$ and we have 
\begin{equation}
\label{ine-tensorisation-os}
\norm{\Id_{\L^p(\cal{M})} \ot T}_{\cb, \L^p(\cal{M},E) \to \L^p(\cal{M},F)} 
\leq \norm{T}_{\cb,E \to F}.
\end{equation}
If $\cal{M}$ is hyperfinite, we will use  sometimes the following notation of \cite[Definition 3.1]{Pis98}\footnote{\thefootnote. It would be more correct to use the notation $\L^\infty_0(\cal{M},E)$}:
\begin{equation}
\label{Def-L0-infty}
\L^\infty(\cal{M},E)
\ov{\mathrm{def}}{=} \cal{M} \ot_{\min} E.
\end{equation}

Let $\cal{M}$ and $\cal{N}$ be hyperfinite von Neumann algebras equipped with normal semifinite faithful traces. Suppose $1 \leq q \leq p \leq \infty$. We have a completely bounded map $\flip \co \L^q(\cal{M},\L^p(\cal{N})) \to \L^p(\cal{N},\L^q(\cal{M}))$, $x \ot y \mapsto y \ot x$ with
\begin{equation}
\label{flip}
\norm{\flip}_{\cb, \L^q(\cal{M},\L^p(\cal{N})) \to \L^p(\cal{N},\L^q(\cal{M}))} 
\leq 1.
\end{equation}

If $\cal{M}$ and $\cal{N}$ are von Neumann algebras equipped with normal semifinite faithful traces and if $1 \leq q \leq p \leq \infty$ and $\frac{1}{q}=\frac{1}{p}+\frac{1}{r}$ , we introduce the norm on $\L^p(\cal{M}) \ot \L^q(\cal{N})$ \cite[p.~71]{JuP10}
\begin{equation}
\label{Norm-LpLq}
\norm{x}_{\L^p(\cal{M},\L^q(\cal{N}))}
=\sup_{\norm{a}_{\L^{2r}(\cal{M})},\norm{b}_{\L^{2r}(\cal{M})} \leq 1} \bnorm{(a \ot 1) x (b \ot 1)}_{\L^q(\cal{M} \otvn \cal{N})}.
\end{equation}
If $\cal{M}$ is hyperfinite, this norm coincide with the norm of the space $\L^p(\cal{M},\L^q(\cal{N}))$ defined by \eqref{Def-vector-valued-Lp-non-com} with $E=\L^q(\cal{N})$. If $x \geq 0$, we can take the infimum with $a=b$. In particular, if $1 \leq p \leq \infty$, we have
\begin{equation}
\label{LinftyLq-norms}
\norm{x}_{\L^\infty(\cal{M},\L^p(\cal{N}))}
=\sup_{\norm{a}_{\L^{2p}(\cal{M})},\norm{b}_{\L^{2p}(\cal{M})} \leq 1} \bnorm{(a \ot 1) x (b \ot 1)}_{\L^p(\cal{M} \otvn \cal{N})}.
\end{equation}

By \cite[lemma 3.11]{GJL20} and \cite[Lemma 4.9]{JuP10}, if $\E= \Id \ot \tau \co \cal{M} \otvn \cal{N} \to \cal{M}$ is the canonical trace preserving normal conditional expectation, we have
\begin{equation}
\label{formula-Junge}
\norm{x}_{\L^\infty(\cal{M},\L^1(\cal{N}))}
=\inf_{x=x_1x_2} \bnorm{\E(x_1x_1^*)}_{\L^\infty(\cal{M})}^{\frac{1}{2}}\bnorm{\E(x_2^*x_2)}_{\L^\infty(\cal{M})}^{\frac{1}{2}}.
\end{equation}


Moreover, if $\cal{N}$ is another approximately finite-dimensional von Neumann algebra equipped with a normal semifinite faithful trace, recall Fubini's theorem \cite[(3.6) p.~40]{Pis98}, which states that for any $1 \leq p<\infty$, there exists a canonical completely isometric isomorphism
\begin{equation}
\label{Fubini}
\L^p(\cal{M},\L^p(\cal{N}))
=\L^p(\cal{M} \otvn \cal{N}),
\end{equation}
where the von Neumann algebra $\cal{M} \otvn \cal{N}$ is endowed with the product of traces.

\paragraph{Tensorizations of linear maps} 
Let $E$ and $F$ be operator spaces. Suppose $1 \leq p \leq \infty$. By \cite[Corollary 1.2 p.~19]{Pis98}, a linear map $T \co E \to F$ is completely bounded if and only if the linear map $\Id_{S^p} \ot T$ extends to a bounded operator $\Id_{S^p} \ot T \co S^p(E) \to S^p(F)$. In this case, the completely bounded norm $\norm{T}_{\cb,E \to F}$ is given by
\begin{equation}
\label{defnormecb}
\norm{T}_{\cb,E \to F}
=\norm{\Id_{S^p} \ot T}_{S^p(E) \to S^p(F)}.
\end{equation}
Let $\cal{M}$ be an approximately finite-dimensional von Neumann algebra equipped with a normal semifinite faithful trace. If a linear map $T \co E \to F$ is completely bounded then by \cite[(3.1) p.~39]{Pis98} it induces a completely bounded map $\Id_{\L^p(\cal{M})} \ot T \co \L^p(\cal{M},E) \to \L^p(\cal{M},F)$ and we have 
\begin{equation}
\label{ine-tensorisation-os}
\norm{\Id_{\L^p(\cal{M})} \ot T}_{\cb, \L^p(\cal{M},E) \to \L^p(\cal{M},F)} 
\leq \norm{T}_{\cb,E \to F}.
\end{equation}

\paragraph{Multipliers and bimodules}

Let $A$ be a Banach algebra and let $X$ be a Banach space equipped with a structure of $A$-bimodule. We say that a pair of maps $(L,R)$ from $A$ into $X$ is a multiplier if
\begin{equation}
\label{multiplier-bimodule}
aL(b)
=R(a)b,\quad a,b \in A.
\end{equation}
We say that $(L,R)$ is a \textit{bounded} multiplier if $L$ and $R$ are bounded. In this case, we let
\begin{equation}
\label{}
\norm{(L,R)}
\ov{\mathrm{def}}{=} \max \big\{\norm{L}_{A \to X}, \norm{R}_{A \to X}\big\}.
\end{equation}
If $X$ is $A$-right module, a left multiplier is a map $L \co A \to X$ such that 
\begin{equation}
\label{left-multiplier-bimodule}
L(ab)
=L(a)b
,\quad a,b \in A.
\end{equation}
Similarly, if $X$ is $A$-left module, a right multiplier is a map $R \co A \to X$ such that
\begin{equation}
\label{right-multiplier-bimodule}
R(ab)
=aR(b)
,\quad a,b \in A.
\end{equation}
We refer to \cite{Daw10} for more information.

\paragraph{Matrix normed modules}
Let $A$ be a Banach algebra equipped with an operator space structure. We need to consider the left $A$-modules $X$ that correspond to completely contractive homomorphisms $A \to \CB(X)$ for an operator space $X$. Following \cite[(3.1.4)]{BLM04}, we call these matrix normed left $A$-modules. It is easy to see that these are exactly the left $A$-modules $X$ satisfying for all $m, n \in \mathbb{N}$, any $[a_{ij}] \in \M_n(A)$ and any $[x_{kl}] \in \M_m(X)$
\begin{equation}
\label{Inequality-module}
\bnorm{[a_{ij}x_{kl}]}_{\M_{nm}(X)}
\leq \bnorm{[a_{ij}]}_{\M_n(A)} \bnorm{[x_{kl}]}_{\M_m(X)}. 
\end{equation}
A reformulation of \eqref{Inequality-module} is that the module action on $X$ extends to a complete contraction $A \otp X \to X$ where $\otp$ denotes the operator space projective tensor product. Similar results hold for matrix normed right modules, e.g. the right analogue of \eqref{Inequality-module} is 
\begin{equation}
\label{Inequality-module-right}
\bnorm{[x_{kl}a_{ij}]}_{\M_{nm}(X)}
\leq \bnorm{[a_{ij}]}_{\M_n(A)} \bnorm{[x_{kl}]}_{\M_m(X)}. 
\end{equation}
If $X$ is a matrix normed left $A$-module, then by \cite[3.1.5 (2)]{BLM04} the dual $X^*$ is a matrix normed \textit{right} $A$-module for the dual action defined by
\begin{equation}
\label{dual-action}
\langle \varphi a, x \rangle_{X^*,X}
\ov{\mathrm{def}}{=} \langle \varphi , a x\rangle_{X^*,X}, \quad \varphi \in X^*, a \in A, x \in X.
\end{equation}
There exists a similar result for matrix normed right $A$-modules.

\paragraph{Entropy}
Let $g \co \R^+ \to \R$ be the function defined by $g(0) \ov{\mathrm{def}}{=} 0$ and $g(t) \ov{\mathrm{def}}{=} t\ln(t)$ if $t > 0$. Let $\cal{M}$ be a von Neumann algebra with a normal faithful \textit{finite} trace $\tau$. If $\rho$ is positive element of $\L^1(\cal{M})$, by a slight abuse, we use the notation $\rho\log \rho \ov{\mathrm{def}}{=} g(\rho)$. In this case, if in addition $\norm{\rho}_{\L^1(\cal{M})}=1$ we can define the Segal entropy 
\begin{equation}
\label{Segal-entropy}
\H(x) \ov{\mathrm{def}}{=} -\tau(x \log_2 x),
\end{equation}
introduced in \cite{Seg60}. If we need to specify the trace $\tau$, we will use the notation $\H_\tau(\rho)$. Note that if the trace $\tau$ is \textit{normalized} and if $x$ is a positive element in the space $\L^1(\cal{M})$ with $\tau(x)=1$ then by \cite[Proposition 2.2]{LoW22} the entropy $\H(x)$ belongs to $[-\infty,0]$ with
\begin{equation}
\label{carac-entropy-max}
\H(x)=0 
\text{ if and only if } x=1.
\end{equation} Recall that by \cite[p.~78]{OcS78}, if $\rho_1, \rho_2$ are positive elements of $\L^1(\cal{M})$ with norm 1 (i.e.~states), we have
\begin{equation}
\label{entropy-tensor-product}
\H(\rho_1 \ot \rho_2)
= \H(\rho_1)+\H(\rho_2).
\end{equation}
Recall that for any $t>0$ and any state $\rho \in \L^1(\cal{M},\tau)$ and any state $\rho_1 \in \L^1(\cal{M},t\tau)$, we have
\begin{equation}
\label{changement-de-trace}
\H_\tau(\rho)
=\H_{t\tau}\bigg(\frac{1}{t}\rho\bigg)-\log t
\quad \text{and} \quad
\H_{\tau}(t\rho_1)
=\H_{t\tau}(\rho_1)-\log t.
\end{equation}

\paragraph{Elementary results}
Let $f \co [0,+\infty) \to \R^+$ be a function with $\lim_{q \to \infty} f(q)=1$. By \cite[Lemma 2.6]{JuP15}, $\frac{\d}{\d p} f(p^*) |_{p=1}$ exists if and only if $\lim_{q \to \infty} q\ln f(q)$ exists. In this case, we have\footnote{\thefootnote. Recall that $p^*=\frac{p}{p-1}$.} 
\begin{equation}
\label{elem-lim}
\frac{\d}{\d p} f(p^*) |_{p=1}
=\lim_{q \to \infty} q\ln f(q).
\end{equation}
Consider a net $(f_\alpha)$ of continuous functions $f_\alpha \co K \to \R$ on a compact topological space $K$. Suppose that $(f_\alpha)$ converges uniformly. Then it is easy to check that
\begin{equation}
\label{minimum-uniform-convergence}
\lim_{\alpha} \min_{x \in K} f_\alpha(x) 
=\min_{x \in K} \lim_{\alpha} f_\alpha(x).
\end{equation}

\begin{prop}
\label{prop-existence-conditional-expectation}
Let $\cal{M}$ be a von Neumann algebra equipped with a normal finite faithful trace $\tau$. For any unital von Neumann subalgebra $\cal{N}$ of $\cal{M}$, there exists a unique faithful normal conditional expectation $\E \co \cal{M} \to \cal{M}$ on $\cal{N}$ such that $\tau \circ \E = \tau$.
\end{prop}
In this situation, it is well-known \cite[Theorem 7 p.~151]{Kad04} that 
\begin{equation}
\label{Esp-dual}
\tau(x\E(y))
=\tau(xy), \quad x \in \cal{N}, y \in \cal{M}.
\end{equation}

\section{Quantum hypergroups and subfactor planar algebras}
\label{sec-Frobenius}
\subsection{Definition of quantum hypergroups}


Recall that a hypergroup \cite{BlH95} \cite{Jew75} is a locally compact space $X$ equipped with a group-like structure which allows the bounded measures on $X$ to convolve in a similar way to convolution on a locally compact group. Note that the authors of \cite{KPC10} use axioms which are different from the ones of \cite{BlH95}.

The idea to define a compact quantum hypergroup by a unital $\C^*$-algebra equipped with a coassociative \textit{completely positive} coproduct instead of a coassociative $*$-homomorphism and some axioms is contained in the paper \cite{ChV99} (see also \cite{DVD11a} and \cite{DVD11b} for a related approach). Here, we use a notion of quantum hypergroup for von Neumann algebras which is a very slight variation of \cite[Definition 2.1]{HuT22} where we use complete positivity of the coproduct. Some axioms are rather different of the ones of \cite{ChV99}. 


The importance of this notion is that there are examples arising from subfactor theory. Indeed, one of the goals of the paper is to connect these structures to the exact computation of minimum output entropies and classical capacities of quantum channels on finite-dimensional algebras. The link between this notion and the one of \cite{ChV99} (and also the one of \cite{DVD11a}) remains to be clarified at the present moment. We refer also to \cite{DFW21}, \cite{FrS09}, \cite{Kal01}, \cite{Wan13} and \cite{Zha20} for other papers on quantum hypergroups. Finally, that the notion of \textit{classical} hypergroups already appears in subfactor theory, see e.g. \cite{BDG21}.  
 
\begin{defi}
\label{def-comproduct}
Let $\mathcal{M}$ be a von Neumann algebra. A normalized coproduct is a normal unital completely positive map $\Delta \co \cal{M} \to \cal{M} \otvn \cal{M}$ such that $(\Delta \ot \Id)\Delta=(\Id \ot \Delta)\Delta$ (co-associativity). A coproduct is a multiple $k\Delta$ of a normalized coproduct $\Delta$ with $k \in \R^+$. Sometimes, we need to take account the norm of $\Delta$ and we say that $\Delta$ is a $\norm{\Delta}$-coproduct.
\end{defi}

By \cite[Theorem 7.2.4]{EfR00} we have $(\cal{M} \otvn \cal{M})_*=\cal{M}_* \otp\cal{M} _*$ where $\otp$ denotes the operator space projective tensor product. Since $\Delta$ is completely bounded, the preadjoint map $\Delta_* \co \cal{M}_* \otp \cal{M}_* \to \cal{M}_*$ is also completely bounded with the same completely bounded norm. We will use the following notation
\begin{equation}
\label{Convolution-L1}
\omega_1 * \omega_2
\ov{\mathrm{def}}{=} \Delta_*(\omega_1 \ot \omega_2) 
= (\omega_1 \ot \omega_2)\circ \Delta
,\quad  \omega_1, \omega_2 \in \cal{M}_*.
\end{equation}
This gives us a completely bounded Banach algebra structure on the predual $\cal{M}_*$ since we have the inequality
\begin{align}
\label{Young inequality-1}
\norm{\omega_1 * \omega_2}_{\cal{M}_*}
\leq \norm{\Delta}\norm{\omega_1}_{\cal{M}_*} \norm{\omega_2}_{\cal{M}_*},\quad \omega_1 \in \cal{M}_*, \omega_2 \in \cal{M}_*.
\end{align} 

In particular, we have
\begin{align}
\label{def-convolution-1}
\langle \omega_1 * \omega_2, z \rangle_{\cal{M}_*, \cal{M}}
=\big\langle \omega_1 \ot \omega_2,\Delta(x)\big\rangle_{\cal{M}_* \otp \cal{M}_*,\cal{M} \otvn \cal{M}},\quad \omega_1 \in \cal{M}_*, \omega_2 \in \cal{M}_*, z \in \cal{M}.
\end{align}

With the property $(\Delta \ot \Id)\Delta=(\Id \ot \Delta)\Delta$, an elementary computation\footnote{\thefootnote. We have $
[((\omega_1 \ot \omega_2) \circ \Delta)  \ot \omega_3] \circ \Delta
=[\omega_1 \ot ((\omega_2 \ot \omega_3) \circ \Delta)] \circ \Delta
$.} shows that $(\omega_1 * \omega_2) * \omega_3=\omega_1 * (\omega_2* \omega_3)$ for any $\omega_1,\omega_2,\omega_3 \in \cal{M}_*$. Indeed, another small computation shows that the associativity of the convolution product is equivalent to the coassociativity.

Let $\cal{M}$ be a von Neumann algebra equipped with a normal semifinite faithful trace $\tau$. trace-preserving involutive $*$-antiautomorphism $R \co \cal{M} \to \cal{M}$. Recall that such a map $R$ is normal and induces an isometry $R \co \L^1(\cal{M}) \to \L^1(\cal{M})$. We identify the Banach space $\L^1(\cal{M})$ with the predual $\cal{M}_*$ with the isometry $i \co \L^1(\cal{M}) \to \cal{M}_*$, $f \mapsto \tau(f\,R(\cdot))$ and we equip $\L^1(\cal{M})$ with the induced operator space structure. This means that the map $i$ is completely isometric and that the bracket duality is
\begin{equation}
\label{bracket-1}
\langle f,h\rangle_{\L^1(\cal{M}),\L^\infty(\cal{M})}
\ov{\mathrm{def}}{=} \tau(fR(h)), \quad f \in \L^1(\cal{M}), h \in \L^\infty(\cal{M}).
\end{equation}
Similarly, we let
\begin{equation}
\label{bracket-1-ot}
\langle f,g \rangle_{\L^1(\cal{M} \otvn \cal{M}),\cal{M} \otvn \cal{M}}
\ov{\mathrm{def}}{=} (\tau \ot \tau)\big[f(R\ot R)(h)\big], \quad f \in \L^1(\cal{M} \otvn \cal{M}), g \in \cal{M} \otvn \cal{M}.
\end{equation}
The use of the antipode $R$ allows us to avoid the opposite operator space structure and gives the beautiful formula \eqref{Frob-recip-bis}. See \cite[p.~139]{Pis03} for a discussion of the opposite structure in the context of noncommutative $\L^p$-spaces. For any $f,g \in \L^1(\cal{M})$, we let 
\begin{equation}
\label{convol-1-1}
f *_{1,1} g
\ov{\mathrm{def}}{=} i^{-1}(i(g)*i(f)).
\end{equation}

A simple computation shows that \eqref{def-convolution-1} translates into the following formula. 

\begin{prop}
\label{prop-convl-bis-65}
For any $h,f \in \L^1(\QG)$ and any $g \in \L^\infty(\QG)$, we have
\begin{equation}
\label{def-convolution-2}
\langle f *_{1,1} g, h \rangle_{\L^1(\QG), \L^\infty(\QG)}
=\big\langle g \ot f,\Delta(h)\big\rangle_{\L^1(\QG) \otp \L^1(\cal{M}),\L^\infty(\QG) \otvn \L^\infty(\QG)}
,
\end{equation}
where the right bracket is defined as in \eqref{bracket-1} with $R \ot R$.
\end{prop}

\begin{proof}
We have
\begin{align*}
\MoveEqLeft
\la f *_{1,1} g,h\ra_{\L^1(\QG),\L^\infty(\QG)}
=\la i^{-1}(i(g)*i(f)),h\ra_{\L^1(\QG),\L^\infty(\QG)}
=\la i(g)*i(f),h\ra_{\L^\infty(\QG)_*,\L^\infty(\QG)} \\
&\ov{\eqref{def-convolution-1}}{=} \la i(g) \ot i(f), \Delta(h) \ra_{\L^\infty(\QG)_* \otp \L^\infty(\QG)_*,\L^\infty(\QG) \otvn \L^\infty(\QG)} 
=(\tau \ot \tau)\big[ (R(g) \ot R(f))\Delta(h) \big] \\
&=(\tau \ot \tau)\big[ (R \ot R)(g \ot f) \Delta(h) \big]
=\la g \ot f,\Delta(h) \ra_{\L^1(\QG \otvn \QG)), \L^\infty(\QG)\otvn \L^\infty(\QG)}.
\end{align*}

\end{proof}


\begin{example} \normalfont
Let $G$ be a locally compact group equipped with a left Haar measure $\mu_G$. In the case of the locally compact quantum group defined by \eqref{Delta_commutatif}, it is easy to check that we obtain for almost all $t \in G$ the classical convolution\footnote{\thefootnote. We have
\begin{align*}
\MoveEqLeft            
\int_{G \times G} g(s)f(t) h(s^{-1}t^{-1}) \d \mu_G(s) \d \mu_G(t)
=\int_{G \times G} g(s)f(t) h\big((ts)^{-1}\big) \d \mu_G(s) \d \mu_G(t) \\
&=\int_G\bigg(\int_G f(t)g(t^{-1}s) \d \mu_G(t)\bigg) h(s^{-1})\d \mu_G(s)
=\int_G (f*g)(s) h(s^{-1})\d \mu_G(s).
\end{align*}
}
\begin{equation}
\label{Convol-usuel}
(f * g)(s)
=\int_G f(t)g(t^{-1}s) \d \mu_G(t), \quad f,g \in \L^1(G).
\end{equation}
\end{example}

\begin{defi}
\label{def-antipode}
Let $\mathcal{M}$ be a von Neumann algebra equipped with a normal semifinite faithful trace $\tau$ and a coproduct $\Delta$. A trace-preserving involutive $*$-antiautomorphism $R \co \cal{M} \to \cal{M}$ is called an antipode with respect to $\tau$ and $\Delta$ if 
\begin{align}
\label{Frob-recip}
\tau\big[(h *_{1,1} f)R(g)\big]
=\tau\big[(f *_{1,1} g)R(h)\big]
,\quad f,g,h \in \frak{n}_\tau.
\end{align}
\end{defi}
Here $\mathfrak{n}_\tau$ is the left ideal of all $x \in \cal{M}$ such that $\tau(|x|^2)<\infty$. Note that this property is an abstract reformulation of \cite[VIII.35]{Bou04} and \cite[(14.10.9)]{Dieu68} (observe that the use of the antipode simplifies the formulas of these references). Moreover, with \eqref{bracket-1}, the equality \eqref{Frob-recip} is equivalent to
\begin{equation}
\label{Frob-recip-bis}
\langle g,h *_{1,1} f\rangle_{\L^\infty(\cal{M}),\L^1(\cal{M})}
=\langle f *_{1,1} g, h \rangle_{\L^1(\cal{M}),\L^\infty(\cal{M})},\quad f,g,h \in \frak{n}_\tau.
\end{equation}

\begin{defi}
\label{def-Frob-vN}
A quantum hypergroup is a quadruple $(\L^\infty(\QH),\tau,\Delta,R)$ where $\L^\infty(\QH)$ is a von Neumann algebra equipped with a normal semifinite faithful trace $\tau$, a coproduct $\Delta$ and an antipode $R$. We will also use the term $\norm{\Delta}$-quantum hypergroup. If $\norm{\Delta}=1$, we say that $\L^\infty(\QH)$ is normalized.
\end{defi}

We write $\QH$ to denote the quantum hypergroup. Finally, note that\footnote{\thefootnote.We have
\begin{align*}
\tau(x *_{1,1} y)
\ov{\eqref{bracket-1}}{=} \left\langle x *_{1,1} y,1\right\rangle
\ov{\eqref{def-convolution-1}}{=} \big\langle x \ot y,\Delta(1)\big\rangle
=\langle x \ot y,\norm{\Delta}1 \rangle
=\norm{\Delta}\tau(x)\tau(y).
\end{align*}}  
\begin{equation}
\label{relation-trace}
\tau(x *_{1,1} y)
=\norm{\Delta} \tau(x)\tau(y), \quad x \in \L^1(\QH), y \in \L^1(\QH).
\end{equation}  

\begin{remark} \normalfont
We think that the following definition could be an alternative to Definition \ref{def-Frob-vN} for finite von Neumann algebras. Let $\cal{M}$ be a finite von Neumann algebra. We will call $(\cal{M},\tau,\Delta,J)$ a tracial \textit{compact} hypergroup structure on $\cal{M}$ if the following properties are satisfied.
\begin{enumerate}
\item $\Delta \co \cal{M} \to \cal{M} \otvn \cal{M}$ is a normal unital completely positive map such that $(\Delta \ot \Id)\Delta =(\Id \ot \Delta)\Delta$.

\item $J \co \cal{M} \to \cal{M}$ is an involutive anti-linear anti-$*$-homomorphism such that 
$
\Delta \circ J
=\Sigma \circ (J \ot J) \circ \Delta$ where $\Sigma$ is the flip on $\cal{M} \otvn \cal{M}$, i.e.~$\Sigma(x \ot y) = y \ot x$.

\item $\tau$ is a normal finite faithful trace on $\cal{M}$ such that \eqref{Def-haar-state} holds for any $x \in \cal{M}$ satisfying
$$ 
R\big((\Id \ot h)(\Delta(z)(1 \ot y))\big)
=(\Id \ot h)\big((1 \ot z)\Delta(y)\big), \quad x,y \in \cal{M}
$$ 
where $R(x) \ov{\mathrm{def}}{=} J(x)^*$.
\end{enumerate}
It is not clear if we must add the existence of a linear map $\epsi \co \cal{M} \to \mathbb{C}$ is a linear map such that $
(\epsi \ot \Id)\Delta
=(\Id \ot\epsi)\Delta
=\Id$ satisfying $
\epsi(xy)
=\epsi(x) \epsi(y)$ for any $x,y \in \cal{M}$.
\end{remark}


\subsection{Examples of quantum hypergroups}

\begin{example} \normalfont
\label{ex-vNF-algebra-from-subfactor}
Suppose $\mathscr{P}_{\bullet}$  is a subfactor planar algebra with finite index $\delta^2$, $\delta>0$. Recall that the convolution on the 2-box spaces is defined as in Figure 1.

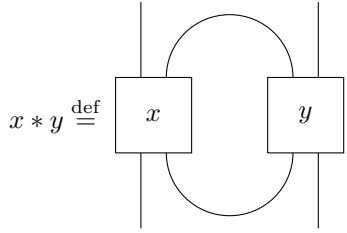
\begin{center}
\raisebox{2cm}{$x*y=$}\hspace{0.2cm}\begin{tikzpicture}[scale=.7]
	\clip (0,0) circle (3cm);
	
	\draw[shaded] (0,0) ellipse (1.2cm and 4cm);
	
	\draw[fill=white] (0,0) ellipse (1cm and 1.5cm);

	\begin{scope}[shift=(10:1cm)]	
		\node at (0,-0.15) [Tbox, inner sep=2mm] {$y$}; 
	\end{scope}
	
	\node at (-35:-2.8cm) {$\star$}; 
	\node at (-2:-1cm) [Tbox, inner sep=2mm] {$x$};
	\node at (0:-1.8cm) {$\star$};
	\node at (0cm:0.1cm) {$\star$};
	
	\draw[very thick] (0,0) circle (3cm); 
\end{tikzpicture}

Figure 1: convolution
\end{center}

We consider the map $\Delta \co \scr{P}_{2,\pm} \to \scr{P}_{2,\pm} \ot \mathscr{P}_{2,\pm}$ defined as the adjoint operator of $*$ with respect to the unnormalized faithful Markov trace $\tr$:
\begin{align}
\label{eq:definition of comultiplication on subfactor}
\left\langle \Delta(z),x \ot y\right\rangle
=\left\langle z,x\ast y\right\rangle,\quad x, y,z \in \scr{P}_{2,\pm}.
\end{align}
By \cite[Theorem 2.4]{HuT22}, the map $\Delta \co \scr{P}_{2,\pm} \to \scr{P}_{2,\pm} \ot \mathscr{P}_{2,\pm}$ is positive. Indeed, the proof shows that this map is \textit{even completely positive} since $\Delta$ is a composition of completely positive maps. Moreover, the same result says that if $\mathscr{P}_{\bullet}$ is irreducible then $\Delta$ is a $\delta^{-1}$-coproduct. The associativity of the convolution map is illustrated by \cite[Figure 22]{BiJ00}.

For any $x \in \scr{P}_{2,\pm}$, let $R(x)$ be the contragredient of $x$. Then by \cite[Lemma 3.4]{JLJ16} $R \co \scr{P}_{2,\pm} \to \scr{P}_{2,\pm}$ is an antipode. The quadruple $(\scr{P}_{2,\pm},\tr,*,R)$ is a $\delta^{-1}$-quantum hypergroup.
\end{example}

\begin{example} \normalfont 
Let $(\cal{A},\cal{B},d,\tau,\cal{F})$ be a fusion bi-algebra \cite{LPW21}. That means that $\cal{A}$ and $\cal{B}$ are two finite-dimensional $\C^*$-algebras equipped with faithful traces $d$ and $\tau$ respectively, that $\cal{A}$ is commutative, and that $\F \co \cal{A} \to \cal{B}$ is a unitary transformation preserving $2$-norms, i.e.~$\tau(|\F(x)|^2)=d(|x|^2)$ for any $x \in \cal{A}$ satisfying the following conditions.
\begin{enumerate}
\item The unitary $\mathcal{F}$ induces a convolution map $* \co \cal{A} \ot \cal{A} \to \cal{A}$, $x \ot y \mapsto x*y$ defined by $x * y \ov{\mathrm{def}}{=} \cal{F}^{-1}(\cal{F}(x)\cal{F}(y))$ such that $x*y \geq 0$ for any positive elements $x,y$ of $\cal{A}$.
\item The map $J \co \cal{A} \to \cal{A}$, $x \mapsto \F^{-1}(\F(x)^*)$ is an anti-linear, $*$-isomorphism on $\cal{A}$.
\item The operator $\F^{-1}(1)$ is a positive multiple of a minimal projection in $\cal{A}$.
\end{enumerate}
In this case, we can introduce a coproduct $\Delta \co \cal{A} \to \cal{A} \ot \cal{A}$ as the adjoint operator of the convolution with respect to the trace $d$. For any $x \in \cal{A}$, let $R(x) \ov{\mathrm{def}}{=} J(x)^*$. The quadruple $(\cal{A},d,\Delta,R)$ is a normalized quantum hypergroup. Indeed, the coproduct $\Delta \co \cal{A} \to \cal{A} \ot \cal{A}$ is positive by \cite[Example 3.2]{HuT22}, hence completely positive since the $\C^*$-algebra $\cal{A}$ is abelian. The co-associativity is obvious. Note that $
\Delta \circ J
=\Sigma \circ (J \ot J) \circ \Delta$ is satisfied.
%
%
%
%
\end{example}

\begin{remark}\normalfont 
We are aware that the spin planar algebra of Remark \ref{Rem-spin-planar-algebra} satisfies a \textit{non-unital} version of Definition \ref{def-comproduct} (by the way the coproduct is a $\sqrt{n}$-coproduct). Note that \eqref{Def-haar-state} is false in this context. Some of the arguments of Section \ref{sec-convolution-by-lp} can be used with the convolution product of this planar algebra.
\end{remark}

\begin{example} \normalfont
A locally compact quantum group of Kac type such that the left Haar weight is tracial defines a normalized quantum hypergroup. It is apparent by a simple computation that the formula \eqref{Antipode-Magic-relation-4} implies \eqref{Frob-recip-bis}.
\end{example}

Recall that a locally compact quantum group $\QG$ is co-amenable if and only if the Banach algebra $\L^1(\QG)$ has a bounded approximate unit with norm not more than 1. See \cite[Theorem 3.1]{BeT03} and \cite{Bra17} for more information. A known class of examples of coamenable locally compact quantum groups is given by finite quantum groups. Let us define a similar notion for quantum hypergroups.

\begin{defi}
We say that a quantum hypergroup $(\L^\infty(\QH),\tau,\Delta,R)$ is co-amenable if and only the Banach algebra $\L^1(\QH)$ has a bounded approximate unit with norm not more than 1.
\end{defi}

\section{Convolutions by elements of $\L^p(\QH)$} 
\label{sec-convolution-by-lp}
\subsection[Convolutions by elements of $\L^p(\QH)$ and c.b. maps from $\L^1(\QH)$ into $\L^p(\QH)$]{Convolutions by elements of $\L^p(\QH)$ and completely bounded maps from $\L^1(\QH)$ into $\L^p(\QH)$} 
\label{Sec-convolution}

If $\QH=(\L^\infty(\QH),\tau,\Delta,R)$ is a (finite) quantum hypergroup and if $1 \leq p \leq \infty$, the authors of \cite[Theorem 3.7]{HuT22} obtained a Young's inequality
\begin{equation}
\label{Convolution-inequality-1-p}
\norm{f*g}_{\L^p(\QH) }
\leq \norm{f}_{\L^1(\QH) }\norm{g}_{\L^p(\QH)}, \quad f \in \L^1(\QH), g \in \L^p(\QH)
\end{equation}
where $*$ is a suitable extension of the previous convolution \eqref{Convolution-L1}. In this section, we prove a completely bounded version of this inequality. 

Let $\QH=(\L^\infty(\QH),\tau,\Delta,R)$ is a (semifinite) quantum hypergroup. Note that \eqref{Convolution-L1} induces a right action and a left action of $\L^\infty(\QH)_*$ on the von Neumann algebra $\L^\infty(\QH)$. Indeed, we can define the right dual action $\star$ of $\L^\infty(\QH)_*$ on $\L^\infty(\QH)$ by
\begin{equation}
\label{dual-action-L1}
\langle g \star \omega_1, \omega_2 \rangle_{\L^\infty(\QH),\L^\infty(\QH)_*}
\ov{\eqref{dual-action}}{=} \langle g, \omega_1*\omega_2 \rangle_{\L^\infty(\QH),\cal{M}_*}, \quad \omega_1,\omega_2 \in \L^\infty(\QH)_*,g \in \L^\infty(\QH)
\end{equation}
and similary the left dual action $\star$ of $\L^\infty(\QH)_*$ on $\L^\infty(\QH)$ by
\begin{equation}
\label{dual-action-L1-bis}
\langle \omega_1 \star g, \omega_2 \rangle_{\L^\infty(\QH),\L^\infty(\QH)_*}
=\langle g, \omega_2*\omega_1 \rangle_{\L^\infty(\QH),\L^\infty(\QH)_*}, \quad \omega_1,\omega_2 \in \L^\infty(\QH)_*,g \in \L^\infty(\QH).
\end{equation}

\begin{lemma}
\label{Lemma-action-concrete}
Let $(\L^\infty(\QH),\tau,\Delta,R)$ be a quantum hypergroup. For any $g \in \L^\infty(\QH)$ and any $\omega \in \L^\infty(\QH)_*$, we have
\begin{equation}
\label{action-module}
g \star \omega
=(\omega \ot \Id) \circ \Delta(g)
\quad \text{and} \quad
\omega \star g
=(\Id \ot \omega) \circ \Delta(g).
\end{equation}
\end{lemma}

\begin{proof}
For any $g \in \L^\infty(\QH)$ and any $\omega,\omega_0 \in \L^\infty(\QH)_*$, we have
\begin{align*}
\MoveEqLeft
\langle g \star \omega, \omega_0 \rangle_{\L^\infty(\QH),\L^\infty(\QH)_*}
\ov{\eqref{dual-action-L1}}{=} \langle g, \omega*\omega_0\rangle_{\L^\infty(\QH),\L^\infty(\QH)_*}  
\ov{\eqref{Convolution-L1}}{=} (\omega \ot \omega_0) \circ \Delta(g) \\
&= \omega_0\big((\omega \ot \Id)\circ \Delta(g)\big) 
=\big\langle (\omega \ot \Id) \circ \Delta(g), \omega_0 \big\rangle_{\L^\infty(\QH),\L^\infty(\QH)_*}
\end{align*} 
and similarly
\begin{align*}
\MoveEqLeft
\langle \omega \star g, \omega_0 \rangle_{\L^\infty(\QH),\L^\infty(\QH)_*}
\ov{\eqref{dual-action-L1-bis}}{=} \langle g , \omega_0 * \omega\rangle_{\L^\infty(\QH),\L^\infty(\QH)_*}            
\ov{\eqref{Convolution-L1}}{=} (\omega_0 \ot \omega) \circ \Delta(g) \\
&=\omega_0\big((\Id \ot \omega)\circ \Delta(g)\big)
=\big\langle (\Id \ot \omega) \circ \Delta(g), \omega_0 \big\rangle_{\L^\infty(\QH),\L^\infty(\QH)_*}.
\end{align*} 
We deduce \eqref{action-module} by duality. 
\end{proof}

For any $f \in \L^1(\cal{M})$ and any $g \in \L^\infty(\QH)$, we define the element $f *_{1,\infty} g$ of $\L^\infty(\QH)$ by
\begin{equation}
\label{1-infty}
f *_{1,\infty} g
\ov{\mathrm{def}}{=} g \star \tau(fR(\cdot)) 
\quad \text{that is} \quad 
f *_{1,\infty} g
\ov{\eqref{action-module}}{=} \big(\tau(fR(\cdot)) \ot \Id \big) \circ \Delta(g).
\end{equation}
For any $f \in \L^\infty(\QH)$ and any $g \in \L^1(\QH)$, we define the element $f *_{\infty,1} g$ of $\L^\infty(\QH)$ by
\begin{equation}
\label{Def-*1-infty}
f *_{\infty,1} g
\ov{\mathrm{def}}{=} \tau(gR(\cdot)) \star f
\quad \text{that is} \quad 
f *_{\infty,1} g
\ov{\eqref{action-module}}{=} \big(\Id \ot \tau(gR(\cdot)) \big) \circ \Delta(f).
\end{equation}

\begin{example} \normalfont
Let $G$ be a locally compact group equipped with a left Haar measure $\mu_G$. In the case of the locally compact quantum group defined by \eqref{Delta_commutatif}, it is easy to check that we obtain for almost all $t \in G$ the classical convolution\footnote{\thefootnote. If $\Delta_G \co G \to \R$ is the modular function and if $\nu_G=\Delta^{-1}\mu_G$ is the associated right Haar measure, the second formula can be proved by 
\begin{align*}
\MoveEqLeft
(f *_{\infty,1} g)(t)            
=\int_G f(ts^{-1})g(s) \d \nu_G(s) 
=\int_G f(ts^{-1})g(s) \Delta_G(s)^{-1}\d \mu_G(s)\\
&=\int_G f(ts)g(s^{-1}) \d \mu_G(s)
=\int_G f(s)g(s^{-1}t) \d \mu_G(s).
\end{align*} 
}
\begin{equation}
\label{Convol-usuel}
(f *_{1,\infty} g)(t)
=\int_G f(s)g(s^{-1}t) \d \mu_G(s), \quad f \in \L^1(G), g \in \L^\infty(G).
\end{equation}
and
$$
(f *_{\infty,1} g)(t)
=\int_G f(s)g(s^{-1}t) \d \mu_G(s), \quad f \in \L^\infty(G), g \in \L^1(G).
$$
\end{example}

Suppose that $\cal{M}$ is a finite von Neumann algebra. Recall that, we have a canonical injective bounded map $j\co \cal{M} \to \L^1(\cal{M})$, $x \mapsto x$. So the pair $(\cal{M},\L^1(\cal{M}))$ of Banach spaces is an interpolation couple. In the next result, we show the compatibility of the convolution maps. 


\begin{prop}
\label{Prop-compatible}
Let $(\L^\infty(\QH),\tau,\Delta,R)$ be a (finite) quantum hypergroup. We have the following commutative diagram.
$$
 \xymatrix @R=1cm @C=3cm{
       \L^1(\QH)  \otp \L^1(\QH) \ar[r]^{*_{1,1}}& \L^1(\QH)  \\
        \L^\infty(\QH) \ot \L^\infty(\QH) \ar[r]_{*_{1,\infty}} \ar@{^{(}->}[u]^{j \ot j} & 
				\L^\infty(\QH) \ar@{^{(}->}[u]_{j} \\
}
$$
\end{prop}

\begin{proof}
For any $f,g,h \in \L^\infty(\QH)$, we have on the one hand
\begin{align*}
\MoveEqLeft
\big\langle j(f *_{1,\infty} g), h \big\rangle_{\L^1(\QH),\L^\infty(\QH)}
= \big\langle f *_{1,\infty} g, h \big\rangle_{\L^1(\QH),\L^\infty(\QH)}  
\ov{\eqref{1-infty}}{=} \big\langle \big(\tau(fR(\cdot) \ot \Id) \big) \circ \Delta(g), h \big\rangle_{\L^1(\QH),\L^\infty(\QH)} \\
&\ov{\eqref{bracket-1}}{=} \tau \big(\big[\big(\tau(fR(\cdot) \ot \Id) \big) \circ \Delta(g)\big)\big]R(h) \big) 
=\big(\tau(R(f)\,\cdot) \ot \tau(R(h)\,\cdot) \big) \circ \Delta(g).
\end{align*}

On the other hand, we have 
\begin{align*}
\MoveEqLeft
\big\langle j(f) *_{1,1} j(g), h \big\rangle_{\L^1(\QH),\L^\infty(\QH)} 
=\big\langle f *_{1,1} g, h \big\rangle_{\L^1(\QH),\L^\infty(\QH)} 
\ov{\eqref{Frob-recip-bis}}{=} \big\langle h *_{1,1} f, g \big\rangle_{\L^1(\QH),\L^\infty(\QH)}\\
&\ov{\eqref{def-convolution-2}}{=} \big\langle f \ot h,\Delta(g)\big\rangle_{\L^1(\QH) \otp \L^1(\QH),\L^\infty(\QH) \otvn \L^\infty(\QH)}
=\big(\tau(R(f)\,\cdot) \ot \tau(R(h)\,\cdot) \big) \circ \Delta(g).
\end{align*}
\end{proof}

\begin{thm}
\label{Th-conv-cb}
Let $(\L^\infty(\QH),\tau,\Delta,R)$ be a normalized quantum hypergroup. Suppose $1 \leq p \leq \infty$. The map $*_{1,\infty}$ induces a completely contractive map $\L^1(\QH) \otp \L^p(\QH) \to \L^p(\QH)$, $f \ot g \mapsto f*g$. The map $*_{\infty,1}$ induces a completely contractive map $\L^p(\QH) \otp \L^1(\QH) \to \L^p(\QH)$, $f \ot g \mapsto f*g$.
\end{thm}

\begin{proof}
We only prove the finite case. The semifinite case is similar using \cite{Ter82} and classical arguments left to the reader. Since the convolution $*$ of \eqref{Convolution-L1} induces a structure of matrix normed left $\L^\infty(\QH)_*$-module on $\L^\infty(\QH)$, the right dual action $\star$ of \eqref{dual-action-L1} induces a structure of matrix normed right $\L^\infty(\QH)_*$-module on $\L^\infty(\QH)$. Moreover, the map $R \co \L^\infty(\QH) \to \L^\infty(\QH)$ is an anti-$*$-automorphism of the von Neumann algebra $\L^\infty(\QH)$.  
For any $[f_{ij}] \in \M_n(\L^1(\QH))$ and any $[g_{kl}] \in \M_m(\L^\infty(\QH))$, we deduce that
\begin{align*}
\MoveEqLeft
\bnorm{[f_{ij} *_{1,\infty} g_{kl}]}_{\M_{nm}(\L^\infty(\QH))}           
\ov{\eqref{1-infty}}{=} \bnorm{[g_{kl} \star \tau(f_{ij}R(\cdot))]}_{\M_{nm}(\L^\infty(\QH))}   \\
&\ov{\eqref{Inequality-module}}{\leq} \bnorm{[\tau(f_{ij}R(\cdot))]}_{\M_n(\L^1(\QH))} \bnorm{[g_{kl}]}_{\M_m(\QH)}
=\bnorm{[f_{ij}]}_{\M_n(\L^1(\QH))} \bnorm{[g_{kl}]}_{\M_m(\QH)}.
\end{align*} 
Note that we have a completely contractive linear map $\Delta_* \co \L^1(\QH) \otp \L^1(\QH) \to \L^1(\QH)$, $f \ot g \mapsto f *_{1,1} g$.
By \cite[p.~126]{EfR00}, for any $[f_{ij}] \in \M_n(\L^1(\QH))$ and any $[g_{kl}] \in \M_m(\L^1(\QH))$ we deduce that 
$$
\bnorm{[f_{ij} * g_{kl}]}_{\M_{nm}(\L^1(\QH))}   
\leq \bnorm{[f_{ij}]}_{\M_n(\L^1(\QH))} \bnorm{[g_{kl}]}_{\M_m(\L^1(\QH))}.
$$
Recall that $\M_n(\L^p(\QH))=(\M_n(\L^\infty(\QH)),\M_n(\L^1(\QH)))_{\frac{1}{p}}$ by \cite[p.~54]{Pis03}. So using Proposition \ref{Prop-compatible},  we obtain by bilinear interpolation \cite[p.~57]{Pis03} the inequality
$$
\bnorm{[f_{ij}*g_{kl}]}_{\M_{nm}(\L^p(\QH))}   
\leq \bnorm{[f_{ij}]}_{\M_n(\L^1(\QH))}\bnorm{[g_{kl}]}_{\M_m(\L^p(\QH))}.
$$
We deduce by \cite[p.~126]{EfR00} a well-defined completely contractive map $\L^1(\QH) \otp \L^p(\QH) \to \L^p(\QH)$, $f \ot g \mapsto f*g$. 

The second part is left to the reader.
\end{proof}

\begin{cor}
\label{Cor-convolution}
Let $(\L^\infty(\QH),\tau,\Delta,R)$ be a normalized quantum hypergroup. Suppose $1 \leq p \leq \infty$. Let $g \in \L^p(\QH)$. The map\footnote{\thefootnote. Here $*$ is an extension of $*_{1,\infty}$.} $T \co \L^1(\QH) \to \L^p(\QH)$, $f \mapsto f * g$ is completely bounded with
\begin{equation}
\label{borne-cb}
\norm{T}_{\cb,\L^1(\QH) \to \L^p(\QH)} 
\leq \norm{g}_{\L^p(\QH)}. 
\end{equation}  
\end{cor}

\begin{proof}
By \cite[Proposition 7.1.4]{EfR00} and \cite[Proposition 7.1.2]{EfR00}, we have an (complete) isometry 
\begin{equation*}
\CB(\L^1(\QH) \otp \L^p(\QH) , \L^p(\QH)) \to \CB(\L^p(\QH) \otp \L^1(\QH), \L^p(\QH)) \to \CB(\L^p(\QH),\CB(\L^1(\QH),\L^p(\QH)))  \\
\end{equation*}
which maps the completely contractive map $\L^1(\QH) \otp \L^p(\QH) \to \L^p(\QH)$, $f \ot g \mapsto f * g$ to the map $\L^p(\QH) \to \CB(\L^1(\QH),\L^p(\QH))$, $g \mapsto (f \mapsto f*g)$. The conclusion is obvious.
\end{proof}

\begin{remark} \normalfont
Let $G$ be a locally compact group with modular function $\Delta_G \co G \to \R$, equipped with a \textit{left} Haar measure. Suppose $1 \leq p,q,r \leq \infty$ with $1+\frac{1}{r}=\frac{1}{q}+\frac{1}{p}$. Here the convolution product $*$ is defined as in \eqref{Convol-usuel}. By \cite[p.~519]{Rob91} (see also \cite[Remark 2.2 p.~183]{KlR78}, we have Young's inequality
\begin{equation}
\label{Young-Lr}
\bnorm{f*\Delta_G^{\frac{1}{q^*}}g}_{\L^r(G)} 
\leq \norm{f}_{\L^q(G)} \bnorm{g}_{\L^p(G)}, \quad f \in \L^q(G), g \in \L^p(G).
\end{equation} 
If $q=1$ then $r=p$ and we recover the inequality $\norm{f*g}_{\L^p(G)} \leq \norm{f}_{\L^1(G)} \norm{g}_{\L^p(G)}$ of \cite[Corollary 20.14 p.~293]{HeR79}. If $p=1$ and if we replace $q$ and $r$ by $p$, we obtain the inequality $\bnorm{f*\Delta_G^{\frac{1}{p^*}}g}_{\L^p(G)} \leq \norm{f}_{\L^p(G)} \norm{g}_{\L^1(G)}$. So the non-unimodularity makes appear a ``twist''. 

Let $X$ be a Banach space. Suppose that $f$ belongs to the Bochner space $\L^p(G,X)$ and that $g\in \L^q(G)$. Using Young's inequality \eqref{Young-Lr} with the functions $\norm{f}_X$ and $|g|$, we see that $\big(\norm{f}_X *\Delta_G^{\frac{1}{q^*}}|g|\big)(t)$ exists almost everywhere, that $\norm{f}_X * \Delta_G^{\frac{1}{q^*}}|g|$ belongs to $\L^r(G)$ and that
\begin{equation}
\label{Inter-3000}
\bnorm{\norm{f}_X*\Delta_G^{\frac{1}{q^*}}|g|}_{\L^r(G)} 
\leq \norm{f}_{\L^q(G,X)} \norm{g}_{\L^p(G)}.
\end{equation}
If $\big(\norm{f}_X * \Delta_G^{\frac{1}{q^*}}|g|\big)(t)$ exists, then clearly $\big(f * \Delta_G^{\frac{1}{q^*}}g\big)(t)$ exists and
\begin{align*}
\MoveEqLeft
\bnorm{(f * \Delta_G^{\frac{1}{q^*}}g)(t)}_X
\ov{\eqref{Convol-usuel}}{=} \norm{\int_G f(s)\Delta_G^{\frac{1}{q^*}}g(s^{-1}t) \d \mu_G(s)}_X \\
&\leq \int_G \norm{f(s)}_X \Delta_G^{\frac{1}{q^*}}|g(s^{-1}t)| \d \mu_G(s)
\ov{\eqref{Convol-usuel}}{=} \big(\norm{f}_X * \Delta_G^{\frac{1}{q^*}} |g|\big)(t).            
\end{align*} 
Raising to power $r$ and integrating on $G$, 
we obtain $\bnorm{f*\Delta_G^{\frac{1}{q^*}}g}_{\L^r(G,X)} \leq \bnorm{\norm{f}_X*\Delta_G^{\frac{1}{q^*}}|g|}_{\L^r(G)} $. Combining with \eqref{Inter-3000}, we finally obtain the vector-valued Young's inequality 
$$
\bnorm{f*\Delta_G^{\frac{1}{q^*}}g}_{\L^r(G,X)} 
\leq \norm{f}_{\L^q(G,X)} \norm{g}_{\L^p(G)}.
$$
In particular, we have $\norm{f*g}_{\L^p(G,X)} \leq \norm{f}_{\L^1(G,X)} \norm{g}_{\L^p(G)}$. We deduce that any function $g \in \L^p(G)$ induces a bounded convolution operator $T \co \L^1(G,X) \to \L^p(G,X)$, $f \mapsto f*g$ satisfying
\begin{equation}
\label{}
\norm{T}_{\L^1(G,X) \to \L^p(G,X)} 
\leq \norm{g}_{\L^p(G)}. 
\end{equation} 
By \cite[p.~86]{DeF93}, this means that a convolution operator $T \co \L^1(G) \to \L^p(G)$ is regular. Actually, each operator $T \co \L^1(\Omega) \to \L^p(\Omega')$ between classical $\L^p$-spaces is regular by \cite[Proposition 2.1.1 p.~68 and p.~156]{HvNVW16} or \cite[Corollary 5.33 p.~207]{AbA02}. Unfortunately, it seems that we cannot obtain the complete boundedness of $T$ from this fact.
\end{remark}

\begin{remark} \normalfont
\label{Rem-Young-pqr}
Suppose $1 \leq p,q,r \leq \infty$ with $\frac{1}{q}+\frac{1}{p}=1+\frac{1}{r}$. We could extend the result \eqref{borne-cb} to the case of a convolution operator $T \co \L^q(\QH) \to \L^r(\QH)$, $f \mapsto f * g$ with $f \in \L^p(\QH)$. Indeed, rewriting a right version of \eqref{borne-cb} by duality with \eqref{Frob-recip-bis} we obtain a completely bounded map $\L^{p^*}(\QH) \to \L^\infty(\QH)$, $f \mapsto f*g$ with completely bounded norm less than $\norm{g}_{\L^p(\QH)}$. We conclude by a suitable interpolation\footnote{\thefootnote. Take $\theta=\frac{p}{q^*}$. We have $\frac{1-\theta}{1}+\frac{\theta}{p^*}=1-\theta(1-\frac{1}{p^*})=\frac{1}{q}$ and $\frac{1-\theta}{p}+\frac{\theta}{\infty}=\frac{1}{p}-\frac{1}{q^*}=\frac{1}{p}-\big(1-\frac{1}{q}\big)=\frac{1}{r}$.} between this map and \eqref{borne-cb} that we have a completely bounded map $\L^q(\QH) \to \L^r(\QH)$, $f \mapsto f*g$ with completely bounded norm less than $\norm{g}_{\L^p(\QH)}$.
However, we does not have interesting application at the present moment. 
\end{remark}

\begin{remark} \normalfont
Of course, we are aware that Theorem \ref{Th-conv-cb} and \ref{Cor-convolution} admits generalizations to some classes of locally compact quantum groups equipped with non-tracial weights. It is apparent that it should be true for a locally compact quantum $\QG$ of Kac type  (although the associated Young's inequality is stated in \cite{LWW17} only for tracial weights), in particular for the locally compact quantum groups defined by \eqref{Delta_commutatif} or \eqref{coproduct-VNG}. A generalization of the previous proof maybe could be work for the more general case of locally compact quantum groups whose scaling group is trivial. We will explore this in a new version of this preprint.

It does not appear possible to drop the assumption ``whose scaling group is trivial''  since it is stated in \cite{LWW17} that the inequality $\norm{f*g}_{\L^\infty(\QG)} \leq \norm{f}_{\L^1(\QG)} \norm{g}_{\L^\infty(\QG)}$ is not true for the compact quantum group $\QG=\SU_q(2)$ with $q \in ]-1,1[-\{0\}$. However, a modified Young's inequality $\bnorm{f*\rho_{\frac{-\i}{p^*}}(g)}_{\L^r(\QG)} \leq \norm{f}_{\L^q(\QG)} \bnorm{g}_{\L^p(\QG)}$ is stated in \cite[Theorem 3.4]{LWW17} for $1 \leq p,q,r \leq 2$  with $\frac{1}{q}+\frac{1}{p}=1+\frac{1}{r}$. A completely bounded version could be useful in quantum information theory since we only need such inequality for $q=1$ and $p=r$ close to 1. Consequently, other investigations will be necessary.  
\end{remark}

\subsection{Bounded versus completely bounded Fourier multipliers from $\L^1(\QH)$ into $\L^p(\QH)$}
\label{Sec-bounded}

Let $(\L^\infty(\QH),\tau,\Delta,R)$ be a normalized quantum hypergroup. Suppose $1 \leq p < \infty$. By Theorem \ref{Th-conv-cb}, we have a structure of $\L^1(\QH)$-bimodule on $\L^p(\QH)$ for any $1 \leq p \leq \infty$. So we can consider the associated multipliers, see the first paragraph of Section \ref{Sec-preliminaries}. We denote by $\frak{M}^{1,p}(\QH)$ the space of multipliers from $\L^p(\QH)$ into $\L^q(\QH)$. In particular, a left  multiplier is a map $T \co \L^1(\QH) \to \L^p(\QH)$ satisfying 
\begin{equation}
\label{Def-left-multiplier-L1}
T(x*y)
=T(x) *y, \quad x,y \in \L^1(\QH).
\end{equation}
We denote by $\frak{M}_l^{1,p}(\QH)$ the space of left bounded multipliers.


Suppose $1 \leq p \leq \infty$. In accordance with \eqref{bracket-1}, we will use the following duality bracket
\begin{equation}
\label{bracket}
\langle x,y\rangle_{\L^p(\QH),\L^{p^*}(\QH)}
\ov{\mathrm{def}}{=} \tau(xR(y))
=\tau(yR(x)), \quad x \in \L^p(\QH), y \in \L^{p^*}(\QH).
\end{equation}


We will use the following lemma which is essentially a consequence of \eqref{Frob-recip-bis}.

\begin{lemma}
\label{Lemma-Dieudonne}
Let $(\L^\infty(\QH),\tau,\Delta,R)$ be a normalized quantum hypergroup. Suppose $1 \leq p < \infty$. If $x \in \L^1(\QH)$, $y \in \L^{p^*}(\QH)$ and $z \in \L^p(\QH)$, we have 
\begin{equation}
\label{convolution-et-crochet-QG}
\langle z*x,y \rangle_{\L^p(\QH),\L^{p^*}(\QH)}
=\langle z,x*y \rangle_{\L^p(\QH),\L^{p^*}(\QH)}
\end{equation}
\end{lemma}

Now, we will show that the bounded Fourier multipliers from $\L^1(\QH)$ into $\L^p(\QH)$ on a \textit{co-amenable} quantum hypergroup are convolution operators. The faithfulness of the convolution for locally compact quantum groups is proved in \cite[Proposition 1]{HNR10}.

\begin{thm}
\label{Thm-description-multipliers-2}
Let $(\QH,\tau,\Delta,R)$ be a co-amenable normalized quantum hypergroup. Suppose $1 < p \leq \infty$. Let $T \co \L^1(\QH) \to \L^p(\QH)$ be a map.  
\begin{enumerate}
\item The operator $T$ is a bounded left multiplier if and only if there exists $f \in \L^p(\QH)$ such that 
\begin{equation}
\label{T-as-convolution-QG}
T(x)
=f*x,\quad x \in \L^1(\QH).
\end{equation}

\item The operator $T$ is a bounded right multiplier if and only if there exists $f \in \L^p(\QH)$ such that 
\begin{equation}
\label{T-as-convolution-right}
T(x)
=x*f,\quad x \in \L^1(\QH).
\end{equation}

\item If $*$ is faithful, the map $\L^p(\QH) \to \frak{M}^{1,p}(\QH)$, $\mu \mapsto (f*\cdot,\cdot*f)$ is a surjective isometry. 
\end{enumerate}

In the situation of the first point or the second point and if $*$ is faithful, $f$ is unique, the map is completely bounded and we have
\begin{equation}
\label{norm-equality-QG}
\norm{f}_{\L^p(\QH)}
=\norm{T}_{\L^1(\QH) \to \L^p(\QH)}
=\norm{T}_{\cb,\L^1(\QH) \to \L^p(\QH)}.
\end{equation}
\end{thm}

\begin{proof}
1. $\Leftarrow$: Suppose that $f$ belongs to $\L^p(\QH)$. The completely boundedness of $T$ is a consequence of Corollary \ref{Cor-convolution} which in addition says that
\begin{equation}
\label{First-inequality-QG}
\norm{T}_{\cb,\L^1(\QH) \to \L^p(\QH)}
\ov{\eqref{borne-cb}}{\leq} \norm{f}_{\L^p(\QH)}.
\end{equation}
Moreover, for any $x,y \in \L^1(\QH)$, we have
$$
T(x*y)
\ov{\eqref{T-as-convolution-QG}}{=} f*(x*y)
=(f*x)*y
\ov{\eqref{T-as-convolution-QG}}{=} T(x)*y.
$$
Hence \eqref{Def-left-multiplier-L1} is satisfied, we conclude that $T$ is a bounded left multiplier.

$\Rightarrow$: Suppose that $T \co \L^1(\QH) \to \L^p(\QH)$ is a bounded left multiplier. Since $\QH$ is co-amenable, there exists a contractive approximative unit $(b_i)$ of the algebra $\L^1(\QH)$. For any $x \in \L^1(\QH)$, we have in $\L^1(\QH)$
\begin{equation}
\label{Approx-Fourier-algebra-QG}
b_i* x
\xra[i]{} x.
\end{equation}
For any $i$, we have
\begin{align}
\MoveEqLeft
\label{eq-inter-658-QG}
\norm{T(b_i)}_{\L^p(\QH)} 
\leq \norm{T}_{\L^1(\QH) \to \L^p(\QH)} \norm{b_i}_{\L^1(\QH)} 
\leq \norm{T}_{\L^1(\QH) \to \L^p(\QH)}.
\end{align}
This implies that the net $(T(b_i))$ lies in a norm bounded subset of the Banach space $\L^p(\QH)$. Note that if $1< p < \infty$, the Banach space $\L^p(\QH)$ is reflexive. By \cite[Theorem 4.2 p.~132]{Con90} if $p<\infty$ and Alaoglu's theorem if $p=\infty$ there exist a subnet $(T(b_j))$ of $(T(b_i))$ and $f \in \L^p(\QH)$ such that the net $(T(b_j))$ converges to $f$ for the weak topology of $\L^p(\QH)$ if $p<\infty$ and for the weak* topology of $\L^\infty(\QH)$ if $p=\infty$, i.e.~for any $z \in \L^{p^*}(\QH)$ we have
\begin{equation}
\label{Limite-faible-QG}
\big\langle T(b_j),z\big\rangle_{\L^p(\QH),\L^{p^*}(\QH)}
\xra[j]{} \langle f, z\rangle_{\L^p(\QH),\L^{p^*}(\QH)}.
\end{equation}

\begin{lemma}
For any $x \in \L^1(\QH)$, we have in $\L^p(\QH)$
\begin{equation}
\label{Limite-utile-QG}
T(x)
=\lim_{i} \big(T(b_j)*x\big).
\end{equation}
\end{lemma}

\begin{proof}
For any $x \in \L^1(\QH)$ and any $j$, we have
\begin{align*}
\MoveEqLeft
\norm{T(x)-T(b_j)*x}_{\L^p(\QH)}            
\ov{\eqref{Def-left-multiplier-L1}}{=} \norm{T(x)-T(b_j*x)}_{\L^p(\QH)} \\
&=\norm{T(x-b_j*x)}_{\L^p(\QH)}
\leq \norm{T}_{\L^1(\QH) \to \L^p(\QH)} \norm{x-b_j*x}_{\L^1(\QH)}
\xra[\, j\, ]{\eqref{Approx-Fourier-algebra-QG}} 0.
\end{align*}
\end{proof}

For any $x \in \L^1(\QH)$ any $y \in \L^{p^*}(\QH)$, we have
\begin{align*}
\MoveEqLeft
\big\langle T(x),y \big\rangle_{\L^p(\QH),\L^{p^*}(\QH)}            
\ov{\eqref{Limite-utile-QG}}{=} \big\langle \lim_j T(b_j)*x,y \big\rangle_{\L^p(\QH),\L^{p^*}(\QH)}    \\ 
&=\lim_{j} \big\langle T(b_j)*x,y \big\rangle_{\L^p(\QH),\L^{p^*}(\QH)} 
\ov{\eqref{convolution-et-crochet-QG}}{=} \lim_{j} \big\langle T(b_j),x*y \big\rangle_{\L^p(\QH),\L^{p^*}(\QH)} \\
&\ov{\eqref{Limite-faible-QG}}{=} \big\langle f,x*y \big\rangle_{\L^p(\QH),\L^{p^*}(\QH)} 
\ov{\eqref{convolution-et-crochet-QG}}{=} \langle  f*x,y \rangle_{\L^p(\QH),\L^{p^*}(\QH)}.
\end{align*} 
So we obtain \eqref{T-as-convolution-QG}. 

Using \cite[Theorem 2.5.21]{Meg98} if $p<\infty$ and \cite[Theorem 2.5.21]{Meg98} if $p=\infty$ in the first inequality, we obtain
\begin{align}
\label{Mysterious-estimate}
\MoveEqLeft
\norm{f}_{\L^p(\QH)}
=\bnorm{\w^*-\lim_j T(b_j)}_{\L^p(\QH)}
\leq \liminf_j \norm{T(b_j)}_{\L^p(\QH)} \\
&\ov{\eqref{eq-inter-658-QG}}{\leq} \norm{T}_{\L^1(\QH) \to \L^p(\QH)}
\leq \norm{T}_{\cb,\L^1(\QH) \to \L^{p}(\QH)}. \nonumber            
\end{align}
Combined with \eqref{First-inequality-QG}, we obtain \eqref{norm-equality-QG}. The uniqueness is essentially a consequence of the faithfulness of the convolution.

2. The proof is similar to the proof of the point 1.

3. Let $(L,R)$ be a bounded multiplier. Note that by \cite[Lemma 2.1]{Daw10}, the map $R \co \L^1(\QH) \to \L^p(\QH)$ is a bounded right multiplier. Then by the point 2, there exists $f \in \L^p(\QH)$ such that $R(x)=x*f$. For any $x,y \in \L^1(\QH)$, we have
\begin{equation}
\label{}
x*L(y)
\ov{\eqref{multiplier-bimodule}}{=} R(x) * y
=(x*f)*y
=x*(f*y).
\end{equation}
Since the convolution $*$ is faithful, we deduce that for any $y \in \L^1(\QH)$ we have $L(y)=f*y$. Conversely, for any $x,y \in \L^1(\QH)$, we have
$$
x*(f*y)
=(x*f) * y.
$$
So by \eqref{multiplier-bimodule}, the couple $(f *\cdot,\cdot*f)$ is a multiplier. The proof is complete.
\end{proof}

\begin{remark} \normalfont
The case $p=1$ for  co-amenable locally compact quantum groups is described in \cite[Proposition 3.1]{HNR11} which says that in this case we have $\frak{M}^{1,1}(\QG)=\M(\QG)$ where $\M(\QG)$ is the measure algebra of $\QG$. See also \cite[Corollary 5.7]{Cra17} and \cite[Theorem 5.2]{Daw12}. 
%
%
%
%
\end{remark}

\begin{remark} \normalfont
A bibliographic search shows that a similar result is proved in \cite[Theorem 1]{Edw55} for locally compact \textit{abelian} groups. H\"ormander also states\footnote{\thefootnote. We reproduce the \textit{complete} proof here: ``follows from the fact that $\L^{p'}$ is the dual space of $\L^p$ when $p < \infty$'' which is essentially the first sentence of the proof of \cite[Theorem 1]{Edw55}.} this result in the particular case of the group $\R$ in its classical paper \cite[Theorem 1.4]{Hor60} without citing Edwards. 
\end{remark}

\begin{remark} \normalfont
The isometric map $\Phi \co \L^p(\QH) \to \CB(\L^1(\QH),\L^p(\QH))$, $f \mapsto f * \cdot$ is completely contractive by the proof of Corollary \ref{Cor-convolution}. It could be possible that this map is even completely isometric, probably by generalizing the argument of \eqref{Mysterious-estimate}. If $p=1$, it is stated without proof in \cite[Proposition 3.1]{HNR11} for any co-amenable locally compact quantum group. 
\end{remark}

\begin{remark} \normalfont
\label{remark-necessary}
In the spirit of \cite[Proposition 3.1]{HNR11}, we conjecture that if $1 < p \leq \infty$ the co-amenability is necessary in order to have the isometry $\frak{M}^{1,p}(\QH)=\L^p(\QH)$. See also the end of Section \ref{completely-bounded-description} for the group case.
\end{remark}

Now, we justify the name ``Fourier multipliers'' of our multipliers in the case of locally compact quantum groups. Of course, we have a similar result for right multipliers which is left to the reader. See \cite[Theorem 4 p.~772]{AMR18} for a related result.

\begin{prop}
\label{prop-bounded-Fourier-multipliers-L1-Lp}
Let $\QG$ be a unimodular co-amenable compact quantum group of Kac type such that the left Haar weight is tracial. Suppose $1<p \leq \infty$. Let $T \co \L^1(\QG) \to \L^p(\QG)$ be a linear transformation. Then $T$ is a bounded left multiplier if and only if there exists a unique $a \in \ell^\infty(\hat{\QG})$ such that
\begin{equation}
\label{En-forme-Fourier}
\widehat{T(x)}(\pi)
=a_\pi\widehat{x}(\pi), \quad \pi \in \Irr(\QG),x \in \L^1(\QG).
\end{equation}
\end{prop}

\begin{proof}
$\Rightarrow$: By the first point of Theorem \ref{Thm-description-multipliers-2}, there exists a unique $f \in \L^p(\QG)$ such that \eqref{T-as-convolution-QG} is satisfied. Since the von Neuman algebra $\L^\infty(\QG)$ is finite we have the contractive inclusion $\L^p(\QG) \subset \L^1(\QG)$. So note that by Proposition \ref{prop-L1-into-c0}, the element $a \ov{\mathrm{def}}{=} \widehat{f}$ belongs to $\ell^\infty(\widehat{\QG})$. Moreover, for any $x \in \L^1(\QG)$ and any $\pi \in \Irr(\QG)$, we have
\begin{align*}
\MoveEqLeft
\widehat{T(x)}(\pi)
=\widehat{f*x}(\pi)            
=\widehat{f}(\pi) \widehat{x}(\pi)
=a_\pi \widehat{x}(\pi).
\end{align*}
We have obtained \eqref{En-forme-Fourier}.

$\Leftarrow$: For any $x,y \in \L^1(\QG)$ and any $\pi \in \Irr(\QG)$, we have
\begin{equation}
\label{}
\reallywidehat{T(x)*y}(\pi)
=\widehat{T(x)}(\pi)\widehat{y}(\pi)
\ov{\eqref{En-forme-Fourier}}{=} a_\pi\widehat{x}(\pi)\widehat{y}(\pi)
\ov{\eqref{conv-L1-et-transfo-fourier}}{=} a_\pi\widehat{x*y}(\pi)
\ov{\eqref{En-forme-Fourier}}{=} \reallywidehat{T(x*y)}(\pi).
\end{equation}
Hence $\reallywidehat{T(x)*y}=\reallywidehat{T(x*y)}$. By Proposition \ref{prop-L1-into-c0}, we infer that $T(x)*y=T(x*y)$. By \eqref{Def-left-multiplier-L1}, we deduce that $T \co \L^1(\QG) \to \L^p(\QG)$ is a left multiplier. It remains to show that this map is bounded. Suppose that $(x_n)$ is a sequence of elements $\L^1(\QG)$ converging to some element $x$ of $\L^1(\QG)$ and that the sequence $(T(x_n))$ converges to some $y$ in $\L^p(\QG)$. Using Proposition \ref{prop-L1-into-c0} in the second inequality and the last inequality (together with the contractive inclusion $\L^p(\QG) \subset \L^1(\QG)$), we obtain
\begin{align*}
\MoveEqLeft
\label{equa-divers-145}
\bnorm{\widehat{T(x)}-\widehat{y}}_{\ell^\infty(\hat{\QG})}
\leq \bnorm{\widehat{T(x)}-\widehat{T(x_n)}}_{\ell^\infty(\hat{\QG})}+\bnorm{\widehat{T(x_n)}-\widehat{y}}_{\ell^\infty(\hat{\QG})} \\   
&= \bnorm{\reallywidehat{T(x-x_n)}}_{\ell^\infty(\hat{\QG})}+\bnorm{\reallywidehat{T(x_n)-y}}_{\ell^\infty(\hat{\QG})} 
\ov{\eqref{En-forme-Fourier}}{\leq} \bnorm{a(\widehat{x}-\widehat{x_n})}_{\ell^\infty(\hat{\QG})}+  \norm{T(x_n)-y}_{\L^1(\QG)} \nonumber \\
&\leq \norm{a}_{\ell^\infty(\hat{\QG})}\bnorm{\widehat{x}-\widehat{x_n}}_{\ell^\infty(\hat{\QG})}+\norm{T(x_n)-y}_{\L^p(\QG)} \\
&\leq \norm{a}_{\ell^\infty(\hat{\QG})}\norm{x_n-x}_{\L^1(\QG)}+\norm{T(x_n)-y}_{\L^p(\QG)}. \nonumber
\end{align*}
Passing to the limit, we infer that $\widehat{T(x)}=\widehat{y}$, hence $T(x)=y$ by the injectivity of the Fourier transform given by Proposition \ref{prop-L1-into-c0}. By the closed graph theorem, we conclude that the operator $T \co \L^1(\QG) \to \L^p(\QG)$ is bounded.

The uniqueness is left to the reader.
\end{proof}

\begin{remark} \normalfont
We could probably extend this result to the case of co-amenable \textit{locally} compact quantum group of Kac type if $1 \leq p \leq 2$ by using the Hausdorff-Young inequality \cite{Coo10}.
\end{remark}

\begin{defi}
Let $\QG$ be a unimodular co-amenable compact quantum group of Kac type such that the left Haar weight is tracial. Suppose $1<p \leq \infty$. Let $T \co \L^1(\QG) \to \L^p(\QG)$ be a bounded left multiplier. The unique $a \in \ell^\infty(\hat{\QG})$ satisfying \eqref{En-forme-Fourier} is called the symbol of $T$. Moreover, we let $M_a \ov{\mathrm{def}}{=} T$.
\end{defi}

Let us state the particular case of co-commutative compact quantum groups. Here, we use the classical notations, e.g.  see \cite{ArK23} for background. Note that the case $p=1$ is well-known by the combination of \cite[Corollary 5.4.11. (ii)]{KaL18} and \cite[Lemma 6.4]{ArK23}.

\begin{cor}
\label{cor-L1-Lp}
Let $G$ be an amenable discrete group. Suppose $1 \leq p \leq \infty$. Any bounded Fourier multiplier $M_\varphi \co \L^1(\VN(G)) \to \L^p(\VN(G))$ is completely bounded and we have
\begin{equation}
\label{cb=bounded}
\norm{M_{\varphi}}_{\cb, \L^1(\VN(G)) \to \L^p(\VN(G))}
=\norm{M_{\varphi}}_{\L^1(\VN(G)) \to \L^p(\VN(G))}.
\end{equation}
\end{cor}

Now, we return to the general case. A map $T \co \L^p(\QH) \to \L^\infty(\QH)$ is a left  multiplier if $T(x*y)=T(x) *y$ for any $x \in \L^p(\QH)$ and any $y \in \L^1(\QH)$. A simple duality argument gives the following result.

\begin{thm}
\label{Thm-description-multipliers-3}
Let $(\QH,\tau,\Delta,R)$ be a co-amenable normalized quantum hypergroup. Suppose $1 \leq p < \infty$. 
\begin{enumerate}
	\item A map $T \co \L^p(\QH) \to \L^\infty(\QH)$ is a bounded left multiplier if and only if there exists $f \in \L^{p^*}(\QH)$ such that 
\begin{equation}
\label{T-as-convolution-QG-2}
T(x)
=f*x,\quad x \in \L^p(\QH).
\end{equation}

\item A map $T \co \L^p(\QH) \to \L^\infty(\QH)$ is a bounded right multiplier if and only if there exists $f \in \L^{p^*}(\QH)$ such that 
\begin{equation}
\label{T-as-convolution-right-2}
T(x)
=x*f,\quad x \in \L^p(\QH).
\end{equation}

\end{enumerate}

If $*$ is faithful, in the situation of the first point or the second point 2, $f$ is unique and we have
\begin{equation}
\label{norm-equality-QG-2}
\norm{f}_{\L^{p^*}(\QH)}
=\norm{T}_{\cb,\L^p(\QH) \to \L^\infty(\QH)}
=\norm{T}_{\L^p(\QH) \to \L^\infty(\QH)}.
\end{equation}
\end{thm}

\begin{proof}
1. $\Rightarrow$: Suppose that there exists $f \in \L^{p^*}(\QH)$ such that \eqref{T-as-convolution-QG-2} is satisfied. By the point 2 of Theorem \ref{Thm-description-multipliers-2}, we have a completely bounded operator $U \co \L^1(\QH) \to \L^{p^*}(\QH)$, $y \mapsto y *f$. For any $x \in \L^p(\QH)$ and any $y \in \L^1(\QH)$, we have
\begin{equation*}
\label{}
\big\langle U(y), x \big\rangle_{\L^{p^*}(\QH),\L^p(\QH)} 
=\langle y *f, x \rangle_{\L^p(\QH),\L^{p^*}(\QH)} 
\ov{\eqref{convolution-et-crochet-QG}}{=}\langle y, f*x \rangle_{\L^p(\QH),\L^{p^*}(\QH)}.
\end{equation*}
By \cite[Proposition 3.2.2]{EfR00}, if $p>1$ we infer that $T \co \L^p(\QH) \to \L^\infty(\QH)$, $x \mapsto f*x$ is a well-defined completely bounded map satisfying \eqref{norm-equality-QG-2}. If $p=1$, it is easy to adapt the argument.

$\Leftarrow$: Suppose that $T \co \L^p(\QH) \to \L^\infty(\QH)$ is a bounded left multiplier. If $p>1$, since $\L^p(\QH)$ is reflexive, it is apparent e.g. from \cite[Lemma 10.1 (3)]{Daw10} that $T$ is weak* continuous. So we can consider the preadjoint $T_* \co \L^1(\QH) \to \L^{p^*}(\QH)$ on $\L^1(\QH)$ of $T$ (if $p=1$ consider the restriction on $\L^1(\QH)$ of the adjoint of $T$ instead of $T_*$). For any $x \in \L^p(\QH)$ and any $y,z \in \L^1(\QH)$ we have
\begin{align*}
\MoveEqLeft
\big\langle x, T_*(z*y) \big\rangle_{\L^p(\QH),\L^{p^*}(\QH)} 
=\big\langle T(x), z*y \big\rangle_{\L^\infty(\QH),\L^1(\QH)}
\ov{\eqref{convolution-et-crochet-QG}}{=} \big\langle T(x)* z, y \big\rangle_{\L^\infty(\QH),\L^1(\QH)} \\
&\ov{\eqref{left-multiplier-bimodule}}{=} \big\langle T(x*z), y \big\rangle_{\L^\infty(\QH),\L^1(\QH)}
=\big\langle x*z, T_*(y) \big\rangle_{\L^p(\QH),\L^{p^*}(\QH)}
\ov{\eqref{convolution-et-crochet-QG}}{=} \big\langle x, z*T_*(y) \big\rangle_{\L^p(\QH),\L^{p^*}(\QH)}.
\end{align*}
Hence for any $y,z \in \L^1(\QH)$ we deduce that 
$$
T_*(z*y)
=z*T_*(y).
$$
Consequently by \eqref{right-multiplier-bimodule} the operator $T_* \co \L^1(\QH) \to \L^{p^*}(\QH)$ is a bounded right multiplier. We infer by the second point of Theorem \ref{Thm-description-multipliers-2} that there exists a $f \in \L^{p^*}(\QH)$ such that $T_*(y)=y*f$ for any $y \in \L^1(\QH)$. For any $x \in \L^p(\QH)$ and any $y \in \L^1(\QH)$, we obtain
\begin{equation*}
\label{}
\big\langle T(x), y \big\rangle_{\L^\infty(\QH),\L^1(\QH)} 
=\big\langle x, T_*(y) \big\rangle_{\L^p(\QH),\L^{p^*}(\QH)} 
=\langle x, y*f \rangle_{\L^p(\QH),\L^{p^*}(\QH)}  
=\langle f*x, y \rangle_{\L^\infty(\QH),\L^1(\QH)}.
\end{equation*}
Therefore $T(x)=f*x$ for any $x \in \L^p(\QH)$.


2. The proof is similar to the proof of the point 1.

The other assertions are very easy and left to the reader.
\end{proof}

\begin{defi}
Let $\QG$ be a unimodular co-amenable compact quantum group of Kac type such that the left Haar weight is tracial. Suppose $1 \leq p \leq \infty$. Let $T \co \L^p(\QG) \to \L^\infty(\QG)$ be a bounded left multiplier. We say that the element $a \ov{\mathrm{def}}{=} \hat{f}$ of  $\ell^\infty(\hat{\QG})$ is the symbol of $T$. We let $M_a \ov{\mathrm{def}}{=} T$.
\end{defi}

\subsection[Another approach on c.~b.~Fourier multipliers from $\L^1(\QG)$ into $\L^p(\QG)$]{Another approach on completely bounded Fourier multipliers from $\L^1(\QG)$ into $\L^p(\QG)$} 
\label{completely-bounded-description}

In this section, we investigate the completely bounded Fourier multipliers from $\L^1(\QG)$ into $\L^p(\QG)$ under weaker assumptions. We will use the following lemma which is an extension of \cite[(3.20)]{DJKRB06} which states a similar result for positive elements of matrix spaces. 


\begin{lemma}
Let $\cal{M}$ and $\cal{N}$ be von Neumann algebras equipped with normal semifinite faithful traces. Suppose $1 \leq p \leq \infty$. For any positive element $x$ of $\L^p(\cal{M},\L^1(\cal{N}))$, we have
\begin{equation}
\label{Norm-LpL1-pos}
\norm{x}_{\L^p(\cal{M},\L^1(\cal{N}))}
=\norm{(\Id \ot \tau)(x)}_{\L^p(\cal{M})}
\end{equation}
where $\tau$ is the trace of $\cal{N}$.
\end{lemma}

\begin{proof}
Here, we denotes the traces of $\cal{M}$ and $\cal{N}$ by the same letter $\tau$. For any positive $x \in \L^p(\cal{M},\L^1(\cal{N}))$, we have using Hahn-Banach theorem with \eqref{dual-by-positive} in the last equality
\begin{align*}
\MoveEqLeft
\norm{x}_{\L^p(\cal{M},\L^1(\cal{N}))}
\ov{\eqref{Norm-LpLq}}{=} \sup_{\norm{a}_{2p^{*}} = 1, a \geq 0} \bnorm{ (a \ot 1) x (a \ot 1)}_{\L^1(\cal{M} \otvn \cal{N})} \\
&= \sup_{\norm{a}_{2p^{*}} = 1, a \geq 0} (\tau \ot \tau)\big( (a \ot 1) x (a \ot 1) \big) 
= \sup_{\norm{a}_{2p^{*}} = 1, a \geq 0} (\tau \ot \tau)\big( (a^2\ot 1) x \big)  \\
&= \sup_{\norm{a}_{2p^{*}} = 1, a \geq 0} \tau \Big[(\Id\ot \tau)\big( (a^2\ot 1)x  \big)\Big] 
=\sup_{\norm{a^2}_{p^{*}} = 1, a \geq 0} \tau\Big[ a^2 (\Id \ot \tau)(x) \Big] \\
&=\sup_{\norm{b}_{p^{*}} = 1, b \geq 0} \tau \Big[ b (\Id \ot \tau)(x) \Big]  
\ov{\eqref{dual-by-positive}}{=} \bnorm{(\Id \ot \tau)(x)}_{\L^p(\cal{M})}. 
\end{align*}
\end{proof}

\begin{remark} \normalfont
\label{Rem-Haagerup-Lp}
If $\cal{M}$ and $\cal{N}$ are equipped with normal states, we could use Haagerup's noncommutative $\L^p$-spaces \cite{Ter81}. Indeed, it seems to the author that we can replace the trace $\tau$ by the Haagerup trace in the previous proof.
\end{remark}

The following lemma is crucial here and in Section \ref{summing-multipliers}. 

\begin{lemma}
\label{Lemma-useful}
Suppose $1 \leq p \leq \infty$. Let $\QG$ be a compact quantum group of Kac type. The coproduct $\Delta$ map induces an isometry $\Delta \co \L^p(\QG) \to \L^\infty(\QG,\L^p(\QG))$ and a bounded map $\Delta \co \L^1(\QG) \to \L^p(\QG,\L^1(\QG))$ with norm less than $\tau(1)^{\frac{1}{p}}$ where $\tau$ is a Haar trace.
\end{lemma}

\begin{proof}
We can write in the case of a \textit{positive} element $x$ of $\L^1(\QG)$
\begin{align}
\MoveEqLeft
\label{divers-56678}
\norm{\Delta(x)}_{\L^\infty(\QG,\L^1(\QG))}              
\ov{\eqref{Norm-LpL1-pos}}{=} \bnorm{(\Id \ot \tau)(\Delta(x))}_{\L^\infty(\QG)} 
\ov{\eqref{Def-haar-state}}{=} \norm{\tau(x)1}_{\L^\infty(\QG)} \\
&=|\tau(x)| \norm{1}_{\L^\infty(\QG)}
=|\tau(x)|
\ov{\eqref{trace-continuity}}{\leq} \norm{x}_{\L^1(\QG)}. \nonumber
\end{align} 
Now, consider an \textit{arbitrary} element $x$ of $\L^1(\QG)$. By \eqref{inverse-Holder}, there exists $y,z \in \L^2(\QG)$ such that $x=yz$ and $\norm{x}_{\L^1(\QG)}=\norm{y}_{\L^2(\QG)} \norm{z}_{\L^2(\QG)}$. Note that a small approximation argument gives $\Delta(x)=\Delta(yz)=\Delta(y)\Delta(z)$. Moreover, we have
\begin{align*}
\MoveEqLeft
\bnorm{(\Id \ot \tau)(\Delta(y)\Delta(y)^*)}_{\L^\infty(\QG)}^{\frac{1}{2}} \bnorm{(\Id \ot \tau)(\Delta(z)^*\Delta(z))}_{\L^\infty(\QG)}^{\frac{1}{2}}  \\
&\ov{\eqref{Norm-LpL1-pos}}{=}\norm{\Delta(y)\Delta(y)^*}_{\L^\infty(\QG,\L^1(\QG))}^{\frac{1}{2}} \norm{\Delta(z)^*\Delta(z)}_{\L^\infty(\QG,\L^1(\QG))}^{\frac{1}{2}} \\
&=\norm{\Delta(yy^*)}_{\L^\infty(\QG,\L^1(\QG))}^{\frac{1}{2}} \norm{\Delta(z^*z)}_{\L^\infty(\QG,\L^1(\QG))}^{\frac{1}{2}} \\      
&\ov{\eqref{divers-56678}}{\leq} \norm{yy^*}_{\L^1(\QG)}^{\frac{1}{2}}\norm{z^*z}_{\L^1(\QG)}^{\frac{1}{2}} 
= \norm{y}_{\L^2(\QG)} \norm{z}_{\L^2(\QG)}
=\norm{x}_{\L^1(\QG)}.
\end{align*}  
By \eqref{formula-Junge}, we deduce that $\norm{\Delta(x)}_{\L^\infty(\QG,\L^1(\QG))} \leq \norm{x}_{\L^1(\QG)}$. So we have a contraction $\Delta \co \L^1(\QG) \to \L^\infty(\QG,\L^1(\QG))$. By interpolation with the compatible map $\Delta \co \L^\infty(\QG) \to \L^\infty(\QG) \otvn \L^\infty(\QG)$, we see that $\Delta$ induces a contraction $\Delta \co \L^p(\QG) \to \L^\infty(\QG,\L^p(\QG))$. 

Until the end of the proof, we can clearly suppose that $\tau$ is normalized. Since $\Delta$ is a trace preserving normal unital injective $*$-homomorphism, it induces a (complete) isometry $\Delta \co \L^p(\QG) \to \L^p(\L^\infty(\QG) \otvn \L^\infty(\QG))=\L^p(\QG,\L^p(\QG))$. By interpolation of $\Delta \co \L^1(\QG) \to \L^1(\QG,\L^1(\QG))$ and $\Delta \co \L^1(\QG) \to \L^\infty(\QG,\L^1(\QG))$, we deduce a contraction $\Delta \co \L^1(\QG) \to \L^p(\QG,\L^1(\QG))$.

Finally, since the von Neumann algebra $\L^\infty(\QG)$ is finite and that the trace $\tau$ is normalized, we have for any $x \in \L^\infty(\QG)$
\begin{equation}
\label{Eq-sans-fin-34}
\norm{\Delta(x)}_{\L^\infty(\QG,\L^p(\QG))}
\geq \norm{\Delta(x)}_{\L^p(\QG,\L^p(\QG))}
=\norm{x}_{\L^p(\QG)}.
\end{equation}
\end{proof}

\begin{remark} \normalfont
Let $G$ be a discrete group. It is possible to show that $\Delta$ induces a contractive map $\Delta \co \L^2(\VN(G)) \to \L^\infty(\VN(G),\L^2(\VN(G)))$ by the following argument. If $a,b$ belongs to the unit ball of $\L^4(\VN(G))$ and if $x=\sum_{s \in G} \alpha_s \lambda_s$ is an element of $\Pol(\hat{G})$ with $\alpha_s \in \mathbb{C}$, we have using the orthonormality of the family $(\lambda_s)_{s \in G}$ and the argument of the proof of \cite[Proposition 3.9]{Pis98} in the fourth equality
\begin{align*}
\MoveEqLeft
\bnorm{(a \ot 1)\Delta(x)(b \ot 1)}_{\L^2(\VN(G) \otvn \VN(G))}
=\norm{(a \ot 1)\Delta\bigg(\sum_{s \in G} \alpha_s \lambda_s\bigg)(b \ot 1)}_{\L^2(\VN(G) \otvn \VN(G))} \\
&\ov{\eqref{coproduct-VNG}}{=}\norm{(a \ot 1)\bigg(\sum_{s \in G} \alpha_s \lambda_s \ot \lambda_s\bigg)(b \ot 1)}_{\L^2(\VN(G) \otvn \VN(G))}  \\          
&=\norm{\sum_{s \in G} \alpha_s a\lambda_sb \ot \lambda_s}_{\L^2(\VN(G) \otvn \VN(G))} \\
&=\bigg(\sum_{s \in G} \alpha_s^2\norm{ a\lambda_sb}_{\L^2(\VN(G)}^2\bigg)^{\frac{1}{2}} 
\ov{\eqref{Holder}}{\leq} \bigg(\sum_{s \in G} \alpha_s^2\norm{a}_{\L^4(\VN(G)}^2\norm{\lambda_s}_{\infty}^2 \norm{b}_{\L^4(\VN(G)}^2\bigg)^{\frac{1}{2}}\\
&\leq \bigg(\sum_{s \in G} \alpha_s^2\bigg)^{\frac{1}{2}}\norm{a}_{\L^4(\VN(G)}  \norm{b}_{\L^4(\VN(G)}
\leq \norm{\sum_{s \in G} \alpha_s \lambda_s}_{\L^2(\VN(G))}
=\norm{x}_{\L^2(\VN(G))}.
\end{align*}
Passing to the supremum on $a$ and $b$, we obtain by \eqref{LinftyLq-norms}
$$
\norm{\Delta(x)}_{\L^\infty(\VN(G),\L^2(\VN(G))}   
\leq \norm{x}_{\L^2(\VN(G))}.
$$ 
\end{remark}



Now, we give a complement to Theorem \ref{Thm-description-multipliers-2}. 

\begin{prop}
\label{Prop-cb-L1-Lp}
Let $\QG$ be a compact quantum group of Kac type. Suppose $1 < p \leq \infty$. We suppose that $\L^\infty(\QG)$ is hyperfinite if $p<\infty$. For any $x \in \L^p(\QG)$, we have
\begin{equation}
\label{norm-cb-Lp}
\norm{x * \cdot}_{\cb, \L^1(\QG) \to \L^p(\QG)}
=\norm{x}_{\L^p(\QG)}.
\end{equation}
\end{prop}

\begin{proof}
If $p=\infty$, using \cite[Proposition 1.2.4]{BLM04}, we have a complete isometry 
$$
\begin{array}{ccccccc}
\L^\infty(\QG) & \longrightarrow  & \L^\infty(\QG) \otvn\L^\infty(\QG) &\ov{\eqref{Ident-magic}}{\longrightarrow} & \CB(\L^\infty(\QG)_*,\L^\infty(\QG)) \\
    x  &       \longmapsto  &    \Delta(x)                     & \longmapsto & (f \mapsto \big(f \ot \Id\big)(\Delta(x)) \ov{\eqref{action-module}}{=} x \star f \\
\end{array}.
$$
Recall that by \eqref{bracket} we have the completely isometric identification $\L^1(\QG) \to \L^\infty(\QG)_*$, $y \to \tau(yR(\cdot))$. 


Suppose $1 < p < \infty$. By Lemma \ref{Lemma-useful}, we have an isometry 
$$
\begin{array}{ccccccc}
\L^p(\QG) & \longrightarrow  & \L^\infty(\QG,\L^p(\QG)) \ov{\eqref{Def-L0-infty}}{=} \L^\infty(\QG) \ot_{\min} \L^p(\QG) &\ov{\eqref{belle-injection}}{\longrightarrow} & \CB(\L^\infty(\QG)_*,\L^p(\QG)) \\
    x  &       \longmapsto  &    \Delta(x)      & \longmapsto & (f \mapsto \big(f \ot \Id\big)(\Delta(x)) \ov{\eqref{action-module}}{=} x \star f \\
\end{array}
$$
which maps $f$ on the convolution operator $f * \cdot$. 
\end{proof}

\begin{remark} \normalfont
\label{Remark-complementation}
Let $\QG$ be a co-amenable compact quantum group of Kac type. Consider the isometric map $\Phi \co \L^p(\QG) \to \CB(\L^1(\QG),\L^p(\QG))$, $f \mapsto f * \cdot$ which is completely contractive. We will explain here why the range $\Ran \Phi$ of $\Phi$ is a completely contractively complemented subspace of $\CB(\L^1(\QG),\L^p(\QG))$. Let $\E$ be the canonical trace preserving normal conditional expectation associated with the coproduct $\Delta \co \L^\infty(\QG) \to \L^\infty(\QG) \otvn \L^\infty(\QG)$. A completely contractive projection from $\CB(\L^1(\QG),\L^p(\QG))$ onto $\Ran \Phi$ is given by $P \co \CB(\L^1(\QG),\L^p(\QG)) \to \CB(\L^1(\QG),\L^p(\QG))$, $T \mapsto \E(\Id_{\L^p(\QG)} \ot T)\Delta$ where $\Delta \co \L^1(\QG) \to \L^p(\QG,\L^1(\QG))$ and $\E \co \L^p(\QG,\L^p(\QG)) \to \L^p(\QG)$. Indeed, note that these two last maps are completely contractive\footnote{\thefootnote. It suffices by interpolation to prove that $\Delta \co \L^1(\QG) \to \L^\infty(\QG,\L^1(\QG))$ is completely contractive by a reiteration argument (repeat the proof of Lemma \ref{Lemma-useful}).}. Now, it suffices to use \cite[Lemma 2.1]{Arh11} and a variant of \cite[Proposition 3.5]{Arh11}.
\end{remark}


\begin{remark} \normalfont
If $\QG$ is a compact quantum group of Kac type such that $\L^\infty(\QG)$ is $\QWEP$ it is \textit{maybe} possible to adapt the proof of Proposition \ref{Prop-cb-L1-Lp} in the case $1 < p <\infty$. The difficulty is that the paper \cite{Jun1} introducing vector-valued noncommutative $\L^p$-spaces $\L^p(\cal{M},E)$ in the case of a $\QWEP$ von Neumann algebra $\cal{M}$ is not definitively finished. 
%
\end{remark}

\begin{remark} \normalfont
It is possible to give variants of Theorem \ref{Thm-description-multipliers-2} for the quantum torus $\T^d_{\theta}$ \cite{CXY13} or more generally for a twisted group von Neumann algebra $\VN(G,\sigma)$ by using the twisted coproducts $\Delta_{\sigma,1} \co \VN(G,\sigma) \to \VN(G,\sigma) \otvn \VN(G)$, $\lambda_{\sigma,s} \mapsto \lambda_{\sigma,s} \ot \lambda_{s}$ and $\Delta_{1,\sigma} \co \VN(G,\sigma) \to \VN(G)\otvn \VN(G,\sigma)$, $\lambda_{\sigma,s} \mapsto \lambda_{s} \ot \lambda_{\sigma,s}$ of \cite[(4.5)]{ArK23} which is a unital injective normal $*$-homomorphism. Indeed, these coproducts induce by preduality an obvious structure of $\L^1(\VN(G))$-bimodule on $\L^p(\VN(G,\sigma))$. So we can consider the  Fourier multipliers from the space $\L^1(\VN(G,\sigma))$ into $\L^p(\VN(G,\sigma))$. If $\VN(G,\sigma)$ is hyperfinite and if $G$ is amenable, we obtain with the same method an isometry $\frak{M}_\cb^{1,p}(G,\sigma) = \L^p(\VN(G))$. In particular, for the quantum tori we have an isometry $\frak{M}_\cb^{1,p}(\T^d_{\theta}) = \L^p(\T)$.
\end{remark}


Let $\cal{M}$ be a von Neumann algebra and $H$ be a Hilbert space. The $H$-valued noncommutative $\L^p$-spaces $\L^\infty(\cal{M},H_c)$ and $\L^\infty(\cal{M},H_r)$ are defined in \cite[Section 2]{JMX06}. Recall that the tensor product $\cal{M} \ot H$ is weak* dense in both spaces. For any element $\sum_{k=1}^n x_k \ot a_k$ of $\cal{M} \ot H$, we have by \cite[(2.9)]{JMX06}
\begin{equation}
\label{Def-LinftyL2c}
\bgnorm{\sum_{k=1}^n x_k \ot a_k}_{\L^\infty(\cal{M},H_c)}
=\bgnorm{\bigg(\sum_{i,j=1}^n \la a_j,a_i \ra_H x_i^*x_j\bigg)^{\frac{1}{2}}}_{\cal{M}}.
\end{equation}
In particular if $H=\L^2(\cal{N})$ is a noncommutative $\L^2$-space then for any $z=\sum_{k=1}^n x_k \ot a_k$  of $\cal{M} \ot \L^2(\cal{N})$ we have
\begin{align}
\MoveEqLeft
\label{normLinfty-L2}
\norm{z}_{\L^\infty(\cal{M},\L^2(\cal{N})_c)}
\ov{\eqref{Def-LinftyL2c}}{=} \bgnorm{\bigg(\sum_{i,j=1}^n \la a_j,a_i \ra_{\L^2(\cal{N})} x_i^*x_j\bigg)^{\frac{1}{2}}}_{\cal{M}}            
=\bgnorm{\sum_{i,j=1}^n \tau(a_i^*a_j) x_i^*x_j}_{\cal{M}}^{\frac{1}{2}} \\
&=\bgnorm{(\Id_{\cal{M}} \ot \tau)\bigg(\sum_{i,j=1}^n a_i^*a_j \ot x_i^*x_j\bigg)}_{\cal{M}}^{\frac{1}{2}}
=\bnorm{(\Id_{\cal{M}} \ot \tau)(z^*z)}_{\cal{M}}^{\frac{1}{2}}. \nonumber
\end{align} 

\begin{prop}
\label{Prop-coproduct-column}
Let $\QG$ be a compact quantum group of Kac type. The coproduct induces a completely isometric map $\Delta \co \L^2(\QG)_c \to \L^\infty(\QG,\L^2(\QG)_c)$ (and similarly for the row case) and a completely contractive map $\Delta \co \L^2(\QG) \to \L^\infty(\QG,\L^2(\QG))$.
\end{prop}

\begin{proof}
Let $x=[x_{ij}]$ be an element of $\M_n(\L^\infty(\QG))$. Since we have the inclusion $\L^\infty(\QG) \subset \L^2(\QG)$, each element $\Delta(x_{ij})$ belongs to $\L^\infty(\QG) \ot \L^2(\QG)$. We see that
\begin{align*}
\MoveEqLeft
\bnorm{(\Id \ot \Delta)(x)}_{\M_n(\L^\infty(\QG,\L^2(\QG)_c)))}             
=\bnorm{[\Delta(x_{ij})]}_{\M_n(\L^\infty(\QG,\L^2(\QG)_c))} 
=\bnorm{[\Delta(x_{ij})]}_{\L^\infty(\M_n(\L^\infty(\QG)), \L^2(\QG)_c)} \\
&\ov{\eqref{normLinfty-L2}}{=} \bnorm{(\Id_{\M_n(\L^\infty(\QG))} \ot \tau)([\Delta(x_{ij})]^*[\Delta(x_{ij})])}_{\L^\infty(\M_n(\L^\infty(\QG)), \L^2(\QG)_c)}^{\frac{1}{2}}\\
&=\bnorm{(\Id_{\M_n(\L^\infty(\QG))} \ot \tau)(\Id_{\M_n} \ot \Delta)([x_{ij}]^*[x_{ij}])}_{\M_n(\L^\infty(\QG))}^{\frac{1}{2}}
\ov{\eqref{Def-haar-state}}{=} 
\bnorm{(\Id_{\M_n} \ot \tau)(x^*x)}_{\M_n}^{\frac{1}{2}} \\
&\ov{\eqref{normLinfty-L2}}{=} \norm{x}_{\M_n(\L^2(\QG)_c)}.
\end{align*}
We conclude by density. The last part is obtained by interpolation with \cite[Corollary 7.11]{Pis03}.
\end{proof}

The following result is a generalization of \cite[Theorem 2.3]{GJP17} and \cite[Remark 2.4]{GJP17}. Here $\L^2(\QG)_r$ is the Hilbert space $\L^2(\QG)$ equipped with the row operator space structure and $\L^2(\QG)_c$ is the same Hilbert space equipped with the column operator space structure. Note that \cite{FHS11} essentially contains a proof of the complete contractivity in the case of group von Neumann algebras with anther argument. Proposition \ref{Prop-coproduct-column} closes a (small) gap in the proof of \cite[Theorem 2.3]{GJP17} and we complete a (small) gap in \cite[Remark 2.4]{GJP17} in the proof of the third part. See also \cite[Lemma 4.5 and Lemma 4.6]{CrN22} for identities which are maybe connected to Theorem \ref{Th-cb-L2-Multiplier}. 

\begin{thm}
\label{Th-cb-L2-Multiplier}
Let $\QG$ be a compact quantum group of Kac type.
\begin{enumerate}
    \item The map $\Phi \co \L^2(\QG)_c \to \CB(\L^2(\QG)_r,\L^\infty(\QG))$, $x \mapsto \cdot *x $ is a complete isometry. 
    \item The map $\Phi \co \L^2(\QG)_r \to \CB(\L^2(\QG)_c,\L^\infty(\QG))$, $x \mapsto \cdot *x $ is a complete isometry.
		\item The map $\Phi \co \L^2(\QG) \to \CB(\L^2(\QG),\L^\infty(\QG))$, $x \mapsto \cdot *x $ is a complete isometry.
\end{enumerate}
\end{thm}

\begin{proof}
1. We will show that we have the following commutative diagram of complete isometries.
\begin{equation*}
  \xymatrix{
    \L^2(\QG)_c \ar[rr]^{\Phi} \ar[d]_{\Delta} & & \CB(\L^2(\QG)_r,\L^\infty(\QG))\\
   \L^\infty(\QG,\L^2(\QG)_c)  \ar@{^{(}->}[rr]_{i} & & \big(\L^1(\QG) \otp \L^2(\QG)_r \big)^* \ar[u]_{\mathcal{I}^{-1}}
  }
\end{equation*} 
The first map $\Delta$ is given by Proposition \ref{Prop-coproduct-column}. According to \cite[pp. 134-135]{EfR00}, we have a completely isometric map\footnote{\thefootnote. Note that is the ``$\subset$'' is actually an equality by \cite[Theorem 2.4]{EKR93} since it is elementary to check that the operator space $\L^2(\QG)_{r}$ has the operator approximation property.}
$$
i \co \L^\infty(\QG,\L^2(\QG)_c) 
=\L^\infty(\QG) \otvn \L^2(\QG)_c \subset
\L^\infty(\QG) \otvn_{\cal{F}} \L^2(\QG)_c
= \big(\L^1(\QG) \otp \L^2(\QG)_r\big)^* 
$$ 
where $\otvn$ is the normal Fubini tensor product. With the identification \eqref{otp-duality}, we can use the canonical completely isometric map $\mathcal{I} \co \CB(\L^2(\QG)_c, \L^\infty(\QG)) \to \big(\L^2(\QG)_{c} \otp \L^1(\QG)\big)^*$, $T \mapsto \big( w \ot y \mapsto \la y,T(w) \ra_{\L^1(\QG),\L^\infty(\QG)}\big)$. If $x,w \in \L^2(\QG)$ note that by Remark \ref{Rem-Young-pqr} the convolution $w*x$ belongs to the space $\L^\infty(\QG)$. For any $x,w,y \in \L^\infty(\QG)$, we obtain
\begin{align*}
\MoveEqLeft
\big(\mathcal{I}(\cdot *x)\big)(w \ot y)             
= \la y,w*x \ra_{\L^1(\QG),\L^\infty(\QG)}
\ov{\eqref{convolution-et-crochet-QG}}{=} \la y*w,x \ra_{\L^1(\QG),\L^\infty(\QG)} 
\ov{\eqref{def-convolution-2}}{=} \la w \ot y,\Delta x \ra.
\end{align*}  
We obtain that the convolution operator $\mathcal{I}(\Phi(x))=\mathcal{I}(\cdot* x)$ is equal to $\la \cdot,\Delta(x) \ra=i(\Delta(x))$.

2. The second part is similar.

3. By \cite[Lemma 0.2]{Har99}, we have a complete contraction $(\CB(E_0,F_0),\CB(E_1,F_1))_\theta \subset \CB(E_\theta,F_\theta)$ for each $0<\theta<1$ with obvious notations. So by interpolation, we obtain a complete contraction $\Phi \co \L^2(\QG) \to \CB(\L^2(\QG),\L^\infty(\QG))$, $x \mapsto \cdot *x $. Consider the ``average'' map $P \co \CB(\L^2(\QG),\L^\infty(\QG)) \to \CB(\L^2(\QG),\L^\infty(\QG))$, $T\mapsto \E (\Id_{\L^\infty(\QG)} \ot T)\Delta$ where the completely contractive map $\Delta \co \L^2(\QG) \to \L^\infty(\QG,\L^2(\QG))$ is defined in Proposition \ref{Prop-coproduct-column} and where $\E \co \L^\infty(\QG,\L^\infty(\QG)) \to \L^\infty(\QG)$ is the conditional expectation associated to the coproduct, see \cite[Theorem 7.5]{DFSW16}. Using \cite[Lemma 2.1]{Arh11} and a variant of \cite[Proposition 3.5]{Arh11}, we see that this map is completely contractive. Now, it suffices to use \cite[1.2.7]{BLM04} to see that the map $\Phi \co \L^2(\QG) \to \CB(\L^2(\QG),\L^\infty(\QG))$ is a complete isometry.
\end{proof}



Recall that a finitely generated discrete group $G$ has rapid decay of order $r \geq 0$ with respect to a length function if we have an estimate
\begin{equation}
\label{Rapid-decay}
\norm{x}_{\VN(G)}
\lesssim (n+1)^r \norm{x}_{\L^2(\VN(G))}
\end{equation}
for any $x=\sum_{|s| \leq n} a_s \lambda_s$. This property has its origin in \cite{Haa78}. We refer to the survey \cite{Cha17} and to \cite[Section 2.2]{Bat21} for more information. We essentially generalize in the next proof an idea of \cite[Remark 2.5]{GJP17}. We refer to \cite{BrN06} for an extensive list of examples where the following result can be applied.  

\begin{prop}
\label{Prop-amenable}
Let $G$ be a discrete group satisfying the rapid decay property with $r=1$ and with respect to a length function which is conditionally negative. Suppose that there exist $K \geq 0$ such that $\norm{M_\varphi}_{\cb,\L^1(\VN(G)) \to \L^2(\VN(G))} \leq K\norm{M_\varphi}_{\L^1(\VN(G)) \to \L^2(\VN(G))}$ for any bounded multiplier $M_\varphi \co \L^1(\VN(G)) \to \L^2(\VN(G))$. Then the group $G$ is amenable.
\end{prop}

\begin{proof}
Since the length function is conditionally negative, we can consider the markovian semigroup $(T_t)_{t \geq 0}$ of completely positive Fourier multipliers on $\VN(G)$ defined by $T_t(\lambda_s)=\e^{-t|s|}\lambda_s$ for any $s \in G$ and any $t \geq 0$. Now, we follow a well-known argument for the free group. For any finite sum $x=\sum_{s \in G} a_s \lambda_s$ of the von Neumann algebra $\VN(G)$, we have using Cauchy-Schwarz inequality
\begin{align*}
\MoveEqLeft 
\label{HaagerupIneq}
\norm{T_t(x)}_\infty 
=\norm{T_t\bigg(\sum_{n \geq 0} \sum_{|s|=n} a_s \lambda_s\bigg) }_\infty 
=\norm{\sum_{n \geq 0} \sum_{|s|=n} \e^{-tn} a_s \lambda_s }_\infty \\
&\leq \sum_{n \geq 0} \e^{-tn}\norm{ \sum_{|s|=n} a_s \lambda_s }_\infty
\ov{\eqref{Rapid-decay}}{\leq} \sum_{n \geq 0} \e^{-tn} (n+1) \bigg( \sum_{|s|=n} |a_s|^2 \bigg)^{\frac12} \\ 
&\leq \bigg(\sum_{n \geq 0}  \e^{-2tn}(n+1)^2\bigg)^{\frac12}\bigg(\sum_{n \geq 0}\sum_{|s|=n} |a_s|^2 \bigg)^{\frac12} \\
&= \bigg(\sum_{n \geq 0}  \e^{-2tn}(n+1)^2\bigg)^{\frac12} \norm{x}_2 
= \bigg(\frac{1+\e^{-2t}}{(1-\e^{-2t})^3}\bigg)^{\frac12} \norm{x}_2.
\end{align*}
Moreover, we have
$$
\bigg(\frac{1+\e^{-2t}}{(1-\e^{-2t})^3}\bigg)^{\frac12}
\underset{0}{\sim} \bigg(\frac{1}{(1-\e^{-2t})^3}\bigg)^{\frac12}
\underset{0}{\sim} \frac{1}{(2t)^{\frac{3}{2}}}.
$$
We deduce the estimate
\begin{equation}
\label{}
\norm{T_t}_{\L^2(\VN(G)) \to \L^\infty(\VN(G))} 
\lesssim \frac{1}{t^{\frac{3}{2}}}, \quad 0 < t < 1. 
\end{equation}
This means that the Coulhon-Varopoulos dimension of the semigroup $(T_t)_{t \geq 0}$ in the sense of \cite{Arh24b} \cite{Are04} is less than $\frac{3}{2}$. 
By duality and the assumption, we deduce that
\begin{equation}
\label{}
\norm{T_t}_{\cb,\L^1(\VN(G)) \to \L^2(\VN(G))} 
\lesssim \frac{1}{t^{\frac{3}{2}}}, \quad 0 < t < 1. 
\end{equation}
For any $t>0$, note that by Theorem \ref{Th-cb-L2-Multiplier} and duality, we have 
$$
\norm{\varphi_t}_{\ell^2_G}
=\norm{T_t}_{\cb,\L^1(\VN(G)) \to \L^2(\VN(G))} <\infty.
$$ 
The weak* continuity of the semigroup $(T_t)_{t \geq 0}$ on the von Neumann algebra $\VN(G)$ implies that for any $s \in G$ we have  $
\varphi_t(s) 
\xra[t \to 0]{} 1$. 
Since each map $T_t$ is completely positive, the function $\varphi_t \ov{\mathrm{def}}{=} \e^{-t|\cdot|}$ is a continuous positive definite function with $\varphi_t(1)$ since each Fourier multiplier $T_t \co \VN(G) \to \VN(G)$ is unital.

Recall that by a combination of \cite[Proposition 18.3.5 p.~357]{Dix77} and \cite[Proposition 18.3.6 p.~358]{Dix77} a locally compact group $G$ is amenable if and only if the function 1 is the uniform limit over every compact set of square-integrable continuous positive-definite functions\footnote{\thefootnote. We can replace ``square-integrable continuous positive-definite functions'' by ``continuous positive definite functions of compact support''.}. So we conclude that $G$ is amenable.
%
\end{proof}

%

\begin{remark} \normalfont
In a subsequent publication, we will give examples of bounded Fourier multipliers on some groups  from $\L^1(\VN(G))$ into $\L^p(\VN(G))$ which are not completely bounded with $p>1$. Note, that it is possible to give an rather elementary proof of this fact for the free group with $p=2$ without the proof of Proposition \ref{Prop-amenable}.
\end{remark}


\subsection{Bounded versus completely bounded operators from $\L^1(\cal{M})$ into $\L^p(\cal{M})$} 
\label{Bounded-vs-cb}

In this section, we prove Theorem \ref{Th-subhomo} that in general the bounded norm and the completely bounded norm are different for operators from $\L^1(\cal{M})$ into $\L^p(\cal{M})$ where $\cal{M}$ is a semifinite von Neumann algebra in sharp contrast with the case of Fourier multipliers on co-amenable compact quantum groups of Kac type. This result in the case $p=\infty$ can be seen as a von Neumann algebra generalization of \cite{Los84} (see also \cite[pp. 118-120]{KaL18}). See also \cite{HuT83}, \cite{Osa91} and \cite{Smi83} for strongly related papers.

Recall that if $T \co X \to Y$ is a bounded operator between Banach spaces we have by \cite[(A.2.3) p. 333]{EfR00}:
\begin{equation}
\label{isom-quotient-mappings}
T \text{ is a quotient mapping }
\iff T^* \text{ is an isometry}.
\end{equation}

Recall that a $\mathrm{C}^*$-algebra $A$ is $n$-subhomogeneous \cite[Definition 2.7.6]{BrO08} if all its irreducible representations have dimension at most $n$ and subhomogeneous if $A$ is $n$-subhomogeneous for some $n$. It is known, e.g. \cite[Lemma 2.4]{ShU1}, that a von Neumann algebra $\cal{M}$ is $n$-subhomogeneous if and only if there exist some distinct integers $n_1,\ldots,n_k \leq n$, some \textit{abelian} von Neumann algebras $A_i \not=\{0\}$ and a $*$-isomorphism 
\begin{equation}
\label{AVN-subh}
\cal{M}
=\M_{n_1}(A_1) \oplus \cdots \oplus \M_{n_k}(A_k).
\end{equation} 

The following generalize \cite[Lemma 2.3]{LeZ22}. Note that the proof is simpler in the case where $\cal{M}$ is finite since in this case we have an inclusion $\cal{M} \subset \L^1(\cal{M})$.

\begin{prop}
\label{prop-embed-subhom}
Let $\cal{M}$ be a semifinite von Neumann algebra equipped with a normal semifinite faithful trace $\tau$. Suppose that $\cal{M}$ is not $n$-subhomogeneous. There exists a non zero $*$-homomorphism $\gamma \co \M_{n+1} \to \cal{M}$ valued in $\cal{M} \cap \L^1(\cal{M})$. 
\end{prop}

\begin{proof}
By \cite[Section V]{Tak02}, we can consider the direct sum decomposition $\cal{M} = \cal{M}_\I \oplus \cal{M}_\II$ of $\cal{M}$ into a type I summand $\cal{M}_\I$ and a type II summand $\cal{M}_\II$.

\vspace{0.2cm}

Assume that $\cal{M}_\II \not= \{0\}$. According to \cite[Lemma 6.5.6]{KaR97}, there exist $n+1$ equivalent mutually orthogonal projections $e_1,\ldots,e_{n+1}$ in $\cal{M}_\II$ such that $e_1+\cdots+e_{n+1}=1$. Then  by \cite[Proposition V.1.22]{Tak02} and its proof, we have a $*$-isomorphism $\cal{M}_\II = \M_{n+1} \otvn (e_1 \cal{M}_\II e_1)$. Let $f$ be a non zero projection of $e_1 \cal{M}_\II e_1$ with finite trace. For any positive matrix $x \in \M_{n+1}$, we have
$$
\tau_{\cal{M}_\II}(x \ot f)
=(\tr \ot \tau_{e_1 \M_\II e_1})(x \ot f)
=\tr(x) \tau_{e_1 \M_\II e_1} (f) 
<\infty.
$$ 
Hence the mapping $\gamma \co \M_{n+1} \to \cal{M}_\II \subset \cal{M}$, $x \mapsto x \ot f$ is a non zero $*$-homomorphism taking values in $\L^1(\cal{M})$.

If $\cal{M}_\II =\{0\}$, then $\cal{M}=\cal{M}_\I$ is of type I. Since $\cal{M}$ is not subhomogeneous, it follows from the structure result \eqref{AVN-subh} and \cite[Theorem V.1.27 p. 299]{Tak02} that there exists a Hilbert space $H$ with  $\dim(H) \geq n+1$ and an abelian von Neumann algebra $A$ such that $\cal{M}$ contains $\B(H) \otvn A$ as a summand. Let $e \in \B(H)$ be a rank one projection and define $\tau_{A} \co A_+ \to [0,\infty]$ by 
$$
\tau_{A}(z) 
\ov{\mathrm{def}}{=} \tau_{\cal{M}}(e \ot z).
$$ 
Then $\tau_{A}$ is a normal semifinite faithful trace and $\tau_{\cal{M}}$ coincides with $\tr \ot \tau_{A}$ on $\B(H)_+ \ot A_+$. Let $f \in A$ be a non zero projection with finite trace. For any finite rank $a \in \B(H)_+$, it follows from the foregoing that 
$$
\tau_\cal{M}(a \ot f)
<\infty.
$$ 
Now let $(e_1,\ldots,e_{n+1})$ be an orthonormal family in $H$. Then the mapping $\gamma \co \M_{n+1} \to \cal{M}_\II$, $[a_{ij}] \to \sum_{i,j=1}^{n+1} a_{ij} \ovl{e_j} \ot e_i \ot f$ is a non zero $*$-homomorphism and the restriction of $\tau_{\cal{M}}$ to the positive part of its range is finite. Hence $\gamma$ is valued in $\L^1(\cal{M})$.
\end{proof}

Below, $\ot_\epsi$ denotes the injective tensor product of \cite[Chapter 3]{Rya02}, $\otpb$ is the projective tensor product of \cite[Chapter 2]{Rya02}, $\otpb$  denotes the operator space projective tensor product. Moreover $\C_n$ and $\mathrm{R}_n$ denotes a Hilbert column space or a Hilbert row space of dimension $n$ defined in \cite[1.2.23]{BLM04}. For any $z \in X \ot Y$, we recall that 
\begin{equation}
\label{def-injective-norm}
\norm{z}_{X \ot_\epsi Y}
\ov{\mathrm{def}}{=} \sup \big\{|\langle z, x' \ot y' \rangle | : \norm{x'}_{X^*} \leq 1, \norm{y'}_{Y^*} \leq 1 \big\}.
\end{equation}
For any Banach spaces $X$ and $Y$ and any operator space $E$ and $F$, we have isometries
\begin{equation}
\label{dual-proj}
(X \otpb Y)^*
=\B(X,Y^*)
\quad \text{and} \quad
(E \otpb F)^*
=\CB(E,F^*).
\end{equation}
The following lemma is stated in the case $p=\infty$ in \cite{Los84} with a more complicated proof.
\begin{lemma}
\label{Elementary-lemma}
Let $n \geq 1$ be an integer. Suppose $1 \leq p \leq \infty$. The canonical map $\Phi_{n,p} \co \M_n \ot_{\min} S^p_n \to \M_n \ot_\epsi S^p_n$, $x \ot y \to x \ot y$ is bijective and satisfies
$$
\norm{\Phi_{n,p}^{-1}} 
\geq n^{1-\frac{1}{2p}}.
$$
\end{lemma}

\begin{proof}
Note that using \cite[Corollary 2.7.7]{Pis03} in the second equality, we have isometrically
\begin{align}
\label{R-C-Cp}
\MoveEqLeft
\mathrm{C}_n \ot_{\min} \mathrm{C}_n^p
=\mathrm{C}_n \ot_{\min} (\mathrm{C}_n,\mathrm{R}_n)_{\frac{1}{p}} 
=(\mathrm{C}_n \ot_{\min} \mathrm{C}_n,\mathrm{C}_n \ot_{\min} \mathrm{R}_n)_{\frac{1}{p}} 
=\big(S^2_n,S^\infty_n\big)_{\frac{1}{p}}
=S^{2p}_n.
\end{align} 
We deduce that
$$
\norm{\sum_{i=1}^{n} e_{i1} \ot e_{i1}}_{\M_n \ot_{\min} S^p_n}
=\norm{\sum_{i=1}^{n}  e_i \ot e_i}_{\mathrm{C}_n \ot_{\min} \mathrm{C}^p_n}
=\norm{\I_n}_{S^{2p}_n}
\ov{\eqref{R-C-Cp}}{=} n^{1-\frac{1}{2p}}.
$$
On the other hand, using the injectivity property \cite[Proposition 4.3]{DeF93} of $\ot_\epsi$ in the first equality and the Cauchy-Schwarz inequality we have
\begin{align*}
\MoveEqLeft
\norm{\sum_{i=1}^{n} e_{i1} \ot e_{i1}}_{\M_n \ot_\epsi S^p_n}
=\norm{\sum_{i=1}^{n} e_{i} \ot e_{i}}_{\ell^2_n \ot_\epsi \ell^2_n}
\ov{\eqref{def-injective-norm}}{=} \sup\left\{\bigg|\sum_{i=1}^{n} y_i z_i\bigg|: \norm{y}_{\ell^2_n} \leq 1,\norm{z}_{\ell^2_n} \leq 1 \right\} 
\leq 1.
\end{align*} 
\end{proof}

With this lemma, we can prove the following result.

\begin{thm}
\label{Th-subhomo}
Let $\cal{M}$ be a semifinite von Neumann algebra. Suppose $1 \leq p \leq \infty$. The following are equivalent.
\begin{enumerate}
	\item The canonical map $\Theta \co \CB(\L^1(\cal{M}),\L^p(\cal{M})) \to \B(\L^1(\cal{M}),\L^p(\cal{M}))$ is an isomorphism.
	\item $\cal{M}$ is subhomogeneous (in particular $\cal{M}$ is of type $\I$).
\end{enumerate}
In this case, if $C$ is a positive constant such that $\norm{T}_{\L^1(\cal{M}) \to \L^p(\cal{M})} \geq C\norm{T}_{\cb,\L^1(\cal{M}) \to \L^p(\cal{M})}$ for any completely bounded map $T \co \L^1(\cal{M}) \to \L^p(\cal{M})$ then $\cal{M}$ is $\big \lceil{  (\frac{1}{C})^{\frac{2p}{2p-1}} }\big \rceil$-subhomogeneous.
\end{thm}

\begin{proof}
Recall that we have canonical isometries 
$$
\CB(\L^1(\cal{M}),\L^p(\cal{M}))
\ov{\eqref{dual-proj}}{=} \big(\cal{M}_* \otp \L^{p^*}(\cal{M})\big)^*
\quad \text{and} \quad
\B(\L^1(\cal{M}),\L^p(\cal{M}))
\ov{\eqref{dual-proj}}{=} \big(\cal{M}_* \otpb \L^{p^*}(\cal{M})\big)^*.
$$

1. $\Rightarrow$ 2: Suppose that there exists a constant $C > 0$ such that 
\begin{equation}
\label{ine-iso}
\norm{\Theta(x)}_{(\cal{M}_* \otpb \L^{p^*}(\cal{M}))^*} 
\geq C \norm{x}_{(\cal{M}_* \otp \L^{p^*}(\cal{M}))^*},\quad x \in \CB(\L^1(\cal{M}),\L^p(\cal{M})).
\end{equation}
We define $n \ov{\mathrm{def}}{=} \big\lfloor{ (\frac{1}{C})^{\frac{2p}{2p-1}} } \big\rfloor +1$. We have $(\frac{1}{C})^{\frac{2p}{2p-1}} < \big\lfloor{ (\frac{1}{C})^{\frac{2p}{2p-1}} } \big\rfloor+1
=n$. Hence 
\begin{equation}
\label{Estimate-impossible}
\frac{1}{C}
<n^{\frac{2p-1}{2p}}
=n^{1-\frac{1}{2p}}.
\end{equation}
Suppose that $\cal{M}$ is not $\big \lfloor{ (\frac{1}{C})^{\frac{2p}{2p-1}} }\big \rfloor$-subhomogeneous. By Proposition \ref{prop-embed-subhom}, there exists a non zero $*$-homomorphism $\gamma \co \M_n \to \cal{M}$ taking values in $\cal{M} \cap \L^1(\cal{M})$. Let $\tau'=\tau \circ \gamma \co \M_n \to \C$. Then $\tau'$ is a non zero trace on the factor $\M_n$ hence there exists $k > 0$ such that $\tau' =k\tr$.  We deduce a complete isometry $J_p\ov{\mathrm{def}}{=}k^{-\frac{1}{p}}\gamma \co S^p_n \to \L^p(\cal{M})$. Since $\M_n$ is finite-dimensional, 
 the map $J_\infty$ is weak* continuous in the case $p=\infty$. The canonical embedding of $J_\infty \ot J_p \co \M_n \ot_{\min} S^p_n \to \cal{M} \ot_{\min} \L^p(\cal{M})$ is an (complete) isometry. By \eqref{isom-quotient-mappings}, the preadjoint map $(J_p)_* \co \L^{p^*}(\cal{M}) \to S^{p^*}_n$ is a quotient map. By \cite[Proposition 2.5]{Rya02}, we deduce that the tensor product $(J_\infty)_* \ot (J_p)_* \co \cal{M}_* \otpb \L^{p^*}(\cal{M}) \to S^1_n \otpb S^{p^*}_n$ is also a quotient map. From \eqref{isom-quotient-mappings}, we obtain that the dual map
$$
\big((J_\infty)_* \ot (J_p)_*\big)^* \co \big(S^1_n \otpb S^{p^*}_n\big)^* \to \big(\cal{M}_* \otpb \L^{p^*}(\cal{M})\big)^* 
$$
is an isometry.

Now, we will show that the following diagram commutes where $j$ is defined by \eqref{belle-injection}:
$$
 \xymatrix @R=1cm @C=2cm{
       \cal{M} \ot_{\min} \L^p(\cal{M})  \ar@{^{(}->}[r]^{j}& \CB(\L^1(\cal{M}),\L^p(\cal{M})) \ar@{^{(}->}[r]^{\Theta}&\big(\cal{M}_* \otpb \L^{p^*}(\cal{M})\big)^*\\
       \M_n \ot_{\min} S^p_n   \ar[u]^{J_\infty \ot J_p} \ar[rr]_{\Phi_{n,p}} & &\big(S^1_n \otpb S^{p^*}_n\big)^*=\M_n \ot_\epsi S^p_n \ar[u]_{((J_\infty)_* \ot (J_p)_*)^*}
}
$$
For any $x \in \M_n$ and any $y \in S^p_n$, we have
\begin{align*}
\MoveEqLeft
\Theta \circ j\circ (J_\infty \ot J_p)(x \ot y)            
=\Theta \circ j(J_\infty(x) \ot J_p(y)) \\
&=\Theta\big(\langle J_\infty(x),\cdot \rangle_{\cal{M},\L^1(\cal{M})} J_p(y)\big)
=\big\langle J_\infty(x), \cdot \big\rangle_{\cal{M},\L^1(\cal{M})} \ot \big\langle J_p(y), \cdot \big\rangle_{\L^p(\cal{M}),\L^{p^*}(\cal{M})}
\end{align*} 
and
\begin{align*}
\MoveEqLeft
((J_\infty)_* \ot (J_p)_*)^*\Phi_{n,p}(x \ot y)            
=((J_\infty)_* \ot (J_p)_*)^*\big(\langle x, \cdot \rangle_{\M_n,S^1_n}\langle y, \cdot \rangle_{\M_n,S^1_n}\big) \\
&=\big\langle x, (J_\infty)_*(\cdot) \big\rangle_{\M_n,S^1_n} \ot \big\langle y, (J_p)_*(\cdot) \big\rangle_{\M_n,S^1_n}.
\end{align*} 
For any $z \in M \ot \L^p(M)$, we deduce that
\begin{align*}
\MoveEqLeft
\bnorm{\Phi_{n,p}(z)}_{\M_n \ot_\epsi S^p_n}            
=\bnorm{\big((J_\infty)_* \ot (J_p)_*\big)^*\Phi_{n,p}(z)}_{(\cal{M}_* \otpb \L^{p^*}(\cal{M}))^*}  
=\bnorm{\Theta \circ j(J_\infty \ot J_p)(z)}_{(\cal{M}_* \otpb \L^{p^*}(\cal{M}))^*}  \\
&\ov{\eqref{ine-iso}}{\geq} C \bnorm{j(J_\infty \ot J_p)(z)}_{\cal{M} \ot_{\min} \L^p(\cal{M})}
= C \norm{z}_{\cal{M} \ot_{\min} \L^p(\cal{M})}.
\end{align*}
So
$$
\norm{\Phi_{n,p}^{-1}(z)}_{\M_n \ot_{\min} S^p_n}
\leq \frac{1}{C}\norm{z}_{\M_n \ot_\epsi S^p_n}
$$
We deduce that $\norm{\Phi_{n,p}^{-1}}_{\M_n \ot_\epsi S^p_n \to \M_n \ot_{\min} S^p_n} \leq \frac{1}{C}$. By Lemma \ref{Elementary-lemma}, we obtain $n^{1-\frac{1}{2p}} \leq \frac{1}{C}$. It is impossible by \eqref{Estimate-impossible}.

2. $\Rightarrow$ 1: It is not difficult with \eqref{AVN-subh} and left to the reader.
\end{proof}

\begin{remark} \normalfont
We can generalize almost verbatim the result to the case of maps from $\L^q(\cal{M})$ into $\L^p(\cal{M})$ using a folklore generalization of \eqref{R-C-Cp}.
\end{remark}

\begin{remark} \normalfont
For the type III case, we can use a completely isometric version of \cite[Theorem 3.5]{PiX03} combined with Theorem \ref{Th-subhomo} in the case of the unique separable hyperfinite factor of type $\II_1$.
\end{remark}

Now, we will examine when the von Neumann algebra $\VN(G)$ of a locally compact group $G$ is subhomogeneous in some particular classes of groups. Let $G$ be a topological group. Recall that a representation $\pi$ of $G$ is said of type I if the von Neumann algebra $\pi(G)''$ is of type I, see \cite[Theorem 7.C.1]{BeH20}. We say thatthe group $G$ is type I \cite[Definition 6D.1]{BeH20} if all its representations are type I. 

By a well-known result of Thoma \cite[Theorem 7.D.1]{BeH20}, a discrete group $G$ is of type I if and only if $G$ is virtually abelian, i.e.~$G$ has an abelian subgroup of finite index. Moreover, it is known by \cite{Kan69} and \cite[Theorem 2]{Smi1} that the von Neumann algebra $\VN(G)$ of a discrete group is of type $\I$ if and only if $G$ is virtually abelian. We deduce that if the von Neumann $\VN(G)$ satisfies the equivalent properties of Theorem \ref{Th-subhomo}, then $G$ is virtually abelian. We will show the converse.

Recall that a locally compact group $G$ is a Moore group \cite[Definition 12.4.14]{Pal01} if all continuous topologically irreducible unitary representations of $G$ are finite-dimensional. If $H$ is a Hilbert space, we denote by $\ovl{H}$ the conjugate Hilbert space of $H$. Suppose that $G$ is unimodular postliminal second countable  locally compact group. Let $\chi \mapsto \mathcal{H}_\chi$ be the canonical field of Hilbert spaces on $\widehat{G}$. By \cite[Theorem p.~368]{Dix77}, there exists a positive measure $\mu$ on $\widehat{G}$ and an isometric isomorphism $W \co  \L^2(G) \to \int_{\hat{G}}^{\oplus}\big(\mathcal{H}_\chi \ot \ovl{\mathcal{H}_\zeta}\big) \d \mu(\chi)$ which implies an identification 
\begin{equation}
\label{Identification-AVN-direct}
\VN(G)
=\int_{\hat{G}}^{\oplus} \big(\B(\mathcal{H}_\chi) \ot \mathbb{C}\big) \d \mu(\chi)
\end{equation}
(the book \cite[18.9.2 p.~371]{Dix77} refers to \cite[p.~260]{Dix3} for the non-second countable case but it is not clear).

\begin{prop}
\label{prop-alpha-unimodular}
Let $G$ be a second countable locally compact group such that there exists an open abelian subgroup of finite index in $G$. Then the von Neumann algebra $\VN(G)$ is subhomogeneous. 
\end{prop}

\begin{proof}
By \cite[Theorem 1]{Moo72} (see also \cite[Theorem 12.4.26]{Pal01}) there exists an integer $N$ such that $\deg \pi \leq N$ for any $\pi \in \widehat{G}$. In particular, $G$ is a Moore-group. By \cite[Theorem 12.4.16]{Pal01}, we infer that $G$ is of type I and unimodular. By \cite[Theorem 8.F.3]{BeH20}, we infer that $G$ is a unimodular postliminal (i.e.~GCR group by \cite[Definition 6.E.2.]{BeH20}), second countable locally compact group. By \eqref{Identification-AVN-direct}, we deduce a $*$-isomorphism $\VN(G)=\int_{\hat{G}}^{\oplus} \B(\mathcal{H}_\chi) \d \mu(\chi)$ where $\dim \mathcal{H}_\zeta \leq N$ for some integer $N \geq 1$. 
Now, we infer that $\VN(G)$ is a subalgebra of $\big(\int_{\hat{G}}^{\oplus} \M_N \d \mu(\chi)\big)_{\L^p}$ which is $*$-isomorphic to $\L^\infty(\hat{G},\mu,\M_N)$. Since the algebra $\L^\infty(\hat{G},\mu,\M_N)$ is subhomogeneous, we conclude that $\VN(G)$ is also subhomogeneous by \cite[Proposition 2.3 (a)]{ShU1}.
\end{proof}

For discrete groups, the converse is available. Note that the assumption ``second countable'' can be removed if the previous Plancherel formula works for non-second countable groups. 

\begin{prop}
\label{th-vir-alpha}
Let $G$ be countable discrete group. Then the von Neumann algebra $\VN(G)$ is subhomogeneous if and only if there exists an abelian subgroup of finite index in $G$, that is if $G$ is virtually abelian.
\end{prop}

\begin{proof}
$\Rightarrow$: Suppose that $\VN(G)$ is subhomogeneous, hence of type $\I$. We conclude with \cite{Kan69} or \cite[Theorem 2]{Smi1} that the group $G$ has an abelian subgroup of finite index, that is $G$ is virtually abelian.

$\Leftarrow$: That is a particular case of Proposition \ref{prop-alpha-unimodular}.
\end{proof}

Let $G$ be a compact group. Recall that we have a $*$-isomorphism
\begin{equation}
\label{Identif-VNG-compact}
\VN(G) \to \bigoplus_{\pi \in \hat{G}} \M_{\deg \pi}, \quad \lambda_s \mapsto (\ovl{\pi}(s))_{s \in G}.
\end{equation}

\begin{prop}
\label{Prop-Lp-NC-compact-alpha}
Let $G$ be a second countable compact group. Then $\VN(G)$ is subhomogeneous if and only if there exists an open abelian subgroup of finite index in $G$.
\end{prop}


\begin{proof}
$\Rightarrow$: Suppose that $\VN(G)$ is subhomogeneous. We infer that the degrees of the representations $\pi$ of $\hat{G}$ are necessarily bounded. Now, we conclude with \cite[Theorem 1]{Moo72} that there exists an open abelian subgroup of finite index in $G$.

$\Leftarrow$: It suffices to use Proposition \ref{prop-alpha-unimodular}. We can give another argument. Suppose that there exists an open abelian subgroup of finite index in $G$. By \cite[Theorem 1]{Moo72}, there is an integer $N$ such that $\deg \pi \leq N$ for any representation $\pi \in \widehat{G}$. So $\VN(G)$ is a subalgebra of $\bigoplus_{\pi \in \hat{G}} \M_N=\ell^\infty_{\hat{G}}(\M_N)$. Since this last algebra is subhomogeneous, we conclude that $\VN(G)$ is also subhomogeneous.
\end{proof}

\begin{remark}
\normalfont
It would be interesting to examine the case of locally compact quantum groups. At present time, no attempt was made. Note that it is known \cite[Theorem 6.3]{KrS18} that a compact quantum group all whose irreducible representations have dimension bounded by a fixed constant must be of Kac type. By \cite[Lemma 4.9]{ACN20}, if $\QG$ is a locally compact quantum group such that its dual $\hat{\QG}$ has bounded degree then the algebra $\L^\infty(\hat{\QG})$ is subhomogeneous. If in addition $\hat{\QG}$ has trivial scaling group, then by \cite[Lemma 4.8]{ACN20} $\QG$ is necessarily co-amenable.
\end{remark}

\begin{cor}
\label{Cor-abelian-AVN}
Let $\cal{M}$ be a semifinite von Neumann algebra. Suppose $1<p<\infty$. The following are equivalent.
\begin{enumerate}
	\item The canonical map $\CB(\L^1(\cal{M}),\L^p(\cal{M})) \to \B(\L^1(\cal{M}),\L^p(\cal{M}))$ is an isometry.
	\item $M$ is abelian.
\end{enumerate}
\end{cor}

\begin{proof}
2. $\Rightarrow$ 1.: Suppose that the von Neumann algebra $\cal{M}$ is abelian. In this case, we have by \cite[p.~72]{Pis03} a complete isometry $\L^1(\cal{M})=\max \L^1(\cal{M})$. Then the point 1 is a consequence of \eqref{max-et-cb}.

1. $\Rightarrow$ 2.: This is a consequence of Theorem \ref{Th-subhomo}.
\end{proof}

A direct consequence is the following corollary. 

\begin{cor}
Let $G$ be a locally compact group. The following are equivalent.
\begin{enumerate}
	\item We have a canonical isometry $\CB(\L^1(\VN(G)),\VN(G))=\B(\L^1(\VN(G)),\VN(G))$.
	\item The group $G$ is abelian.
\end{enumerate}
\end{cor}

We close this Section by investigating the Schur multipliers from the Schatten trace class $S^1_I$ into the space $\B(\ell^2_I)$ of bounded operators on the Hilbert space $\ell^2_I$ where $I$ is an index set. 



\begin{prop}
\label{prop-cb-Schur-mult}
Let $I$ be an index set. Any $A \in \B(\ell^2_I)$ induces a completely bounded Schur multiplier $M_A \co S^1_I \to \B(\ell^2_I)$ and we have
\begin{equation}
\label{Schur-S1-to-Sinfty}
\norm{M_A}_{\cb, S^1_I \to \B(\ell^2_I)}
=\norm{A}_{\B(\ell^2_I)}.
\end{equation}
\end{prop}

\begin{proof}
Recall that we have a normal injective $*$-homomorphism $\Delta \co \B(\ell^2_I) \to \B(\ell^2_I) \otvn \B(\ell^2_I) $, $e_{ij} \mapsto e_{ij} \ot e_{ij}$. Using \cite[Proposition 1.2.4]{BLM04}, we obtain a complete isometry 
$$
\begin{array}{ccccccc}
\Psi \co\B(\ell^2_I) & \longrightarrow  & \B(\ell^2_I) \otvn \B(\ell^2_I) &\ov{\eqref{Ident-magic}}{\longrightarrow} & \CB(S^1_I,\B(\ell^2_I)) \\
    A  &       \longmapsto  &    \Delta(A)                     & \longmapsto & (B \mapsto \big(\la \cdot , B \ra_{\B(\ell^2_I),S^1_I} \ot \Id\big)(\Delta(A))  \\
\end{array}.
$$
 For any $i,j,k,l \in I$, we have
\begin{align*}
\MoveEqLeft
\Psi(e_{ij})(e_{kl})           
\ov{\eqref{Ident-magic}}{=} \langle e_{ij}, e_{kl} \rangle_{\B(\ell^2_I),S^1_I} e_{ij}
=\delta_{i=k}\delta_{j=l} e_{ij}
=M_{e_{ij}}(e_{kl}).
\end{align*} 
Hence $\Psi(e_{ij})=M_{e_{ij}}$. It is not difficult to conclude that $\Psi(A)=M_A$ for any $A \in \B(\ell^2_I)$. 
\end{proof}

\subsection{$\ell^p(S^p_d)$-summing Fourier multipliers on $\L^\infty(\QG)$}
\label{summing-multipliers}

Suppose $1 \leq p <\infty$ and let $d$ be an element of $\N \cup\{\infty\}$. Let $E$ and $F$ be operator spaces and let $T \co E \to F$ be a linear map. Following \cite[p.~58]{Pis98}, we say that $T$ is $\ell^p(S^p_d)$-summing if $T$ induces a bounded map $\Id_{\ell^p(S^p_d)} \ot T \co \ell^p(S^p_d) \ot_{\min} E \to \ell^p(S^p_d(F))$. In this case, we let
\begin{equation}
\label{Def-ellp-Spd-summing-norm}
\norm{T}_{\pi_{p,d},E \to F}
\ov{\mathrm{def}}{=} \bnorm{\Id_{\ell^p(S^p_d)} \ot T}_{\ell^p(S^p_d) \ot_{\min} E \to \ell^p(S^p_d(F))}.
\end{equation}
If $d=\infty$, it is not difficult to see that the $\ell^p(S^p)$-summing maps coincide with the completely $p$-summing maps\footnote{\thefootnote. Recall that a map $T \co E \to F$ is completely $p$-summing \cite[p.~51]{Pis98} if $T$ induces a bounded map $\Id_{S^p} \ot T \co S^p \ot_{\min} E \to S^p(F)$.} and we have 
\begin{equation}
\label{}
\norm{T}_{\pi_{p,\infty},E \to F}
=\norm{T}_{\pi_p^\circ,E \to F}.
\end{equation}

Let $T \co E \to F$ be a $\ell^p(S^p_d)$-summing map. Since we have the isometric inclusion $S^p_d(E) \subset S^p_d \ot_{\min} E \subset \ell^p(S^p_d) \ot_{\min} E$ and the isometric inclusion $S^p_d(F) \subset \ell^p(S^p_d(F))$, the map $T$ is bounded and
\begin{equation}
\label{d-bounded}
\bnorm{\Id_{S^p_d} \ot T}_{S^p_d(E) \to S^p_d(F)}
\leq \norm{T}_{\pi_{p,d},E \to F}.
\end{equation}


We need the following easy lemma.

\begin{lemma}
\label{Lemma-is-bounded}
Suppose $1 \leq p <\infty$ and let $d$ be an element of $\N \cup\{\infty\}$. Let $E,F,G,H$ be operator spaces. Let $T \co E \to F$ be a $\ell^p(S^p_d)$-summing map. If the maps $T_1 \co F \to G$ and $T_2 \co H \to E$ are completely bounded then the map $T_1 T T_2$ is $\ell^p(S^p_d)$-summing and we have
\begin{equation}
\label{Ideal-ellp}
\norm{T_1 T T_2}_{\pi_{p,d}, H \to G}
\leq \norm{T_1}_{\cb, F \to G} \norm{T}_{\pi_{p,d}, E \to F} \norm{T_2}_{\cb,H \to E}.
\end{equation}
\end{lemma}

\begin{proof}
By \cite[Proposition 8.1.5]{EfR00}, we have a well-defined (completely) bounded map $\Id_{\ell^p(S^p_d)} \ot  T_2\co \ell^p(S^p_d) \ot_{\min} H \to \ell^p(S^p_d) \ot_{\min} E$ with
\begin{equation}
\label{}
\bnorm{\Id_{\ell^p(S^p_d)} \ot T_2}_{\ell^p(S^p_d) \ot_{\min} H \to \ell^p(S^p_d) \ot_{\min} E}
\leq \norm{T_2}_{\cb,H \to E}.
\end{equation} 
Since $T \co E \to F$ is a $\ell^p(S^p_d)$-summing map, we have by definition a well-defined bounded map $\Id_{\ell^p(S^p_d)} \ot T \co \ell^p(S^p_d) \ot_{\min} E \to \ell^p(S^p_d(F))$. Moreover, we have
\begin{equation}
\label{}
\bnorm{\Id_{\ell^p(S^p_d)} \ot T_1}_{\ell^p(S^p_d(F)) \to \ell^p(S^p_d(G))}
\ov{\eqref{ine-tensorisation-os}}{\leq} \norm{T_1}_{\cb, F \to G}.
\end{equation}
We obtain the result by composition. 
\end{proof}

The following proposition gives a crucial example of $\ell^p(S^p_d)$-summing map for the sequel.

\begin{prop}
\label{Prop-inj-finite-avn}
Suppose $1 \leq p < \infty$ and let $d$ be an element of $\N \cup\{\infty\}$. If $\cal{M}$ is a hyperfinite finite von Neumann algebra equipped with a normal finite faithful trace $\tau$ then the canonical inclusion $i_p \co \cal{M} \hookrightarrow \L^p(\cal{M})$ is $\ell^p(S^p_d)$-summing and
\begin{equation}
\label{norm-ip}
\norm{i_p}_{\pi_{p,d},\cal{M} \to \L^p(\cal{M})}
=(\tau(1))^{\frac{1}{p}}.
\end{equation}
\end{prop}

\begin{proof}
We can suppose that $\tau(1)=1$. Recall that if $E$ is an operator space, we have essentially\footnote{\thefootnote. The result is stated with completely bounded instead of completely contractive.} by \cite[p.~37]{Pis98} a contractive inclusion $\cal{M} \ot_{\min} E \hookrightarrow \L^1(\cal{M},E)$. Moreover, we have $\L^p(\cal{M},E)=(\cal{M} \ot_{\min} E,\L^1(\cal{M},E))_{\frac{1}{p}}$ by definition. By \cite[Proposition 2.4]{Lun18}, we deduce a contractive inclusion $\cal{M} \ot_{\min} E \hookrightarrow \L^p(\cal{M},E)$. Taking $E=\ell^p(S^p_d)$ and using \cite[(3.6)]{Pis98} in the last equality, we obtain a contractive inclusion
$$
\ell^p(S^p_d) \ot_{\min} \cal{M}
\cong \cal{M} \ot_{\min} \ell^p(S^p_d) 
\hookrightarrow \L^p(\cal{M},\ell^p(S^p_d)) 
\cong \ell^p(S^p_d(\L^p(\cal{M})).
$$
We conclude by \eqref{Def-ellp-Spd-summing-norm} that the inclusion map $i_p \co \cal{M} \hookrightarrow \L^p(\cal{M})$ is $\ell^p(S^p_d)$-summing with $\norm{i_p}_{\pi_{p,d},\cal{M} \to \L^p(\cal{M})} \leq 1$. Note that $\norm{1}_{\L^p(\cal{M})}=1$. Since 
$$
\norm{i_p}_{\pi_{p,d},\cal{M} \to \L^p(\cal{M})} \ov{\eqref{d-bounded}}{\geq} \bnorm{\Id_{S^p_d} \ot i_p}_{S^p_d(\cal{M}) \to S^p_d(\L^p(\cal{M}))} \geq \norm{i_p}_{\cal{M} \to \L^p(\cal{M})},
$$
we obtain the reverse inequality.  
\end{proof}

\begin{remark} \normalfont
If $\mathbb{F}_n$ is the free group with $n$ generators where $n \geq 2$ is an integer then it is stated in \cite[Remark 3.2.2.5]{Jun99} that the canonical contractive inclusion $i_p \co \VN(\mathbb{F}_n) \hookrightarrow \L^p(\VN(\mathbb{F}_n))$ is not completely $p$-summing, i.e.~is not $\ell^p(S^p)$-summing. So the hyperfiniteness assumption in Proposition \ref{Prop-inj-finite-avn} cannot be removed. It would be interesting to characterize the finite von Neumann algebras such that the canonical inclusion $i_p \co  \cal{M}\hookrightarrow \L^p(\cal{M})$ is $\ell^p(S^p_d)$-summing.
\end{remark}

Note that in the case of a co-amenable compact quantum group of Kac type, the assumption ``hyperfinite'' is useless by \cite[Theorem 4.5]{Rua96}.

\begin{prop}
\label{prop-carac-Fourier}
Let $(\cal{M},\tau,\Delta,R)$ is a hyperfinite finite quantum hypergroup. Suppose $1 \leq p <\infty$ and let $d$ be an element of $\N \cup\{\infty\}$. If $T \co \L^p(\cal{M}) \to \L^\infty(\cal{M})$ is a completely bounded left multiplier then $T$ induces a $\ell^p(S^p_d)$-summing left multiplier $T \co \L^\infty(\cal{M}) \to \L^\infty(\cal{M})$. In this case, we have
\begin{equation}
\label{Fourier-sommant}
\norm{T}_{\pi_{p,d},\L^\infty(\cal{M}) \to \L^\infty(\cal{M})}
\leq (\tau(1))^{\frac{1}{p}}\norm{T}_{\cb,\L^p(\cal{M}) \to \L^\infty(\cal{M})}.
\end{equation}
\end{prop}

\begin{proof}
By Proposition \ref{Prop-inj-finite-avn}, the map $i_p \co \L^\infty(\cal{M}) \xhookrightarrow{} \L^p(\cal{M})$ is $\ell^p(S^p_d)$-summing and the norm is $\norm{i_p}_{\pi_{p,d},\L^\infty(\cal{M}) \to \L^p(\cal{M})}=(\tau(1))^{\frac{1}{p}}$. By considering the composition
$$
\L^\infty(\cal{M}) \xhookrightarrow{i_p} \L^p(\cal{M}) \xra{T} \L^\infty(\cal{M})
$$
we deduce by the ideal property \eqref{Ideal-ellp} that $T \co \L^\infty(\cal{M}) \to \L^\infty(\cal{M})$ is $\ell^p(S^p_d)$-summing and that
\begin{align*}
\MoveEqLeft
\norm{T}_{\pi_{p,d}, \L^\infty(\cal{M}) \to \L^\infty(\cal{M})}
\ov{\eqref{Ideal-ellp}}{\leq} \norm{T}_{\cb,\L^p(\cal{M}) \to \L^\infty(\cal{M})}\norm{i_p}_{\pi_{p,d}, \L^\infty(\cal{M}) \to \L^p(\cal{M})} \\
&\ov{\eqref{norm-ip}}{=} (\tau(1))^{\frac{1}{p}}\norm{T}_{\cb,\L^p(\cal{M}) \to \L^\infty(\cal{M})}.            
\end{align*} 
\end{proof}

\begin{remark} \normalfont
If $G$ is an \textit{abelian} compact group $G$ equipped with its \textit{normalized} Haar measure and if $\varphi \co \hat{G} \to \mathbb{C}$ is a complex function, we have the equalities
\begin{equation}
\label{All-equal}
\norm{M_{\varphi}}_{\pi_p^\circ, \L^\infty(G) \to \L^\infty(G)}
= \norm{M_{\varphi}}_{\pi_p, \L^\infty(G) \to \L^\infty(G)}
=\norm{M_{\varphi}}_{\L^p(G) \to \L^\infty(G)}
=\norm{M_{\varphi}}_{\cb,\L^p(G) \to \L^\infty(G)}
\end{equation}
where $\norm{\cdot}_{\pi_p, \L^\infty(G) \to \L^\infty(G)}$ is the classical $p$-summing norm. 
Indeed, the first equality is essentially \cite[Remark 3.7]{JuP15} (see also \cite[Remark 5.13]{Pis98}). The third equality is \eqref{min-et-cb} since $\L^\infty(G)=\min \L^\infty(G)$ as operator space by \cite[p.~72]{Pis03}. The second equality is stated in \cite[Proposition 2]{BeM73}\footnote{\thefootnote. The mathscinet review contains useful remarks.} but the proof or the idea is not given. In respect to this beautiful tradition of not providing proofs, we does not give the argument here. However, we (really) confirm that a rather elementary proof exists\footnote{\thefootnote. At least in the case $b=\infty$ of \cite[Proposition 2]{BeM73}.}.
\end{remark}

The following easy lemma is left to the reader.

\begin{lemma}
Let $\QG$ be a co-amenable compact quantum group of Kac type. Suppose $1<p \leq \infty$. Let $T \co \L^1(\QG) \to \L^p(\QG)$ be a completely bounded Fourier multiplier. Then
\begin{equation}
\label{Equation-Delta}
\Delta T
=(\Id_{} \ot T) \Delta
\end{equation}
where we use $\Delta \co \L^p(\QG) \to \L^p(\QG \otvn \QG)$ (on the left) and $\Delta \co \L^1(\QG) \to \L^p(\QG,\L^1(\QG))$ (on the right ; see Lemma \ref{Lemma-useful}).
\end{lemma}

Let $\cal{M}$ be a von Neumann algebra equipped with a normal semifinite faithful trace. We denote by $\Mult_{a,b} \co \cal{M} \to \L^p(\cal{M})$, $x \mapsto axb$ the multiplication map.

The following is a slight variation of \cite[Theorem 2.4]{JuP15}, \cite[Theorem 2.4]{JuP15} and \cite[Remark 2.2]{JuP15}. We skip the details.

\begin{thm}
\label{thm-factorization-positivity}
Let $\cal{M}$ and $\cal{N}$ be hyperfinite von Neumann algebras such that $\cal{M}$ is equipped with a normal finite faithful trace. Let $T \co \cal{M} \to \cal{N}$ be a map. Then, the following assertions are equivalent:
\begin{enumerate}
\item $\norm{T}_{\pi_{p,d},\cal{M} \to \cal{N}} \leq C$.

\item There exists elements $a,b \in \cal{M}$ satisfying $\norm{a}_{\L^{2p}(\cal{M})} \leq 1$, $\norm{b}_{\L^{2p}(\cal{M})} \leq 1$ such that for every $x \in S^p_d \ot_{\min} \cal{M}$ we have 
\begin{equation}
\label{Ine-lp-summing}
\norm{\big(\Id_{S^p_d} \ot T\big)(x)}_{S^p(\cal{N})}
\leq C\norm{(1 \ot a)x(1 \ot b)}_{S^p(\L^p(\cal{M}))}.
\end{equation}

\item There exist elements $a,b \in \cal{M}$ satisfying $\norm{a}_{\L^{2p}(\cal{M})} \leq 1$, $\norm{b}_{\L^{2p}(\cal{M})} \leq 1$ and a bounded map $\tilde{T} \co \L^p(\cal{M}) \to \cal{N}$ such that
$$
T
=\tilde{T} \circ \Mult_{a,b}
\quad \text{and} \quad 
\norm{\tilde{T}}_{\M_d(\L^p(\cal{M})) \to \M_d(\cal{N})} 
\leq C.
$$
\end{enumerate}
Furthermore, $
\norm{T}_{\pi_{p,d},\cal{M} \to \cal{N}}=\inf\big\{C : C \text{ satisfies any of the previous conditions}\big\}$.
\end{thm}

Now, we complete Proposition \ref{prop-carac-Fourier}. This result gives implicitly the exact value of the  entanglement-assisted classical capacity of quantum channels which are Fourier multipliers.

\begin{thm}
\label{thm-carac-Fourier}
Let $\QG$ be a co-amenable compact quantum group of Kac type. We equip $\L^\infty(\QG)$ with its normalized normal finite faithful trace. Suppose $1 \leq p<\infty$. The element $a$ of $\ell^\infty(\hat{\QG})$ induces a completely $p$-summing Fourier multiplier $M_{a} \co \L^\infty(\QG) \to \L^\infty(\QG)$ if and only if it induces a completely bounded Fourier multiplier $M_{a} \co \L^p(\QG) \to \L^\infty(\QG)$. In this case, we have
\begin{equation}
\label{pidcirc}
\norm{M_a}_{\pi_p^\circ, \L^\infty(\QG) \to \L^\infty(\QG)}
=\norm{M_a}_{\cb,\L^p(\QG) \to \L^\infty(\QG)}
=\norm{M_a}_{\L^p(\QG) \to \L^\infty(\QG)}.
\end{equation}
\end{thm}

\begin{proof}
$\Rightarrow$: Suppose that we have a completely $p$-summing Fourier multiplier $M_{a_1} \co \L^\infty(\QG) \to \L^\infty(\QG)$. By Theorem \ref{thm-factorization-positivity}, there exist $a,b \in \L^{2p}(\QG)$ satisfying $\norm{a}_{\L^{2p}(\QG)} \leq 1$ and $\norm{b}_{\L^{2p}(\QG)} \leq 1$ and a completely bounded map $\tilde{M}_{a_1} \co \L^p(\QG) \to \L^\infty(\QG)$ such that 
\begin{equation}
\label{Facto-Fourier-multiplier}
M_{a_1}
=\tilde{M}_{a_1} \circ \Mult_{a,b}
\quad \text{and} \quad 
\norm{\tilde{M}_{a_1}}_{\cb,\L^p(\QG) \to \L^\infty(\QG)} 
\leq \norm{M_{a_1}}_{\pi_p^\circ, \L^\infty(\QG) \to \L^\infty(\QG)} 
\end{equation}
where $\Mult_{a,b} \co \L^\infty(\QG) \to \L^p(\QG)$. We can suppose that $a$ and $b$ are positive elements with full supports. Let $\Delta \co \L^\infty(\QG) \to \L^\infty(\QG) \otvn \L^\infty(\QG)$ be the coproduct defined by \eqref{coproduct-VNG}. Let $\E \co \L^\infty(\QG) \otvn \L^\infty(\QG) \to \L^\infty(\QG)$ be the canonical trace preserving normal faithful conditional expectation associated with $\Delta$. By using Lemma \ref{Lemma-useful}, we have a contraction $\Delta \co \L^1(\QG) \to \L^{p^*}(\QG,\L^1(\QG))$. By duality, we deduce that we have a well-defined contraction $\E \co \L^p(\QG,\L^\infty(\QG)) \to \L^\infty(\QG)$. 

Since $\tilde{M}_{a_1} \co \L^p(\QG) \to \L^\infty(\QG)$ is completely bounded, we deduce by \eqref{ine-tensorisation-os} that the map $\Id_{\L^p} \ot \tilde{M}_{a_1} \co \L^p(\QG,\L^p(\QG)) \to \L^p(\QG,\L^\infty(\QG))$ is completely bounded and that we have
\begin{align}
\MoveEqLeft
\label{Equa-138676}
\norm{\Id_{\L^p} \ot \tilde{M}_{a_1}}_{\cb,\L^p(\QG,\L^p(\QG)) \to \L^p(\QG,\L^\infty(\QG))}            
\ov{\eqref{ine-tensorisation-os}}{\leq} \norm{\tilde{M}_{a_1}}_{\cb,\L^p(\QG) \to \L^\infty(\QG)} 
\ov{\eqref{Facto-Fourier-multiplier}}{\leq}  \norm{M_{a_1}}_{\pi_p^\circ, \L^\infty(\QG) \to \L^\infty(\QG)}. 
\end{align}
By Lemma \ref{Lemma-useful} applied with the opposite quantum group $\QG^\op$ \cite[Section 4]{KuV03}, we have an isometry $\Delta^\op \co \L^2(\QG^\op) \to \L^\infty(\QG^\op,\L^2(\QG^\op))=\L^\infty(\QG,\L^2(\QG))$, i.e.~for any $x \in \L^2(\QG)=\L^2(\QG^\op)$ we have
$$
\norm{\Delta^\op(x)}_{\L^\infty(\QG,\L^2(\QG))}   
= \norm{x}_{\L^2(\QG)}.
$$ 
By \eqref{Norm-LpLq} applied with $p=\infty$ and $q=r=2$, we deduce that for any elements $c,d$ in the unit ball of $\L^4(\QG)$
\begin{equation}
\label{}
\bnorm{(c \ot 1)\Delta^\op(x)(d \ot 1)}_{\L^2(\QG \otvn \QG)}   
\leq \norm{x}_{\L^2(\QG)}.
\end{equation}
We infer that the map
\begin{equation}
\label{Theta-magic}
\L^2(\QG) \to \L^2(\QG \otvn \QG),\,  
x \mapsto (c \ot 1) \Delta^\op(x) (d \ot 1).
\end{equation}
is a well-defined contraction. Since $\Delta=\Sigma\Delta^\op$, by Fubini's theorem the map
\begin{equation}
\label{Theta-magic}
\L^2(\QG) \to \L^2(\QG \otvn \QG),\,  
x \mapsto (1 \ot c) \Delta(x) (1 \ot d).
\end{equation}
is also a contraction.

\begin{prop}
Suppose $1 \leq p \leq \infty$. The map
\begin{equation}
\label{Theta-bbis}
\Theta \co \L^p(\QG) \to \L^p(\QG \otvn \QG),\,  
x \mapsto (1 \ot a) \Delta(x) (1 \ot b).
\end{equation}
is a well-defined complete contraction.
\end{prop} 

\begin{proof}
Suppose $2 \leq p \leq \infty$. We will use Stein interpolation \cite{CwJ84} \cite[Theorem 2.1]{Voi92} \cite[Theorem 2.7]{Lun18}. For any complex number $z$ of the closed strip $\ovl{S}=\{z \in \mathbb{C} : 0 \leq \Re z \leq 1 \}$, we consider the map
$$
T_z \co \Pol(\QG) \to \L^1(\QG),\quad x \mapsto \big(1 \ot a^{\frac{pz}{2}}\big) \Delta(x) \big(1 \ot b^{\frac{pz}{2}}\big)
$$
where $\Pol(\QG)$ is the space of polynomials in $\L^\infty(\QG)$. Since $a,b$ are elements of the unit ball of $\L^{2p}(\QG)$, the elements $a^{\frac{p}{2}}$ and $b^{\frac{p}{2}}$ are elements of the unit ball of $\L^4(\QG)$. For any $t \in \R$ and any $x \in \Pol(\QG)$, we deduce that
\begin{align*}
\MoveEqLeft
\norm{T_{1+\i t}(x)}_{\L^2(\QG \otvn \QG)}            
=\bnorm{\big(1 \ot a^{\frac{p+\i t p}{2}}\big) \Delta(x) \big(1 \ot b^{\frac{p+\i t p}{2}}\big)}_{\L^2(\QG \otvn \QG)} \\
&\leq \bnorm{\big(1 \ot a^{\frac{p}{2}}\big) \Delta(x) \big(1 \ot b^{\frac{p}{2}}\big)}_{\L^2(\QG \otvn \QG)}
\ov{\eqref{Theta-magic}}{\leq} \norm{x}_{\L^2(\QG)}.
\end{align*}
We infer that the map  $T_{1+\i t} \co \L^2(\QG) \to \L^2(\QG \otvn \QG)$ is contractive. Note that by \cite[p.~139]{Pis03}, the operator space $\L^2(\QG)$ is an operator Hilbert space. We conclude that this map is even completely contractive by \cite[Proposition 7.2 (ii)]{Pis03}.

For any $t \in \R$ and any $x \in \Pol(\QG)$, we have
\begin{align*}
\MoveEqLeft
\norm{T_{\i t}(x)}_{\L^\infty(\QG) \otvn \L^\infty(\QG)}            
=\bnorm{\big(1 \ot a^{\frac{\i t p}{2}}\big) \Delta(x) \big(1 \ot b^{\frac{\i t p}{2}}\big)}_{\L^\infty(\QG) \otvn \L^\infty(\QG)}
=\norm{\Delta(x)}_{\L^\infty(\QG) \otvn \L^\infty(\QG)}.
\end{align*} 
We deduce that the map $T_{\i t} \co \L^\infty(\QG) \to \L^\infty(\QG) \otvn \L^\infty(\QG)$ is completely contractive. Moreover, for any $x \in \Pol(\QG)$, the function $\ovl{S} \to \L^1(\QG)$, $z \mapsto T_z(x)$ is analytic on $S$ and continuous and bounded on $\ovl{S}$. Finally, the functions $\R \to \L^2(\QG \otvn \QG)$, $t \mapsto T_{1+\i t}(x)$ and $\R \to \L^\infty(\QG) \otvn \L^\infty(\QG)$, $t \mapsto T_{\i t}(x)$ are continuous. By Stein's interpolation, we conclude by taking $z=\frac{2}{p}$ that the map \eqref{Theta-bbis} is a complete contraction.

The case $1 \leq p \leq 2$ is similar. 
\end{proof}

We have
\begin{align*}
\MoveEqLeft
\Delta M_{a_1}
\ov{\eqref{Equation-Delta}}{=} (\Id \ot M_{a_1}) \Delta
\ov{\eqref{Facto-Fourier-multiplier}}{=}(\Id \ot \tilde{M}_{a_1} \Mult_{a,b}) \Delta \\
&=(\Id \ot \tilde{M}_{a_1}) (\Id \ot \Mult_{a,b}) \Delta
\ov{\eqref{Theta-bbis}}{=} (\Id \ot \tilde{M}_{a_1}) \Theta.           
\end{align*} 
Since $\E \Delta=\Id_{\L^\infty(\QG)}$, we deduce that $M_{a_1}=\E \Delta M_{a_1}=\E (\Id \ot M_{a_1}) \Theta$. That means that we have the following commutative diagram.
$$
\xymatrix @R=1.5cm @C=2cm{
\L^p(\QG,\L^p(\QG)) \ar[r]^{\Id_{\L^p} \ot \tilde{M}_{a_1}}   & \ar[d]^{\E} \L^p(\QG,\L^\infty(\QG)) \\
\L^p(\QG) \ar[r]_{M_{a_1}} \ar[u]^{\Theta}   & \L^\infty(\QG)\\
  }
$$
We conclude that ${a_1}$ induces a bounded multiplier $M_{a_1} \co \L^p(\QG) \to \L^\infty(\QG)$. Since $\QG$ is co-amenable, by Theorem \ref{Thm-description-multipliers-3} this multiplier is completely bounded and we have
\begin{align*}
\MoveEqLeft
\norm{M_{a_1}}_{\cb,\L^p(\QG) \to \L^\infty(\QG)}
=\norm{M_{a_1}}_{\L^p(\QG) \to \L^\infty(\QG)} 
=\norm{\Theta(\Id_{\L^p} \ot \tilde{M}_{a_1})\E}_{\L^p(\QG) \to \L^\infty(\QG)} \\
&\leq \norm{\Theta} \norm{\Id_{\L^p} \ot \tilde{M}_{a_1}}_{\L^p(\L^p) \to \L^p(\L^\infty)} \norm{\E} 
\ov{\eqref{Equa-138676}}{\leq} \norm{M_{a_1}}_{\pi_p^\circ, \L^\infty(\QG) \to \L^\infty(\QG)}.
\end{align*}
\end{proof}


In the end of this section, we need the following which is a slight deepening of \cite[(1.7')]{Pis98}.

\begin{lemma}
Let $E,F$ be operator spaces and $T \co E \to F$ be a bounded map. If $n \geq 1$ is an integer, we have
\begin{equation}
\label{d-norm-1}
\norm{\Id_{S^p_n} \ot T}_{S^p_n(E) \to S^p_n(F)} 
=\norm{\Id_{\M_n} \ot T}_{\M_n(E)\to \M_n(F)}.
\end{equation}
\end{lemma}

\begin{proof}
Let $x \in \M_n(E)$. Consider some $a \in S^{2p}_n$, $b \in S^{2p}_n$ with $\norm{a}_{S^{2p}_n} \leq 1$, $\norm{b}_{S^{2p}_n} \leq 1$. Using \cite[(1.7)]{Pis98}, in the last equality, we obtain
\begin{align*}
\MoveEqLeft
\norm{a \cdot ((\Id_{\M_n} \ot T)(x)) \cdot b}_{S^p_n(F)}
=\norm{(\Id_{S^p_n} \ot T)(a \cdot x \cdot b)}_{S^p_n(F)} \\
&\leq \bnorm{\Id_{S^p_n} \ot T}_{S^p_n(E) \to S^p_n(F)} \norm{a \cdot x \cdot b}_{S^p_n(E)} \leq \bnorm{\Id_{S^p_n} \ot T}_{S^p(E) \to S^p_n(F)} \norm{x}_{\M_n(E)}.
\end{align*}
Taking the supremum and using again \cite[(1.7)]{Pis98}, we obtain
$$
\norm{(\Id_{\M_n} \ot T)(x)}_{\M_n(F)} 
\leq  \bnorm{(\Id_{S^p_n} \ot T)}_{S^p_n(E) \to S^p_n(F)} \norm{x}_{\M_n(E)}.
$$
Hence
$$  
\bnorm{\Id_{\M_n} \ot T}_{\M_n(E) \to \M_n(F)} 
\leq \bnorm{\Id_{S^p_n} \ot T}_{S^p_n(E) \to S^p_n(F)}.
$$
Let $x \in S^p_n(E)$. Let $a \in S^{2p}_n$, $b \in S^{2p}_n$ et $y \in \M_n(E)$ such that $x=a\cdot y\cdot b$. Using \cite[Lemma 1.6 (ii)]{Pis98} in the first inequality, we obtain
\begin{align*}
\MoveEqLeft
\norm{(\Id_{S^p_n} \ot T)(x)}_{S^p_n(F)}
=\norm{(\Id_{S^p_n} \ot T)(a\cdot y \cdot b)}_{S^p_n(F)} 
=\norm{a \cdot (\Id_{\M_n} \ot T)(y)) \cdot b}_{S^p_n(F)} \\
&\leq \norm{a}_{S^{2p}_n} \bnorm{(\Id_{\M_n} \ot T)(y)}_{\M_n(F)} \norm{b}_{S^{2p}_n} 
\leq \norm{\Id_{\M_n} \ot T}_{\M_n(E) \to \M_n(F)} \norm{a}_{S^{2p}_n} \norm{y}_{\M_n(E)} \norm{b}_{S^{2p}_n}.
\end{align*}
Passing to the infimum with \cite[Theorem 1.5]{Pis98}, we obtain that
$$
\bnorm{(\Id_{S^p_n} \ot T)(x)}_{S^p_n(F)} 
\leq \bnorm{\Id_{\M_n} \ot T}_{\M_n(E) \to \M_n(F)} \norm{x}_{S^p_n(E)}. 
$$
Finally
$$ 
\norm{\Id_{S^p_n} \ot T}_{S^p_n(E) \to S^p_n(F)} 
\leq \norm{\Id_{\M_n} \ot T}_{\M_n(E)\to \M_n(F)}. 
$$
\end{proof}

\begin{remark} \normalfont
\label{Remark-avec-d}
Let $\QG$ be a finite quantum group. We equip $\L^\infty(\QG)$ with its normalized normal finite faithful trace. We let $d_{\QG} \ov{\mathrm{def}}{=} \max_{\pi \in \hat{\QG}} d_\pi$. Suppose $1 \leq p<\infty$ and $d \geq d_{\QG}$. If $a$ of $\ell^\infty(\hat{\QG})$  we have
\begin{equation}
\label{}
\norm{M_a}_{\pi_{p,d}, \L^\infty(\QG) \to \L^\infty(\QG)}
=\norm{M_a}_{\cb,\L^p(\QG) \to \L^\infty(\QG)}
=\norm{M_a}_{\L^p(\QG) \to \L^\infty(\QG)}.
\end{equation}
%
The proof is similar. We have a $*$-isomorphism $\L^\infty(\QG) = \oplus_{\pi \in \hat{\QG}} \M_{d_\pi}$. The point is that we need to control the norm of $\Id_{\L^p(\QG)} \ot \tilde{M}_{a_1}$ by the $\ell^p(S^p_d)$-summing norm of $M_{a_1}$. We have
\begin{equation}
\label{}
\norm{\tilde{M}_{a_1}}_{S^p_d(\L^p(\QG)) \to S^p_d(\L^\infty(\QG))}
\ov{\eqref{d-norm-1}}{=} \norm{\tilde{M}_{a_1}}_{\M_d(\L^p(\QG)) \to \M_d(\L^\infty(\QG))}
\leq \norm{M_{a_1}}_{\pi_{p,d},\L^\infty(\QG) \to \L^\infty(\QG)}.
\end{equation}
Moreover, we have
$$
\Id_{\L^p(\QG)} \ot \tilde{M}_{a_1}
=\Id_{\L^p(\oplus_{\pi \in \hat{\QG}} \M_{d_\pi})} \ot \tilde{M}_{a_1}
=\Id_{\oplus_{\pi \in \hat{\QG}} S^p_{d_\pi}} \ot \tilde{M}_{a_1}
=\oplus_{\pi \in \hat{\QG}} \big(\Id_{S^p_{d_\pi}} \ot \tilde{M}_{a_1}\big).
$$ 
\end{remark}

If $d=1 < d_{\QG}$, we can prove the following result similarly with \eqref{estimate-in-d}.

\begin{thm}
\label{thm-carac-Fourier-2}
Let $\QG$ be a finite quantum group. We equip $\L^\infty(\QG)$ with its normalized normal finite faithful trace. Suppose $1 \leq p<\infty$. If $a \in \ell^\infty(\hat{\QG})$ we have
\begin{equation}
\label{}
\frac{1}{d_{\QG}}\norm{M_a}_{\L^p(\QG) \to \L^\infty(\QG)} 
\leq \norm{M_a}_{\pi_{p,1}, \L^\infty(\QG) \to \L^\infty(\QG)}
\leq \norm{M_a}_{\L^p(\QG) \to \L^\infty(\QG)}.
\end{equation}
\end{thm}

\subsection{$\ell^p(S^p_d)$-summing Schur multipliers on $\M_n$}
\label{summing-Schur-multipliers}

We refer to Theorem \ref{Th-CEA-Herz-Schur} for an application of the results of this section.

\begin{prop}
\label{prop-carac-Schur}
Suppose $1 \leq p <\infty$ and let $d$ be an element of $\N \cup\{\infty\}$. If $M_A \co \M_n \to \M_n$ is a Schur multiplier then we have
\begin{equation}
\label{Schur-sommant}
\norm{M_A}_{\pi_{p,d},\M_n \to \M_n}
\leq n^{\frac{1}{p}} \norm{M_A}_{\cb,S^p_n \to \M_n}.
\end{equation}
\end{prop}

\begin{proof}
By Proposition \ref{Prop-inj-finite-avn}, the linear map $i_p \co \M_n \xhookrightarrow{} S^p_n$ satisfies $\norm{i_p}_{\pi_{p,d},\M_n \to S^p_n}=n^{\frac{1}{p}}$. By considering the composition
$$
\M_n \xhookrightarrow{i_p} S^p_n \xra{M_A} \M_n,
$$
we deduce by the ideal property \eqref{Ideal-ellp} that
\begin{align*}
\MoveEqLeft
\norm{M_A}_{\pi_{p,d}, \M_n \to \M_n}
\ov{\eqref{Ideal-ellp}}{\leq} \norm{M_A}_{\cb,S^p_n \to \M_n} \norm{i_p}_{\pi_{p,d}, \M_n \to S^p_n} 
\ov{\eqref{norm-ip}}{=} n^{\frac{1}{p}}\norm{M_A}_{\cb,S^p_n \to \M_n}.            
\end{align*} 
\end{proof}

In the same spirit than Theorem \ref{thm-carac-Fourier}, we can prove the following result. The second inequality is sharp.

\begin{prop}
\label{thm-carac-Schur}
Suppose $1<p<\infty$. Consider some integer $n \geq 0$. If $M_A \co \M_n \to \M_n$ is a Schur multiplier then we have
\begin{equation}
\label{Summing-Schur}
\norm{M_A}_{S^p_n \to \M_n} \leq \norm{M_A}_{\pi_p^\circ, \M_n \to \M_n}
\leq n^{\frac{1}{p}}\norm{M_A}_{\cb,S^p_n \to \M_n}.
\end{equation}
\end{prop}

\begin{proof}
Suppose that the matrix $A$ induces a completely $p$-summing Schur multiplier $M_{A} \co \M_n \to \M_n$. By Theorem \ref{thm-factorization-positivity}, there exist $a,b \in S^{2p}_n$ satisfying $\norm{a}_{S^{2p}_n} \leq 1$ and $\norm{b}_{S^{2p}_n} \leq 1$ and a map $T \co S^p_n \to \M_n$ such that 
\begin{equation}
\label{Facto-Schur-multiplier}
M_A
=T \circ \Mult_{a,b}
\text{ and } 
\norm{T}_{\cb,S^p_n \to \M_n} 
\leq \norm{M_{A}}_{\pi_p^\circ, \M_n \to \M_n} 
\end{equation}
where $\Mult_{a,b} \co \M_n \to S^p_n$. It is easy to check that we can suppose that $a$ and $b$ are positive elements with full supports. Consider the map $\Delta \co \M_n \to \M_n \otvn \M_n$, $e_{ij} \mapsto e_{ij} \ot e_{ij}$. This map is a trace preserving $*$-homomorphism. Let $\E \co \M_n \otvn \M_n \to \M_n$ be the canonical trace preserving normal faithful conditional expectation associated with $\Delta$. The map $\Delta$ induces a contraction $\Delta \co S^1_n \to S^1_n(S^1_n)$. By composition, we deduce a contraction $\Delta \co S^1_n \to S^{p^*}_n(S^1_n)$. By duality, we deduce that we have a well-defined contraction $\E \co S^p_n(\M_n) \to \M_n$. 

Since $T \co S^p_n \to \M_n$ is completely bounded, we deduce by \eqref{ine-tensorisation-os} that the map $\Id_{S^p_n} \ot T \co S^p_n(S^p_n) \to S^p_n(\M_n)$ is completely bounded and that we have
\begin{align}
\MoveEqLeft
\label{Equa-138676-Schur}
\norm{\Id_{S^p_n} \ot T}_{\cb,S^p_n(S^p_n) \to S^p_n(\M_n)}            
\ov{\eqref{ine-tensorisation-os}}{\leq} \norm{T}_{\cb,S^p_n \to \M_n} 
\ov{\eqref{Facto-Schur-multiplier}}{\leq} \norm{M_A}_{\pi_p^\circ, \M_n \to \M_n}. 
\end{align}
We have a contraction $\Delta \co S^2_n \to S^2_n(S^2_n)$, hence by composition a contraction $\Delta \co S^2_n \to \M_n(S^2_n)$, i.e.~for any $x \in S^2_n$ we have
$$
\norm{\Delta(x)}_{\M_n(S^2_n)}   
\leq \norm{x}_{S^2_n}.
$$ 
By the formula \eqref{Norm-LpLq} applied with $p=\infty$ and $q=r=2$, we deduce that for any elements $c,d$ in the unit ball of $S^4_n$
\begin{equation*}
\norm{(c \ot 1)\Delta(x)(d \ot 1)}_{S^2_n(S^2_n)}   
\leq \norm{x}_{S^2_n}.
\end{equation*}
We infer that the map
\begin{equation*}
S^2_n \to S^2_n(S^2_n),\,  
x \mapsto (c \ot 1) \Delta(x) (d \ot 1).
\end{equation*}
is a well-defined contraction and thus by Fubini's theorem the map
\begin{equation}
\label{Theta-magic-Schur}
S^2_n \to S^2_n(S^2_n),\,  
x \mapsto (1 \ot c) \Delta(x) (1 \ot d).
\end{equation}
is also a contraction.

\begin{prop}
Suppose $1 \leq p \leq \infty$. The map
\begin{equation}
\label{Theta-bis-Schur}
\Theta \co S^p_n \to S^p_n(S^p_n),\,  
x \mapsto (1 \ot a) \Delta(x) (1 \ot b).
\end{equation}
is a well-defined complete contraction.
\end{prop} 

\begin{proof}
Suppose $2 \leq p \leq \infty$. We will use Stein interpolation \cite{CwJ84} \cite[Theorem 2.1]{Voi92} \cite[Theorem 2.7]{Lun18}. For any complex number $z$ of the closed strip $\ovl{S}=\{z \in \mathbb{C} : 0 \leq \Re z \leq 1 \}$, we consider the map
$$
T_z \co \M_n \to \M_n,\quad x \mapsto \big(1 \ot a^{\frac{pz}{2}}\big) \Delta(x) \big(1 \ot b^{\frac{pz}{2}}\big).
$$
Note that the elements $a^{\frac{p}{2}}$ and $b^{\frac{p}{2}}$ have unit norms in $S^4_n$. For any $t \in \R$ and any $x \in \M_n$, we deduce that
\begin{align*}
\MoveEqLeft
\norm{T_{1+\i t}(x)}_{S^2_n(S^2_n)}            
=\bnorm{\big(1 \ot a^{\frac{p+\i t p}{2}}\big) \Delta(x) \big(1 \ot b^{\frac{p+\i t p}{2}}\big)}_{S^2_n(S^2_n)} \\
&\leq \bnorm{\big(1 \ot a^{\frac{p}{2}}\big) \Delta(x) \big(1 \ot b^{\frac{p}{2}}\big)}_{S^2_n(S^2_n)}
\ov{\eqref{Theta-magic-Schur}}{\leq} \norm{x}_{S^2_n}.
\end{align*}
We infer that the map  $T_{1+\i t} \co S^2_n \to S^2_n(S^2_n)$ is contractive. Note that by \cite[p.~139]{Pis03}, the operator space $S^2_n$ is an operator Hilbert space. We conclude that this map is even completely contractive by \cite[Proposition 7.2 (ii)]{Pis03}.

For any $t \in \R$ and any $x \in \M_n$, we have
\begin{align*}
\MoveEqLeft
\norm{T_{\i t}(x)}_{\M_n \otvn \M_n}            
=\bnorm{\big(1 \ot a^{\frac{\i t p}{2}}\big) \Delta(x) \big(1 \ot b^{\frac{\i t p}{2}}\big)}_{\M_n \otvn \M_n}
=\norm{\Delta(x)}_{\M_n \otvn \M_n}.
\end{align*} 
We deduce that the map $T_{\i t} \co \M_n \to \M_n \otvn \M_n$ is completely contractive. Moreover, for any $x \in \M_n$, the function $\ovl{S} \to S^1_n$, $z \mapsto T_z(x)$ is analytic on $S$ and continuous and bounded on $\ovl{S}$. Finally, the functions $\R \to S^2_n(S^2_n)$, $t \mapsto T_{1+\i t}(x)$ and $\R \to \M_n \otvn \M_n$, $t \mapsto T_{\i t}(x)$ are continuous. By Stein's interpolation, we conclude by taking $z=\frac{2}{p}$ that the map \eqref{Theta-bis-Schur} is a complete contraction.

The case $1 \leq p \leq 2$ is similar. 
\end{proof}

We have
\begin{align*}
\MoveEqLeft
\Delta M_A
=(\Id \ot M_A) \Delta
\ov{\eqref{Facto-Schur-multiplier}}{=}(\Id \ot T \Mult_{a,b}) \Delta
=(\Id \ot T) (\Id \ot \Mult_{a,b}) \Delta
\ov{\eqref{Facto-Schur-multiplier}}{=} (\Id \ot T) \Theta.           
\end{align*} 
Since $\E \Delta=\Id_{\M_n}$, we deduce that $M_{A}=\E \Delta M_A=\E (\Id \ot T) \Theta$. That means that we have the following commutative diagram.
$$
\xymatrix @R=1.5cm @C=2cm{
S^p_n(S^p_n) \ar[r]^{\Id_{S^p_n} \ot T}   & \ar[d]^{\E} S^p_n(\M_n) \\
S^p_n \ar[r]_{M_A} \ar[u]^{\Theta}   & \M_n \\
  }
$$
We conclude that
\begin{align*}
\MoveEqLeft
\norm{M_A}_{S^p_n\to \M_n}
=\norm{\E(\Id_{S^p_n} \ot T)\Theta}_{S^p_n \to \M_n} \\
&\leq \norm{\E} \norm{\Id_{S^p_n} \ot T}_{S^p_n(S^p_n) \to S^p_n(\M_n)}  \norm{\Theta} 
\ov{\eqref{Equa-138676-Schur}}{\leq} \norm{M_{A}}_{\pi_p^\circ, \M_n \to \M_n}.
\end{align*}

The second inequality is a particular case of Proposition \ref{prop-carac-Schur}.
\end{proof}

\begin{remark}\normalfont
At the time of the writing, it is not clear if there is an equality in the second inequality. 
\end{remark}


\begin{remark}\normalfont 
\label{Rem-spin-planar-algebra}
The spin planar algebra $\mathscr{P}_{\bullet}$ on the finite set $\{1,\ldots,n\}$ was introduced by Jones in \cite[Example 4.2]{Jon00} \cite[Example 2.8]{Jon21}. See \cite{KSSS19} for a rigorous definition and presentation. This planar algebra (with modulus $\sqrt{n}$) is not a subfactor planar algebra. Its 2-box spaces $\scr{P}_{2,+}$ and $\scr{P}_{2,-}$ identifies to $\M_n$ and to $\ell^\infty_{\{1,\ldots,n\} \times \{1,\ldots,n\}}$.  It is left to the reader to check that the convolution $A*B$ of two matrices $A,B$ of $\scr{P}_{2,+}=\M_n$ identifies to $\sqrt{n}A \circ B$ where $\circ$ is the Hadamard product, defined by $[A \circ B]_{ij}=a_{ij}b_{ij}$. 

\end{remark}

The following is a variant of a result of \cite{Arh11}.

\begin{prop}
Suppose $1 \leq p \leq \infty$. If $I$ is an index set then the bilinear map
$$
\begin{array}{cccc}
    &   S^p_{I} \times S^{p^*}_{I}  &  \longrightarrow   & S^\infty_{I}   \\
    &   (A,B)  &  \longmapsto       &  A*B  \\
\end{array}
$$
is jointly completely contractive.
\end{prop}

\begin{proof}
We denote by $\beta \co \ell^2_I \to \ell^{\infty}_I$ the canonical contractive map. By \cite[(1.10)]{BLM04}, we have
$$
\norm{\beta}_{\cb,\mathrm{R}_I \rightarrow \ell^{\infty}_I} 
=\norm{\beta}_{\ell^{2}_I \rightarrow \ell^{\infty}_I}
\leq 1 
\ \ \ \text{and} \ \ \ 
\norm{\beta}_{\cb,\C_I \rightarrow {\ell^{\infty}_I}} 
=\norm{\beta}_{\ell^{2}_I \rightarrow \ell^{\infty}_I} 
\leq 1.
$$
By \cite[1.5.5 p.~31]{BLM04}, the tensor product $\alpha \ot \beta \co \mathrm{R}_I \ot_h \C_I \rightarrow \ell^{\infty}_I \ot_h \ell^{\infty}_I$ is completely contractive. Now, recall that we have a completely isometric canonical map $\ell^{\infty}_I \ot_h \ell^{\infty}_I \rightarrow \frak{M}_{\infty}^{I}$. Hence the map
$$
\begin{array}{ccccccc}
 S^1_{I}=\mathrm{R}_I \ot_h \C_I &  \longrightarrow  &   \ell^{\infty}_I \ot_h \ell^{\infty}_I  &  \longrightarrow   &  \frak{M}_{\infty}^{I} \\
   e_{ij}  &       \longmapsto  &    e_i   \ot e_j                &  \longmapsto       &     M_{e_{ij}}   \\
\end{array}
$$
is completely contractive. This means that the map $A\mapsto M_A$ from $S^1_{I}$ into $\CB(S^\infty_{I})$ is completely contractive. Then the map $(A,B) \mapsto A*B$ from $S^1_{I} \times S^\infty_{I}$ into $S^\infty_{I}$ is completely jointly contractive. By the commutativity of $*$ and $\otp$, the map from $S^1_{I} \times
S^{\infty}_{I}$ into $S^\infty_{I}$ is also completely jointly contractive. Finally, we obtain the result by bilinear interpolation.
\end{proof}

By \cite[Proposition 7.1.2]{EfR00}, we deduce a completely contractive map $S^p_I \to \CB(S^{p^*}_I,S^\infty_I)$, $A \mapsto M_A$. Finally, we obtain by duality a completely contractive map $S^{p}_I \to \CB(S^1_I,S^p_I)$, $A \mapsto M_A$. In particular
\begin{equation}
\label{Schur-Sp}
\norm{M_A}_{\cb,S^1_I \to S^p_I}
\leq \norm{A}_{S^p_I}.
\end{equation}

\begin{remark} \normalfont
\label{Schur-not-completely-bounded}
Note that it is easy to check that the identity $\Id \co S^1_n \to S^p_n$ is a Schur multiplier satisfying 
$$
\norm{\Id}_{\cb,S^1_n \to S^p_n} 
= n^{1-\frac{1}{p}} 
\not= 1
= \norm{\Id}_{S^1_n \to S^p_n}.
$$
We may note along the way that $\H_{\cb,\min}(\Id)=-\log(n)$ and $\H_{\min}(\Id)=0$.
\end{remark}

\section{Entropies and capacities}
\label{sec-entropies}
\subsection{Definitions and transference}
\label{Entropy-capacity-I}

Generalizing the case of matrix algebras, it is natural to introduce the following definition.

\begin{defi}
\label{Def-capacities}
Let $\cal{M}$ and $\cal{N}$ be finite von Neumann algebra equipped with normal finite faithful  traces. Let $T \co \L^1(\cal{M}) \to \L^1(\cal{N})$ be a quantum channel.

\begin{enumerate}
	\item A value $\alpha \geq 0$ is an achievable rate for classical information transmission
through $T$ if 
\begin{enumerate}
	\item $\alpha = 0$, or
	\item $\alpha > 0$ and the following holds for every positive real number $\epsi> 0$: for all but finitely many positive integers $n$ and for $m = \lfloor \alpha n \rfloor$, there exist two quantum channels $\Phi_1 \co \ell^1_{2^m} \to \L^1(\cal{M})$ and $\Phi_2 \co \L^1(\cal{N}) \to \ell^1_{2^m}$ such that
\begin{equation}
\label{emulation-1}
\bnorm{\Id_{\ell^1_{2^m}}-\Phi_2 T^{\ot n} \Phi_1}_{\cb,\ell^1_{2^m} \to \ell^1_{2^m}} 
< \epsi.
\end{equation}
\end{enumerate}

\item The classical capacity of $T$, denoted $\C(T)$, is the supremum value of all achievable rates for classical information transmission through $T$.
\end{enumerate}
Replacing $\ell^1_{2^m}$ by the Schatten space $S^1_{2^m}$, we define similarly the quantum capacity $\Q(T)$ of $T$.
\end{defi}

\begin{remark} \normalfont
Using the diagonal embedding $J \co \ell^1_{2^m} \to S^1_{2^m}$ and the canonical conditional expectation $\E \co S^1_{2^m} \to \ell^1_{2^m}$, it is easy to check that
\begin{equation}
\label{C-leq-Q}
\Q(T)
\leq \C(T).
\end{equation}
\end{remark}
\begin{remark} \normalfont 
With obvious notations and standard arguments, we have
\begin{equation}
\label{capacity-composition}
\C(T_1T_2) 
\leq \min\big\{\C(T_1), \C(T_2) \big\}
\quad \text{and} \quad
\Q(T_1T_2) 
\leq \min\big\{\Q(T_1), \Q(T_2) \big\}.
\end{equation}
Moreover, we have 
\begin{equation}
\label{capacity-sum}
\C(T_1 \ot T_2) 
\geq \C(T_1)+\C(T_2)
\quad \text{and} \quad
\Q(T_1 \ot T_2) 
\geq \Q(T_1)+\Q(T_2).
\end{equation}
\end{remark}

Generalizing again the case of matrix algebras by replacing the von Neumann entropy by the Segal entropy, we equally introduce the following definition.

\begin{defi}
Let $\cal{M}$ be a finite von Neumann algebra equipped with a normal finite faithful trace. Let $T \co \L^1(\cal{M}) \to \L^1(\cal{M})$ be a quantum channel. We define the Holevo capacity of $T$ by
\begin{equation}
\label{Def-Holevo-capacity-bis}
\chi(T)
\ov{\mathrm{def}}{=} \sup \bigg\{ \H\bigg(\sum_{i=1}^{N} \lambda_iT(\rho_i)\bigg)-\sum_{i=1}^{N} \lambda_i \H(T(\rho_i))\bigg\}
\end{equation}
where the supremum is taken over all $N \in \mathbb{N}$, all probability distributions $(\lambda_1,\ldots,\lambda_N)$ and all states $\rho_1,\ldots,\rho_N$ of $\L^1(\cal{M})$,
\end{defi}

\begin{remark} \normalfont
\label{replace-trace}
A simple computation with \eqref{changement-de-trace} shows that the Holevo capacity $\chi(T)$ is invariant if we replace the trace $\tau$ by the trace $t\tau$ where $t>0$. 
\end{remark}

\begin{remark} \normalfont
\label{remark-maj-Smin}
If $\tau$ is \textit{normalized}, we have for all $N \in \mathbb{N}$, all probability distributions $(\lambda_1,\ldots,\lambda_N)$ and all states $\rho_1,\ldots,\rho_N$ of $\L^1(\cal{M})$
\begin{align*}
\MoveEqLeft
\H\bigg(\sum_{i=1}^{N} \lambda_iT(\rho_i)\bigg)-\sum_{i=1}^{N} \lambda_i \H(T(\rho_i))            
\leq -\sum_{i=1}^{N} \lambda_i \H(T(\rho_i)) 
\ov{\eqref{Def-minimum-output-entropy-ini}}{\leq} -\H_{\min}(T).
\end{align*} 
Hence taking the supremum, we obtain the following observation 
\begin{equation}
\label{Holevo-Smin}
\chi(T) 
\leq -\H_{\min}(T).
\end{equation}
\end{remark}

\begin{remark} \normalfont
\label{replace-trace}
In \cite{Arh25}, we will prove that the classical capacity $\C(T)$ of a quantum channel $T \co \L^1(\cal{M}) \to \L^1(\cal{M})$ is \eqref{Classical-capacity} generalizing the case of matrix algebras. The author did not try to obtain a generalization of the formula $\Q(T)
=\lim_{n \to +\infty} \frac{\Q^{(1)}(T^{\ot n})}{n}$ (which seems more delicate) to von Neumann algebras.
\end{remark}

%

Following \cite[Definition 1]{Pis12}, we say that a linear map $T \co E \to F$ between operator spaces is co-completely bounded if it is completely bounded as a mapping from $E$ into the opposite operator space $F^{\op}$. We let
\[
\norm{T}_{\ccb, E \to F} 
\ov{\mathrm{def}}{=} \norm{T}_{\cb, E \to F^{\op}}
\]
(possibly equal to $\infty$ if $T$ is not co-completely bounded). It is well known that the transposition on $\M_n$ has completely bounded norm equal to $n$ (cf. e.g. \cite[pp. 418-419]{Pis03}). So by duality $\norm{\Id_{S^1_n}}_{\ccb,S^1_n \to S^1_n} = n$.

%
%

Now, we generalize the observation of \cite{HoW01} which gives an upper bound on the quantum capacity of a quantum channel.

\begin{prop}
\label{Estimate-with-ccb}
Let $\cal{M}$ be a finite von Neumann algebra equipped with a normal finite faithful trace. Let $T \co \L^1(\cal{M}) \to \L^1(\cal{M})$ be a quantum channel. Then
\begin{equation}
\Q(T) 
\leq \log_2 \norm{T}_{\ccb,\L^1(\cal{M}) \to \L^1(\cal{M})}.
\end{equation}
\end{prop}

\begin{proof}
Suppose that $m_i/n_i \to c$ with $c \leq \Q(T)$ and $i$ large enough such that $\bnorm{\Id_2^{m_i} - E_i T^{\ot n_i} D_i} \leq \epsi$ with appropriate encodings and decodings $E_i, D_i$. We obtain
\begin{align*}
\MoveEqLeft
2^{m_i}
=\bnorm{\Id_{S^1_{2^{m_i}}} }_\ccb 
\leq  \bnorm{\Id_{S^1_{2^{m_i}}} - E_i T^{\ot n_i} D_i}_\ccb + \norm{E_i T^{\ot n_i} D_i}_\ccb \\
& \leq \bnorm{\Id_{S^1_{2^{m_i}}}}_\ccb  \bnorm{\Id_{S^1_{2^{m_i}}} - E_i T^{\ot n_i} D_i}_\cb + \norm{E_i}_{\cb} \norm{ T^{\ot n_i}}_\ccb \norm{D_i}_\cb 
\leq 2^{N_i} \epsi + \norm{T}_\ccb^{n_i}.
\end{align*}
Hence $
\big[2^{m_i} (1-\epsi)\big]^{\frac{1}{n_i}}
\leq \norm{T}_\ccb$. Taking on both sides the logarithm $\log_2$ we get
\begin{equation}
\frac{m_i}{n_i} + \frac{\log_2(1 - \epsi)}{n_i} 
\leq \log_2 \norm{T}_\ccb.
\end{equation}
Letting $i \to \infty$ this implies $c \leq \log_2\norm{T}_\ccb$. Consequently, we conclude that $\Q(T) \leq \log_2\norm{T}_\ccb$.
\end{proof}

%
%

Of course, this estimate does not give in general sharp estimates on the quantum capacities. 

\begin{example} \normalfont
\label{ex-estimating-quantum}
Consider a Schur multiplier $M_A \co S^1_n \to S^1_n$ such that the matrix $A$ is positive and all its diagonal entries are equal to one. Such map is a quantum channel (also known as generalized dephasing or Hadamard channel). Then using \cite[Theorem 4]{Pis12}, we have by duality $
\norm{M_A}_{\ccb, S^1_n \to S^1_n}
=\norm{M_A}_{\ccb, \M_n \to \M_n}
=\bnorm{[|a_{ij}|]}_{\M_n}$. In particular, we have the estimate $\Q(M_A) \leq \log_2 \bnorm{[|a_{ij}|]}_{\M_n}$.
\end{example}


\begin{lemma}
Let $\cal{M}$ and $\cal{N}$ be finite von Neumann algebra equipped with normal faithful finite traces. Let $J \co \cal{M} \to \cal{N}$ be a trace preserving normal injective $*$-monomorphism. Let $x \in \L^1(\cal{M})$ be a positive element with $\norm{x}_1=1$. If $J_1 \co \L^1(\cal{M}) \to \L^1(\cal{N})$ is the canonical extension of $J$ then we have
\begin{equation}
\label{Preservation-entropy}
\H(J_1(x))
=\H(x).
\end{equation}
\end{lemma}

\begin{proof}
Suppose $1<p<\infty$. By Lemma \ref{lemma-trace-preserving}, note that $J$ induces an (not necessarily onto) isometry $J_p \co \L^p(\cal{M}) \to \L^p(\cal{N})$. Hence if $x \in \L^1(\cal{M})_+$ with $\norm{x}_1=1$, we have
\begin{equation}
\H(J_1(x))
\ov{\eqref{deriv-norm-p}}{=} -\frac{\d}{\d p} \norm{J_p(x)}_{\L^p(\cal{N})}^p|_{p=1}
=-\frac{\d}{\d p} \norm{x}_{\L^p(\cal{M})}^p|_{p=1}
\ov{\eqref{deriv-norm-p}}{=} \H(x).
\end{equation}
\end{proof}

\begin{prop}
Let $\cal{M}$ and $\cal{N}$ be finite von Neumann algebra equipped with normal faithful finite traces. Suppose that $J \co \cal{M} \to \cal{N}$ is a normal injective $*$-homomorphism and that $\E \co \cal{N} \to \cal{M}$ is the associated trace preserving normal conditional expectation. Let $T \co \L^1(\cal{M}) \to \L^1(\cal{M})$ be a quantum channel. We have 
\begin{equation}
\label{EJ-first}
\H_{\min}(T)=\H_{\min}(J_1T\E_1), \quad \chi(T)=\chi(J_1T\E_1), \quad \C(T)=\C(J_1T\E_1), \quad \Q(T)=\Q(J_1T\E_1)
\end{equation}
and if $\cal{M},\cal{N}$ are finite-dimensional
\begin{equation}
\label{EJ-second}
\H_{\cb,\min}(T)=\H_{\cb,\min}(J_1T\E_1).
\end{equation}
\end{prop}

\begin{proof}
1) We have
\begin{equation}
\label{}
\H_{\min}(J_1T\E_1)
\ov{\eqref{Def-minimum-output-entropy}}{=} \min_{\rho \geq 0, \tau (\rho) = 1} \H(J_1T\E_1(\rho))
\ov{\eqref{Preservation-entropy}}{=} \min_{\rho \geq 0, \tau(\rho) = 1} \H(T\E_1(\rho)).
\end{equation}
Note that when $\rho$ describe the states of $\L^1(\cal{N})$, the element $\E_1(\rho)$ describe the states of $\L^1(\cal{M})$. Hence, we conclude that
$$
\H_{\min}(J_1T\E_1)
=\min_{\rho' \geq 0, \tau(\rho') = 1} \H(T(\rho'))
\ov{\eqref{Def-minimum-output-entropy}}{=} \H_{\min}(T).
$$
 
2) We have
\begin{align}
\MoveEqLeft
\label{Eq-sans-fin-prime}
\chi(J_1T\E_1) \ov{\eqref{Def-Holevo-capacity}}{=} \sup \bigg\{ \H\bigg(\sum_{i=1}^{N} \lambda_i J_1T\E_1(\rho_i)\bigg)-\sum_{i=1}^{N} \lambda_i \H(J_1T\E_1(\rho_i))\bigg\}
\end{align}
where the supremum is taken over all $N \in \N$, all probability distributions $(\lambda_1,\ldots,\lambda_N)$ and all states $\rho_1,\ldots,\rho_N$ of $\L^1(\cal{N})$. Note that for any $(\lambda_1,\ldots,\lambda_N)$ and any states $\rho_1,\ldots,\rho_N$, we have
\begin{equation}
\label{}
\H\bigg(\sum_{i=1}^{N} \lambda_i J_1T\E_1(\rho_i)\bigg)
=\H\bigg(J_1\bigg(\sum_{i=1}^{N} \lambda_i T\E_1(\rho_i)\bigg)\bigg)
\ov{\eqref{Preservation-entropy}}{=} \H\bigg(\sum_{i=1}^{N} \lambda_i T\E_1(\rho_i)\bigg)
\end{equation}
and $\H(J_1T\E_1(\rho_i)) \ov{\eqref{Preservation-entropy}}{=} \H(T\E_1(\rho_i))$. So we have
\begin{equation}
\label{}
\chi(J_1T\E_1)
=\sup \bigg\{ \H\bigg(\sum_{i=1}^{N} \lambda_i T\E_1(\rho_i)\bigg)-\sum_{i=1}^{N} \lambda_i \H(T\E_1(\rho_i))\bigg\}
\end{equation}
where the supremum is taken as before. Note that when $\rho$ describe the states of $\L^1(\cal{N})$, the element $\E_1(\rho)$ describe the states of $\L^1(\cal{M})$. Hence, we conclude that
$$
\chi(J_1T\E_1)
=\sup \bigg\{ \H\bigg(\sum_{i=1}^{N} \lambda_i T(\rho_i')\bigg)-\sum_{i=1}^{N} \lambda_i \H(T(\rho_i'))\bigg\}
$$
where the supremum is taken over all $N \in \N$, all probability distributions $(\lambda_1,\ldots,\lambda_N)$ and all states $\rho_1',\ldots,\rho_N'$ of $\L^1(\cal{M})$. Hence $\chi(J_1T\E_1)=\chi(T)$. 

3) Suppose that \eqref{emulation-1} is true with $\cal{M}=\cal{N}$. Since $\E_1 J_1=\Id_{\L^1(\cal{M})}$, we have $\E_1^{\ot n} J_1^{\ot n}=\Id_{\L^1(\cal{M}^{\ot n})}$. We infer that
\begin{align*}
\MoveEqLeft
\bnorm{\Id_{\ell^1_{2^m}}-\Phi_2\E_1^{\ot n} (J_1 T \E_1)^{\ot n} J_1^{\ot n} \Phi_1}_{\cb,\ell^1_{2^m} \to \ell^1_{2^m}} \\
&=\bnorm{\Id_{\ell^1_{2^m}}-\Phi_2\E_1^{\ot n} J_1^{\ot n} T^{\ot n} \E_1^{\ot n} J_1^{\ot n} \Phi_1}_{\cb,\ell^1_{2^m} \to \ell^1_{2^m}} 
=\bnorm{\Id_{\ell^1_{2^m}}-\Phi_2 T^{\ot n} \Phi_1}_{\cb,\ell^1_{2^m} \to \ell^1_{2^m}}
< \epsi.            
\end{align*}
So $\alpha$ is an achievable rate for classical information transmission through $J_1 T\E_1$. Conversely, if $\bnorm{\Id_{\ell^1_{2^m}}-\Phi_2 (J_1 T\E_1)^{\ot n} \Phi_1}_{\cb,\ell^1_{2^m} \to \ell^1_{2^m}}
< \epsi$ then
$$
\bnorm{\Id_{\ell^1_{2^m}}-\Phi_2 J_1^{\ot n} T^{\ot n}\E_1^{\ot n} \Phi_1}_{\cb,\ell^1_{2^m} \to \ell^1_{2^m}}
< \epsi.
$$  
So using the encoding channel $\E_1^{\ot n} \Phi_1$ and the decoding channel $\Phi_2 J_1^{\ot n}$, we see that $\alpha$ is an achievable rate for classical information transmission through $T$. 

4) The proof is similar to the point 3.

5) It is left to the reader.
%
\end{proof}

\begin{example} \normalfont
Consider the finite-dimensional von Neumann algebras $\cal{M}=\M_{n_1} \oplus \cdots \oplus \M_{n_K}$ and $\cal{N}=\M_{n_1+ \cdots +n_K}$. Using the injective map $J \co \cal{M} \to \cal{N}$, the conditional expectation $\E \co \cal{N} \to \cal{M}$ and \cite{FuW07} \cite[Proposition 5]{GJL18a}, we obtain $\C(\Id_{\L^1(\cal{M})}) \ov{\eqref{EJ-first}}{=} \C(J_1\Id_{\L^1(\cal{M})}\E_1)=\log \big(\sum_{i=1}^{K} n_i\big)$ and $\Q(\Id_{\L^1(\cal{M})}) \ov{\eqref{EJ-first}}{=} \Q(J_1\Id_{\L^1(\cal{M})}\E_1)=\log(\max_{1 \leq i \leq K} n_i)$  where we use that $\E$ is degradable.
\end{example}

\subsection{Entropies and classical capacities of multipliers on quantum hypergroups}
\label{Entropy-capacity}


Let $\cal{M}$ be a finite von Neumann algebra equipped with a normal finite faithful trace $\tau$. Recall that a quantum channel is a trace preserving completely positive map $T \co \L^{1}(\mathcal{M}) \to \L^{1}(\cal{M})$. 

Let $(\L^\infty(\QH),\tau,\Delta,R)$ be a normalized quantum hypergroup. By \eqref{relation-trace} and the complete positivity of the coproduct, if $f \in \L^1(\cal{M})$ then the convolution map $T \co \L^1(\QH) \to \L^1(\QH)$, $x \mapsto f * x$ is a quantum channel if and and only if $f$ is positive and $\norm{f}_{\L^1(\QH)}=1$ (see \cite[Theorem 2.1]{SkV19} for the case of locally compact quantum groups).

If the trace $\tau$ is normalized, note that the numerical value $\H(f) \ov{\mathrm{def}}{=} -\tau(f\log f)$ is negative.

Note that the compact quantum group $\QG$ is of Kac type by \cite{VDa97} and it is well-known that a finite quantum group is co-amenable. So the next result can used with these quantum groups.

\begin{thm}
\label{thm-entropy-cb}
Let $(\L^\infty(\QH),\tau,\Delta,R)$ be a finite-dimensional normalized quantum hypergroup.  Let $f$ be a element of $\L^\infty(\QH)$. Suppose that the convolution map $T \co \L^1(\cal{M}) \to \L^1(\cal{M})$, $x \mapsto f * x$ is a quantum channel. Then
\begin{equation}
\label{Scbmin}
\H_{\cb,\min,\tau}(T)
=\H_{\min,\tau}(T)
=-\tau(f\log f)
=\H(f).
\end{equation}
\end{thm}

\begin{proof}
 Hence, we can use Theorem \ref{Thm-description-multipliers-2} or Proposition \ref{Prop-cb-L1-Lp}. Using $\norm{f}_{\L^1(\cal{M})}=1$ in the third equality, we obtain
\begin{align*}
\MoveEqLeft
\H_{\cb,\min,\tau}(T)
\ov{\eqref{Def-intro-Scb-min}}{=} -\frac{\d}{\d p} \norm{T}_{\cb,\L^1(\QH) \to \L^p(\QH)}^p|_{p=1} 
\ov{\eqref{norm-equality-QG}}{=} -\frac{\d}{\d p} \norm{f}_{\L^p(\QH)}^p|_{p=1} 
\ov{\eqref{deriv-norm-p}}{=} -\tau(f\log f)
=\H(f).
\end{align*}
By \eqref{Smin-as-derivative-intro}, the same computation is true without the subscript ``$\cb$''. So we equally have $\H_{\min,\tau}(T)=\H(f)$. The proof is complete.
\end{proof}

\begin{remark} \normalfont
In particular, $\H_{\min,\tau}$ is additive on convolution operators, i.e.~$\H_{\min,\tau}(T_1 \ot T_2)=\H_{\min,\tau}(T_1)+\H_{\min,\tau}(T_2)$ if $T_1$ and $T_2$ are convolution operators.
\end{remark}

Let $n \geq 1$ be an integer. Given a quantum channel $T \co S^1_n \to S^1_n$,  by \cite[pp.~4-5]{JuP15} the Holevo bound is given by
\begin{equation}
\label{chi-as-C1}
\chi(T)
=\frac{\d}{\d p}\norm{T^*}_{\pi_{p^*,1}, \M_n \to \M_n}|_{p=1}
\end{equation}
where we refer to \eqref{Def-Holevo-capacity} for the definition of $\chi(T)$. 

We extend this equality to maps acting on finite-dimensional vo Neumann algebras by using the method of approximation of the proof of \cite[Theorem 3.24]{ArK23}.

\begin{prop}
Let $\cal{M}=\M_{n_1} \oplus \cdots \oplus \M_{n_K}$ be a finite-dimensional von Neumann algebra equipped with a faithful trace $\tau=\lambda_1 \tr_{n_1} \oplus \cdots \oplus \lambda_K \tr_{n_K}$ where $\lambda_1,\ldots,\lambda_K \in \mathbb{Q}$. Let $T \co \L^1(\cal{M}) \to \L^1(\cal{M})$ be a quantum channel. Then we have
\begin{equation}
\label{chiT-as-derivative}
\chi(T)
=\frac{\d}{\d p}\norm{T^*}_{\pi_{p^*,1}, \cal{M} \to \cal{M}}|_{p=1}.
\end{equation}
\end{prop}

\begin{proof}
\noindent
\textit{Case 1: All $\lambda_k$ belong to $\N$.}
Then let $n \ov{\mathrm{def}}{=} \sum_{k = 1}^K \lambda_k n_k$. We consider the normal unital trace preserving $*$-homomorphism $J \co \cal{M} \to \M_n$ defined by
\[
J(x_1 \oplus \cdots \oplus x_K) 
= \begin{bmatrix} x_1 & & & & & &  \\ 
& \ddots & & & & & \\ & & x_1 & & &0 & \\ 
& & & \ddots & & & \\ & & & & x_K & & \\ 
& 0& & & & \ddots & \\ 
 & & & & & & x_K 
\end{bmatrix} ,
\]
where $x_k$ appears $\lambda_k$ times on the diagonal, $k = 1,\ldots, K$. Let $\E \co \M_n \to \cal{M}$ be the associated normal faithful conditional expectation. Since $\E J=\Id_{\cal{M}}$, we have $T^*=\E J T^* \E J$. Hence
\begin{align*}
\MoveEqLeft
\norm{T^*}_{\pi_{p,1},\cal{M} \to \cal{M}}            
=\bnorm{\E J T^* \E J}_{\pi_{p,1},\cal{M} \to \cal{M}}
\ov{\eqref{Ideal-ellp}}{\leq} \norm{\E}_{\cb} \norm{JT^*\E}_{\pi_{p,1}, \M_n \to \M_n} \norm{J}_{\cb}
\leq \norm{JT^*\E}_{\pi_{p,1}, \M_n \to \M_n}.
\end{align*}
In a similar manner, we have 
$$
\norm{JT^*\E}_{\pi_{p,1}, \M_n \to \M_n} 
\ov{\eqref{Ideal-ellp}}{\leq} \norm{J}_{\cb} \norm{T^*}_{\pi_{p,1},\M_n \to \M_n} \norm{\E}_{\cb} 
\leq \norm{T^*}_{\pi_{p,1},\cal{M} \to \cal{M}}.
$$ 
Hence, we conclude that
\begin{equation}
\label{Divers-12567}
\norm{T^*}_{\pi_{p,1},\cal{M} \to \cal{M}} 
=\norm{J T^* \E}_{\pi_{p,1},\M_n \to \M_n}.
\end{equation}
Consequently, we have
\begin{align*}
\MoveEqLeft
\frac{\d}{\d p}\norm{T^*}_{\pi_{p^*,1}, \cal{M} \to \cal{M}}|_{p=1}
\ov{\eqref{Divers-12567}}{=} \frac{\d}{\d p}\norm{JT^*\E}_{\pi_{p^*,1}, \M_n \to \M_n}|_{p=1} 
=\frac{\d}{\d p}\norm{(J_1T\E_1)^*}_{\pi_{p^*,1},\M_n \to \M_n}|_{p=1} \\
&\ov{\eqref{chi-as-C1}}{=} \chi(J_1T\E_1) 
\ov{\eqref{EJ-first}}{=} \chi(T). 
\end{align*}

\noindent
\textit{Case 2: All $\lambda_k$ belong to $\mathbb{Q}_{+}$.}
Then there exists a common denominator of the $\lambda_k$'s, that means that there exists $t \in \N$ such that $\lambda_k = \frac{\lambda_k'}{t}$ for some integers $\lambda_k'$. Note that the right-hand side of \eqref{chiT-as-derivative} is invariant if we multiply the trace $\tau$ by $t$. Moreover, Remark \eqref{replace-trace} says that the left-hand side of \eqref{chiT-as-derivative} is also invariant if we replace the trace $\tau$ by $t\tau$. Thus, Case 1.2 follows from Case 1.1. 
%
\end{proof}

\begin{remark} \normalfont
Undoubtedly, with a painful approximation argument in the spirit of the proof of \cite[Theorem 3.24]{ArK23}, we can extend the previous result to the case of an arbitrary trace $\tau$.
\end{remark}


Now, we obtain a bound of the classical capacity of multipliers.

\begin{thm}
\label{thm-capacity}
Let $(\cal{M},\tau,\Delta,R)$ be a finite-dimensional normalized quantum hypergroup such $\tau$ is normalized. Let $f$ be a element of $\L^\infty(\QG)$. Suppose that the convolution map $T \co \L^1(\cal{M}) \to \L^1(\cal{M})$, $x \mapsto f * x$ is a quantum channel. Then 
\begin{equation}
\label{CT-Fourier}
\C(T)
\leq -\H(f)=\tau(f\log f)
\end{equation}
\end{thm}

\begin{proof}
1. Note that the proof of Theorem \ref{Thm-description-multipliers-3} contains the description of the adjoint $T^*$. Moreover, observe that for any element $d$ of $\N \cup\{\infty\}$, we have 
$$
\lim_{q \to \infty} \norm{T^*}_{\pi_{q,d}, \L^\infty(\cal{M}) \to \L^\infty(\cal{M})}
=\norm{\Id_{\M_d} \ot T^*}_{\M_d(\L^\infty(\cal{M})) \to \M_d(\L^\infty(\cal{M}))}
=1.
$$ 
We let $a \ov{\mathrm{def}}{=} \widehat{f}$. We have
\begin{align}
\MoveEqLeft
\label{chiT}
\chi(T)
\ov{\eqref{chiT-as-derivative}}{=} \frac{\d}{\d p}\norm{T^*}_{\pi_{p^*,1}, \L^\infty(\cal{M}) \to \L^\infty(\cal{M})}|_{p=1} 
=\frac{\d}{\d p}\norm{f*\cdot}_{\pi_{p^*,1}, \L^\infty(\cal{M}) \to \L^\infty(\cal{M})} |_{p=1} \\
&\ov{\eqref{elem-lim}}{=} \lim_{q \to +\infty} q \ln \norm{f*\cdot}_{\pi_{q,1}, \L^\infty(\cal{M}) \to \L^\infty(\cal{M})} 
\ov{\eqref{Fourier-sommant}}{\leq} \lim_{q \to +\infty} q \ln \norm{f*\cdot}_{\cb,\L^q(\cal{M}) \to \L^\infty(\cal{M})} \nonumber \\
&=\frac{\d}{\d p} \norm{f*\cdot}_{\cb,\L^{p^*}(\cal{M}) \to \L^\infty(\cal{M})} |_{p=1}
\ov{\eqref{norm-equality-QG-2}}{=} \frac{\d}{\d p} \norm{f}_{\L^{p}(\cal{M})} |_{p=1}
\ov{\eqref{deriv-norm-p-bis}}{=} -\H(f). \nonumber
\end{align} 
Another argument is the inequality $\chi(T) \ov{\eqref{Holevo-Smin}}{\leq} -\H_{\min}(T) \ov{\eqref{Scbmin}}{=} -\H(f)$.

%

We can write $\L^\infty(\QG)=\M_{n_1} \oplus \cdots \oplus \M_{n_K}$ and its faithful trace satisfies the assumptions of Proposition \ref{chiT-as-derivative}. Now, with the notations of the proof of Proposition \ref{chiT-as-derivative}, we have the following ``HSW formula''
\begin{equation}
\label{HSW}
\C(T)
\ov{\eqref{EJ-first}}{=} \C(J_1T\E_1)
\ov{\eqref{Classical-capacity}}{=} \lim_{n \to +\infty} \frac{\chi((J_1T\E_1)^{\ot n})}{n} 
=\lim_{n \to +\infty} \frac{\chi(J_1^{\ot n}T^{\ot n}\E_1^{\ot n})}{n}
\ov{\eqref{EJ-first}}{=}\lim_{n \to +\infty} \frac{\chi(T^{\ot n})}{n}. 
\end{equation}
By observing that the map $f^{\ot n} * \cdot \co \L^1(\cal{M}^{\ot^n}) \to \L^1(\cal{M}^{\ot^n})$ is a convolution operator on the quantum hypergroup $\cal{M} \otvn \cdots \otvn \cal{M}$ (see \cite[Section 2.5]{Tro16} for the case of quantum groups), we deduce that
\begin{align*}
\MoveEqLeft
\C(T)
\ov{\eqref{HSW}}{=} \lim_{n \to +\infty} \frac{\chi(T^{\ot n})}{n}     
=\lim_{n \to +\infty} \frac{\chi(f^{\ot n} * \cdot)}{n} 
\ov{\eqref{chiT}}{\leq} \lim_{n \to +\infty} \frac{-\H(f^{\ot n})}{n} \\
&\ov{\eqref{entropy-tensor-product}}{=} \lim_{n \to +\infty} \frac{-n \H(f)}{n} 
=-\H(f).
\end{align*} 
\end{proof}

\begin{remark} \normalfont
Note that we can write $-\H(f)=\tau(f\log f)=\cal{S}(\tau(f\,\cdot),\tau)$ \cite{LoW22} where $\cal{S}(\tau(f\,\cdot),\tau)$ is the relative entropy between the states $\tau(f\,\cdot)$ and $\tau$. We deduce by \cite[Corollary 5.6]{OhP93} that $-\H(f)=0$ if and only if $\tau(f\,\cdot)=\tau$. 

If $f=1$, we have $-\H(f)=0$ so $\C(T)=0$. It is not really surprising since for any state $\varphi$ and the Haar state $\tau$ we have $\tau * \varphi = \tau$. So all the information of $\varphi$ is lost by applying $T$. 
\end{remark}

\begin{remark} \normalfont
\label{Remark-commutative}
Let $G$ be a finite (classical) group and consider a measure of probability $\nu$ on $G$. It is elementary to check that the convolution operator $T \co \ell^1_G \to \ell^1_G$, $g \mapsto \nu*g$ where 
$$
(\nu*g)(t)
\ov{\mathrm{def}}{=} \sum_{s \in G} g\big(s^{-1}t\big)\nu(s)
,\quad t \in G,
$$
is classical channel. By \cite[p.~104]{CST08}, the matrix of $T$ is given\footnote{\thefootnote. For any $t,r \in G$, we have $[T]_{t,r}=(\nu*e_r)(t)=\sum_{s \in G} e_r(s^{-1}t)\nu(s)=\nu(t r^{-1})$.}  by the Toeplitz matrix $[\nu(t r^{-1})]_{t,r \in G}$. Observe that the rows of such a matrix are permutations of each other. So such a convolution operator $T \co \ell^1_G \to \ell^1_G$ is symmetric in the sense of \cite[p.~190]{CoT06}. We deduce by \cite[(7.25)]{CoT06} that its classical capacity is given by
\begin{equation}
\label{classical-of-classical-convolution}
\C(T)
=\log |G|-\H(\nu)
\end{equation} 
(and is achieved by a uniform distribution on the input) where $\H(\nu)$ is the Shannon entropy of $\nu$. If $\mu_G$ is the \textit{normalized} Haar measure of $G$ (i.e.~the uniform probability measure on $G$) and if $\mathrm{D}$ denotes the relative entropy (Kullback-Leibler divergence) \cite[p.~19]{CoT06}, a simple computation left to the reader shows that
$$
\C(T)
=\mathrm{D}(\nu \| \mu_G).
$$
Finally, we can also obtain the formula \eqref{classical-of-classical-convolution} by a combination of the formula $\mathrm{C}(T)=\frac{\d}{\d p}\norm{T^*}_{\pi_{p^*}, \ell^\infty_G \to \ell^\infty_G}|_{p=1}$ of \cite[(1.2)]{JuP15}\footnote{\thefootnote. If $X$ and $Y$ are Banach spaces, recall that here $\norm{\cdot}_{\pi_p,X \to Y}$ denotes the $p$-summing norm $\pi_p$ of \cite[p.~31]{DJT95}.} and Theorem \ref{Thm-description-multipliers-3}. Indeed, first we consider the function $f \co G \to \mathbb{C}$, $s \mapsto f(s)=|G|\nu(s)$. Note that $\norm{f}_{\L^1(G,\mu_G)}=\frac{1}{|G|}\sum_{s \in G} f(s)=\frac{1}{|G|}\sum_{s \in G} \nu(s)=1$. Now, we have\footnote{\thefootnote. Recall that $(f*g)(t)=\int_G g(s^{-1}t)f(s)\d\mu_G(s)=\frac{1}{|G|} \sum_{s \in G} g(s^{-1}t) f(s)$.}
\begin{align*}
\MoveEqLeft
\mathrm{C}(T)
=\frac{\d}{\d p}\norm{T^*}_{\pi_{p^*},\ell^\infty_G \to \ell^\infty_G}|_{p=1} 
=\frac{\d}{\d p}\norm{\nu*\cdot}_{\pi_{p^*}, \ell^\infty_G \to \ell^\infty_G} |_{p=1} \\
&=\frac{\d}{\d p}\norm{f*\cdot}_{\pi_{p^*}, \ell^\infty_G \to \ell^\infty_G} |_{p=1}
\ov{\eqref{All-equal}}{=} \frac{\d}{\d p} \norm{f * \cdot }_{\ell^{p^*}_G \to \ell^\infty_G} |_{p=1} \\
&\ov{\eqref{norm-equality-QG-2}}{=} \frac{\d}{\d p} \norm{f}_{\ell^p_G} |_{p=1}
\ov{\eqref{deriv-norm-p-bis}}{=}
-\H_{\mu_G}(f) 
\ov{\eqref{changement-de-trace}}{=} \log |G|-\H\bigg(\frac{1}{|G|}f\bigg) 
=\log |G|-\H(\nu) 
\end{align*}
where we recall that $\H(\cdot,\mu_G)$ is the normalized entropy. So the bound of Theorem \ref{thm-capacity} is \textit{sharp}.
\end{remark}

\section{Examples}
\label{Sec-Examples}


\subsection{Fourier multipliers on the dihedral group $\mathbb{D}_3$ and dephasing channels on $\M_2$}
\label{Sec-Dihedral}


Recall that if $G$ is a finite (classical) group, the von Neumann algebra $\VN(G)$ identifies to a finite sum of matrix algebras. More precisely, we have by essentially \cite[Theorem 5.5.6 p.~66]{Ste12} a $*$-isomorphism $\VN(G) \to \oplus_{\pi \in \hat{G}} \M_{d_\pi}$, $\lambda_s \mapsto \oplus_{\pi \in \hat{G}} \pi(s)$. 
Note that the dihedral group $\mathbb{D}_3$ can be identified with the symmetric group $\mathrm{S}_3=\{e,(123),(132),(12),(23),(13)\}$. Recall that $\mathrm{S}_3$ has three irreducible unitary representations: $1$, $\mathrm{sign}$ and a two dimensional representation $\pi \co \mathrm{S}_3 \to \B(\mathbb{C}^2)$ defined by
$$
\pi(e)=\begin{bmatrix}
  1  & 0  \\
  0  & 1  \\
\end{bmatrix},\quad \pi(123)=\begin{bmatrix}
  \e^{\frac{\i 2\pi}{3}}   & 0  \\
  0   &  \e^{-\frac{\i 2\pi}{3}} \\
\end{bmatrix},\quad \pi(132)
=\begin{bmatrix}
  \e^{-\frac{\i 2\pi}{3}}   & 0  \\
  0   & \e^{\frac{\i 2\pi}{3}}  \\
\end{bmatrix}
$$
and
$$
\pi(12)=\begin{bmatrix}
  0  & 1  \\
  1  & 0  \\
\end{bmatrix},\quad \pi(23)=\begin{bmatrix}
  0   &  \e^{-\frac{\i 2\pi}{3}} \\
 \e^{\frac{\i 2\pi}{3}}    & 0  \\
\end{bmatrix}, \quad \pi(13)=\begin{bmatrix}
   0  &  \e^{\frac{\i 2\pi}{3}} \\
   \e^{-\frac{\i 2\pi}{3}}  & 0  \\
\end{bmatrix}.
$$
In particular, we have a $*$-isomorphism $
\Phi \co \VN(\mathbb{D}_3)
\mapsto \mathbb{C} \oplus \mathbb{C} \oplus \M_2(\mathbb{C})$ defined by
\[
\begin{split}
&\Phi(1)
= 1 \oplus 1 \oplus \begin{bmatrix}
  1   & 0  \\
  0   & 1 \\
\end{bmatrix}, \qquad
\Phi\big(\lambda_{(12)}\big)
=1 \oplus -1 \oplus \begin{bmatrix}
  0   &  1 \\
  1   &  0 \\
\end{bmatrix},\\
&
\Phi\big(\lambda_{(123)}\big)
= 1 \oplus 1 \oplus \begin{bmatrix}
  \e^{\frac{\i 2\pi}{3}}   & 0  \\
   0  &  \e^{-\frac{\i 2\pi}{3}} \\
\end{bmatrix}, \quad
\Phi\big(\lambda_{(23)}\big)
= 1 \oplus -1 \oplus \begin{bmatrix}
  0   &  \e^{-\frac{\i 2\pi}{3}} \\
 \e^{\frac{\i 2\pi}{3}}    &  0 \\
\end{bmatrix}, \\
&\Phi\big(\lambda_{(132)}\big)
=1 \oplus 1 \oplus \begin{bmatrix}
  \e^{-\frac{\i 2\pi}{3}}   &  0 \\
   0  &  \e^{\frac{\i 2\pi}{3}} \\
\end{bmatrix},
\quad
\Phi\big(\lambda_{(13)}\big)
= 1 \oplus -1 \oplus \begin{bmatrix}
  0   &  \e^{\frac{\i 2\pi}{3}} \\
  \e^{-\frac{\i 2\pi}{3}}   &  0 \\
\end{bmatrix}.\\
\end{split} 
\]
We will use the basis $(e_1,e_2,e_{11},e_{12},e_{21},e_{22})$\footnote{\thefootnote. Here we have $e_2=0 \oplus 0 \oplus 1 \oplus 0 \oplus \begin{bmatrix}
  0   & 0  \\
  0   & 0  \\
		\end{bmatrix}$ and $e_{21}=0 \oplus 0 \oplus \begin{bmatrix}
  0   & 0  \\
  1   & 0  \\
		\end{bmatrix}$.}. 
Using the computer code of Section \ref{Appendix3}, we obtain
\[
\begin{split}
&\Delta(e_1)
=e_{1} \ot e_{1} + e_{2} \ot e_{2} +\frac{1}{2}e_{11} \ot e_{22} +\frac{1}{2}e_{12} \ot e_{21}+\frac{1}{2}e_{21} \ot e_{12}+\frac{1}{2}e_{22} \ot e_{11}, \\
&
\Delta(e_2)
=e_{1} \ot e_{2} + e_{2} \ot e_{1} -\frac{1}{2}e_{12} \ot e_{21}-\frac{1}{2}e_{21} \ot e_{12}+\frac{1}{2}e_{11} \ot e_{22} +\frac{1}{2}e_{22} \ot e_{11}, \\
&
\Delta(e_{11})
=e_{1} \ot e_{11} + e_{2} \ot e_{11} +e_{11} \ot e_{1} +e_{11} \ot e_{2}+e_{22} \ot e_{22},\\
&
\Delta(e_{12})
=e_{1} \ot e_{12} + e_{12} \ot e_{1} +e_{21} \ot e_{21} - e_2 \ot e_{12} -e_{12} \ot e_2, \\
&
\Delta(e_{21})
=e_{1} \ot e_{21} + e_{12} \ot e_{12} +e_{21} \ot e_{1} - e_2 \ot e_{21}- e_{21} \ot e_2, \\
&
\Delta(e_{22})
=e_{1} \ot e_{22} + e_{2} \ot e_{22} +e_{11} \ot e_{11} +e_{22} \ot e_{1}+e_{22} \ot e_{2}.
\end{split}
\]
The Haar state identifies to
\begin{equation}
\label{Haar-D3}
h\bigg(x_1 \oplus x_2 \oplus \begin{bmatrix}
 a_{11}    &  a_{12} \\
  a_{21}   &  a_{22} \\
\end{bmatrix}\bigg)
=\frac{1}{6}\bigg(x_1+x_2+2\tr\begin{bmatrix}
 a_{11}    &  a_{12} \\
  a_{21}   &  a_{22} \\
\end{bmatrix}\bigg)
=\frac{1}{6}x_1+\frac{1}{6}x_2+\frac{1}{3}(a_{11}  + a_{22}).
\end{equation}
A state on $\L^\infty(\VN(\mathbb{D}_3))$ identifies to a convex combination of two complex numbers and a state of $\M_2(\mathbb{C})$. Recall that by \cite[p.~5]{Pet08} any state $\psi$ of $\M_2(\mathbb{C})$ can be written
\begin{equation}
\label{State-M2}
\psi(A)
=\tr\big(A\rho(x,y,z)\big)
\quad \text{where} \quad 
\rho(x,y,z)
=\frac{1}{2}\begin{bmatrix}
 1+z    & x+\i y  \\
 x-\i y    & 1-z  \\
\end{bmatrix}
\end{equation}
with $A \in \M_2(\mathbb{C})$, $x,y,z \in \R$ and $x^2+y^2+z^2 \leq 1$. So a state of $\L^\infty(\VN(\mathbb{D}_3))$ identifies to a sum 
$
\dsp f
=6\mu_1 \oplus 6\mu_2 \oplus \frac{3\mu_3}{2}\begin{bmatrix}
 1+z    &  x+ \i y \\
  x - \i y   &  1-z \\
\end{bmatrix}
$ where $\mu_1, \mu_2, \mu_3 \geq 0$, $x,y,z \in \R$ with $\mu_1 + \mu_2 + \mu_3=1$ and $x^2+y^2+z^2 \leq 1$. The duality bracket \eqref{bracket} is given by
\begin{align*}
\MoveEqLeft
\left\langle 6\mu_1 \oplus 6\mu_2 \oplus \frac{3\mu_3}{2}\begin{bmatrix}
 1+z    &  x+ \i y \\
  x - \i y   &  1-z \\
\end{bmatrix}, x_1 \oplus x_2 \oplus \begin{bmatrix}
 a_{11}    &  a_{12} \\
  a_{21}   &  a_{22} \\
\end{bmatrix}\right\rangle_{\L^1(\VN(\mathbb{D}_3)),\L^\infty(\VN(\mathbb{D}_3))}  \\ 
&\ov{\eqref{bracket}}{=} x_1\mu_1+x_2\mu_2+\frac{\mu_3}{2}\tr\bigg(
\begin{bmatrix}
 a_{11}    &  a_{12} \\
  a_{21}   &  a_{22} \\
\end{bmatrix}\begin{bmatrix}
 1+z    &  x+ \i y \\
  x - \i y   &  1-z \\
\end{bmatrix}^T\bigg) \\         
&=x_1\mu_1+x_2\mu_2+\frac{\mu_3}{2}\big[(1+z)a_{11}+(x-\i y)a_{21}+(x+\i y)a_{12}+(1-z)a_{22}\big].
\end{align*}
We deduce that $f(e_1)=\mu_1, f(e_2)=\mu_2, f(e_{11})=\frac{\mu_3}{2}(1+z),f(e_{12})= \frac{\mu_3}{2}(x-\i y),f(e_{21})= \frac{\mu_3}{2}(x+\i y), f(e_{22})=\frac{\mu_3}{2}(1-z) 
$. If $f$ is a state, we infer that
\begin{align*}
\MoveEqLeft
(f \ot \Id) \circ \Delta(e_1) \\
&=(f \ot \Id)\Big(e_{1} \ot e_{1} + e_{2} \ot e_{2} +\frac{1}{2}e_{11} \ot e_{22} +\frac{1}{2}e_{12} \ot e_{21}+\frac{1}{2}e_{21} \ot e_{12}
+\frac{1}{2}e_{22} \ot e_{11} \Big) \\ 
&=f(e_{1}) \ot e_{1} + f(e_{2}) \ot e_{2} +\frac{1}{2}f(e_{11}) \ot e_{22} +\frac{1}{2}f(e_{12}) \ot e_{21}+\frac{1}{2}f(e_{21}) \ot e_{12}+\frac{1}{2}f(e_{22}) \ot e_{11} \\
&=\mu_1 e_1 +\mu_2 e_2 +\frac{\mu_3}{4}\big[(1+z)e_{22} +(x-\i y)e_{21}+(x+\i y)e_{12} + (1-z)e_{22} \big], \\
\end{align*}
\\[-1.9cm]
\begin{align*}
\MoveEqLeft
(f \ot \Id) \circ \Delta(e_2) \\
&=(f \ot \Id)\Big(e_{1} \ot e_{2} + e_{2} \ot e_{1} -\frac{1}{2}e_{12} \ot e_{21}-\frac{1}{2}e_{21} \ot e_{12}+\frac{1}{2}e_{11} \ot e_{22} +\frac{1}{2}e_{22} \ot e_{11}\Big)\\        
&=f(e_{1}) \ot e_{2} + f(e_{2}) \ot e_{1} -\frac{1}{2}f(e_{12}) \ot e_{21}-\frac{1}{2}f(e_{21}) \ot e_{12}+\frac{1}{2}f(e_{11}) \ot e_{22}  +\frac{1}{2}f(e_{22}) \ot e_{11} \\
&=\mu_1e_2+\mu_2e_1+\frac{\mu_3}{4}\big[ (x-\i y)e_{21}-(x+\i y)e_{12}+  (1+z)e_{22} +(1-z) e_{11}\big], \\
\end{align*}
\\[-1.7cm]
\begin{align*}
\MoveEqLeft
(f \ot \Id) \circ \Delta(e_{11}) 
=(f \ot \Id)(e_{1} \ot e_{11} + e_{2} \ot e_{11} +e_{11} \ot e_{1} +e_{11} \ot e_{2}+e_{22} \ot e_{22})\\     
&=f(e_{1}) \ot e_{11} + f(e_{2}) \ot e_{11} +f(e_{11}) \ot e_{1} +f(e_{11}) \ot e_{2}+f(e_{22}) \ot e_{22} \\
&=\mu_1e_{11}+\mu_2e_{11}+\frac{\mu_3}{2}\big[(1+z) e_{1} + (1+z)e_{2}+ (1-z)e_{22}\big],\\
\end{align*}
\\[-1.9cm]
\begin{align*}
\MoveEqLeft
(f \ot \Id) \circ \Delta(e_{12}) 
=(f \ot \Id)(e_{1} \ot e_{12} + e_{12} \ot e_{1} +e_{21} \ot e_{21} - e_2 \ot e_{12} -e_{12} \ot e_2)\\       
&=f(e_{1}) \ot e_{12} - f(e_2) \ot e_{12}+ f(e_{12}) \ot e_{1} +f(e_{21}) \ot e_{21}  -f(e_{12}) \ot e_2 \\
&=\mu_1 e_{12}-\mu_2 e_{12}+\frac{\mu_3}{2}\big[ (x-\i y)e_{1} + (x+\i y)e_{21} - (x-\i y)e_2\big],\\
\end{align*}
\\[-1.9cm]
\begin{align*}
\MoveEqLeft
(f \ot \Id) \circ \Delta(e_{21}) 
=(f \ot \Id)(e_{1} \ot e_{21} + e_{12} \ot e_{12} +e_{21} \ot e_{1} - e_2 \ot e_{21}- e_{21} \ot e_2)\\       
&=f(e_{1}) \ot e_{21} - f(e_2) \ot e_{21}+ f(e_{12}) \ot e_{12} +f(e_{21}) \ot e_{1} - f(e_{21}) \ot e_2 \\
&=\mu_1e_{21}-\mu_2e_{21}+\frac{\mu_3}{2}\big[ (x-\i y)e_{12} + (x+\i y)e_{1} -(x+\i y) e_2\big], \\
\end{align*}
\\[-1.9cm]
\begin{align*}
\MoveEqLeft
(f \ot \Id) \circ \Delta(e_{22}) 
=(f \ot \Id)(e_{1} \ot e_{22} + e_{2} \ot e_{22} +e_{11} \ot e_{11} +e_{22} \ot e_{1}+e_{22} \ot e_{2})\\     
&= f(e_{1}) \ot e_{22} + f(e_{2}) \ot e_{22} +f(e_{11}) \ot e_{11} +f(e_{22}) \ot e_{1}+f(e_{22}) \ot e_{2} \\
&=\mu_1 e_{22}+\mu_2e_{22}+\frac{\mu_3}{2}\big[(1+z)e_{11} +(1-z)e_{1}+(1-z)e_{2}\big].
\end{align*}
With respect to the basis $(e_1,e_2,e_{11},e_{12},e_{21},e_{22})$ of $\L^\infty(\VN(\mathbb{D}_3))$, with \eqref{action-module} we conclude that the matrix of the convolution operator $T \co \L^\infty(\VN(\mathbb{D}_3)) \mapsto \L^\infty(\VN(\mathbb{D}_3))$, $y \to y*f$ is 
\begin{equation}
\label{direct-sum}
\begin{bmatrix}
\mu_1 & \mu_2 & \frac{\mu_3}{2}(1+z) & \frac{\mu_3}{2}(x-\i y) & \frac{\mu_3}{2}(x+\i y) & \frac{\mu_3}{2}(1-z) \\
\mu_2 & \mu_1 & \frac{\mu_3}{2}(1+z) & \frac{\mu_3}{2}(x-\i y) & \frac{\mu_3}{2}(x+\i y) & \frac{\mu_3}{2}(1-z)  \\
0 & \frac{\mu_3}{4}(1-z) & \mu_1+\mu_2 & 0 & 0 & \frac{\mu_3}{2}(1+z) \\
\frac{\mu_3}{2}(x+\i y)& \frac{\mu_3}{4}(x+\i y) & 0 & \mu_1-\mu_2 & \frac{\mu_3}{2}(x-\i y) & 0\\
\frac{\mu_3}{2}(x-\i y)& \frac{\mu_3}{4}(x-\i y) & 0 & \frac{\mu_3}{2}(x+\i y) & \mu_1-\mu_2 & 0 \\
\mu_3& \frac{\mu_3}{2}(1+z) & \frac{\mu_3}{2}(1-z) & 0 & 0 & \mu_1+\mu_2 \\
\end{bmatrix}.
\end{equation}

If $\mu_3=0$, note that the subspaces $\vect \{e_{1},e_{2}\}$ and $\vect \{e_{11},e_{12},e_{21},e_{22}\}$ are stable under this convolution operator. Furthermore, this operator identifies to a direct sum $T_1 \oplus T_2 \co \ell^\infty_2 \oplus \M_2 \to \ell^\infty_2 \oplus \M_2$. 

Since $\mu_1+\mu_2=1$, the restriction $T_2$ of $T$ on the second subspace identifies to the Schur multiplier $M_A \co \M_2 \to \M_2$ associated to the matrix $
A=\begin{bmatrix}
  1 & \mu_1-\mu_2  \\
  \mu_1-\mu_2   &  1 \\
\end{bmatrix}$.
We recover that this map is a quantum channel by essentially \cite[Lemma 3.1 p.~27]{Pau02} since $|\mu_1-\mu_2| \leq \mu_1+\mu_2=1$. Moreover, if $\mu_1=1-t$ and $\mu_2=t$ for some $0 \leq t \leq 1$, we obtain $A=\begin{bmatrix}
  1 & 1-2t  \\
  1-2t   &  1 \\
\end{bmatrix}$. This channel is the dephasing channel (phase damping channel) of \cite[Example~2.14 p.~19]{Pet08} (see also \cite[p.~155]{Wil17}). The restriction $T_1$ of $T$ on the first subspace identifies to a binary symmetric channel \cite[p.~187]{CoT06} (or \cite[Example~7.2 p.~93]{Pet08}). 


Consequently the capacities and entropies of $T,T_1$ and $T_2$ are related. For example, consider the canonical maps $J \co \M_2 \to \VN(\VN(\mathbb{D}_3))$, $x \mapsto 0 \oplus 0 \oplus x$ and $\E \co \VN(\VN(\mathbb{D}_3)) \to \M_2$. If $x \in S^1_2$, we have 
\begin{equation}
\label{inter-987}
\norm{J(x)}_{\L^1(\VN(\mathbb{D}_3))} 
=\norm{0 \oplus 0 \oplus x}_{\L^1(\VN(\mathbb{D}_3))} 
\ov{\eqref{Haar-D3}}{=} \frac{1}{3}\norm{x}_{S^1_2}.
\end{equation}
Moreover, if $y$ belongs to $\L^\infty(\VN(\mathbb{D}_3))$ we have 
$$
\norm{\E(y)}_{S^p_2}
=3^{\frac{1}{p}} \norm{0 \oplus 0 \oplus \E(y)}_{\L^p(\VN(\mathbb{D}_3))} 
\leq 3^{\frac{1}{p}} \norm{y}_{\L^p(\VN(\mathbb{D}_3))}.
$$ 
We deduce that
\begin{align*}
\MoveEqLeft
\norm{T_2}_{\cb,S^1_2 \to S^p_2}
=\norm{\E T J}_{\cb,S^1_2 \to S^p_2} \\
&\leq \norm{\E}_{\cb,\L^p(\VN(\mathbb{D}_3)) \to S^p_2} \norm{T}_{\cb,\L^1(\VN(\mathbb{D}_3)) \to \L^p(\VN(\mathbb{D}_3))} \norm{J}_{\cb, S^1_2 \to \L^1(\VN(\mathbb{D}_3))}  \\
&\leq 3^{\frac{1}{p}} \frac{1}{3} \norm{T}_{\cb,\L^1(\VN(\mathbb{D}_3)) \to \L^p(\VN(\mathbb{D}_3))} 
\ov{\eqref{norm-equality-QG}}{=} 3^{\frac{1}{p}-1} \bnorm{(6\mu_1,6\mu_2) \oplus 0}_{\L^p(\VN(\mathbb{D}_3))} \\
&\ov{\eqref{Haar-D3}}{=} 3^{\frac{1}{p}-1} \frac{6}{6^{\frac{1}{p}}} \bnorm{(\mu_1,\mu_2)}_{\ell^p_3} 
=\frac{1}{2^{\frac{1}{p}-1}}  \big(\mu_1^p+\mu_2^p\big)^{\frac{1}{p}}            
= \frac{1}{2^{\frac{1}{p}-1}} \big((1-t)^p+t^p\big)^{\frac{1}{p}}.
\end{align*}     
So by differentiation, we are able to recover a sharp estimate on the completely bounded minimal entropy $\H_{\cb,\min}(T_2)$ of the dephasing channel $T_2$ which is equal to $-t\log_2(t)-(1-t)\log_2(1-t)-1$, see \eqref{Hcbmin-HS}.

That the classical capacity of $T_1$ is known. Using the methods of \cite{Sto07} \cite[Proposition 1]{FuW07}, we see that the knowledge of the classical capacity of $T_2$ (which is also known) is equivalent to the knowledge of the classical capacity of $T$. So we can compute the classical capacity of $T$. Theorem \ref{thm-carac-Fourier} combined with \cite[Theorem 1.1]{JuP15} gives the entanglement-assisted classical capacity $\C_{\EA}(T)$ of the Fourier multiplier $T$. We can also recover this result from $T_1$ and $T_2$.

\begin{remark}\normalfont
If the matrix algebra $\M_6$ is equipped with its normalized trace, we consider the normal unital trace preserving $*$-homomorphism $J \co \L^\infty(\VN(\mathbb{D}_3)) \to \M_6$ defined by
\[
J \bigg( x_1 \oplus x_2 \oplus\begin{bmatrix}
  a_{11}   & a_{12}  \\
  a_{12}   & a_{22} \\
\end{bmatrix}\bigg) 
= \begin{bmatrix} 
x_1 & 0&0 &0 &0 &0  \\ 
0 & x_2 & 0& 0& 0& 0& \\ 
0 &0 & a_{11} & a_{12} & 0&0  \\ 
0 &0 & a_{12} &  a_{22} & 0&0 \\ 
0 &0 &0 &0 &  a_{11} & a_{12} \\
0 &0 &0 & 0&  a_{21} & a_{22} \\
\end{bmatrix}
\]
and the associated canonical trace preserving normal conditional expectation $\E \co \M_6 \to \L^\infty(\VN(\mathbb{D}_3)$. Now, we can consider the quantum channel $JT\E \co \M_6 \to \M_6$ and combine the results of Section \ref{Entropy-capacity-I} and Section \ref{Entropy-capacity}.
\end{remark}

$$
\frac{1}{2}\begin{bmatrix}
  1 & 0 & 0 & 0 & 0 & 0 & 0 & 1\\
 0 & 0 & 0 & 0 & 0 & 0 & 0 & 0\\
 0 & 0 & 0 & 0 & 0 & 0 & 0 & 0\\
 0 & 0 & 0 & 0 & 0 & 0 & 0 & 0\\
 0 & 0 & 0 & 0 & 0 & 0 & 0 & 0\\
 0 & 0 & 0 & 0 & 0 & 0 & 0 & 0\\
 0 & 0 & 0 & 0 & 0 & 0 & 0 & 0\\
 1 & 0 & 0 & 0 & 0 & 0 & 0 & 1\\
\end{bmatrix}
$$

\subsection{Fourier multipliers on the quaternion group $\mathbb{Q}_8$ and unital qubit channels on $\M_2$}
\label{Sec-quaternion}

Recall that $\mathbb{Q}_8=\{\pm 1, \pm \i, \pm j, \pm k \}$. The unique two-dimensional irreducible complex representation is given by 
$$
\pi(\pm 1)
=\begin{bmatrix}
 \pm 1   &  0 \\
   0 &  \pm 1 \\
\end{bmatrix},\quad 
\pi(\pm  \i)
=\begin{bmatrix}
  \pm  \i  &  0 \\
   0  &  \mp  \i \\
\end{bmatrix},\quad 
\pi(\pm j)
=\begin{bmatrix}
   0  &  \pm 1 \\
  \mp 1  &  0 \\
\end{bmatrix}
\textrm{ and }
\pi(\pm k)
=\begin{bmatrix}
  0  & \pm \i  \\
 \pm \i  & 0  \\
\end{bmatrix}.
$$
In particular, we have a $*$-isomorphism $
\Phi \co \VN(\mathbb{Q}_8)
\mapsto \mathbb{C} \oplus \mathbb{C} \oplus \mathbb{C} \oplus \mathbb{C} \oplus \M_2(\mathbb{C})$ defined by
\[
\begin{split}
&\Phi(1)
= 1 \oplus 1 \oplus 1 \oplus 1 \oplus 
\begin{bmatrix}
  1   & 0  \\
  0   & 1 \\
\end{bmatrix}, \qquad
\Phi(-1)
= 1 \oplus 1 \oplus 1 \oplus 1 \oplus 
\begin{bmatrix}
   -1  & 0  \\
   0  & -1  \\
\end{bmatrix},\\
&
\Phi\big(\lambda_{\i}\big)
= 1 \oplus 1 \oplus -1 \oplus -1 \oplus 
\begin{bmatrix}
   \i  &  0 \\
   0  & -\i \\
\end{bmatrix}, \quad
\Phi\big(\lambda_{-\i}\big)
= 1 \oplus 1 \oplus -1 \oplus -1  \oplus 
\begin{bmatrix}
  -\i   &  0 \\
 0  &  \i \\
\end{bmatrix}, \\
&\Phi\big(\lambda_{j}\big)
= 1 \oplus -1 \oplus 1 \oplus -1 \oplus 
\begin{bmatrix}
  0  &  1 \\
  -1   & 0  \\
\end{bmatrix},
\quad
\Phi(\lambda_{-j})
= 1 \oplus -1 \oplus 1 \oplus -1 \oplus 
\begin{bmatrix}
   0  & -1 \\
   1 & 0  \\
\end{bmatrix}.\\
&\Phi(\lambda_{k})
= 1 \oplus -1 \oplus -1 \oplus 1 \oplus 
\begin{bmatrix}
  0  &  \i \\
   \i  &  0 \\
\end{bmatrix},
\quad
\Phi(\lambda_{-k})
=  1 \oplus -1 \oplus -1 \oplus 1 \oplus 
\begin{bmatrix}
   0  & -\i \\
  -\i  & 0  \\
\end{bmatrix}.\\
\end{split} 
\]
We will use the basis $(e_1,e_2,e_3,e_4,e_{11},e_{12},e_{21},e_{22})$. Using the computer code of Section \ref{Appendix3}, we obtain the following formulas for the coproduct
\[
\begin{split}
&\Delta(e_1)
=e_{1} \ot e_{1}+e_{2} \ot e_{2}+e_{3} \ot e_{3}+e_{4} \ot e_{4}+\frac{1}{2}(e_{11} \ot e_{22}-e_{12} \ot e_{21}-e_{21} \ot e_{12}+e_{22} \ot e_{11}), \\
&
\Delta(e_2)
=e_{1} \ot e_{2}+e_{2} \ot e_{1}+e_{3} \ot e_{4}+e_{4} \ot e_{3}+\frac{1}{2}(e_{11} \ot e_{22}+e_{12} \ot e_{21}+e_{21} \ot e_{12}+e_{22} \ot e_{11}), \\
&\Delta(e_3)
=e_{1} \ot e_{3}+e_{2} \ot e_{4}+e_{3} \ot e_{1}+e_{4} \ot e_{2}+\frac{1}{2}(e_{11} \ot e_{11}+e_{12} \ot e_{12}+e_{21} \ot e_{21}+e_{22} \ot e_{22}), \\
&\Delta(e_4)
=e_{1} \ot e_{4}+e_{2} \ot e_{3}+e_{3} \ot e_{2}+e_{4} \ot e_{1}+\frac{1}{2}(e_{11} \ot e_{11}-e_{12} \ot e_{12}-e_{21} \ot e_{21}+e_{22} \ot e_{22}), \\
&
\Delta(e_{11})
=e_{1} \ot e_{11}+e_{2} \ot e_{11}+e_{3} \ot e_{22}+e_{4} \ot e_{22}+e_{11} \ot e_{1}+e_{11} \ot e_{2}+e_{22} \ot e_{3}+e_{22} \ot e_{4},\\
&
\Delta(e_{12})
=e_{1} \ot e_{12}-e_{2} \ot e_{12}-e_{3} \ot e_{21}+e_{4} \ot e_{21}+e_{12} \ot e_{1}-e_{12} \ot e_{2}+e_{21} \ot e_{3}+e_{21} \ot e_{4}, \\
&
\Delta(e_{21})
=e_{1} \ot e_{21}-e_{2} \ot e_{21}-e_{3} \ot e_{12}+e_{4} \ot e_{12}-e_{12} \ot e_{3}+e_{12} \ot e_{4}+e_{21} \ot e_{1}-e_{21} \ot e_{2}, \\
&
\Delta(e_{22})
= e_{1} \ot e_{22}+e_{2} \ot e_{22}+e_{3} \ot e_{11}+e_{4} \ot e_{11}+e_{11} \ot e_{3}+e_{11} \ot e_{4}+e_{22} \ot e_{1}+e_{22} \ot e_{2}.
\end{split}
\]
The Haar state $h$ identifies to
\begin{align}
\MoveEqLeft
\label{Haar-Q8}
h\bigg(x_1 \oplus x_2 \oplus x_3 \oplus x_4 \oplus \begin{bmatrix}
 a_{11}    &  a_{12} \\
  a_{21}   &  a_{22} \\
\end{bmatrix}\bigg)
=\frac{1}{8}\bigg(x_1+ x_2 + x_3 + x_4 + 2\tr\begin{bmatrix}
 a_{11}    &  a_{12} \\
  a_{21}   &  a_{22} \\
\end{bmatrix}\bigg) \\
&=\frac{1}{8}x_1+\frac{1}{8}x_2+\frac{1}{8}x_3+\frac{1}{8}x_4+\frac{1}{4}(a_{11}  + a_{22}). \nonumber
\end{align}
A state on $\L^\infty\VN(\mathbb{Q}_8))$ identifies to a convex combination of four complex numbers and a state of $\M_2(\mathbb{C})$. Using \eqref{State-M2} a state of $\L^\infty(\VN(\mathbb{Q}_8))$ identifies to a sum
\begin{equation}
\label{state-f}
f
=8\mu_1 \oplus 8\mu_2 \oplus 8\mu_3 \oplus 8\mu_4 \oplus 
2\mu_5\begin{bmatrix}
 1+z    &  x+ \i y \\
  x - \i y   &  1-z \\
\end{bmatrix}
\end{equation}
where $\mu_1, \mu_2, \mu_3, \mu_4, \mu_5 \geq 0$, $x,y,z \in \R$ with $\mu_1 + \mu_2 + \mu_3 + \mu_4 + \mu_5=1$ and $x^2+y^2+z^2 \leq 1$. The duality bracket \eqref{bracket} is given by
\begin{align*}
\MoveEqLeft
\left\langle x_1 \oplus x_2 \oplus x_3 \oplus x_4 \oplus \begin{bmatrix}
 a_{11}    &  a_{12} \\
  a_{21}   &  a_{22} \\
\end{bmatrix}, 8\mu_1 \oplus 8\mu_2 \oplus 8\mu_3 \oplus 8\mu_4 \oplus 2\mu_5 \begin{bmatrix}
 1+z    &  x+ \i y \\
  x - \i y   &  1-z \\
\end{bmatrix}\right\rangle_{\L^\infty(\VN(\mathbb{Q}_8)),\L^1(\VN(\mathbb{Q}_8))}  \\ 
&=x_1\mu_1+x_2\mu_2+x_3\mu_3+x_4\mu_4+\frac{\mu_5}{2}\tr\bigg(
\begin{bmatrix}
 a_{11}    &  a_{12} \\
  a_{21}   &  a_{22} \\
\end{bmatrix}\begin{bmatrix}
 1+z    &  x+ \i y \\
  x - \i y   &  1-z \\
\end{bmatrix}^T\bigg) \\         
&=x_1\mu_1+x_2\mu_2+x_3\mu_3+x_4\mu_4 +\frac{\mu_5}{2}\big[(1+z)a_{11}+(x-\i y)a_{21}+(x+\i y)a_{12}+(1-z)a_{22}\big].
\end{align*} 
We have $
f(e_1)=\mu_1, f(e_2)=\mu_2,f(e_3)=\mu_3,f(e_4)=\mu_4, f(e_{11})=\frac{\mu_5}{2}(1+z),f(e_{12})= \frac{\mu_5}{2}(x-\i y),f(e_{21})= \frac{\mu_5}{2}(x+\i y), f(e_{22})=\frac{\mu_5}{2}(1-z) 
$. We deduce that{\small
\begin{align*}
\MoveEqLeft
(f \ot \Id) \circ \Delta(e_{1}) \\
&=(f \ot \Id)\Big(e_{1} \ot e_{1}+e_{2} \ot e_{2}+e_{3} \ot e_{3}+e_{4} \ot e_{4}+\frac{1}{2}(e_{11} \ot e_{22}-e_{12} \ot e_{21}-e_{21} \ot e_{12}+e_{22} \ot e_{11})\Big) \\        
&=f(e_{1}) \ot e_{1}+f(e_{2}) \ot e_{2}+f(e_{3}) \ot e_{3}+f(e_{4}) \ot e_{4}+\frac{1}{2}\big(f(e_{11}) \ot e_{22}-f(e_{12}) \ot e_{21}\\
&-f(e_{21}) \ot e_{12}+f(e_{22}) \ot e_{11}\big) \\
&= \mu_1 e_{1}+ \mu_2 e_{2}+ \mu_3 e_{3}+ \mu_4 e_{4}+\frac{\mu_5}{4}(1+z) e_{22}- \frac{\mu_5}{4}(x-\i y) e_{21}- \frac{\mu_5}{4}(x+\i y) e_{12}+\frac{\mu_5}{4}(1-z) e_{11}.
\end{align*} 
\\[-1cm]
\begin{align*}
\MoveEqLeft
(f \ot \Id) \circ \Delta(e_{2}) \\
&=(f \ot \Id)\Big(e_{1} \ot e_{2}+e_{2} \ot e_{1}+e_{3} \ot e_{4}+e_{4} \ot e_{3}+\frac{1}{2}\big(e_{11} \ot e_{22}+e_{12} \ot e_{21}+e_{21} \ot e_{12}+e_{22} \ot e_{11}\big)\Big)\\        
&=f(e_{1}) \ot e_{2}+f(e_{2}) \ot e_{1}+f(e_{3}) \ot e_{4}+f(e_{4}) \ot e_{3}+\frac{1}{2}\big(f(e_{11}) \ot e_{22}+f(e_{12}) \ot e_{21}+f(e_{21}) \ot e_{12}\\
&+f(e_{22}) \ot e_{11} \big)\\
&= \mu_1 e_{2}+ \mu_2 e_{1}+ \mu_3 e_{4}+ \mu_4 e_{3}+\frac{\mu_5}{4}(1+z) e_{22}+ \frac{\mu_5}{4}(x-\i y) e_{21}+ \frac{\mu_5}{4}(x+\i y) e_{12}+\frac{\mu_5}{4}(1-z) e_{11}.
\end{align*} 
\\[-1cm]
\begin{align*}
\MoveEqLeft
(f \ot \Id) \circ \Delta(e_{3}) \\
&=(f \ot \Id)\Big( e_{1} \ot e_{3}+e_{2} \ot e_{4}+e_{3} \ot e_{1}+e_{4} \ot e_{2}+\frac{1}{2}(e_{11} \ot e_{11}+e_{12} \ot e_{12}+e_{21} \ot e_{21}+e_{22} \ot e_{22})\Big)\\        
&= f(e_{1}) \ot e_{3}+f(e_{2}) \ot e_{4}+f(e_{3}) \ot e_{1}+f(e_{4}) \ot e_{2}+\frac{1}{2}\big(f(e_{11}) \ot e_{11}+f(e_{12}) \ot e_{12}\\
&+f(e_{21}) \ot e_{21}+f(e_{22}) \ot e_{22}\big) \\
&= \mu_1 e_{3}+ \mu_2 e_{4}+ \mu_3 e_{1}+ \mu_4 e_{2}+\frac{\mu_5}{4}(1+z) e_{11}+ \frac{\mu_5}{4}(x-\i y) e_{12}+ \frac{\mu_5}{4}(x+\i y) e_{21}+ \frac{\mu_5}{4}(1-z) e_{22}.
\end{align*} 
\\[-1cm]
\begin{align*}
\MoveEqLeft
(f \ot \Id) \circ \Delta(e_{4}) \\
&=(f \ot \Id)\Big(e_{1} \ot e_{4}+e_{2} \ot e_{3}+e_{3} \ot e_{2}+e_{4} \ot e_{1}+\frac{1}{2}(e_{11} \ot e_{11}-e_{12} \ot e_{12}-e_{21} \ot e_{21}+e_{22} \ot e_{22})\Big)\\        
&=f(e_{1}) \ot e_{4}+f(e_{2}) \ot e_{3}+f(e_{3}) \ot e_{2}+f(e_{4}) \ot e_{1}+\frac{1}{2}\big(f(e_{11}) \ot e_{11}-f(e_{12}) \ot e_{12}\\
&-f(e_{21}) \ot e_{21}+f(e_{22}) \ot e_{22}\big) \\
&=\mu_1 e_{4}+ \mu_2 e_{3}+ \mu_3 e_{2}+ \mu_4 e_{1}+\frac{\mu_5}{4}(1+z) e_{11}- \frac{\mu_5}{4}(x-\i y) e_{12}- \frac{\mu_5}{4}(x+\i y) e_{21}+\frac{\mu_5}{4}(1-z) e_{22}.
\end{align*} 
\\[-1cm]
\begin{align*}
\MoveEqLeft
(f \ot \Id) \circ \Delta(e_{11}) \\
&=(f \ot \Id)(e_{1} \ot e_{11}+e_{2} \ot e_{11}+e_{3} \ot e_{22}+e_{4} \ot e_{22}+e_{11} \ot e_{1}+e_{11} \ot e_{2}+e_{22} \ot e_{3}+e_{22} \ot e_{4}) \\        
&=f(e_{1}) \ot e_{11}+f(e_{2}) \ot e_{11}+f(e_{3}) \ot e_{22}+f(e_{4}) \ot e_{22}+f(e_{11}) \ot e_{1}+f(e_{11}) \ot e_{2}+f(e_{22}) \ot e_{3} \\
&+f(e_{22}) \ot e_{4} \\
&= \mu_1 e_{11}+ \mu_2 e_{11}+ \mu_3 e_{22}+ \mu_4 e_{22}+ \frac{\mu_5}{2}(1+z) e_{1}+ \frac{\mu_5}{2}(1+z) e_{2}+ \frac{\mu_5}{2}(1-z) e_{3}+ \frac{\mu_5}{2}(1-z)  e_{4}.
\end{align*} 
\\[-1cm]
\begin{align*}
\MoveEqLeft
(f \ot \Id) \circ \Delta(e_{12}) \\ 
&=(f \ot \Id)(e_{1} \ot e_{12}-e_{2} \ot e_{12}-e_{3} \ot e_{21}+e_{4} \ot e_{21}+e_{12} \ot e_{1}-e_{12} \ot e_{2}+e_{21} \ot e_{3}+e_{21} \ot e_{4}) \\        
&=f(e_{1}) \ot e_{12}-f(e_{2}) \ot e_{12}-f(e_{3}) \ot e_{21}+f(e_{4}) \ot e_{21}+f(e_{12}) \ot e_{1}-f(e_{12}) \ot e_{2}+f(e_{21}) \ot e_{3}\\
&+f(e_{21}) \ot e_{4}  \\
&= \mu_1 e_{12}- \mu_2 e_{12}- \mu_3 e_{21}+ \mu_4 e_{21}+ \frac{\mu_5}{2}(x-\i y) e_{1}-\frac{\mu_5}{2}(x-\i y)  e_{2}+ \frac{\mu_5}{2}(x+\i y) e_{3}+ \frac{\mu_5}{2}(x+\i y) e_{4}.
\end{align*} 
\\[-1cm]
\begin{align*}
\MoveEqLeft
(f \ot \Id) \circ \Delta(e_{21}) \\
&=(f \ot \Id)(e_{1} \ot e_{21}-e_{2} \ot e_{21}-e_{3} \ot e_{12}+e_{4} \ot e_{12}-e_{12} \ot e_{3}+e_{12} \ot e_{4}+e_{21} \ot e_{1}-e_{21} \ot e_{2})\\        
&=f(e_{1}) \ot e_{21}-f(e_{2}) \ot e_{21}-f(e_{3}) \ot e_{12}+f(e_{4}) \ot e_{12}-f(e_{12}) \ot e_{3}+f(e_{12}) \ot e_{4}\\
&+f(e_{21}) \ot e_{1}-f(e_{21}) \ot e_{2} \\
&= \mu_1 e_{21}- \mu_2 e_{21}- \mu_3 e_{12}+ \mu_4 e_{12}- \frac{\mu_5}{2}(x-\i y) e_{3}+ \frac{\mu_5}{2}(x-\i y) e_{4}+\frac{\mu_5}{2}(x+\i y) e_{1}-\frac{\mu_5}{2}(x+\i y) e_{2}.
\end{align*} 
\\[-1cm]
\begin{align*}
\MoveEqLeft
(f \ot \Id) \circ \Delta(e_{22}) \\
&=(f \ot \Id)\big(e_{1} \ot e_{22}+e_{2} \ot e_{22}+e_{3} \ot e_{11}+e_{4} \ot e_{11}+e_{11} \ot e_{3}+e_{11} \ot e_{4}+e_{22} \ot e_{1}+e_{22} \ot e_{2}\big)
\\
&=f(e_{1}) \ot e_{22}+f(e_{2}) \ot e_{22}+f(e_{3}) \ot e_{11}+f(e_{4}) \ot e_{11}+f(e_{11}) \ot e_{3}+f(e_{11}) \ot e_{4}+f(e_{22}) \ot e_{1}\\
&+f(e_{22}) \ot e_{2}\\        
&= \mu_1 e_{22}+ \mu_2 e_{22}+ \mu_3 e_{11}+ \mu_4 e_{11}+ \frac{\mu_5}{2}(1+z) e_{3}+ \frac{\mu_5}{2}(1+z) e_{4}+ \frac{\mu_5}{2}(1-z) e_{1}+ \frac{\mu_5}{2}(1-z) e_{2}.
\end{align*} }
\noindent With respect to the basis $(e_1,e_2,e_3,e_4,e_{11},e_{12},e_{21},e_{22})$ of $\L^\infty(\VN(\mathbb{Q}_8))$, we deduce that the matrix of the convolution operator $T \co \L^\infty(\VN(\mathbb{Q}_8)) \mapsto \L^\infty(\VN(\mathbb{Q}_8))$, $y \to f*y$ is 
$$
\begin{bmatrix}
\mu_1 & \mu_2 & \mu_3 & \mu_4 & \frac{\mu_5(1+z)}{2} & \frac{\mu_5(x-\i y)}{2} & \frac{\mu_5(x+\i y)}{2}& \frac{\mu_5(1-z)}{2}\\
\mu_2 & \mu_1 & \mu_4 & \mu_3 & \frac{\mu_5(1+z)}{2} & -\frac{\mu_5(x-\i y)}{2} & -\frac{\mu_5(x+\i y)}{2} & \frac{\mu_5(1-z)}{2} \\
\mu_3 & \mu_4 & \mu_1 & \mu_2 & \frac{\mu_5(1-z)}{2} & \frac{\mu_5(x+\i y)}{2} & - \frac{\mu_5(x-\i y)}{2} & \frac{\mu_5(1+z)}{2} \\
\mu_4 & \mu_3 & \mu_2 &\mu_1 & \frac{\mu_5(1-z)}{2} & \frac{\mu_5(x+\i y)}{2} & \frac{\mu_5(x-\i y)}{2} & \frac{\mu_5(1+z)}{2} \\
\frac{\mu_5(1-z)}{4} & \frac{\mu_5(1-z)}{4} & \frac{\mu_5(1+z)}{4} & \frac{\mu_5(1+z)}{4} & \mu_1+\mu_2 & 0 & 0 & \mu_3+\mu_4 \\
- \frac{\mu_5(x+\i y)}{4} & \frac{\mu_5(x+\i y)}{4} & \frac{\mu_5(x-\i y)}{4} & - \frac{\mu_5(x-\i y)}{4} & 0 & \mu_1-\mu_2 & \mu_4-\mu_3 & 0 \\
- \frac{\mu_5(x-\i y)}{4} & \frac{\mu_5(x-\i y)}{4} & \frac{\mu_5(x+\i y)}{4} & - \frac{\mu_5(x+\i y)}{4} & 0 & \mu_4-\mu_3 & \mu_1-\mu_2 & 0 \\
\frac{\mu_5(1+z)}{4} & \frac{\mu_5(1+z)}{4} & \frac{\mu_5(1-z)}{4} & \frac{\mu_5(1-z)}{4} & \mu_3+\mu_4 & 0 & 0 & \mu_1+\mu_2 \\
\end{bmatrix}.
$$
If $\mu_5=0$, the subspaces $\vect \{e_{1},e_{2},e_{3},e_{4}\}$ and $\vect \{e_{11},e_{12},e_{21},e_{22}\}$ are obviously stable under the convolution operator $T$ and this operator identifies to a direct sum $T_1 \oplus T_2 \co \ell^\infty_4\oplus\M_2 \to \ell^\infty_4 \oplus \M_2$. In this case, the entropy is given by
\begin{align*}
\MoveEqLeft
\H(8\mu_1 \oplus 8\mu_2 \oplus 8\mu_3 \oplus 8\mu_4 \oplus 0) \\         
&=-h[8\mu_1\log(8\mu_1) \oplus 8\mu_2\log(8\mu_2) \oplus 8\mu_3\log( 8\mu_3) \oplus 8\mu_4\log( 8\mu_4) \oplus 0]\\
&=-\mu_1\log( 8\mu_1) -\mu_2\log( 8\mu_2) -\mu_3\log( 8\mu_3) -\mu_4\log( 8\mu_4)  \\
&=-\mu_1\log\mu_1 -\mu_2\log \mu_2 -\mu_3\log\mu_3 -\mu_4\log\mu_4-\log_2 8.
\end{align*}
Observe that if we consider
\begin{equation}
\label{mu-DEP}
\mu_1=1-\frac{3t}{4},\quad \mu_2=\frac{t}{4},\quad \mu_3=\frac{t}{4},\quad \mu_4=\frac{t}{4}
\end{equation}
for some $0 \leq t \leq 1$ then the matrix of $T_2$ is given by 
$
\begin{bmatrix}
   1-\frac{t}{2}   & 0  & 0 &  \frac{t}{2}  \\
  0   &  t& 0  & 0  \\
  0   &  0  & t  & 0  \\
	\frac{t}{2}    & 0  & 0 &   1-\frac{t}{2} \\
\end{bmatrix}
$. Of course, the reader will notice with a small computation that it is the matrix of the depolarizing channel $
T_2 \co S^1_2 \to S^1_2, x \mapsto (1-t)x+t\tr(x)\frac{\I_2}{2}$ of \cite[Example 2.13 p.~18]{Pet08} with $n=2$ and with $1-t$ instead of $p$. 

Note that the classical channel $T_1 \co \ell^1_4 \to \ell^1_4$ is of course entanglement breaking. Using \cite[(7.25)]{CoT06} (see \eqref{classical-of-classical-convolution}), we know that
\begin{align*}
\MoveEqLeft
\chi(T_1)
=\C(T_1)
=\log_2(4)-\H(\mu_1,\mu_2,\mu_3,\mu_4) \\
&=\log_2(4)+\mu_1\log_2(\mu_1)+\mu_2\log_2(\mu_2)+\mu_3\log_2(\mu_3)+\mu_4\log_2(\mu_4) \\
&=2+\bigg(1-\frac{3t}{4}\bigg)\log_2\bigg(1-\frac{3t}{4}\bigg)+\frac{3t}{4}\log_2\bigg(\frac{t}{4}\bigg).
\end{align*}
Finally, the completely bounded norm of $T_1 \co \ell^1_4 \to \ell^p_4$ is equal to the norm of $T_1$ (see Corollary \ref{Cor-abelian-AVN}) and this last quantity is easy to compute.

By \cite[p.~84]{Key02}  \cite[p.~100]{Pet08} \cite[Theorem 20.4.3 p.~561]{Wil17}, the Holevo capacity of the depolarizing channel $T_2$ is given by
\begin{equation}
\label{}
\chi(T_2)
=\C(T_2)
=1-\H_{\min}(T_2)
=1+\bigg(1-\frac{t}{2}\bigg)\log_2\bigg(1-\frac{t}{2}\bigg)+\frac{t}{2}\log_2\bigg(\frac{t}{2}\bigg).
\end{equation}
It is equivalent to say that $\H_{\min}(T_2)=-\big(1-\frac{t}{2}\big)\log_2\big(1-\frac{t}{2}\big)-\frac{t}{2}\log_2\big(\frac{t}{2}\big)$ (for the non-normalized trace). By \cite[p.~84]{Key02} \cite[p.~597]{Wil17}, its entanglement-assisted classical capacity is equal to
\begin{equation}
\label{CEA-depo}
\C_{\EA}(T_2)
=2+\bigg(1-\frac{3t}{4}\bigg)\log_2\bigg(1-\frac{3t}{4}\bigg)+\frac{3t}{4}\log_2\bigg(\frac{t}{4}\bigg).
\end{equation}
Finally, the value of the opposite of the completely bounded minimal entropy is
\begin{equation}
\label{cbmin-depo}
-\H_{\cb,\min}(T_2)
=1+\bigg(1-\frac{3t}{4}\bigg)\log_2\bigg(1-\frac{3t}{4}\bigg)+\frac{3t}{4}\log_2\bigg(\frac{t}{4}\bigg)
\end{equation}
since we have by \cite[(5.6) p.~53]{DJKRB06} the formula
\begin{equation}
\label{cb-depo-norm}
\norm{T_2}_{\cb,S^1_2 \to S^p_2}
=\frac{1}{2^{1+\frac{1}{p}}} \big[(4-3t)^p+3 t^p\big]^{\frac{1}{p}}
=\frac{1}{2^{\frac{1}{p}-1}} \bigg(\Big(1-\frac{3t}{4}\Big)^p+\frac{3 t^p}{4^p}\bigg)^{\frac{1}{p}}.
\end{equation}
Note that $-\H_{\cb,\min}(T_2)$ is equal to $\Q^{(1)}(T_2)$ by \cite[p.~662]{Wil17}.

Our results allows us to recover some (sharp) inequalities provided by these formulas. Indeed, For example, consider the canonical maps $J \co \M_2 \to \VN(\mathbb{Q}_8)$, $x \mapsto 0 \oplus 0 \oplus 0 \oplus 0 \oplus x$ and $\E \co \VN(\mathbb{Q}_8) \to \M_2$. If $x \in S^1_2$, we have 
\begin{equation}
\label{inter-987}
\norm{J(x)}_{\L^1(\VN(\mathbb{Q}_8))} 
=\norm{0 \oplus 0 \oplus 0 \oplus 0 \oplus x}_{\L^1(\VN(\mathbb{Q}_8))} 
\ov{\eqref{Haar-Q8}}{=} \frac{1}{4}\norm{x}_{S^1_2}.
\end{equation}
Moreover, if $y$ belongs to $\L^\infty(\VN(\mathbb{Q}_8))$ we have 
$$
\norm{\E(y)}_{S^p_2}
=4^{\frac{1}{p}} \norm{0 \oplus 0 \oplus 0 \oplus 0 \oplus \E(y)}_{\L^p(\VN(\mathbb{Q}_8))} 
\leq 4^{\frac{1}{p}} \norm{y}_{\L^p(\VN(\mathbb{Q}_8))}.
$$ 
We deduce that
\begin{align*}
\MoveEqLeft
\norm{T_2}_{\cb,S^1_2 \to S^p_2}
=\norm{\E T J}_{\cb,S^1_2 \to S^p_2} \\
&\leq \norm{\E}_{\cb,\L^p(\VN(\mathbb{Q}_8)) \to S^p_2} \norm{T}_{\cb,\L^1(\VN(\mathbb{Q}_8)) \to \L^p(\VN(\mathbb{Q}_8))} \norm{J}_{\cb, S^1_2 \to \L^1(\VN(\mathbb{Q}_8))}  \\
&\leq 4^{\frac{1}{p}} \frac{1}{4} \norm{T}_{\cb,\L^1(\VN(\mathbb{Q}_8)) \to \L^p(\VN(\mathbb{Q}_8))} 
\ov{\eqref{norm-equality-QG}}{=} 4^{\frac{1}{p}-1} \bnorm{(8\mu_1,8\mu_2,8\mu_3,8\mu_4) \oplus 0}_{\L^p(\VN(\mathbb{Q}_8))} \\
&\ov{\eqref{Haar-Q8}}{=} 4^{\frac{1}{p}-1} \frac{8}{8^{\frac{1}{p}}} \bnorm{(\mu_1,\mu_2,\mu_3,\mu_4)}_{\ell^p_4} 
=\frac{1}{2^{\frac{1}{p}-1}} \bigg(\sum_{i=1}^4 \mu_i^p\bigg)^{\frac{1}{p}}            
\ov{\eqref{mu-DEP}}{=} \frac{1}{2^{\frac{1}{p}-1}} \bigg(\Big(1-\frac{3t}{4}\Big)^p+\frac{3 t^p}{4^p}\bigg)^{\frac{1}{p}}.
\end{align*}   
So by differentiation we recover a sharp estimate on $-\H_{\cb,\min}(T_2)$\footnote{\thefootnote. The author believes that we can obtain an equality since $\H_{\cb,\min}(T) \not=\H_{\cb,\min}(T_1)$}.

Using the methods of \cite{Sto07} \cite[Proposition 1]{FuW07}, we see that the knowledge of the classical capacity of $T_2$ is equivalent to the knowledge of the classical capacity of $T$. So we can compute the classical capacity of $T$. 

Observe that Theorem \ref{thm-carac-Fourier} combined with \cite[Theorem 1.1]{JuP15} gives the entanglement-assisted classical capacity $\C_{\EA}(T)$ of the Fourier multiplier $T$. We can recover this result with \eqref{CEA-depo}.

\begin{remark} \normalfont
By \cite[Theorem 2.2]{ChL23}, for any unital qubit channel $\Phi \co\M_2 \to \M_2$, there exist unitaries $u,v \in \M_2$ such that the Choi matrix of the map $\Phi_{u,v} \co \M_2 \to \M_2$, $x \mapsto v^*\Phi(uxu^*)v$ is  
\begin{equation*}
\label{qubit-abc}
\frac{1}{2} {\small 
\begin{bmatrix} 
\lambda_1 + \lambda_2 & 0 & 0 & \lambda_1-\lambda_2 \\
0 & \lambda_3 + \lambda_4 & \lambda_3 - \lambda_4 & 0 \\
0 & \lambda_3-\lambda_4 & \lambda_3+\lambda_4 & 0 \\
\lambda_1-\lambda_2 & 0 & 0 & \lambda_1+\lambda_2 \\ 
\end{bmatrix}}
\end{equation*}
where $\lambda_1 \geq \cdots \geq \lambda_4 \geq 0$ are the eigenvalues of the Choi matrix of $\Phi$, which satisfy $\lambda_1 + \lambda_2 + \lambda_3 + \lambda_4=2$. That means that the matrix of $\Phi_{u,v}$ with respect to the basis $(e_{11},e_{12},e_{21},e_{22})$ is given by
\begin{equation*}
\label{qubit-bis}
\frac{1}{2} {\small 
\begin{bmatrix} 
\lambda_1 + \lambda_2 & 0 & 0 & \lambda_3+\lambda_4 \\
0 & \lambda_1 - \lambda_2 & \lambda_3 - \lambda_4 & 0 \\
0 & \lambda_3-\lambda_4 & \lambda_1-\lambda_2 & 0 \\
 \lambda_3 + \lambda_4 & 0 & 0 & \lambda_1+\lambda_2 \\ 
\end{bmatrix}}.
\end{equation*}
Looking at the matrix of the convolution operator, we conclude that \textit{any} unital qubit channel is unitarily equivalent to a summand $T_1$ of a Fourier multiplier $T=T_1 \oplus T_2$ on the quaternion group $\mathbb{Q}_8$. We refer to \cite{RSW02} and \cite{Rus03} more information on unital qubit channels.

We plan to investigate finite groups with irreducible representations of higher degree in the future.
\end{remark}

\subsection{Fourier multipliers on the dihedral group $\mathbb{D}_4$ and transpose depolarizing channels}
\label{Sec-Dihedral-bis}

With \eqref{rep-dihedral}, we have a $*$-isomorphism $
\Phi \co \VN(\mathbb{D}_4)
\mapsto \mathbb{C} \oplus \mathbb{C} \oplus \mathbb{C} \oplus \mathbb{C} \oplus \M_2(\mathbb{C})$ defined by
\[
\begin{split}
&\Phi(1)
= 1 \oplus 1 \oplus 1 \oplus 1 \oplus 
\begin{bmatrix}
   1  &  0 \\
   0  & 1 \\
\end{bmatrix}, \qquad
\Phi(\lambda_s)
= 1 \oplus -1 \oplus -1 \oplus 1 \oplus 
\begin{bmatrix}
   0  & 1  \\
   1  &  0 \\
\end{bmatrix},\\
&
\Phi\big(\lambda_{r}\big)
= 1 \oplus 1 \oplus -1 \oplus -1 \oplus 
\begin{bmatrix}
  \i   &  0 \\
  0   & -\i \\
\end{bmatrix}, \quad
\Phi\big(\lambda_{sr}\big)
= 1 \oplus -1 \oplus 1 \oplus -1 \oplus 
\begin{bmatrix}
  0  & -\i  \\
  \i & 0  \\
\end{bmatrix}, \\
&\Phi\big(\lambda_{r^2}\big)
= 1 \oplus 1 \oplus 1 \oplus 1 \oplus 
\begin{bmatrix}
  -1  & 0 \\
   0  & -1  \\
\end{bmatrix},
\quad
\Phi(\lambda_{sr^2})
= 1 \oplus -1 \oplus -1 \oplus 1 \oplus 
\begin{bmatrix}
   0  & -1 \\
   -1 & 0  \\
\end{bmatrix}.\\
&\Phi(\lambda_{r^3})
= 1 \oplus 1 \oplus -1 \oplus -1 \oplus 
\begin{bmatrix}
  -\i  & 0  \\
   0  & \i  \\
\end{bmatrix},
\quad
\Phi(\lambda_{sr^3})
= 1 \oplus -1 \oplus 1 \oplus -1 \oplus 
\begin{bmatrix}
   0  & \i \\
   -\i & 0  \\
\end{bmatrix}.\\
\end{split} 
\]
We will use the basis $(e_1,e_2,e_3,e_4,e_{11},e_{12},e_{21},e_{22})$. Using the computer code of Section \ref{Appendix3}, we obtain
\[
\begin{split}
&\Delta(e_1)
=e_{1} \ot e_{1}+e_{2} \ot e_{2}+e_{3} \ot e_{3}+e_{4} \ot e_{4}+\frac{1}{2}\big(e_{11} \ot e_{22}+e_{12} \ot e_{21}+e_{21} \ot e_{12}+e_{22} \ot e_{11}\big), \\
&
\Delta(e_2)
=e_{1} \ot e_{2}+e_{2} \ot e_{1}+e_{3} \ot e_{4}+e_{4} \ot e_{3}+\frac{1}{2}(e_{11} \ot e_{22}-e_{12} \ot e_{21}-e_{21} \ot e_{12}+e_{22} \ot e_{11}), \\
&\Delta(e_3)
=e_{1} \ot e_{3}+e_{2} \ot e_{4}+e_{3} \ot e_{1}+e_{4} \ot e_{2}+\frac{1}{2}(e_{11} \ot e_{11}-e_{12} \ot e_{12}-e_{21} \ot e_{21}+e_{22} \ot e_{22}), \\
&\Delta(e_4)
=e_{1} \ot e_{4}+e_{2} \ot e_{3}+e_{3} \ot e_{2}+e_{4} \ot e_{1}+\frac{1}{2}(e_{11} \ot e_{11}+e_{12} \ot e_{12}+e_{21} \ot e_{21}+e_{22} \ot e_{22}), \\
&
\Delta(e_{11})
=e_{1} \ot e_{11}+e_{2} \ot e_{11}+e_{3} \ot e_{22}+e_{4} \ot e_{22}+e_{11} \ot e_{1}+e_{11} \ot e_{2}+e_{22} \ot e_{3}+e_{22} \ot e_{4},\\
&
\Delta(e_{12})
=e_{1} \ot e_{12}-e_{2} \ot e_{12}-e_{3} \ot e_{21}+e_{4} \ot e_{21}+e_{12} \ot e_{1}-e_{12} \ot e_{2}-e_{21} \ot e_{3}+e_{21} \ot e_{4}, \\
&
\Delta(e_{21})
=e_{1} \ot e_{21}-e_{2} \ot e_{21}-e_{3} \ot e_{12}+e_{4} \ot e_{12}-e_{12} \ot e_{3}+e_{12} \ot e_{4}+e_{21} \ot e_{1}-e_{21} \ot e_{2}, \\
&
\Delta(e_{22})
=e_{1} \ot e_{22}+e_{2} \ot e_{22}+e_{3} \ot e_{11}+e_{4} \ot e_{11}+e_{11} \ot e_{3}+e_{11} \ot e_{4}+e_{22} \ot e_{1}+e_{22} \ot e_{2}.
\end{split}
\]
If $f$ is a state as in \eqref{state-f}, we have $
f(e_1)=\mu_1, f(e_2)=\mu_2,f(e_3)=\mu_3,f(e_4)=\mu_4, f(e_{11})=\frac{\mu_5}{2}(1+z),f(e_{12})= \frac{\mu_5}{2}(x-\i y),f(e_{21})= \frac{\mu_5}{2}(x+\i y), f(e_{22})=\frac{\mu_5}{2}(1-z) 
$.  We deduce that{\small
\begin{align*}
\MoveEqLeft
(f \ot \Id) \circ \Delta(e_{1}) \\
&=(f \ot \Id)\Big(e_{1} \ot e_{1}+e_{2} \ot e_{2}+e_{3} \ot e_{3}+e_{4} \ot e_{4}+\frac{1}{2}\big(e_{11} \ot e_{22}+e_{12} \ot e_{21}+e_{21} \ot e_{12}+e_{22} \ot e_{11}\big)\Big) \\        
&=f(e_{1}) e_{1}+f(e_{2}) e_{2}+f(e_{3}) e_{3}+f(e_{4}) e_{4}+\frac{1}{2}\big(f(e_{11}) e_{22}+f(e_{12}) e_{21}+f(e_{21}) e_{12}+f(e_{22}) e_{11}\big) \\
&=\mu_1 e_{1}+\mu_2 e_{2}+\mu_3 e_{3}+\mu_4 e_{4}+\frac{\mu_5}{4}(1+z) e_{22}+\frac{\mu_5}{4}(x-\i y) e_{21}+\frac{\mu_5}{4}(x+\i y) e_{12}+\frac{\mu_5}{4}(1-z) e_{11}.
\end{align*} 
\\[-1cm]
\begin{align*}
\MoveEqLeft
(f \ot \Id) \circ \Delta(e_{2}) \\
&=f(e_{1}) e_{2}+f(e_{2}) e_{1}+f(e_{3}) e_{4}+f(e_{4}) e_{3}+\frac{1}{2}\big(f(e_{11}) e_{22}-f(e_{12}) e_{21}-f(e_{21}) e_{12}+f(e_{22}) e_{11}) \big)\\        
&=f(e_1)  e_{2}+f(e_2) e_{1}+f(e_3) e_{4}+f(e_4) e_{3}+\frac{1}{2}\big(f(e_{11}) e_{22}-f(e_{12}) e_{21}-f(e_{21}) e_{12}+f(e_{22}) e_{11} \big)\\
&=\mu_1 e_{2}+\mu_2 e_{1}+\mu_3 e_{4}+\mu_4 e_{3}+\frac{\mu_5}{4}(1+z) e_{22}-\frac{\mu_5}{4}(x-\i y) e_{21}-\frac{\mu_5}{4}(x+\i y) e_{12}+\frac{\mu_5}{4}(1-z) e_{11}.
\end{align*} 
\\[-1cm]
\begin{align*}
\MoveEqLeft
(f \ot \Id) \circ \Delta(e_{3}) \\
&=(f \ot \Id)\Big(e_{1} \ot e_{3}+e_{2} \ot e_{4}+e_{3} \ot e_{1}+e_{4} \ot e_{2}+\frac{1}{2}(e_{11} \ot e_{11}-e_{12} \ot e_{12}-e_{21} \ot e_{21}+e_{22} \ot e_{22})\Big)\\        
&= f(e_{1}) e_{3}+f(e_{2}) e_{4}+f(e_{3}) e_{1}+f(e_{4}) e_{2}+\frac{1}{2}\big(f(e_{11}) e_{11}-f(e_{12}) e_{12}-f(e_{21}) e_{21}+f(e_{22}) e_{22})\big)\\
&=\mu_1 e_{3}+\mu_2 e_{4}+\mu_3 e_{1}+\mu_4 e_{2}+\frac{\mu_5}{4}(1+z) e_{11}-\frac{\mu_5}{4}(x-\i y) e_{12}-\frac{\mu_5}{4}(x+\i y) e_{21}+\frac{\mu_5}{4}(1-z) e_{22}.
\end{align*} 
\\[-1cm]
\begin{align*}
\MoveEqLeft
(f \ot \Id) \circ \Delta(e_{4}) \\
&=(f \ot \Id)\Big(e_{1} \ot e_{4}+e_{2} \ot e_{3}+e_{3} \ot e_{2}+e_{4} \ot e_{1}+\frac{1}{2}(e_{11} \ot e_{11}+e_{12} \ot e_{12}+e_{21} \ot e_{21}+e_{22} \ot e_{22})\Big)\\        
&=f(e_{1}) e_{4}+f(e_{2}) e_{3}+f(e_{3}) e_{2}+f(e_{4}) e_{1}+\frac{1}{2}\big(f(e_{11}) e_{11}+f(e_{12}) \ot e_{12}+f(e_{21}) e_{21}+f(e_{22}) e_{22}\big)\\
&=\mu_1 e_{4}+\mu_2 e_{3}+\mu_3 e_{2}+\mu_4 e_{1}+\frac{\mu_5}{4}(1+z) e_{11}+\frac{\mu_5}{4}(x-\i y) e_{12}+\frac{\mu_5}{4}(x+\i y) e_{21}+\frac{\mu_5}{4}(1-z) e_{22}.
\end{align*} 
\\[-1cm]
\begin{align*}
\MoveEqLeft
(f \ot \Id) \circ \Delta(e_{11}) \\
&=(f \ot \Id)(e_{1} \ot e_{11}+e_{2} \ot e_{11}+e_{3} \ot e_{22}+e_{4} \ot e_{22}+e_{11} \ot e_{1}+e_{11} \ot e_{2}+e_{22} \ot e_{3}+e_{22} \ot e_{4}) \\        
&=f(e_{1}) e_{11}+f(e_{2}) e_{11}+f(e_{3}) e_{22}+f(e_{4}) e_{22}+f(e_{11}) e_{1}+f(e_{11}) e_{2}+f(e_{22}) e_{3}+f(e_{22}) e_{4} \\
&=\mu_1  e_{11}+\mu_2  e_{11}+\mu_3  e_{22}+\mu_4  e_{22}+\frac{\mu_5}{2}(1+z) e_{1}+\frac{\mu_5}{2}(1+z) e_{2}+\frac{\mu_5}{2}(1-z) e_{3}+\frac{\mu_5}{2}(1-z) e_{4} .
\end{align*} 
\\[-1cm]
\begin{align*}
\MoveEqLeft
(f \ot \Id) \circ \Delta(e_{12}) \\ 
&=(f \ot \Id)(e_{1} \ot e_{12}-e_{2} \ot e_{12}-e_{3} \ot e_{21}+e_{4} \ot e_{21}+e_{12} \ot e_{1}-e_{12} \ot e_{2}-e_{21} \ot e_{3}+e_{21} \ot e_{4}) \\        
&=f(e_{1}) e_{12}-f(e_{2}) e_{12}-f(e_{3}) e_{21}+f(e_{4}) e_{21}+f(e_{12}) e_{1}-f(e_{12}) e_{2}-f(e_{21}) e_{3}+f(e_{21}) e_{4}\\
&=\mu_1  e_{12}-\mu_2 e_{12}-\mu_3  e_{21}+\mu_4  e_{21}+\frac{\mu_5}{2}(x-\i y) e_{1}-\frac{\mu_5}{2}(x-\i y) e_{2}-\frac{\mu_5}{2}(x+\i y) e_{3}+\frac{\mu_5}{2}(x+\i y) e_{4} .
\end{align*} 
\\[-1cm]
\begin{align*}
\MoveEqLeft
(f \ot \Id) \circ \Delta(e_{21}) \\
&=(f \ot \Id)(e_{1} \ot e_{21}-e_{2} \ot e_{21}-e_{3} \ot e_{12}+e_{4} \ot e_{12}-e_{12} \ot e_{3}+e_{12} \ot e_{4}+e_{21} \ot e_{1}-e_{21} \ot e_{2})\\        
&=f(e_{1}) e_{21}-f(e_{2}) e_{21}-f(e_{3}) e_{12}+f(e_{4}) e_{12}-f(e_{12}) e_{3}+f(e_{12}) e_{4}+f(e_{21}) e_{1}-f(e_{21}) e_{2}\\
&=\mu_1  e_{21}-\mu_2 e_{21}-\mu_3 e_{12}+\mu_4 e_{12}-\frac{\mu_5}{2}(x-\i y) e_{3}+\frac{\mu_5}{2}(x-\i y) e_{4}+\frac{\mu_5}{2}(x+\i y) e_{1}-\frac{\mu_5}{2}(x+\i y) e_{2} .
\end{align*} 
\\[-1cm]
\begin{align*}
\MoveEqLeft
(f \ot \Id) \circ \Delta(e_{22}) \\
&=(f \ot \Id)\big(e_{1} \ot e_{22}+e_{2} \ot e_{22}+e_{3} \ot e_{11}+e_{4} \ot e_{11}+e_{11} \ot e_{3}+e_{11} \ot e_{4}+e_{22} \ot e_{1}+e_{22} \ot e_{2}\big)
\\
&=f(e_{1}) e_{22}+f(e_{2}) e_{22}+f(e_{3}) e_{11}+f(e_{4}) e_{11}+f(e_{11}) e_{3}+f(e_{11}) e_{4}+f(e_{22}) e_{1}+f(e_{22}) e_{2}\\       
&=\mu_1  e_{22}+\mu_2 e_{22}+\mu_3 e_{11}+\mu_4 e_{11}+\frac{\mu_5}{2}(1+z) e_{3}+\frac{\mu_5}{2}(1+z) e_{4}+\frac{\mu_5}{2}(1-z)  e_{1}+\frac{\mu_5}{2}(1-z)  e_{2} .
\end{align*} 
}
With respect to the basis $(e_1,e_2,e_3,e_4,e_{11},e_{12},e_{21},e_{22})$ of $\L^\infty(\VN(\mathbb{D}_4))$, we deduce that the matrix of the convolution operator $T \co \L^\infty(\VN(\mathbb{D}_4)) \mapsto \L^\infty(\VN(\mathbb{D}_4))$, $y \to f*y$ is 
$$
\begin{bmatrix}
\mu_1 & \mu_2 & \mu_3 & \mu_4 & \frac{\mu_5(1+z)}{2} & \frac{\mu_5(x-\i y)}{2} & \frac{\mu_5(x+\i y)}{2} & \frac{\mu_5(1-z)}{2}\\
\mu_2 & \mu_1 & \mu_4 & \mu_3  & \frac{\mu_5(1+z)}{2} & -\frac{\mu_5(x-\i y)}{2} & -\frac{\mu_5(x+\i y)}{2} & \frac{\mu_5(1-z)}{2} \\
\mu_3 & \mu_4 & \mu_1 & \mu_2 & \frac{\mu_5(1-z)}{2} & -\frac{\mu_5(x+\i y)}{2}  & -\frac{\mu_5(x-\i y)}{2} & \frac{\mu_5(1+z)}{2} \\
\mu_4 & \mu_3 & \mu_2 & \mu_1 & \frac{\mu_5(1-z)}{2} & \frac{\mu_5(x+\i y)}{2} & \frac{\mu_5(x-\i y)}{2} & \frac{\mu_5(1+z)}{2} \\
\frac{\mu_5(1-z)}{4} & \frac{\mu_5(1-z)}{4} & \frac{\mu_5(1+z)}{4} & \frac{\mu_5(1+z)}{4} & \mu_1+\mu_2 & 0 & 0 & \mu_3+\mu_4 \\
\frac{\mu_5(x+\i y)}{4} & -\frac{\mu_5(x+\i y)}{4} & -\frac{\mu_5(x-\i y)}{4} & \frac{\mu_5(x-\i y)}{4} & 0 & \mu_1-\mu_2 & \mu_4-\mu_3& 0 \\
 -\frac{\mu_5(x-\i y)}{4} & -\frac{\mu_5(x-\i y)}{4} & -\frac{\mu_5(x+\i y)}{4} & \frac{\mu_5(x+\i y)}{4} & 0 & \mu_4-\mu_3 & \mu_1-\mu_2 & 0 \\
\frac{\mu_5(1+z)}{4} & \frac{\mu_5(1+z)}{4} & \frac{\mu_5(1-z)}{4} & \frac{\mu_5(1-z)}{4} & \mu_3+\mu_4 & 0 & 0 & \mu_1+\mu_2 \\
\end{bmatrix}.
$$
Observe that if $\mu_1=\frac{1+t}{4}$, $\mu_2=\frac{1+t}{4}$, $\mu_3=\frac{1-3t}{4}$, $\mu_4=\frac{1+t}{4}$ then the matrix of $T_2$ is given by 
$
\begin{bmatrix}
   \frac{1+t}{2}   & 0  & 0 &  \frac{1-t}{2}  \\
  0   &  0& t  & 0  \\
  0   &  t  & 0  & 0  \\
	\frac{1-t}{2}    & 0  & 0 &   \frac{1+t}{2} \\
\end{bmatrix}.
$
An easy computation that it is the matrix of the transpose depolarizing channel $T_2 \co S^1_2 \to S^1_2$, $x \mapsto tx^T+(1-t)\tr(x)\frac{\I_2}{2}$ of \cite[Example 2.18 p.~21]{Pet08} with $n=2$.

The Werner-Holevo channel $T_{\WH} \co S^1_2 \to S^1_2$ was investigated in \cite{WeH02} and is a particular case with $t=-1$. It is defined by $T_{\WH}(x) = (\tr x) \I_2 - x^T $, see also \cite[p.~21]{Pet08}. It is known \cite[p.~54]{DJKRB06} that 
$$
\norm{T_{\WH}}_{\cb, S^1_2 \to S^p_2} 
=2^{1- \frac{1}{p}}. 
$$
In particular, we have $\H_{\cb,\min}(T_{\WH}) = -1$.

With the same method than the one of Section \ref{Sec-Dihedral-bis}, we can recover some information. For example, with $J \co \M_2 \to \VN(\mathbb{Q}_8)$, $x \mapsto 0 \oplus 0 \oplus 0 \oplus 0 \oplus x$ and $\E \co \VN(\mathbb{Q}_8) \to \M_2$, we obtain the estimate
\begin{align*}
\MoveEqLeft
\norm{T_2}_{\cb,S^1_2 \to S^p_2}
=\norm{\E T J}_{\cb,S^1_2 \to S^p_2} \\
&\leq \norm{\E}_{\cb,\L^p(\VN(\mathbb{Q}_8)) \to S^p_2} \norm{T}_{\cb,\L^1(\VN(\mathbb{Q}_8)) \to \L^p(\VN(\mathbb{Q}_8))} \norm{J}_{\cb, S^1_2 \to \L^1(\VN(\mathbb{Q}_8))}  \\
&\leq 4^{\frac{1}{p}} \frac{1}{4} \norm{T}_{\cb,\L^1(\VN(\mathbb{Q}_8)) \to \L^p(\VN(\mathbb{Q}_8))} 
\ov{\eqref{norm-equality-QG}}{=} 4^{\frac{1}{p}-1} \bnorm{(8\mu_1,8\mu_2,8\mu_3,8\mu_4) \oplus 0}_{\L^p(\VN(\mathbb{Q}_8))} \\
&\ov{\eqref{Haar-Q8}}{=} 4^{\frac{1}{p}-1} \frac{8}{8^{\frac{1}{p}}} \bnorm{(\mu_1,\mu_2,\mu_3,\mu_4)}_{\ell^p_4} 
=\frac{1}{2^{\frac{1}{p}-1}} \bigg(\sum_{i=1}^4 \mu_i^p\bigg)^{\frac{1}{p}}            
\ov{\eqref{mu-DEP}}{=} \frac{1}{2^{\frac{1}{p}-1}}
=2^{1- \frac{1}{p}}.
\end{align*}   
We refer to \cite{DHS04} \cite{DHS06} \cite{FHMV04} for more information on this quantum channel.

%

\subsection{Fourier multipliers on the Kac-Paljutkin quantum group $\mathbb{KP}$}
\label{Sec-Kac-Paljukin}

Now, we describe the Kac-Paljutkin quantum group $\KP$ introduced in \cite{KaP66} (see also \cite{FrG06} for some information). This is the smallest finite quantum group that is neither commutative nor cocommutative. We consider the finite-dimensional von Neumann algebra $\L^\infty(\KP) \ov{\mathrm{def}}{=} \mathbb{C} \oplus \mathbb{C} \oplus \mathbb{C} \oplus \mathbb{C} \oplus \M_2(\mathbb{C})$. We will use the basis $(e_1,e_2,e_3,e_4,e_{11},e_{12},e_{21},e_{22})$\footnote{\thefootnote. Here we have $e_3=0 \oplus 0 \oplus 1 \oplus 0 \oplus \begin{bmatrix}
  0   & 0  \\
  0   & 0  \\
		\end{bmatrix}$ and $e_{21}=0 \oplus 0 \oplus 0 \oplus 0 \oplus \begin{bmatrix}
  0   & 0  \\
  1   & 0  \\
		\end{bmatrix}$.}. The formulas
\[
\begin{split}
&\Delta(e_1)
=e_1 \ot e_1+e_2 \ot e_2+e_3 \ot e_3+e_4 \ot e_4 
+\frac{1}{2}(e_{11} \ot e_{11}+e_{12} \ot e_{12}+e_{21} \ot e_{21}+e_{22} \ot e_{22}) \\
&
\Delta(e_2)
=e_1 \ot e_2+e_2 \ot e_1+e_3 \ot e_4+e_4 \ot e_3
+\frac{1}{2}(e_{11} \ot e_{22}+e_{22} \ot e_{11}+\i e_{21} \ot e_{12}-\i e_{12} \ot e_{21}) \\
&
\Delta(e_3)
=e_1 \ot e_3+e_3 \ot e_1+e_2 \ot e_4+e_4 \ot e_2
+\frac{1}{2}(e_{11} \ot e_{22}+e_{22} \ot e_{11}-\i e_{21} \ot e_{12}+\i e_{12} \ot e_{21}) \\
&
\Delta(e_4)
=e_1 \ot e_4+e_4 \ot e_1+e_2 \ot e_3+e_3 \ot e_2
+\frac{1}{2}(e_{11} \ot e_{11}+e_{22} \ot e_{22}-e_{12} \ot e_{12}-e_{21} \ot e_{21}) \\
&
\Delta(e_{11})
=e_1 \ot e_{11}+e_{11} \ot e_1+e_2 \ot e_{22}+e_{22} \ot e_2
+e_3 \ot e_{22}+e_{22} \ot e_3+e_4 \ot e_{11}+ e_{11} \ot e_4 \\
&
\Delta(e_{12})
=e_1 \ot e_{12}+e_{12} \ot e_1+\i e_2 \ot e_{21}-\i e_{21} \ot e_2
-\i e_3 \ot e_{21}+\i e_{21} \ot e_3-e_4 \ot e_{12}-e_{12} \ot e_4 \\
&
\Delta(e_{21})
=e_1 \ot e_{21}+e_{21} \ot e_1-\i e_2 \ot e_{12}+\i e_{12} \ot e_2
+\i e_3 \ot e_{12}-\i e_{12} \ot e_3-e_4 \ot e_{21}-e_{21} \ot e_4 \\
&
\Delta(e_{22})
=e_1 \ot e_{22}+e_{22} \ot e_1+ e_2 \ot e_{11}+e_{11} \ot e_2
+e_3 \ot e_{11}+e_{11} \ot e_3+e_4 \ot e_{22}+e_{22} \ot e_4
\end{split} 
\]
define the coproduct $\Delta \co \L^\infty(\KP)\to \L^\infty(\KP) \ot \L^\infty(\KP)$. The Haar state $\tau$ on $\L^\infty(\KP)$ is given by
\begin{equation}
\label{}
\tau\bigg(x_1 \oplus x_2 \oplus x_3 \oplus x_3 \oplus x_4 \oplus \begin{bmatrix}
  a_{11}   & a_{12}  \\
  a_{12}   & a_{22} \\
\end{bmatrix}\bigg)
\ov{\mathrm{def}}{=} \frac{1}{8}(x_1 + x_2 + x_3 + x_3 + x_4+2a_{11}+2a_{22}).
\end{equation}
The antipode $R \co \L^\infty(\KP) \to \L^\infty(\KP)$ is given by 
$$
R(e_k)
\ov{\mathrm{def}}{=} e_k,\quad 
R(e_{ij}) 
\ov{\mathrm{def}}{=} e_{ji}
$$
where $1 \leq k \leq 4$ and $1 \leq i,j \leq 2$. We refer to \cite{EnV96} for another construction of this finite quantum group.

We can consider the convolution operator $T \co \L^\infty(\KP) \mapsto \L^\infty(\mathbb{KP})$, $g \to f*g$ defined in \eqref{Def-*1-infty}. By \cite[p.~30]{FrG06}\footnote{\thefootnote. We warn the reader that for obscure reasons the authors of \cite{FrG06} use another basis which did not simplify anything. Moreover, we use the antipode in our bracket. So we replaced $y$ by $-y$.}, with respect to the basis $(e_1,e_2,e_3,e_4,e_{11},e_{12},e_{21},e_{22})$ of $\L^\infty(\KP)$, its matrix is given by 
$$
\begin{bmatrix}
\mu_1 & \mu_2 & \mu_3 & \mu_4 & \frac{1+z}{2}\mu_5 & \frac{x-\i y}{2}\mu_5&\frac{x+\i y}{2}\mu_5&\frac{1-z}{2}\mu_5 \\
\mu_2 & \mu_1 & \mu_4 & \mu_3 & \frac{1-z}{2}\mu_5 & \frac{-\i x+ y}{2}\mu_5& \frac{ \i x +y}{2}\mu_5& \frac{1+z}{2}\mu_5 \\
\mu_3 & \mu_4 & \mu_1 & \mu_2 & \frac{1-z}{2}\mu_5&\frac{\i x-y}{2}\mu_5 & \frac{-\i x-y}{2}\mu_5 & \frac{1+z}{2}\mu_5 \\
\mu_4 & \mu_3 & \mu_2 & \mu_1 & \frac{1+z}{2} \mu_5&\frac{-x+\i y}{2}\mu_5& \frac{-x-\i y}{2}\mu_5 & \frac{1-z}{2}\mu_5 \\
\frac{1+z}{4}\mu_5   & \frac{1-z}{4}\mu_5     & \frac{1-z}{4}\mu_5     & \frac{1+z}{4}\mu_5&\mu_1+\mu_4 & 0 & 0 & \mu_2+\mu_3\\
\frac{x-\i y}{4}\mu_5&\frac{\i x-y}{4}\mu_5   & \frac{-\i x+y}{4}\mu_5 &\frac{-x+\i y}{4}\mu_5 & 0 & \mu_1-\mu_4 &-\i\mu_2+\i\mu_3 & 0\\
\frac{x+\i y}{4}\mu_5& \frac{-\i x-y}{4}\mu_5 & \frac{\i x+y}{4}\mu_5  & \frac{-x-\i y}{4}\mu_5 & 0 & \i\mu_2-\i\mu_3 & \mu_1-\mu_4 & 0\\
\frac{1-z}{4}\mu_5	 & \frac{1+z}{4}\mu_5     & \frac{1+z}{4}\mu_5     & \frac{1-z}{2}\mu_5 & \mu_2+\mu_3 & 0 & 0 & \mu_1+\mu_4\\
\end{bmatrix}.
$$
If $\mu_5=0$, the subspaces $\vect \{e_{1},e_{2},e_{3},e_{4}\}$ and $\vect \{e_{11},e_{12},e_{21},e_{22}\}$ are obviously stable under the convolution operator $T$ and this operator identifies to a direct sum $T_1 \oplus T_2 \co \ell^\infty_4 \oplus \M_2 \to \ell^\infty_4 \oplus \M_2$.

\begin{remark} \normalfont
A similar analysis could be done for the Sekine quantum groups \cite{Sek96}.
\end{remark}

\section{The principle of transference in noncommutative analysis}

The idea of this section is to start with a convolution operator $T_f \co \L^1(\QG) \to \L^1(\QG)$, $x \mapsto f * x$ on a quantum group $\QG$ and to use a right action $\alpha$ of $\QG$ on a von Neumann algebra $\cal{M}$ (e.g. the matrix algebra $\M_n$ or a direct sum $\M_{n_1} \oplus \cdots \oplus \M_{n_K}$ of matrix algebras) to construct a quantum channel $T_{f,\alpha} \co \L^1(\cal{M}) \to \L^1(\cal{M})$. These maps are related by the intertwining relation \eqref{commute-transference} and we can often obtain some information on the channel $T_{f,\alpha}$ from the convolution operator $T_f$.

\subsection{Actions of locally compact quantum groups}
\label{Sec-actions-Transference}


\paragraph{Group actions}  Recall the following notion of \cite[Definition I.1]{Eno77} and \cite[Definition 1.1]{Vae01}. We also refer to \cite{DeC17} for more information on actions of \textit{compact} quantum groups. 

\begin{defi}
\label{Defi-action}
A right action of a locally compact quantum group $\QG$ on a von Neumann algebra $\cal{M}$ is a normal unital injective $*$-homomorphism $\alpha \co \cal{M} \to \cal{M} \otvn \L^\infty(\QG)$ such that
\begin{equation}
\label{Def-corep}
(\alpha \ot \Id) \circ \alpha
=(\Id \ot \Delta) \circ \alpha.
\end{equation}
\end{defi}

Similarly, a left action of a locally compact quantum group $\QG$ on a von Neumann algebra $\cal{M}$ is a normal unital injective $*$-homomorphism $\beta \co \cal{M} \to \L^\infty(\QG) \otvn \cal{M}$ such that
\begin{equation}
\label{Def-corep-left}
(\Id \ot \beta) \circ \beta
=(\Delta \ot \Id) \circ \beta.
\end{equation}
Note that we can define verbatim the notion of an action of a quantum hypergroup and a part of the results of this section generalize to this setting. Unfortunately, we does not know interesting examples of actions of quantum hypergroups which are not quantum groups.

\begin{defi}
We say that two right actions $\alpha_1 \co \cal{M} \to \cal{M} \otvn \L^\infty(\QG)$ and $\alpha_2 \co \cal{N} \to \cal{N} \otvn \L^\infty(\QG)$ are conjugate if there exists a $*$-isomorphism $\theta \co \cal{M} \to \cal{N}$ such that
\begin{equation}
\label{conjugate-actions}
\alpha_2 \circ \theta
=(\theta \ot \Id) \circ \alpha_1.
\end{equation}
The conjugacy class of $\alpha$ is denoted $[\alpha]$. 
\end{defi}

\begin{remark} \normalfont
\label{lien-actions-droite-gauche}
Let us recall that there is a simple passage from right to left actions (and conversely). Indeed, if $\flip \co \cal{M} \otvn \L^\infty(\QG) \to \L^\infty(\QG) \otvn \cal{M}$, $x \ot y \mapsto y \ot x$ denotes the flip map it is easy to check that a map $\alpha \co \cal{M} \to \cal{M} \otvn \L^\infty(\QG)$ is a right action of $\QG$ on $\cal{M}$ if and only if $\beta \ov{\mathrm{def}}{=} \flip \circ \alpha \co \cal{M} \to \L^\infty(\QG) \otvn \cal{M}$ is a left action of the opposite quantum group $\QG^\op$ \cite[Section 4]{KuV03}, i.e.~the locally compact quantum group defined by $\L^\infty(\QG^\op) \ov{\mathrm{def}}{=} \L^\infty(\QG)$ and the reversed coproduct $\Delta_{\QG^\op} \ov{\mathrm{def}}{=}\flip \circ \Delta_\QG$ where $\flip$ denotes again a suitable flip map. This allows simple translation of many results about right actions to left actions.
\end{remark}

\begin{remark} \normalfont
\label{Rem-classical-action}
\label{Ex-classical-action}
Let $\cal{M}$ be a von Neumann algebra and $G$ be a locally compact group. Suppose that we have a group homomorphism $\Phi \co G \to \Aut(\cal{M})$ such that for each $x \in \cal{M}$ the function $G \to \cal{M}$, $s \mapsto \Phi_s(x)$ is continuous if $\cal{M}$ is endowed with the weak* topology. This homomorphism can be translated to a map $\alpha \co \cal{M} \to \cal{M} \otvn \L^\infty(G)=\L^\infty(G,\cal{M})$, $x \mapsto \alpha(x)$ where 
\begin{equation}
\label{classical-useful}
\alpha(x)(s) 
\ov{\mathrm{def}}{=} \Phi_s(x).
\end{equation}
By \cite[p.~263]{Str81} or \cite[Proposition I.3]{Eno77}, the map $\alpha$ is a right action of the group $G$ on the von Neumann algebra $\cal{M}$. A reverse procedure is described in \cite[p.~265]{Str81} and \cite[Proposition I.3]{Eno77} (see also \cite[Proposition 2.1]{NaT79} and \cite[Proposition 15.2.2 p.~439]{Tus22}). So for classical locally compact groups, there is a bijective correspondence between actions of the group $G$ on von Neumann algebras and actions of the quantum group $(\L^\infty(G),\Delta)$ on von Neumann algebras.

In this case, two ergodic actions $\Phi \co G \to \Aut(\cal{M})$ and $\Psi \co G \to \Aut(\cal{M})$ are conjugate if there is a $*$-isomorphism $\theta \co \cal{M} \to \cal{N}$ such that
\begin{equation}
\label{}
\theta \circ \Phi_s \circ \theta^{-1} 
= \Psi_s, \quad s \in G.
\end{equation}
\end{remark}


Recall \cite[Definition 6.1.1 p.~93]{Vae03} 
that a measurable unitary $2$-cocycle on a locally compact quantum group $\QG$ is a unitary element $\Omega \in \L^{\infty}(\QG) \otvn \L^{\infty}(\QG)$ such that
\begin{equation}
\label{equation-2-cocycle-quantum}
(\Omega \ot 1)(\Delta \ot \Id)(\Omega)
=(1 \ot \Omega)(\Id \ot \Delta)(\Omega).
\end{equation}

\begin{example} \normalfont
\label{2-cocycle-locally-compact}
If $\QG$ is a locally compact group $G$ then a measurable unitary 2-cocycle identifies to a measurable map $\sigma \co G \times G \to \T$ such that 
\begin{equation}
\label{equation-2-cocycle}
\sigma(s,t)\sigma(st,r)
=\sigma(s,tr)\sigma(t,r), \quad s,t,r \in G.
\end{equation}
For example, the 2-cocycles of the group $G=\R^n$ are described in \cite[Theorem 7.38 p.~276]{Var85}.
\end{example}

Given a $2$-cocycle $\Omega$ on a locally compact quantum group $\mathbb{G}$ we can introduce (see \cite{DMN22}) the right and left twisted co-multiplications with respect to $\Omega$
$$
\begin{array}{rcclcrccl}
{}_{\Omega}\Delta \co &\L^{\infty}(\QG)& \longrightarrow & \L^{\infty}(\QG) \otvn \L^{\infty}(\QG) &,\ & \Delta_{\Omega^*} \co & \L^{\infty}(\QG)& \longrightarrow & \L^{\infty}(\QG) \otvn \L^{\infty}(\QG)\\
&x& \longmapsto & \Omega\cdot \Delta(x) & & &x& \longmapsto & \Delta(x)\cdot \Omega^*
\end{array}.
$$
Let $\QG$ be a compact quantum group and $\Omega$ be a $2$-cocycle on $\QG$. A measurable unitary $\Omega$-representation of $\QG$ on a complex Hilbert space $H$ is a unitary element $u \in \B(H)\otvn \L^{\infty}(\QG)$ such that $
(\Id \ot{}_{\Omega}\Delta)(u)
=u_{12}u_{13}$.  
A measurable unitary $\Omega^*$-representation on $H$  is a unitary element $u \in \B(H) \otvn \L^{\infty}(\QG)$ satisfying $(\Id \ot \Delta_{\Omega^*})(u)=u_{12}u_{13}$.	


\begin{example}[Projective representations] \normalfont
\label{projective-quantum-groups}
Following \cite[Definition 3.1.1]{DMN22}, we say that a right action $\alpha \co \B(H) \rightarrow \B(H) \otvn \L^{\infty}(\QG)$ of a compact quantum group $\QG$ on the von Neumann algebra $\B(H)$, where $H$ is a complex Hilbert space, is a measurable right projective representation. In this situation, by the combination of \cite[Theorem 3.1.15]{DMN22} and \cite[Proposition 3.1.19]{DMN22} there exists a $2$-cocycle $\Omega \in \L^{\infty}(\QG) \otvn \L^{\infty}(\QG)$ and a unitary $\Omega^*$-representation  $u \in \B(H) \overline{\otimes} \L^{\infty}(\QG)$ such that 
$$
\alpha(x) 
= u(x \ot 1)u^*, \quad x \in \B(H).
$$
Conversely, if $u$ is a unitary $\Omega^*$-representation, we can construct a measurable right coaction
\begin{equation}
\label{eq.deltau2}
\alpha_u \co \B(H) \to  \B(H) \otvn \L^{\infty}(\QG), \qquad \alpha_u(a) = u(a \ot 1)u^*,\qquad a\in \B(H),
\end{equation}
where the coaction property follows immediately from the $\Omega^*$-representation property of $u$. Similarly, any unitary $\Omega$-representation $u$ provides a measurable left coaction
\begin{equation}
\label{eq.deltau}
\beta_u \co \B(H) \to \L^{\infty}(\QG) \otvn \B(H), \qquad 
\beta_u(a) 
= \flip (u^*(a \ot 1)u),\qquad a \in \B(H).
\end{equation}


\end{example}

\begin{example} \normalfont
\label{rep-action}
A measurable projective unitary representation of a locally compact group $G$ on a complex Hilbert space $H$ is a measurable $u \co G \to \B(H)$ such that
\begin{equation}
\label{projective-rep}
u(s)u(t)
=\sigma(s,t)u(st), \quad s,t \in G
\end{equation}
for some measurable 2-cocycle $\sigma \co G \times G \to \T$. This concept originates from Schur's works \cite{Sch04}, \cite{Sch07}, and \cite{Sch11} on finite groups. We refer to \cite[pp.~166-172]{Cur99} for a concise historical summary of Schur's work. Generalizing slightly \cite[p.~238]{Tak03}, it is easy to check
\footnote{\thefootnote. For any $s,t \in G$ and any $x \in \B(H)$, we have
$$
\tilde{\alpha}_s(\tilde{\alpha}_t(x))
=u(s)u(t)xu(t)^*u(s)^*
=u(s)u(t)x (u(s)u(t))^*
\ov{\eqref{projective-rep}}{=} \sigma(s,t)u(st) x \ovl{\sigma(s,t)}u(st)^*
=u(st) x u(st)^*
=\tilde{\alpha}_{st}(x).
$$} 
that the map $\tilde{\alpha} \co G \to \Aut(\B(H))$, $s \mapsto u(s)xu(s)^*$ is an action of the group $G$ on the von Neumann algebra $\B(H)$. Combining with Example \ref{Ex-classical-action}, we can obtain a lot of actions on matrix algebras $\M_n$ by considering finite-dimensional projective unitary representations of groups.
We refer to the books \cite{Kar85} and \cite{Mor17} for more information on projective unitary representations and the classical papers \cite{Wig39} and \cite{Bar54}. 
\end{example}

\begin{example} \normalfont
Consider a right action $\alpha \co \cal{M} \to \cal{M} \otvn \L^\infty(\QG)$ of a locally compact quantum group $\QG$ on a von Neumann algebra $\cal{M}$. Following \cite[Definition 2.1 p.~434]{Vae01}, the crossed product $\cal{M} \rtimes_\alpha \QG$ of $\cal{M}$ by the action $\alpha$ of $\QG$ is by definition the von Neumann subalgebra of $\cal{M} \otvn \B(\L^2(\QG))$ generated by $\alpha(\cal{M})$ and $1 \ot \L^\infty(\hat{\QG})$. By \cite[Proposition 2.2 p.~434]{Vae01}, this crossed product admits a distinguished action of the locally compact quantum group $\hat{\QG}^\op$, called the dual action. This action $\hat{\alpha}\co \cal{M}\rtimes_\alpha \QG \to( \cal{M} \rtimes_\alpha \QG)\otvn \L^\infty(\hat{\QG}^\op)$ is defined by
\[
\begin{aligned}
\hat{\alpha}(\alpha(x))&=\alpha(x) \ot 1,&\quad{x} \in \L^\infty(\QG),\\
\hat{\alpha}(1 \ot y)&= 1 \ot \Delta_{\hat{\QG}^\op}(y),&\quad{y} \in \L^\infty(\hat{\QG}).
\end{aligned}
\]
\end{example}

\begin{example} \normalfont
The conjugation action of $\SU(2)$ on the matrix algebra $\M_2$ induces an action of the group $\SO(3)$ on $\M_2$. Note that the special orthogonal group $\SO(3)$ acts on the von Neumann algebra $\M_2$. Indeed, by \cite[Theorem 1.1]{Sol10} this compact group is even the quantum automorphism group of $(\M_2,\tr)$. More generally, if $0 < q \leq 1$ and if we consider the faithful state\footnote{\thefootnote. Up to conjugation by a unitary $2 \times 2$ matrix these are all the faithful states on the matrix algebra $\M_2$.} $\psi_q \co [a_{ij}] \to \frac{1}{1+q^2}(a_{11}+q^2 a_{22})$ on the matrix algebra $\M_2$, then Soltan has proved \cite[Theorem 1.1]{Sol10} that the quantum automorphism group of $(\M_2,\psi_q)$ is isomorphic to the quantum group $\SO_q(3)$ of Podle\'s. See also \cite{Mat21} for a related paper.
%
%
%
%
%
\end{example}

If $\omega \in \L^\infty(\QG)_*$ and if $\alpha \co \cal{M} \to \cal{M} \otvn \L^\infty(\QG)$ is a right action of a locally compact group $\QG$ on a von Neumann algebra $\cal{M}$, we introduce the maps
\begin{equation}
\label{convol-def-bis}
T_\omega
\ov{\mathrm{def}}{=} (\Id \ot \omega) \circ \Delta \co \L^\infty(\QG) \to \L^\infty(\QG)
\quad \text{and} \quad
T_{\omega,\alpha}
\ov{\mathrm{def}}{=} (\Id \ot \omega) \circ \alpha \co \cal{M} \to \cal{M}
\end{equation}
and if $\beta$ is a left action
\begin{equation}
\label{convol-def-left}
R_\omega
\ov{\mathrm{def}}{=} (\omega \ot \Id) \circ \Delta \co \L^\infty(\QG) \to \L^\infty(\QG)
\quad \text{and} \quad
R_{\omega,\beta}
\ov{\mathrm{def}}{=} (\omega \ot \Id) \circ \beta \co \cal{M} \to \cal{M}.
\end{equation}
If $\omega$ is a normal state, note that the maps $T_{\omega,\alpha}$ and $R_{\omega,\beta}$ are unital and completely positive. If $\QG$ is compact, it is obvious with \eqref{Def-haar-state} that $T_\omega$ preserves the Haar state $h_{\QG}$, i.e.~$h_{\QG} \circ T_\omega = h_{\QG}$.


We have the following relation between these maps. 

\begin{prop}
\label{prop-commuting}
Let $\QG$ be a locally compact quantum group and $\cal{M}$ be a von Neumann algebra. Let $\alpha \co \cal{M} \to \cal{M} \otvn \L^\infty(\QG)$ be a right action and $\beta \co \cal{M} \to \L^\infty(\QG) \otvn \cal{M}$ be a left action. If $\omega \in \L^\infty(\QG)_*$, we have
\begin{equation}
\label{commute-transference}
\alpha \circ T_{\omega,\alpha}
=(\Id_{\cal{M}} \ot T_\omega) \circ \alpha
\quad \text{and} \quad
\beta \circ R_{\omega,\beta}
=(R_\omega \ot \Id_{\cal{M}}) \circ \beta.
\end{equation}
\end{prop}

\begin{proof}
We have
\begin{align*}
\MoveEqLeft
(\Id \ot T_\omega) \circ \alpha           
\ov{\eqref{convol-def-bis}}{=} (\Id \ot \Id \ot \omega)(\Id \ot \Delta) \circ \alpha
\ov{\eqref{Def-corep}}{=} (\Id \ot \Id \ot \omega)(\alpha \ot \Id) \circ \alpha \\
&=(\alpha \ot \omega) \circ \alpha
=\alpha \circ (\Id \ot \omega) \circ \alpha 
\ov{\eqref{convol-def-bis}}{=} \alpha \circ T_{\omega,\alpha}.
\end{align*}
and
\begin{align*}
\MoveEqLeft
(R_\omega \ot \Id) \circ \beta           
\ov{\eqref{convol-def-left}}{=} (\omega \ot \Id \ot \Id)(\Delta \ot \Id) \circ \beta
\ov{\eqref{Def-corep-left}}{=} (\omega \ot \Id \ot \Id)(\Id \ot \beta) \circ \beta \\
&=(\omega \ot \beta) \circ \beta
=\beta \circ (\omega \ot \Id) \circ \beta
\ov{\eqref{convol-def-left}}{=} \beta \circ R_{\omega,\beta}.
\end{align*} 
\end{proof}

The following observation was also stated in \cite[Theorem 2]{Kal13}. We were not aware of this result prior to writing the first version of this paper. For the sake of completeness, we give the proof. 

\begin{prop}
\label{prop-theta-alpha}
Let $T \co \L^\infty(\QG) \to \L^\infty(\QG)$ be a completely bounded left Fourier multiplier. There exists a unique linear map $\Theta_{\alpha,T} \co \cal{M} \to \cal{M}$ such that
\begin{equation}
\label{a1}
\alpha \circ \Theta_{\alpha,T}
= (\Id \ot T) \circ \alpha.
\end{equation}
Moreover, $\Theta_{\alpha,T}$ is normal and completely bounded. Finally, if $\omega \in \L^1(\QG)$ then we have 
\begin{equation}
\label{theta-467}
\Theta_{\alpha,T_\omega}
= T_{\omega, \alpha}.
\end{equation}
\end{prop}

\begin{proof}
We have $\Delta \circ T \ov{\eqref{covariance-left}}{=} (\Id \ot T) \circ \Delta$. Hence
\begin{align*}
\MoveEqLeft
(\Id \ot \Delta) (\Id \ot T) \alpha 
= (\Id \ot \Id \ot T) (\Id \ot \Delta) \alpha \\
&\ov{\eqref{Def-corep}}{=} (\Id \ot \Id \ot T) (\alpha \ot \Id )\alpha 
= (\alpha \ot \Id) (\Id \ot T) \alpha.
\end{align*}
Consequently for any $x \in \cal{M}$, we deduce that by \cite[Theorem 2.7 p.~437]{Vae01} the element $(\Id \ot T) \alpha(x)$ belongs to the algebra $\alpha(\cal{M})$. Hence we can consider the map $\Theta_{\alpha,T}
\ov{\mathrm{def}}{=} \alpha^{-1} (T \ot \Id) \alpha \co \cal{M} \to \cal{M}$. Since $\alpha$ is an injective $*$-homomorphism and $T$ is a completely bounded normal map, we conclude that $\Theta_{\alpha,T}$ is a completely bounded normal map. Uniqueness is obvious. The last formula is a consequence of uniqueness and Proposition \ref{prop-commuting}.
\end{proof}

Of course, we have a similar result for left actions. 
%
Now, we highlight how to obtain quantum channels.

\begin{prop}
Let $\QG$ be a compact quantum group with Haar state $h_{\QG}$ and $\cal{M}$ be a von Neumann algebra equipped with a normal finite faithful trace $\tau_{\cal{M}}$. Let $\alpha \co \cal{M} \to \cal{M} \otvn \L^\infty(\QG)$ be a state preserving right action. If $\omega \in \L^\infty(\QG)_*$ then $T_{\omega,\alpha} \co \cal{M} \to \cal{M}$ is trace preserving, hence induces a quantum channel $T_{\omega,\alpha} \co \L^1(\cal{M}) \to \L^1(\cal{M})$ if $\omega$ is a state. 
\end{prop}

\begin{proof}
Using the preservation of the Haar state $h_{\QG}$ by the operator $T_\omega \co \L^\infty(\QG) \to \L^\infty(\QG)$ in the third equality, we obtain
\begin{align*}
\MoveEqLeft
\tau_{\cal{M}} \circ T_{\omega,\alpha}            
=(\tau_{\cal{M}} \ot h_{\QG}) \circ \alpha \circ T_{\omega,\alpha} 
\ov{\eqref{commute-transference}}{=} (\tau_{\cal{M}} \ot h_{\QG}) \circ (\Id \ot T_\omega) \circ \alpha 
=(\tau_{\cal{M}} \ot h_{\QG}) \circ \alpha 
=\tau_{\cal{M}}.
\end{align*} 
\end{proof}

Of course, we have a similar result for trace preserving left actions.

\subsection{Generalization of Werner's Quantum Harmonic Analysis}
\label{sec-Werner}

Werner introduced in \cite{Wer84} the framework of Quantum Harmonic Analysis. This framework provides a new analytical point of view that benefit multiple areas of analysis offering a new perspective on classic problems. We refer to the papers \cite{BBF24}, 
\cite{BBLS22}, 
\cite{BGNT21}, 
\cite{DLS24}, 
\cite{FuH24}, 
\cite{FuR23}, 
\cite{FLW24}, 
\cite{Ful20}, 
\cite{FulG23}, 
\cite{Hal23}, 
\cite{LuS21},  
\cite{LuS18a}, 
\cite{LuS19}, 
\cite{LuS20} and 
\cite{LuS18b}
\cite{Skr20}, 
for more information and some applications. 

The starting point is a continuous square-integrable\footnote{\thefootnote. A continuous unitary representation $u$ of $G$ on $H$ is said to be square-integrable if for any vectors $\xi,\eta \in H$  the coefficient $\la u(\cdot)\xi,\eta \ra$ is square-integrable on $G$, i.e. $\int_G |\la u(s)\xi,\eta \ra|^2\d\mu_G(s) <\infty$.} unitary representation continuous $u \co G \to \B(H)$ of a locally compact group $G$ on a complex Hilbert space $H$. The central objects in quantum harmonic analysis are a \textit{variant} of the convolution $*$ on the locally compact group $G$ defined in \eqref{Convol-usuel}, the <<function-operator>> and the <<operator-operator convolution>>. More precisely, assuming that the group $G$ is compact for simplicity, it is defined in the paper \cite{Hal23} the function-operator convolution
\begin{equation}
\label{convol-Ha}
f * x 
\ov{\mathrm{def}}{=} \int_G f(s) u_s^* x u_s \d \mu_G(s), \quad f \in \L^1(G),  x \in S^1(H),
\end{equation}
where $\mu_G$ is the normalized Haar measure, and the operator-operator convolution
\begin{equation}
\label{convol-op-op}
(x * y)(s) 
\ov{\mathrm{def}}{=} \tr(x u_s^* y u_s),\quad x \in \B(H), y \in S^1(H), s \in G.
\end{equation}
By \cite[Proposition 3.12 p.~15]{Hal23}, we have the following remarkable <<associativity>> properties
\begin{equation}
\label{asso-1}
f*(g*x)
=(f*_{\mathrm{Hal}} g) * x, \quad f,g \in \L^1(G),x \in S^1(H)
\end{equation}
and
\begin{equation}
\label{asso-1}
f*(x*y)
=(f*x)*y) \quad f \in \L^1(G),x \in S^1(H), y \in \B(H).
\end{equation}
and where the convolution
\begin{equation}
\label{convol-Hal}
(f*_{\mathrm{Hal}}g)(s)
\ov{\mathrm{def}}{=} \int_G f(t)g(s t^{-1}) \d \mu_G(t).
\end{equation}
is not the standard convolution defined in \eqref{Convol-usuel}. We will demonstrate in this section that the concepts introduced in the previous section encompass the setting of Quantum Harmonic Analysis. Fortunately, the terminology is consistent: Werner's Quantum Harmonic Analysis is a part of the harmonic analysis of quantum groups.

Let $\QG$ be a locally compact quantum group. Recall that a linear map $T \co \L^1(\QG) \to \L^1(\QG)$ is called a left centralizer of $\L^1(\QG)$ if it satisfies $T(f \star g) 
= T(f)\star g$ for all $f, g \in \L^1(\QG)$. We denote by $\C_{\cb}^{\ell}(\L^1(\QG))$ the space of all
completely bounded left centralizers of $\L^1(\QG)$. A weak* continuous completely bounded map $T \co \L^\infty(\QG) \to \L^\infty(\QG)$ is said to be a left Fourier multiplier (or a left covariant map) if it satisfies
\begin{equation}
\label{covariance-left}
\Delta \circ T
=(\Id \ot T) \circ \Delta.
\end{equation}
We denote by $\frak{M}^\infty_{\cb,\ell}(\QG)$ the algebra of all normal completely bounded Fourier multipliers on $\L^\infty(\QG)$. A map $T \co \L^1(\QG) \to \L^1(\QG)$ belongs to $\C_{\cb}^\ell(\L^1(\QG))$ if and only if its adjoint $T^* \co \L^\infty(\QG) \to \L^\infty(\QG)$ belongs to $\frak{M}^\infty_{\cb,\ell}(\QG)$, see e.g.~\cite[Proposition 2.3 p.~7]{DFSW16}. 

Consider the universal measure algebra $\M_u(\QG)$ and the measure algebra $\M(\QG) \ov{\mathrm{def}}{=} \C_0(\QG)^*$, which contains $\L^\infty(\QG)_*$ as a norm closed two-sided ideal via the embedding $\L^\infty(\QG)_*\to \M(\QG)$, $\omega \mapsto \omega_{|\C_0(\QG)}$. We have canonical inclusions 
$$
\L^\infty(\QG)_* \subset \M(\QG) \subset \M_u(\QG) \subset \C_\cb^\ell(\L^1(\QG)).
$$ 
Moreover, it is known \cite[Proposition 3.1]{HNR10} that the canonical map $\M(\QG) \to \C_\cb^\ell(\L^1(\QG))$, $\mu \mapsto \mu *\, \cdot$ is a completely isometric algebra isomorphism if and only if $\QG$ is co-amenable, i.e.~if the algebra $\L^\infty(\QG)_*$ has a contractive (or bounded) approximate identity.

\begin{defi}
\label{def-function-operator-convolution}
Consider a right action $\alpha \co \cal{M} \to \cal{M} \otvn \L^\infty(\QG)$ of a locally compact quantum group $\QG$ on a von Neumann algebra $\cal{M}$. If $T \co \L^\infty(\QG) \to \L^\infty(\QG)$ is a left Fourier multiplier and $x \in \L^\infty(\cal{M})$, we introduce the element
\begin{equation}
\label{funct-op}
T *_\QG x
\ov{\mathrm{def}}{=} \Theta_{\alpha,T}(x)
\end{equation}
in the von Neumann algebra $\L^\infty(\cal{M})$, where we use the notation of Proposition \ref{prop-theta-alpha}. 
\end{defi}

Recall that each $\omega \in \L^\infty(\QG)_*$ defines a Fourier multiplier $T_\omega \co \L^\infty(\QG) \to \L^\infty(\QG)$, introduced in \eqref{convol-def-bis}. For any $x \in \L^\infty(\cal{M})$, we let $\omega *_\QG x \ov{\mathrm{def}}{=} T_\omega*_\QG x$. We have in particular
\begin{equation}
\label{funct-op-bis}
\omega *_\QG x
\ov{\eqref{funct-op}}{=} \Theta_{\alpha,T_\omega}(x)
\ov{\eqref{theta-467}}{=} T_{\omega, \alpha}
\ov{\eqref{convol-def-bis}}{=} (\Id_{\cal{M}} \ot \omega) \circ \alpha(x), \quad x \in \L^\infty(\cal{M}).
\end{equation}
In other words, with the notations of \eqref{convol-def-bis}, we have $T_{\omega,\alpha}(x)=\omega *_\QG x$. It is rather easy to obtain the following associativity property.

\begin{prop}[first associativity]
\label{prop-asso-1}
For any $\omega_1,\omega_2 \in \L^\infty(\QG)_*$ and any $x \in \L^\infty(\cal{M})$, we have the equality
\begin{equation}
\label{asso-magic}
\omega_1 *_\QG ( \omega_2 *_\QG x)
= (\omega_1 * \omega_2 ) *_\QG x.
\end{equation}
\end{prop}

\begin{proof}
On the one hand, we have
\begin{align*}
\MoveEqLeft       
\omega_1 *_\QG (\omega_2 *_\QG x)  
\ov{\eqref{funct-op-bis}}{=} (\Id  \ot \omega_1) \circ \alpha \big((\Id \ot \omega_2) \circ \alpha(x)\big) \\
&= (\Id \ot \omega_1) \circ (\alpha \ot \omega_2) \circ \alpha(x) 
=\big[\Id \ot (\omega_1 \ot \omega_2) \big](\alpha \ot \Id) \circ \alpha(x).
\end{align*}
On the other hand, we have
\begin{align*}
\MoveEqLeft
(\omega_1 * \omega_2 ) *_\QG x 
\ov{\eqref{funct-op}}{=} \big[\Id \ot(\omega_1 * \omega_2)   \big] \circ \alpha(x) 
\ov{\eqref{Convolution-L1}}{=} \big[\Id \ot ((\omega_1 \ot \omega_2) \circ \Delta) \big] \circ \alpha(x) \\
&=\big[\Id \ot (\omega_1 \ot \omega_2) \big](\Id \ot \Delta) \circ \alpha(x) 
\ov{\eqref{Def-corep}}{=} \big[ \Id \ot (\omega_1 \ot \omega_2) \big](\alpha \ot \Id) \circ \alpha(x).
\end{align*}
\end{proof}

\begin{defi}[function-operator convolution]
\label{def-function-operator-convolution-bis}
Let $\QG$ be a locally compact quantum group. Consider a right action $\alpha \co \cal{M} \to \cal{M} \otvn \L^\infty(\QG)$. Suppose that $\QG$ admits a tracial left Haar weight $h$. If $f \in \L^1(\QG)$ and $x \in \L^\infty(\cal{M})$, we introduce the element
\begin{equation}
\label{funct-op-3}
f *_\QG x
\ov{\mathrm{def}}{=} h(f \, \cdot) *_\QG x
\ov{\eqref{funct-op-bis}}{=}(\Id \ot h(f \, \cdot)) \circ \alpha(x)
\end{equation}
in the von Neumann algebra $\L^\infty(\cal{M})$.
\end{defi}

Now, we define the operator-operator convolution.

\begin{defi}
\label{def-operator-operator-convolution-2}
Consider a right action $\alpha \co \cal{M} \to \cal{M} \otvn \L^\infty(\QG)$ of a locally compact quantum group $\QG$ on a von Neumann algebra $\cal{M}$. If $\eta \in \L^\infty(\cal{M})_*$ and $x \in \L^\infty(\cal{M})$, we introduce the element
\begin{equation}
\label{op-op}
x *^\QG \eta
\ov{\mathrm{def}}{=} (\eta \ot \Id_{}) \circ \alpha(x)
\end{equation}
in the von Neumann algebra $\L^\infty(\QG)$.
\end{defi}

In the case, where the von Neumann algebra $\cal{M}$ is semifinite equipped with a normal semifinite faithful trace $\tau$, we can introduce the following definition, since we have an isometric identification $\L^1(\cal{M}) \to \L^\infty(\cal{M})_*$, $y \mapsto \tau(y\,\cdot)$.


\begin{defi}[operator-operator convolution]
\label{def-operator-operator-convolution-2}
Consider a right action $\alpha \co \cal{M} \to \cal{M} \otvn \L^\infty(\QG)$ of a locally compact quantum group $\QG$ on a von Neumann algebra $\cal{M}$ equipped with a normal semifinite faithful trace $\tau$. If $x \in \L^\infty(\cal{M})$ and $y \in \L^1(\cal{M})$, we define the element
\begin{equation}
\label{op-op-true-bis}
x *^\QG y
\ov{\mathrm{def}}{=} (\tau_y \ot \Id_{\L^\infty(\QG)}) \circ \alpha(x)
\end{equation}
in the space $\L^\infty(\QG)$, where $\tau_y \ov{\mathrm{def}}{=} \tau(y \, \cdot)$.
\end{defi}


%
%

In this context, we can prove the following associativity.

\begin{prop}
\label{prop-other-asso}
Consider a right action $\alpha \co \cal{M} \to \cal{M} \otvn \L^\infty(\QG)$ of a locally compact quantum group $\QG$ on a finite von Neumann algebra $\cal{M}$ equipped with a normal finite faithful trace $\tau$. For any $\omega \in \L^\infty(\QG)_*$ and any $x,y \in \L^\infty(\cal{M})$, we have
\begin{equation}
\label{}
\omega \star (x *^\QG y)
= (\omega *_\QG x) *^\QG y.
\end{equation}
\end{prop}

\begin{proof}
For any $\omega \in \L^\infty(\QG)_*$ and any $x,y \in \L^\infty(\cal{M})$, we have
\begin{align*}
\MoveEqLeft
(\omega *_\QG x) *^\QG y         
\ov{\eqref{op-op-true-bis}}{=} (\tau_y \ot \Id_{}) \circ \alpha(\omega *_\QG x) 
\ov{\eqref{funct-op-bis}}{=} (\tau_y \ot \Id_{}) \circ \alpha(T_{\omega, \alpha}(x)) \\
&\ov{\eqref{commute-transference}}{=} (\tau_y \ot \Id_{}) \circ (\Id_{\cal{M}} \ot T_\omega) \circ \alpha(x) 
= (\tau_y \ot T_\omega) \circ \alpha(x) 
=T_\omega \circ (\tau_y \ot \Id) \circ \alpha(x) \\
&\ov{\eqref{op-op-true-bis}}{=} T_\omega(x *^\QG y) 
\ov{\eqref{convol-def-bis}}{=} (\Id \ot \omega) \circ \Delta(x *^\QG y) 
\ov{\eqref{action-module}}{=} \omega \star (x *^\QG y).
\end{align*}
\end{proof}

\begin{remark} \normalfont
Replacing the trace $\tau$ by the Haagerup trace associated to a normal semifinite faithful weight, we could generalize the previous result to arbitrary von Neumann algebras. 
\end{remark}

\begin{example} \normalfont
\label{QHA-1}
Consider a continuous unitary representation $u \co G \to \B(H)$ of a compact group $G$ on a complex Hilbert space $H$. Following Example \ref{Ex-classical-action} and Example \ref{rep-action}, we can introduce the associated map $\alpha \co \B(H) \to \B(H) \otvn \L^\infty(G)=\L^\infty(G,\B(H))$, $x \mapsto (s \mapsto u(s) x u(s)^*)$. Here we consider the quantum group $\QG=(\L^\infty(G),\Delta)$ of Example \ref{Ex-quantum-group-classical}. 

Each function $f \in \L^1(G)$ identifies to an element $\omega_f \in \L^\infty(G)$ defined by $\omega_f(g)\ov{\mathrm{def}}{=} \int_G f \check{g} \d \mu_G$, where we use the duality bracket \eqref{bracket-1}. For any function $f \in \L^1(G)$ and any operator $x \in \B(H)$, we deduce that
\begin{align*}
\MoveEqLeft
f *_\QG x
\ov{\eqref{funct-op-bis}}{=} (\Id_{\cal{M}} \ot \omega_f) \circ \alpha(x)
=\int_G f(s)\alpha(x)(s^{-1}) \d \mu_G(s) \\
&=\int_G f(s)u(s^{-1}) x u(s^{-1})^* \d \mu_G(s) 
=\int_G f(s)u(s)^* x u(s) \d \mu_G(s).
\end{align*}
Consequently, in this context, the function-operator convolution is the one of \eqref{convol-Ha}. So we recover the associativity property \eqref{asso-1} with Proposition \ref{prop-asso-1}. Moreover, we have
\begin{align*}
\MoveEqLeft
x *^\QG y 
\ov{\eqref{op-op-true-bis}}{=} \big(\tr(y\, \cdot)\ot \Id_{\L^\infty(G)}  \big) \circ \alpha(x)
= \big(\tr(y\, \cdot) \ot \Id_{\L^\infty(G)}\big)\big(u(\cdot) x u(\cdot)^* \big) 
= \tr\big(y u(\cdot) x u(\cdot)^* \big).
\end{align*}
Consequently, for all $s \in G$, we obtain
$$
(x *^\QG y)(s)
=\tr\big(y u(s) x u(s)^* \big).
$$
This means that the function-function convolution is identical the one of \eqref{convol-op-op}. It is easy to see with \eqref{Delta_commutatif} that the convolution product $\omega_f * \omega_g$ identifies to the one of \eqref{convol-Hal}.
\end{example}

\subsection{Young's inequalities}
\label{Section-Young}
Recall that by \cite[Corollary 2.2.3 p.~24]{EfR00} a bounded linear form $\varphi \co E \to \mathbb{C}$ on an operator space $E$ is completely bounded and that its completely bounded norm is equal to its norm, i.e.
\begin{equation}
\label{cb-linear-form}
\norm{\varphi}_{}
=\norm{\varphi}_{\cb}.
\end{equation}

\begin{prop}
Consider a right action $\alpha \co \cal{M} \to \cal{M} \otvn \L^\infty(\QG)$ of a locally compact quantum group $\QG$ on a von Neumann algebra $\cal{M}$. For any $\omega \in \L^\infty(\QG)_*$ and any $x \in \L^\infty(\cal{M})$, we have
$$
\norm{\omega *_\QG x}_{\L^\infty(\cal{M})}   
\leq \norm{\omega}_{\L^\infty(\QG)_*} \norm{x}_{\L^\infty(\cal{M})}.
$$
\end{prop}

\begin{proof}
A simple computation shows that
\begin{align*}
\MoveEqLeft
\norm{\omega *_\QG x}_{\L^\infty(\cal{M})}         
\ov{\eqref{funct-op-bis}}{=} \norm{(\Id_{\cal{M}} \ot \omega) \circ \alpha(x)}_{\L^\infty(\cal{M})} \\
&\leq \norm{\Id \ot \omega}_{\cal{M} \otvn \L^\infty(\QG) \to \L^\infty(\cal{M})} \norm{\alpha(x)}_{\cal{M} \otvn \L^\infty(\QG)} 
\ov{\eqref{cb-linear-form}}{\leq} \norm{\omega}_{\L^\infty(\QG)_*} \norm{x}_{\L^\infty(\cal{M})}.
\end{align*}
\end{proof}

\begin{prop}
Consider a trace preserving right action $\alpha \co \cal{M} \to \cal{M} \otvn \L^\infty(\QG)$ of a compact quantum group $\QG$ of Kac type on an approximately finite-dimensional von Neumann algebra $\cal{M}$ equipped with a normal semifinite faithful trace $\tau$. For any $\omega \in \L^\infty(\QG)_*$ and any $x \in \L^1(\cal{M})$, we have
$$
\norm{\omega *_\QG x}_{\L^1(\cal{M})}   
\leq \norm{\omega}_{\L^\infty(\QG)_*} \norm{x}_{\L^1(\cal{M})}.
$$
\end{prop}

\begin{proof}
Note that $\alpha$ induces an isometry $\alpha \co \L^1(\cal{M}) \to \L^1(\cal{M} \otvn \L^\infty(\QG))$. For any $\omega \in \L^1(\QG)$ and any $x \in \L^1(\cal{M})$, we have
\begin{align*}
\MoveEqLeft
\norm{\omega *_\QG x}_{\L^1(\cal{M})} 
\ov{\eqref{funct-op-bis}}{=} \norm{T_{\omega,\alpha}(x)}_{\L^1(\cal{M})} 
= \norm{\alpha \circ T_{\omega,\alpha}(x)}_{\L^1(\cal{M} \otvn \L^\infty(\QG))} \\
&\ov{\eqref{commute-transference}}{=} \norm{(\Id_{\L^1(\cal{M})} \ot T_\omega) \circ \alpha(x)}_{\L^1(\cal{M} \otvn \L^\infty(\QG))}\\
&\leq \norm{\Id_{\L^1(\cal{M})} \ot T_\omega}_{\L^1(\cal{M} \otvn \L^\infty(\QG)) \to \L^1(\L^\infty(\QG) \otvn \cal{M})} \norm{\alpha(x)}_{\L^1(\L^\infty(\QG) \otvn \cal{M})}\\
&\ov{\eqref{ine-tensorisation-os}}{\leq} \norm{T_\omega}_{\cb,\L^1(\QG) \to \L^1(\QG)} \norm{x}_{\L^1(\cal{M})} 
\ov{\eqref{borne-cb} \eqref{action-module}}{\leq} \norm{\omega}_{\L^1(\QG)} \norm{x}_{\L^1(\cal{M})}.         
\end{align*}
\end{proof}

\begin{remark} \normalfont
We could replace in the previous result <<approximately finite-dimensional>> by QWEP by using the results of \cite{Jun1} and \cite{Jun04}.
\end{remark}

By bilinear interpolation of these compatible maps, we obtain the following result.

\begin{cor}
Suppose that $1 \leq p \leq \infty$. Consider a trace preserving right action $\alpha \co \cal{M} \to \cal{M} \otvn \L^\infty(\QG)$ of a locally compact quantum group $\QG$ on an approximately finite-dimensional  von Neumann algebra $\cal{M}$ equipped with a normal semifinite faithful trace $\tau$. We have a well-defined map $* \co \L^{p}(\QG) \times \L^{p}(\cal{M}) \to \L^p(\cal{M})$ such that 
\begin{equation}
\label{interpoled}
\norm{f *_\QG x}_{\L^p(\cal{M})}
\leq \norm{f}_{\L^{p}(\QG)} \norm{x}_{\L^1(\cal{M})}, \quad f \in \L^{p}(\QG),x \in \L^1(\cal{M}).
\end{equation} 
\end{cor}

In particular, in the setting of Example \ref{QHA-1}, we recover the inequality of \cite[Proposition 4.11 (ii)]{Hal23} in the case where $G$ is compact. 
Similarly, we have the following weaker result for the operator-operator convolution.

\begin{prop}
Consider a right action $\alpha \co \cal{M} \to \cal{M} \otvn \L^\infty(\QG)$ of a locally compact quantum group $\QG$ on a von Neumann algebra $\cal{M}$. For any $\eta \in \L^\infty(\cal{M})_*$ and any $x \in \L^\infty(\cal{M})$, we have
$$
\norm{\eta *^\QG x}_{\L^\infty(\QG)}   
\leq \norm{\eta}_{\L^\infty(\cal{M})_*} \norm{x}_{\L^\infty(\cal{M})}.
$$
\end{prop}

\begin{proof}
We have
\begin{align*}
\MoveEqLeft
\norm{\eta *^\QG x}_{\L^\infty(\QG)}       
\ov{\eqref{op-op}}{=} \norm{( \eta \ot \Id_{} ) \circ \alpha(x)}_{\L^\infty(\QG)} 
\leq \norm{\eta \ot \Id}_{\L^\infty(\cal{M}) \otvn \L^\infty(\QG)\to\L^\infty(\QG)} \norm{\alpha(x)}_{\L^\infty(\cal{M}) \otvn \L^\infty(\QG)} \\ 
&\ov{\eqref{cb-linear-form}}{\leq} \norm{\eta}_{\L^\infty(\cal{M})_*} \norm{x}_{\L^\infty(\cal{M})}.
\end{align*}
%
%
\end{proof}

\subsection{Ergodic actions of quantum groups}
\label{sec-ergodic-action}

Consider a right action $\alpha \co \cal{M} \to \cal{M} \otvn \L^\infty(\QG)$ of a locally compact quantum group $\QG$ on a von Neumann algebra $\cal{M}$, as in Definition \ref{Defi-action}. We can introduce the fixed point subalgebra 
$$
\cal{M}^\alpha 
\ov{\mathrm{def}}{=} \{x \in \cal{M} : \alpha(x) = x \ot 1\},
$$ 
as in \cite[Definition 1.2]{Vae02}.  Note that this algebra is a unital sub-von Neumann algebra of the von Neumann algebra $\cal{M}$. By \cite[Proposition 1.3 p.~432]{Vae02} or \cite[Proposition 15.2.4 p.~441]{Tus22}, for any $x \in \cal{M}^+$ the element $\E(x) \ov{\mathrm{def}}{=} (\Id \ot \psi) \circ \alpha(x)$ of the extended positive part $\cal{M}_{\mathrm{ext}}^+$ belongs to $(\cal{M}^\alpha)_{\mathrm{ext}}^+$ and $\E$ is a normal faithful operator-valued weight from $\cal{M}$ to $\cal{M}^\alpha$. The action $\alpha$ is said to be integrable if the operator-valued weight is semifinite. Thus an action $\alpha \co  \cal{M} \to \cal{M} \ot \L^\infty(\QG)$ is integrable if the set
\begin{equation}
\label{integr+}
\big\{x \in \cal{M}_+ : (\Id \ot \psi) \circ \alpha(x) \in \cal{M}_+ \big\}
\end{equation}
is weak* dense in $\cal{M}_+$. Elements of \eqref{integr+} are called integrable for $\alpha$. An element $x \in \cal{M}$ is square integrable for $\alpha$ if $x^* x$ is integrable for $\alpha$. Note that integrability of a right action is defined with respect to the left Haar weight. 

 If $\QG$ is compact, this operator-valued weight is a normal faithful conditional expectation $\E \co \cal{M}\to \cal{M}$ on the subalgebra $\cal{M}^\alpha $ defined by 
\begin{equation}
\label{cond-exp-fixed}
\E
=(\Id \ot h_{\QG}) \circ \alpha.
\end{equation}
The action is said to be ergodic if $\cal{M}^\alpha$ equals $\mathbb{C}1$. If $\QG$ is \textit{compact} and equipped with its Haar state $ h_\QG$, then by adapting the reasoning found in \cite[pp.~94-95]{Boc95}, originally applied to $\mathrm{C}^*$-algebras, to the context of von Neumann algebras, the von Neumann algebra $\cal{M}$ admits a unique normal state $\varphi_\cal{M}$ determined by the following condition
\begin{equation}
\label{Def-ergodic}
(\Id_{\cal{M}} \ot h_\QG) \circ \alpha(x)
=\varphi_\cal{M}(x)1_{\cal{M}}, \quad x \in \cal{M}.
\end{equation}
Note that $\varphi_\cal{M}$ is faithful by essentially \cite[Remark 2]{Boc95}. From \eqref{Def-ergodic}, it is very easy to see that $\alpha$ is state preserving, as the following result shows.

\begin{prop}
We have
$$
(\varphi_\cal{M} \ot h_\QG) \circ \alpha
=\varphi_\cal{M}.
$$
\end{prop}

\begin{proof}
If $x \in \cal{M}$, we have
$$
(\varphi_\cal{M} \ot h_\QG) \circ \alpha(x)
=\varphi_\cal{M}\big((\Id \ot h_\QG) \circ \alpha(x)\big)
\ov{\eqref{Def-ergodic}}{=} \varphi_\cal{M}\big(\varphi_\cal{M}(x)1_{\cal{M}}\big)
=\varphi_\cal{M}(x)\varphi_\cal{M}(1_{\cal{M}})
=\varphi_\cal{M}(x).
$$
\end{proof}



\begin{example} \normalfont
Examples of ergodic actions are given by embeddable quantum homogeneous spaces of compact quantum groups. These spaces are von Neumann subalgebras $\cal{M} \subset \L^\infty(\QG)$ which are right coideals, i.e.~$\Delta(\cal{M}) \subset \cal{M} \otvn \L^\infty(\QG)$, equipped with the restriction of the coproduct. Indeed, consider some $x \in \cal{M}$ such that $\Delta(x) = x \ot 1$. We see that 
$$
x
=(\Id \ot h_\QG)(x \ot 1)
=(\Id \ot h_\QG) \circ \Delta(x)
\ov{\eqref{Def-haar-state}}{=} h_\QG(x)1.
$$ 
Hence $x$ belongs to $\mathbb{C}1$. Note that a coduality between embeddable quantum homogeneous spaces for $\QG$ and for $\hat{\QG}$ is described in the paper \cite{KaS14}. Indeed, if $\cal{M}$ is an embeddable quantum homogeneous space of $\QG$, then $\widetilde{\cal{M}} \ov{\mathrm{def}}{=} \{ y \in \L^\infty(\hat{\QG}): xy=yx \textrm{ for any } x \in \cal{M}\}$ is an embeddable quantum homogeneous space of $\hat{\QG}$ and we have $\widetilde{\widetilde{\cal{M}}}=\cal{M}$. 
\end{example}

\begin{example} \normalfont
\label{Example-classical-444}
Consider a right action $\alpha \co \cal{M} \to \cal{M} \otvn \L^\infty(G)$ of a \textit{classical} locally compact group $G$. If $\Phi \co G \to \Aut(\cal{M})$ is the associated group homomorphism, we have by \cite[(4) p.~264]{Str81}
$$
\cal{M}^\alpha=\{x \in \cal{M}: \Phi_s(x)=x \text{ for any } s \in G\}.
$$ 
The right action $\alpha$ is ergodic if and only if for any $x \in \cal{M}$ the equality $\Phi_s(x)=x$ for any $s \in G$ implies that $x$ belongs to $\mathbb{C}1$. Furthermore, if $G$ is compact, the condition \eqref{Def-ergodic} means that for any $x \in \cal{M}$ the vector-valued integral $\int_G \alpha(x)(s) \d s$ (Gelfand integral) is a scalar operator $\varphi_\cal{M}(x)1_{\cal{M}}$. Then $\varphi_\cal{M}$ is the unique $G$-invariant normal state on the von Neumann algebra $\cal{M}$. 


If $G$ is compact and if $\mu_G$ is the normalized Haar measure on $G$ then for each $x \in \cal{M}$ it is known that the average $\int_G \tilde{\alpha}_s(x) \d\mu_G(s)$ is a scalar operator $\varphi_\cal{M}(x)1$ with $\varphi_\cal{M}(x) \in \mathbb{C}$. Then $\varphi_\cal{M}$ is the unique $G$-invariant state on $\cal{M}$.

If there exists a right ergodic action of a compact group $G$ on a von Neumann algebra $\cal{M}$ then it is known \cite[Corollary 4.2]{HLS81} (and its proof) that the von Neumann algebra $\cal{M}$ is necessarily finite and injective. Moreover, the proof of \cite[Corollary 4.2]{HLS81} shows that the state $\varphi_\cal{M}$ is necessarily a trace. Finally, we refer to \cite{AHK80} and \cite{OPT80} for a complete classification of ergodic actions of \textit{abelian} compact groups\footnote{\thefootnote. Note that \cite[Lemma 1.24]{ChH1} says that abelian compact groups are ergodically rigid, i.e.~the only ergodic actions of abelian compact groups are on type I von Neumann algebras. Unfortunately, this contradicts the results of \cite{OPT80}. So the result \cite[Lemma 1.24]{ChH1} is false (confirmed by email).}. 
\end{example}

\begin{example}[projective representations] \normalfont
\label{proj-ergodic-QG}
Consider the context of Example \ref{projective-quantum-groups}. It is known \cite[Remark 3.2.7]{DMN22} that a measurable right projective representation $\alpha \co \B(H) \to \B(H) \otvn \L^{\infty}(\QG)$ of a compact quantum group $\QG$,  with implementing unitary $u$, is ergodic if and only if $u$ is irreducible in the sense of \cite[Definition 3.2.1]{DMN22}. Suppose that $u \co G \to \B(H)$ is a continuous projective unitary representation of a classical \textit{locally compact} group $G$. Consider the action $\tilde{\alpha} \co G \to \Aut(\B(H))$, $s \mapsto u(s)xu(s)^*$ of $G$ on the von Neumann algebra $\B(H)$ of Example \ref{rep-action} and the associated right action $\alpha \co \B(H) \to \B(H) \otvn \L^\infty(G)$ from Example \ref{Ex-classical-action}. From Example \ref{Example-classical-444}, we have $\B(H)^\alpha=\{x \in \B(H) : \tilde{\alpha}(x)=x \text{ for all }s \in G\}=\{x \in \B(H) : u(s)xu(s)^*=x \text{ for all }s \in G\}$. By Schur's lemma \cite[Proposition 1.35]{KaT13} \cite[Lemma 3.5]{Fol16} (and a well-known extension of this result to projective unitary representations), we deduce that the right action $\alpha$ is ergodic if and only if the projective representation $u$ is irreducible.
\end{example}


A deeper analysis of the structure of ergodic actions of compact groups has been made by Wassermann in \cite{Was89}, \cite{Was88a} and \cite{Was88b}. In particular, Wassermann remarkably showed\footnote{\thefootnote. It is an intriguing open question to know if the group $\SU(n)$ can act ergodically for some $n \geq 3$ on the unique injective factor of type $\II_1$ with separable predual.} that the compact group $\SU(2)$ only admits ergodic actions on von Neumann algebras of finite type I. Moreover, Wassermann proved \cite[Theorem 20]{Was89} that each ergodic action of a compact group $G$ on a von Neumann algebra of Type I is isomorphic to an induced action $\Ind_{H}^{G} u$, where $H$ is a closed subgroup of $G$ acting by an irreducible projective representation $u$. 

The abstract theory of ergodic actions of compact quantum groups on operator algebras has been initiated by Boca in \cite{Boc95}. One significant difference with the compact group case is that the unique invariant state $h_{\cal{M}}$ is not necessarily a trace. Indeed, Wang \cite{Wan99} gave examples of ergodic actions of universal unitary quantum groups on type III factors. However, by \cite[Corollary 5.2]{DeC11a}, the invariant state associated to an ergodic action of a compact quantum group of Kac type on a factor of type I is tracial. Note that by \cite{BRV06} this is not true for ergodic actions of compact quantum groups of Kac type on arbitrary von Neumann algebras. Finally, we refer to the preprint \cite{ChH1} for information on ergodic actions of the quantum group $\O_{-1}(2)$ and to \cite{BaG10}, \cite{Ban17} \cite{Ban20}, \cite{BSS12} and \cite{FrTS23} for other examples and information.

It is important to note that the traces are \textit{normalized}. In this situation, $\alpha$ is injective by \cite[Lemma 2.1 p.~2283]{Arh19}.

\begin{prop}
\label{prop-fundamental}
Let $\cal{M}, \cal{N}_1$ and $\cal{N}_2$ be finite von Neumann algebra equipped with normal finite faithful normalized traces. Let $\alpha \co \cal{M} \to \cal{N}_1 \otvn \cal{N}_2$ be a trace preserving normal $*$-homomorphism such that
\begin{equation}
\label{rel-fund-455}
(\Id_{\cal{N}_1} \ot \tau_{\cal{N}_2}) \circ \alpha(x)
=\tau_{\cal{M}}(x) 1_{\cal{N}_1}, \quad x \in \cal{M}.
\end{equation}
Suppose that $1 \leq p \leq \infty$. Then $\alpha$ induces a complete contraction $\beta \co \L^1(\cal{M}) \to \L^p(\cal{N}_1,\L^1(\cal{N}_2))$.
\end{prop}

\begin{proof}
In the case of a \textit{positive} element $u$ in the space $S^\infty_k(\L^1(\cal{M}))$, we can write
\begin{align}
\MoveEqLeft
\label{divers-566789-bis}
\bnorm{(\Id_{S^\infty_k} \ot \alpha)(u)}_{S^\infty_k(\L^\infty(\cal{N}_1,\L^1(\cal{N}_2)))}              
\ov{\eqref{Norm-LpL1-pos}}{=} \bnorm{(\Id_{S^\infty_k(\cal{N}_1)} \ot \tau_{\cal{N}_2})((\Id_{S^\infty_k} \ot \alpha)(u))}_{S^\infty_k(\cal{N}_1)} \\
&= \bnorm{\big[\Id_{S^\infty_k} \ot \big(\Id_{\cal{N}_1} \ot \tau_{\cal{N}_2} \big) \circ \alpha\big](u) }_{S^\infty_k(\cal{M})} 
\ov{\eqref{rel-fund-455}}{=} \bnorm{\big(\Id_{S^\infty_k} \ot \tau_{\cal{M}} \big)(z) \ot 1_{\cal{N}_1}}_{S^\infty_k(\cal{M})} \nonumber \\
&=\bnorm{\big(\Id_{S^\infty_k} \ot \tau_{\cal{M}} \big)(u)}_{S^\infty_k} 
\ov{\eqref{Norm-LpL1-pos}}{=} \norm{u}_{S^\infty_k(\L^1(\cal{M}))}\nonumber. 
\end{align} 
Now, consider an \textit{arbitrary} element $x$ in the space $S^\infty_k(\L^1(\cal{M}))$. Let $\epsi > 0$. According to \eqref{formula-Junge}, there exists element $y$ and $z$ such that $x = y z$ and 
\begin{equation}
\label{fin1}
\norm{yy^*}_{S^\infty_k(\L^1(\cal{M})}^{\frac{1}{2}} \norm{z^*z}_{S^\infty_k(\L^1(\cal{M}))}^{\frac{1}{2}} \leq \norm{x}_{S^\infty_k(\L^1(\cal{M}))} +\epsi.
\end{equation}
Now, we have
\begin{align*}
\MoveEqLeft
\norm{(\Id \ot \alpha)(y)((\Id \ot \alpha)(y))^*}_{S^\infty_k(\L^\infty(\cal{N}_1,\L^1(\cal{N}_2)))}^{\frac{1}{2}} 
\norm{((\Id \ot \alpha)(z))^*(\Id \ot \alpha)(z)}_{S^\infty_k(\L^\infty(\cal{N}_1,(\L^1(\cal{N}_2)))}^{\frac{1}{2}} \\
&=\norm{(\Id \ot \alpha)(yy^*)}_{S^\infty_k(\L^\infty(\cal{N}_1,\L^1(\cal{N}_2)))}^{\frac{1}{2}} 
\norm{(\Id \ot \alpha)(z^*z)}_{S^\infty_k(\L^\infty(\cal{N}_1,\L^1(\cal{N}_2)))}^{\frac{1}{2}} \\      
&\ov{\eqref{divers-566789-bis}}{=} \norm{yy^*}_{S^\infty_k(\L^1(\cal{M}))}^{\frac{1}{2}}\norm{z^*z}_{S^\infty_k(\L^1(\cal{M}))}^{\frac{1}{2}} 
\ov{\eqref{fin1}}{\leq} \norm{x}_{S^\infty_k(\L^1(\cal{M}))} +\epsi.
\end{align*}  
Moreover, it is immediate that 
$$
(\Id \ot \alpha)(x)
=(\Id \ot \alpha)(yz)
=(\Id \ot \alpha)(y)(\Id \ot \alpha)(z).
$$ 
From \eqref{formula-Junge}, we infer that $\norm{(\Id \ot \alpha)(x)}_{S^\infty_k(\cal{N}_1,\L^1(\cal{N}_2)))} \leq \norm{x}_{S^\infty_k(\L^1(\cal{M})} +\epsi$. Since $\epsi > 0$ is arbitrary, we obtain that the map $\alpha \co \L^1(\cal{M}) \to \L^\infty(\cal{N}_1,\L^1(\cal{N}_2))$ is a complete contraction. Now, by Lemma \ref{lemma-trace-preserving} we have a complete isometry $\alpha \co \L^1(\cal{M}) \to \L^1(\cal{N}_1,\L^1(\cal{N}_2))$. We conclude by interpolation that the map $\alpha$ induces a complete contraction $\alpha \co \L^1(\cal{M}) \to \L^p(\cal{N}_1,\L^1(\cal{N}_2))$. 
\end{proof}

For this paper, the interest of the ergodicity lies in the second part of the following result and the next theorem. 

\begin{cor}
\label{Prop-transfer-1}
Let $\QG$ be a compact group of Kac type equipped with its Haar state $h_{\QG}$ and let $\cal{M}$ be a finite von Neumann algebra. Let $\alpha \co \cal{M} \to \cal{M} \otvn \L^\infty(\QG)$ be a right action. Consider $1 \leq p <\infty$.
\begin{enumerate}
	\item If $\cal{M}$ is equipped with a trace and if $\alpha$ is trace preserving then the map $\alpha$ induces a complete isometry $\alpha \co \L^p(\cal{M}) \to \L^p(\cal{M},\L^p(\QG))$.
	\item Suppose that $\alpha$ is ergodic and that the state $h_\cal{M}$ is a trace. The map $\alpha$ induces an isometry $\alpha \co \L^1(\cal{M}) \to \L^p(\cal{M},\L^1(\QG))$ which is completely contractive.
\end{enumerate}
\end{cor}

\begin{proof}
1. Since the normal unital injective $*$-homomorphism $\alpha$ is trace preserving, it suffices to use Lemma \ref{lemma-trace-preserving}.

2. It is a consequence of Proposition \ref{prop-fundamental} and \eqref{Def-ergodic}.
For the reverse inequality, it suffices to use the contractive inclusion $\L^p(\QG,\L^1(\cal{M})) \subset \L^1(\QG,\L^1(\cal{M}))$ and the first point.
\end{proof}

\begin{remark} \normalfont
It seems to the author that this map is completely isometric with elementary additional arguments.
\end{remark}

We have the following generalization of Proposition \ref{Prop-coproduct-column}.

\begin{prop}
\label{Prop-coaction-column}
Let $\QG$ be a compact quantum group of Kac type. An ergodic right action $\alpha \co \cal{M} \to \cal{M} \otvn \L^\infty(\QG)$ induces a completely isometric map $\alpha \co \L^2(\cal{M})_c \to \L^\infty(\cal{M},\L^2(\QG)_c)$ (and similarly for the row case) and a completely contractive map $\alpha \co \L^2(\cal{M}) \to \L^\infty(\cal{M},\L^2(\QG))$.
\end{prop}

\begin{proof}
We have
\begin{equation}
\label{normLinfty-L2-bis}
\norm{x}_{\L^\infty(\cal{M},\L^2(\QG)_c)}            
=\norm{(\Id_{\cal{M}} \ot h_{\QG})(x^*x)}_{\L^\infty(\cal{M})}^{\frac{1}{2}}.
\end{equation}
We have
\begin{align*}
\MoveEqLeft
\norm{\alpha(x)}_{\L^\infty(\cal{M},\L^2(\QG)_c))}            
\ov{\eqref{normLinfty-L2-bis}}{=} \norm{(\Id \ot h_{\QG})(\alpha(x^*x))}_{\L^\infty(\cal{M})}^{\frac{1}{2}} 
\ov{\eqref{Def-ergodic}}{=} \norm{\varphi_{\cal{M}}(x^*x)1_\cal{M}}_{\L^\infty(\cal{M})}^{\frac{1}{2}} \\
&=\varphi_{\cal{M}}(x^*x)^{\frac{1}{2}}
=\norm{x}_{\L^2(\cal{M})}.
\end{align*} 
Let $x=[x_{ij}]$ be an element of $\M_n(\cal{M})$. We have
\begin{align*}
\MoveEqLeft
\bnorm{(\Id_{\M_n} \ot \alpha)(x)}_{\M_n(\L^\infty(\cal{M},\L^2(\QG)_c)))}             
=\bnorm{[\alpha(x_{ij})]}_{\M_n(\L^\infty(\cal{M},\L^2(\QG)_c))} 
 =\bnorm{[\alpha(x_{ij})]}_{\L^\infty(\M_n(\L^\infty(\cal{M})), \L^2(\QG)_c)} \\
&\ov{\eqref{normLinfty-L2}}{=}  \bnorm{(\Id_{\M_n(\cal{M})} \ot h_{\QG})([\alpha(x_{ij})]^*[\alpha(x_{ij})])}_{\L^\infty(\M_n(\cal{M}), \L^2(\QG)_c)}^{\frac{1}{2}}
\ov{\eqref{Def-ergodic}}{=} \bnorm{(\Id_{\M_n} \ot h_{\QG})(x^*x)}_{\M_n}^{\frac{1}{2}} \\
&\ov{\eqref{normLinfty-L2}}{=}  \norm{x}_{\M_n(\L^2(\cal{M})_c)}.
\end{align*}
We conclude by density. The last part is obtained by interpolation.
\end{proof}

\begin{remark} \normalfont
In \cite{DRVV10}, the authors proved that there exists a bijective correspondence between (not necessarily) ergodic actions of monoidally equivalent compact quantum groups on von Neumann algebras. Concrete examples of monoidally equivalent compact quantum groups are given in \cite[Section 4]{DRVV10} and \cite[Sections 5 and 6]{BRV06}.  
\end{remark}

\begin{remark} \normalfont
Let $G$ be a discrete group and $\alpha \co \M_n \to \VN(G) \otvn \M_n$ be an ergodic action of $(\VN(G),\Delta)$, where $\Delta$ is defined in \eqref{coproduct-VNG}, on the matrix algebra $\M_n$. By \cite[Corollary 7.2]{DeC11a}, there exists a \textit{finite} subgroup $H$ of $G$ such that $\alpha(\M_n) \subset \VN(H) \ot \M_n$ and the action $\alpha' \co \M_n \to \VN(H) \otvn \M_n$ obtained from the action $\alpha$ with range restricted to $\VN(H) \otvn \M_n$ is ergodic.
\end{remark}

\begin{remark} \normalfont
If the state $\varphi_{\cal{M}}$ is not a trace, we can use Haagerup's noncommutative $\L^p$-spaces described in \cite{Ter81}. If $D$ is the density operator belonging to $\L^1(\cal{M})=\L^1(\cal{M},\varphi_{\cal{M}})$, recall that by \cite[(1.13)]{HJX10} the Haagerup trace $\tr$ associated to $\varphi_{\cal{M}}$ satisfies
\begin{equation}
\label{HJX-1.13}
\varphi_{\cal{M}}(x)
=\tr(xD), \quad x \in \cal{M}.
\end{equation}
So using Remark \ref{Rem-Haagerup-Lp} in the first equality, for any positive $x \in \cal{M}$, we can write 
\begin{align*}
\MoveEqLeft
\label{divers-566789-bis}
\norm{\alpha(x)}_{\L^\infty(\cal{M},\L^1(\QG))}              
\ov{\eqref{Norm-LpL1-pos}}{=} \bnorm{(\Id \ot h_\cal{\QG})(\alpha(x))}_{\L^\infty(\cal{M})} 
\ov{\eqref{Def-ergodic}}{=} \norm{\varphi_\cal{M}(x) 1_{\L^\infty(\QG)}}_{\L^\infty(\cal{M})} \\
&=|\varphi_\cal{M}(x)| \norm{1}_{\L^\infty(\cal{M})} 
\ov{\eqref{HJX-1.13}}{=} |\tr(xD)| \nonumber
\leq \norm{xD}_{\L^1(\cal{M})}. \nonumber
\end{align*} 
Here we use a identification between $\L^1(\QG,\tau)$ and the associated noncommutative Haagerup $\L^1$-space. By \cite[Lemma 3]{Wat88}, the subspace $\cal{M}D$ is dense in the Banach space $\L^1(\cal{M})$. So we obtain a contractive map from $\L^1(\cal{M})$ into $\L^\infty(\cal{M},\L^1(\QG))$.
\end{remark}

The following theorem is the main result of this section.

\begin{thm}
\label{thm-Smin-transfert}
Let $\QG$ be a compact quantum group of Kac type equipped with its normalized Haar state and $\cal{M}$ be a hyperfinite finite von Neumann algebra. Let $\alpha \co \cal{M} \to \cal{M} \otvn \L^\infty(\QG)$ be an ergodic right action. Suppose that the normal state $\varphi_\cal{M}$ is a trace. Consider $1 \leq p <\infty$. Then for any $f \in \L^1(\QG)$ we have
\begin{equation}
\label{transference-norms}
\norm{T_{f,\alpha}}_{\cb,\L^1(\cal{M}) \to \L^p(\cal{M})}
\leq \norm{T_f}_{\cb,\L^1(\QG) \to \L^p(\QG)}.
\end{equation}
and if $\L^\infty(\QG)$ is in addition hyperfinite 
\begin{equation}
\label{belles-estimations}
\H_{\cb,\min,\tau}\big(T_{f,\alpha}\big)
\geq \H_{h_\QG}(f)
\end{equation}
\end{thm}

\begin{proof}
Using the point 1 of Proposition \ref{Prop-transfer-1} in the first equality and \cite[(3.1) p.~39]{Pis98} in the second inequality and the point 2 of Proposition \ref{Prop-transfer-1} in last inequality, we have 
\begin{align*}
\MoveEqLeft
\norm{T_{f,\alpha}}_{\cb,\L^1(\cal{M}) \to \L^p(\cal{M})}
=\norm{\alpha \circ T_{f,\alpha}}_{\cb,\L^1(\cal{M}) \to \L^p(\cal{M},\L^p(\QG))} 
\ov{\eqref{commute-transference}}{=} \bnorm{\big(\Id_{\L^p(\cal{M})} \ot T_f\big) \circ \alpha}_{\cb,\L^1(\cal{M}) \to \L^p(\cal{M},\L^p(\QG))}  \\        
&\leq \bnorm{\Id_{\L^p(\cal{M})} \ot T_f}_{\cb,\L^p(\cal{M},\L^1(\QG)) \to \L^p(\cal{M},\L^p(\QG))} \norm{\alpha}_{\cb,\L^1(\cal{M}) \to \L^p(\cal{M},\L^1(\QG))}\\
&\leq \norm{T_f}_{\cb,\L^1(\QG) \to \L^p(\QG)} \norm{\alpha}_{\cb,\L^1(\cal{M}) \to \L^p(\cal{M},\L^1(\QG))}
\ov{\eqref{norm-equality-QG}}{\leq} \norm{T_f}_{\cb,\L^1(\QG) \to \L^p(\QG)}.
\end{align*}
Moreover, $\norm{T_{f,\alpha}}_{\cb,\L^1(\cal{M}) \to \L^1(\cal{M})} = \norm{T_f}_{\cb,\L^1(\QG) \to \L^1(\QG)}=1$. Now, with the additional assumption, we have
\begin{align*}
\MoveEqLeft
\H_{\cb,\min,\tau}(T_{f,\alpha})
\ov{\eqref{Def-Scb-min-without-p}}{=} -\frac{\d}{\d p} \norm{T_{f,\alpha}}_{\cb,\L^1(\cal{M}) \to \L^p(\cal{M})}|_{p=1} 
\ov{\eqref{elem-lim}}{=} -\lim_{q \to \infty} q\ln \norm{T_{f,\alpha}}_{\cb,\L^1(\cal{M}) \to \L^{q^*}(\cal{M})} \\
&\ov{\eqref{transference-norms}}{\geq} -\lim_{q \to \infty} q\ln \norm{T_f}_{\cb,\L^1(\QG) \to \L^{q^*}(\QG)}
\ov{\eqref{elem-lim}}{=} -\frac{\d}{\d p} \norm{T_f}_{\cb,\L^1(\QG) \to \L^p(\QG)}|_{p=1}
\ov{\eqref{Def-Scb-min-without-p}}{=} \H_{\cb,\min,h_\QG}(T_f) \\
&\ov{\eqref{Scbmin}}{=} \H(f).
\end{align*} 
\end{proof}

\begin{remark} \normalfont
In particular, with \eqref{Holevo-Smin} and \eqref{compar-ent-min}, we obtain the inequality (non-optimal in general)
\begin{equation*}
\chi(T_{f,\alpha}) 
\leq -\H(f).
\end{equation*}
\end{remark}

\begin{remark} \normalfont
In the case where $\cal{M}$ is equipped with a trace, $\QG$ is a compact quantum group of Kac type and where $\alpha$ is trace preserving, a similar argument shows that 
\begin{equation}
\label{transference-norms-bis}
\norm{T_{f,\alpha}}_{\cb,\L^p(\cal{M}) \to \L^p(\cal{M})}
\leq \norm{T_f}_{\cb,\L^p(\QG) \to \L^p(\QG)}.
\end{equation} 
This observation generalizes the transference result \cite[Proposition 2.20]{ArK22}.
\end{remark}

Recall that we say that a quantum channel $T \co S^1_n \to S^1_n$ is covariant if there exists a continuous projective unitary representation $u \co G \to \M_n$ of a group $G$ such that
\begin{equation}
\label{}
T(u_sxu_s^*)
=u_sT(x)u_s, \quad s \in G,x \in \M_n.
\end{equation}
If the group $G$ is compact and if $u$ is \textit{irreducible}, Holevo gives in \cite[p.~42]{Hol05} (see also \cite{Hol16} for a nice complement) a simple argument of the equalities 
\begin{equation}
\label{Holevo-result}
\chi(T) 
=\log n -\H_{\min}(T) 
\quad \text{and} \quad
\max_{\rho} \H(T(\rho))
=\H\bigg(T\bigg(\frac{1}{n}\I\bigg)\bigg).
\end{equation}
The first equality is equally proved in \cite[Corollary 4.2 p.~385]{JuP15} (it is the case $d=1$) but only for irreducible \textit{representations}. Indeed, the first assumption of \cite[Definition 4.1 p.~383]{JuP15} is satisfied for any irreducible representation by \cite[Lemma 27.16]{HeR70}. The problem with this formula is the computation of the entropy $\H_{\min}(T)$. 

The result \cite[Corollary 15]{MSD17} (see also \cite{KMS20} for a related paper) precisely describes the quantum channels which are covariant with respect to an irreducible representation $u$ of a finite group $G$ such that the representation $u \ot u^c$ is multiplicity free\footnote{\thefootnote. Here $u^c$ is the contragredient representation.}.
In particular, such quantum channel can be written $T_{f,\alpha}$ as in \eqref{channel-commutatif} for some positive function $f \co G \to \mathbb{C}$ with $\norm{f}_{\L^1(G)}=1$. Now,
we obtain the following estimate:
$$
-\H_{\min,\tau}(T)
\ov{\eqref{compar-ent-min}}{\leq} -\H_{\cb,\min,\tau}(T)
\ov{\eqref{belles-estimations}}{\leq} -\H(f)
$$ 
where the entropies are taken for the \textit{normalized} traces \textit{contrary} to \eqref{Holevo-result}. We have the same upper bound for the classical capacity, see Remark \ref{Rem-classical-capacity-Tfalpha}.

\subsection{Systems of eigenoperators of ergodic co-actions of discrete groups}
\label{sec-eigenoperators}


\paragraph{2-cocycles of discrete groups} 
We first recall that a 2-cocycle on a discrete group $G$ with values in the torus $\T \ov{\mathrm{def}}{=} \{z \in \mathbb{C}: |z|=1 \}$ is a map $\sigma \co G \times G \to \T$ satisfying \eqref{equation-2-cocycle}. The terminologies <<multiplier>> or <<factor system>> are also used in the literature.  This implies that\footnote{\thefootnote. Indeed, the cocycle condition $\sigma(s,t)\sigma(st,r)
=\sigma(s,tr)\sigma(t,r)$ of \eqref{equation-2-cocycle} evaluated with $(s^{-1},s,e)$ instead of $(s,t,r)$ gives
$$
\sigma(s^{-1},s)\sigma(e,e)
=\sigma(s^{-1},s)\sigma(s,e).
$$ 
Hence $
\sigma(e,e) 
= \sigma(s,e)$. 
Similarly the 2-cocycle condition applied with $(e,s,s^{-1})$ instead of $(s,t,r)$ gives
$$
\sigma(e,s)\sigma(s,s^{-1})
=\sigma(e,e)\sigma(s,s^{-1}) 
$$
Thus $\sigma(e,s) = \sigma(e,e)$.}
$$
\sigma(s,e)
= \sigma(e,e) 
= \sigma(e,s) , \quad s \in G.
$$ 
The 2-cocycle $\sigma$ is said to be normalized if $\sigma(e,e)=1$. This condition entails\footnote{\thefootnote. Indeed, the cocycle condition $\sigma(s,t)\sigma(st,r)
=\sigma(s,tr)\sigma(t,r)$ of \eqref{equation-2-cocycle} evaluated with $(s,s^{-1},s)$ instead of $(s,t,r)$ gives
$$
\sigma(s,s^{-1})\sigma(e,s)
=\sigma(s,e)\sigma(s^{-1},s).
$$ } that
\begin{equation}
\label{ss-1-2-cocycle}
\sigma(s,s^{-1}) 
= \sigma(s^{-1},s), \quad s \in G.
\end{equation}
For any arbitrary function $\rho \co G \to \T$, the map $\sigma' \co G \times G\to \T$, $(s,t) \mapsto \rho(s)\rho(t)\rho(st)^{-1}\sigma(s,t)$ is again a 2-cocycle. In the case, we say that $\sigma$ and $\sigma'$ are equivalent (or cohomologous).  

\paragraph{Second cohomology group} The set of 2-cocycles on $G$ is an abelian group under pointwise multiplication, denoted by $\mathrm{Z}^2(G,\T)$. The identity element is the constant function on $G$ denoted by $1$. The inverse operation corresponds to conjugation, i.e.~$\sigma^{-1} = \ovl{\sigma}$, where $\ovl{\sigma}(s,t) \ov{\mathrm{def}}{=} \ovl{\sigma(s,t)}$ for any $s,t \in G$. The set of 2-cocycles which are equivalent to 1 is a subgroup $\mathrm{B}^2(G,\T)$ of the group $\mathrm{Z}^2(G,\T)$. Hence, we can introduce the quotient group
\begin{equation}
\label{def-H2}
\H^2(G,\T)
\ov{\mathrm{def}}{=} \mathrm{Z}^2(G,\T)/\mathrm{B}^2(G,\T),
\end{equation}
commonly referred to as the second cohomology group of $G$, also known as the Schur multiplier or multiplicator.




\paragraph{Ergodic co-actions of discrete groups}
A co-action of a discrete group $G$ is a left action $\alpha \co \cal{M} \to \cal{M} \otvn \VN(G)$ of the compact quantum group $(\VN(G),\Delta)$, defined in Example \ref{example-QG}, on a von Neumann algebra $\cal{M}$. The set of conjugacy classes (see \eqref{conjugate-actions}) of faithful ergodic co-actions of a discrete group $G$ is denoted by $[G]$. As in \cite[p.~163]{KaS82}, we can equip $[G]$ with a structure of semigroup. It is showed in \cite[Theorem 4.7 p.~170]{KaS82} that $[\Delta]$ is a unit element in the semigroup $[G]$. Moreover, the same result says if $\alpha$ is a faithful ergodic action then $[\tilde{\alpha}]$ is the inverse of $[\alpha]$, where 
$\tilde{\alpha} \co \cal{M}^\op \to \cal{M}^\op \otvn \VN(G)$, $x^\op \mapsto (\theta \ot R)\circ \alpha(x)$. Here, we use the antipode $R \co \VN(G) \to \VN(G)$, $\lambda_s \mapsto \lambda_{s^{-1}}$ and the map $\theta \co \cal{M} \to \cal{M}^\op$, $x \mapsto x^\op$. 

Consequently, $[G]$ is a group. Furthermore, it is showed in \cite[Theorem 4.6 p.~169]{KaS82} that the group $[G]$ of conjugacy classes of faithful ergodic actions is isomorphic to the group $\H^2(G,\T)$. We also refer to \cite{OPT80}, to \cite{AHK80} and to \cite{DeS83} for the case where $G$ is abelian. 
If $\alpha \co \cal{M} \to \cal{M} \otvn \VN(G)$ is a faithful action then by essentially \cite[Proposition 3.7 p.~166]{KaS82} the unique invariant state on $\cal{M}$ is a normal finite faithful trace $\tau_{\cal{M}}$.

Finally, it is known \cite[Proposition 6.2 p.~173]{KaS82} that if $\alpha \co \cal{M} \to \cal{M} \otvn \VN(G)$ is a faithful co-action then $\cal{M}$ is injective if and only if the discrete group $G$ is amenable.




\paragraph{Systems of eigenoperators}

Recall that $\mathrm{A}(G)$ is the Fourier algebra of $G$. Suppose that that the discrete group $G$ is abelian (see Remark \ref{rem-more general} for the non-abelian case). Then $\hat{G}$ is a compact abelian group with dual $G$.

\begin{defi}
We say that a family of unitary operators $(v_s)_{s \in G}$ in $\cal{M}$ is a family of eigenoperators for $\alpha \co \cal{M} \to \cal{M} \otvn \VN(G)$ if the von Neumann algebra $\cal{M}$ is generated by this family and if
\begin{equation}
\label{eigen-def}
\Phi_\chi(v_s)
=\la \chi,s \ra_{\hat{G},G} v_s, \quad s \in G, \chi \in \hat{G}
\end{equation}
where $\Phi \co \hat{G} \to \Aut(\cal{M})$ is the map associated to the ergodic action $\alpha \co \cal{M} \to \cal{M} \otvn\VN(G)=\cal{M} \otvn \L^\infty(\hat{G})$ with \eqref{classical-useful}.
\end{defi}


The existence of such system is explained in \cite[p.~240]{OPT80} (see also \cite[p.~165]{KaS82}). In this case, it is easy to check (see \cite{OPT80}) that the complex function $\sigma \co G \times G \to \mathbb{C}$ defined by
\begin{equation}
\label{calculus-Vilenkin-prod}
\sigma(s,t)
\ov{\mathrm{def}}{=} v_s v_t v_{s t}^*, 
\quad \textrm{i.e.} \quad 
v_s v_t
=\sigma(s,t) v_{s t} \quad s,t \in G,
\end{equation}
is a 2-cocycle. 
We consider the unique $G$-invariant normal state on the von Neumann algebra $\cal{M}$, which is a normal finite faithful trace $\tau$. The family $(v_s)_{s \in G}$ is clearly an orthonormal basis of the complex Hilbert space $\L^2(\cal{M})$. We have
\begin{equation}
\label{calculus-Vilenkin-inv}
v_s^* 
=c_s v_{s^{-1}}, \quad s \in G.
\end{equation}
for some complex number $c_s$. We can suppose that the 2-cocycle is normalized.
We have
\begin{equation}
\label{trace-fer}
\tau(v_s) 
= 0 \quad \text{if } s \not= e, 
\quad \text{i.e.} \quad 
\tau(v_s)
=\delta_{s,1}.
\end{equation}

\begin{remark} \normalfont
\label{rem-more general}
Our results in this section and in the two following section can be adapted to the case of ergodic faithful integrable of locally compact \textit{abelian} groups on von Neumann algebras, as in \cite{DeS83}. In this case, by \cite[Proposition 1.15 p.~275]{DeS83} the formula of \eqref{Def-ergodic} defines a normal semifinite faithful trace $\varphi_{\cal{M}}$ on $\cal{M}$. Similarly, we can generalize to the case of the case of a \textit{non-abelian} discrete group $G$ with \cite{KaS82}.
\end{remark}

\paragraph{Some maps}
Following Remark \ref{lien-actions-droite-gauche}, we introduce the map $\beta \ov{\mathrm{def}}{=} \flip \circ \alpha\co \cal{M} \to \VN(G) \otvn \cal{M}$. It is easy to check that the maps $\alpha$ and $\beta$ can be rewritten under the form
\begin{equation}
\label{actions-alpha-beta}
\alpha(v_s)
\ov{\mathrm{def}}{=} v_s \ot \lambda_s,\quad
\beta(v_s)
\ov{\mathrm{def}}{=} \lambda_s \ot v_s,\quad s \in G.
\end{equation}
Indeed, we have
\begin{equation}
\label{}
\alpha(v_s)(\chi) 
\ov{\eqref{classical-useful}}{=} \Phi_\chi(v_s)
\ov{\eqref{eigen-def}}{=} \la \chi,s \ra_{\hat{G},G} v_s.
\end{equation}

\paragraph{Multipliers on algebras of eigenoperators}
If the complex function $\phi \co G \to \mathbb{C}$ belongs to the Fourier algebra $\mathrm{A}(G)$ then for any $s \in G$ we have
\begin{align*}
\MoveEqLeft
R_{\la \phi, \cdot \ra,\beta}(\lambda_s)
\ov{\eqref{convol-def-left}}{=} (\la \phi, \cdot \ra \ot \Id) \circ \Delta(\lambda_s)         
= (\la \phi, \cdot \ra \ot \Id) (\lambda_s \ot \lambda_s) 
= \la \phi, \lambda_s \ra \lambda_s = \phi(s) \lambda_s.
\end{align*}
Consequently, $R_{\la \phi, \cdot \ra,\beta}$ is the Fourier multiplier $M_\phi \co \VN(G) \to \VN(G)$ with symbol $\phi$. Moreover
\begin{align*}
\MoveEqLeft
R_{\la \phi, \cdot \ra,\beta}(v_s)         
\ov{\eqref{convol-def-left}}{=} (\la \phi, \cdot \ra_{\mathrm{A}(G),\VN(G)} \ot \Id) \circ \beta(v_s) 
\ov{\eqref{actions-alpha-beta}}{=} (\la \phi, \cdot \ra_{\mathrm{A}(G),\VN(G)} \ot \Id) (\lambda_s \ot v_s) \\
&=\la \phi, \lambda_s \ra_{\mathrm{A}(G),\VN(G)} v_s 
=\phi(s)v_s.
\end{align*}
So the map $R_{\phi,\beta}$ can be seen as a <<multiplier>> which parallels the previous notion of Fourier multiplier. 

\begin{defi}
\label{def-multiplier-eigen}
Consider a complex function $\phi \co G \to \mathbb{C}$. A weak* continuous linear map $R_\phi \co \cal{M} \to \cal{M}$ defined by
\begin{equation}
\label{def-rad-mult}
R_\phi(v_s)
= \phi(s) v_s, \quad s \in G.
\end{equation}
is called multiplier with symbol $\phi$. 
\end{defi}

The previous maps $R_{\la \phi, \cdot \ra,\beta}$ are multipliers. Similarly, we have
\begin{align*}
\MoveEqLeft
T_{\la \phi, \cdot \ra,\alpha}(v_s)
\ov{\eqref{convol-def-bis}}{=} (\Id \ot \la \phi, \cdot \ra) \circ \alpha(v_s) 
\ov{\eqref{actions-alpha-beta}}{=} (\Id \ot \la \phi, \cdot \ra)(v_s \ot \lambda_s) \\
&=\la \phi, \lambda_s \ra_{\mathrm{A}(G),\VN(G)} v_s
=\phi(s)v_s.
\end{align*}
Hence $T_{\la \phi, \cdot \ra,\alpha}=R_\phi$. In particular, for any $\phi \in \mathrm{A}(G)$ we have the two intertwining relations
\begin{equation}
\label{Equation-Delta-66}
\beta \circ R_\phi
\ov{\eqref{commute-transference}}{=} \big(\Id_{\L^\infty(G)} \ot M_\phi\big) \circ  \beta
\end{equation}
and 
\begin{equation}
\label{Equation-Delta-op}
\alpha \circ R_\phi
\ov{\eqref{commute-transference}}{=} \big(\Id_{\cal{M}} \ot M_\phi\big) \circ \alpha.
\end{equation}

\begin{remark} \normalfont
\label{remark-trace-preserving}
It is easy to see that such a radial multiplier $R_\phi$ is unital if and only if $\phi(0)=1$. This is also equivalent to say that the map $R_\phi$ is trace preserving.
\end{remark}

\begin{prop}
\label{prop-autre-extension}
Suppose that $1 \leq p \leq \infty$. The map $\beta \co \cal{M} \to \VN(G) \otvn \cal{M}$ is a trace preserving unital injective $*$-homomorphism and induces a complete contraction $\beta \co \L^1(\cal{M}) \to \L^p(\VN(G),\L^1(\cal{M}))$.
\end{prop}

\begin{proof}
Since $\beta = \flip \circ \alpha$ and since $\alpha$ is a trace preserving unital injective $*$-homomorphism, we see that $\beta$ is trace preserving unital injective $*$-homomorphism. It is not difficult to prove that
\begin{equation}
\label{ergo-biz-biss}
(\Id_{\VN(G)} \ot \tau_{\cal{M}}) \circ \beta(x)
=\tau_{\cal{M}}(x) 1_{\VN(G)}, \quad x \in \cal{M}.
\end{equation}
Indeed, for any $s \in G$, we have
\begin{align*}
\MoveEqLeft
(\Id_{\VN(G)} \ot \tau_{\cal{M}}) \circ \beta(v_s)
\ov{\eqref{actions-alpha-beta}}{=} (\Id_{\VN(G)} \ot \tau_{\cal{M}}) (\lambda_s \ot v_s)
=\tau_{\cal{M}}(v_s) \lambda_s
=\tau_{\cal{M}}(v_s) 1_{\VN(G)}.
\end{align*}
Hence the result a consequence of Proposition \ref{prop-fundamental}.
\end{proof}

%


%

\paragraph{A $*$-homomorphism} We equally consider the linear map $\eta \co \VN(G) \to \cal{M} \otvn \cal{M}^\op$ defined by
\begin{equation}
\label{autre-eta}
\eta(\lambda_s)
\ov{\mathrm{def}}{=} v_s \ot (v_{s}^*)^\op,\quad s \in G.
\end{equation}
As the following result shows, this map allows the intertwining of multipliers on the group von Neumann algebra $\VN(G)$ and the multipliers on the von Neumann algebra $\cal{M}$.

\begin{prop}
\label{Prop-eta}
For any complex function $\phi \co G \to \mathbb{C}$ we have
\begin{equation}
\label{commute-bis}
(\Id_{\cal{M}} \ot R_\phi) \circ \eta
=\eta \circ M_\phi.
\end{equation}
Suppose that $1 \leq p \leq \infty$. The map $\eta$ is a trace preserving unital injective $*$-homomorphism and induces a complete contraction $\eta \co \L^1(\VN(G)) \to \L^p(\cal{M},\L^1(\cal{M}))$.
\end{prop}

\begin{proof}
The fact that $\eta$ is normal is easy and left to the reader. First, we prove the equality \eqref{commute-bis}. On the one hand, for any $s \in G$, we have
\begin{align*}
\MoveEqLeft
(\Id \ot R_\phi)\eta(\lambda_s)         
\ov{\eqref{autre-eta}}{=} (\Id \ot R_\phi)(v_s \ot v_s)
= v_s \ot R_\phi(v_s)
\ov{\eqref{def-rad-mult}}{=} \phi(s) v_s \ot v_s.
\end{align*}
On the other hand, we have
\begin{align*}
\MoveEqLeft
\eta M_\phi(\lambda_s)         
= \phi(s)\eta(\lambda_s) 
\ov{\eqref{autre-eta}}{=} \phi(s) v_s \ot v_s.
\end{align*}
For any $s,t \in G$, we have
\begin{align*}
\MoveEqLeft
\eta(\lambda_s\lambda_t)
\ov{\eqref{autre-eta}}{=}  v_{st} \ot (v_{st}^*)^\op
\ov{\eqref{calculus-Vilenkin-prod}}{=}  \ovl{\sigma(s,t)} v_{s} v_{t} \ot ((\ovl{\sigma(s,t)} v_{s} v_{t})^*)^\op 
= v_{s} v_{t}  \ot ((v_{s} v_{t})^*)^\op \\
&=v_s v_t \ot (v_s^*)^\op  (v_t^*)^\op 
=(v_s \ot (v_s^*)^\op)(v_t \ot (v_t^*)^\op)
\ov{\eqref{autre-eta}}{=} \eta(\lambda_s)\eta(\lambda_t)
\end{align*}
and
$$
\eta(\lambda_s^*)
= \eta(\lambda_{s^{-1}})
\ov{\eqref{autre-eta}}{=} v_{s^{-1}} \ot v_{s^{-1}}^*
\ov{\eqref{calculus-Vilenkin-inv}}{=} c_s v_{s^{-1}} \ot  v_{s}
\ov{\eqref{calculus-Vilenkin-inv}}{=} v_s^* \ot v_s
=(v_s \ot v_s^*)^*
\ov{\eqref{autre-eta}}{=} \eta(\lambda_s)^*.
$$
We deduce that $\eta$ is a $*$-homomorphism, which is clearly unital. Moreover, we also have
\begin{align*}
\MoveEqLeft
(\tau \ot \tau)(\eta(\lambda_s))
\ov{\eqref{autre-eta}}{=} (\tau \ot \tau)(v_s \ot v_s) 
=\tau(v_s)^2
\ov{\eqref{trace-fer}}{=} \delta_{s,e}
=\tau(\lambda_s).         
\end{align*}
Consequently, the map $\eta$ is trace preserving. The injectivity follows from \cite[Lemma 2.1 p.~2283]{Arh19}, but it is also possible to establish it through a direct proof. It is easy to check that
\begin{equation}
\label{ergo-biz}
(\Id_{\cal{M}} \ot \tau_{\cal{M}}) \circ \eta(f)
=\tau_{\VN(G)}(x) 1_{\cal{M}}, \quad x \in \VN(G).
\end{equation}
Indeed, for any $s \in G$, we have
\begin{align*}
\MoveEqLeft
(\Id_{\cal{M}} \ot \tau_{\cal{M}}) \circ \eta(\lambda_s)
\ov{\eqref{autre-eta}}{=}(\Id_{\cal{M}} \ot \tau_{\cal{M}}) (v_s \ot v_s) 
=\tau(v_s) v_s \\
&\ov{\eqref{trace-fer}}{=} \delta_{s,e} v_s 
=\delta_{s,e} 1_{\cal{M}}
= \tau_{\VN(G)}(\lambda_s) 1_{\cal{M}}.
\end{align*}
The last assertion is a consequence of Proposition \ref{prop-fundamental}.
\end{proof}

\subsection{Completely bounded norms $\L^1$ to $\L^p$ of multipliers on algebras of eigenoperators}
\label{sec-completely-bounded-multipliers}

If $G$ is a discrete group, recall that a function $\varphi \co G \to \mathbb{C}$ is positive definite if for any function $\xi \co G \to \mathbb{C}$ we have $\sum_{s,t \in G} \varphi(t^{-1}s) \ovl{\xi(t)} \xi(s) \geq 0$, see \cite[Proposition 13.4.4 p.~286]{Dix77}. Such a function satisfies $\varphi(s^{-1})=\ovl{\varphi(s)}$ and $|\varphi(s)| \leq \varphi(e)$ for any $s \in G$.  We use the notations that the ones of Section \ref{sec-eigenoperators}. In the next result, we characterize completely positive multipliers.

\begin{prop}
\label{Prop-cp}
Let $\phi \co G \to \mathbb{C}$ be a complex function. Then the following properties are equivalent.
\begin{enumerate}
	\item The multiplier $R_\phi \co \cal{M} \to \cal{M}$ is completely positive.
	\item The complex function $\phi \co G \to \mathbb{C}$ is definite positive on the discrete group $G$.
\end{enumerate} 
\end{prop}

\begin{proof}
1 $\Rightarrow$ 2: Suppose that the multiplier $R_\phi \co \cal{M} \to \cal{M}$ is completely positive. We denote by $\E \co \cal{M} \otvn \cal{M} \to \VN(G)$ the canonical trace preserving faithful conditional expectation associated to the trace preserving injective unital $*$-homomorphism $\eta \co \VN(G) \to \cal{M} \otvn \cal{M}$, provided by Proposition \ref{prop-existence-conditional-expectation}. Using the equality $\E \circ \eta=\Id$, we see that
$$
M_\phi
=\E \circ \eta \circ M_\phi
\ov{\eqref{commute-bis}}{=} \E \circ (\Id \ot R_\phi) \circ \eta.
$$ 
As a composition of completely positive maps, we deduce that the multiplier $M_\phi \co \VN(G) \to \VN(G)$ is (completely) positive. Consequently, the function $\phi \co G \to \mathbb{C}$ is definite positive.

2 $\Rightarrow$ 1: Suppose that the complex function $\phi \co G \to \mathbb{C}$ is definite positive. In this case, the linear operator $M_\phi \co \VN(G) \to \VN(G)$ is completely positive. Now, we consider the canonical trace preserving faithful conditional expectation $\E \co \cal{M} \otvn \VN(G) \to \cal{M}$ associated to the trace preserving injective unital $*$-homomorphism $\alpha \co \cal{M} \to \cal{M} \otvn \VN(G)$, provided by Proposition \ref{prop-existence-conditional-expectation}. With the equality $\E \circ \alpha = \Id$, we infer that
$$
R_\phi
=\E \circ \alpha \circ R_\phi
\ov{\eqref{Equation-Delta-op}}{=} \E \circ \big(\Id_{\cal{M}} \ot M_\phi\big) \circ \alpha.
$$
We obtain by composition that the radial multiplier $R_\phi \co \cal{M} \to \cal{M}$ is completely positive.
\end{proof}

\begin{thm}
\label{th-norm-cb-rad}
Suppose that $1 < p \leq \infty$. For any function $\phi \co \hat{G} \to \mathbb{C}$, we have
\begin{equation}
\label{norm-cb}
\norm{R_\phi}_{\cb, \L^1(\cal{M}) \to \L^p(\cal{M})}
= \norm{\sum_s \phi(s) \lambda_s}_{\L^p(\VN(G))}
\end{equation}
and
\begin{equation}
\label{norm-cb-2}
\norm{R_\phi^*}_{\cb, \L^{p^*}(\cal{M}) \to \L^\infty(\cal{M})}
=\norm{\sum_s \phi(s) \lambda_s}_{\L^{p}(\VN(G))}.
\end{equation}
\end{thm}

\begin{proof}
The second formula is an immediate consequence of the first formula by duality. So, it suffices to prove the first one. Note that by Proposition \ref{Prop-transfer-1}, we have
\begin{equation}
\label{Ine-div-4457}
\norm{\alpha}_{\cb, \L^1(\cal{M}) \to \L^p(\cal{M},\L^1(\VN(G)))}
\leq 1.
\end{equation}
Recall that the linear map $\alpha$ is a trace preserving unital $*$-homomorphism. We denote by $\E \co \cal{M} \otvn \VN(G) \to \cal{M}$ the associated canonical trace preserving faithful conditional expectation, provided by Proposition \ref{prop-existence-conditional-expectation}. Thus, the linear map $\alpha$ induces a complete isometry $\alpha \co \L^p(\cal{M}) \to \L^p(\cal{M} \otvn \VN(G))$ and we have a complete contraction $\E \co \L^p(\cal{M} \otvn \VN(G)) \to \L^p(\cal{M})$. On the one hand, using $\E \alpha=\Id_{\L^p(\cal{M})}$ in the first equality, we deduce that
\begin{align*}
\MoveEqLeft
\norm{R_\phi}_{\cb,\L^1(\cal{M}) \to \L^p(\cal{M})}
=\norm{\alpha R_\phi}_{\cb,\L^1(\cal{M}) \to \L^p(\cal{M} \otvn \VN(G))} \\
&\ov{\eqref{Equation-Delta-op}}{=} \norm{(\Id_{} \ot M_\phi ) \alpha}_{\cb,\L^1(\cal{M}) \to \L^p(\cal{M} \otvn \VN(G))} \\     
&\leq \norm{\Id \ot M_\phi}_{\cb,\L^p(\cal{M},\L^1(\VN(G)) \to \L^p(\cal{M} \otvn \VN(G))} \norm{\alpha}_{\cb,\L^1(\cal{M}) \to \L^p(\cal{M},\L^1(\VN(G)))} \\
&\ov{\eqref{ine-tensorisation-os}\eqref{Ine-div-4457}}{\leq} \norm{M_\phi}_{\cb, \L^1(\VN(G)) \to \L^p(\VN(G))}
\ov{\eqref{norm-equality-QG}}{=} \norm{\sum \phi(s) \lambda_s}_{\L^{p}(\VN(G))}.
\end{align*}
On the other hand, using the trace preserving unital injective $*$-homomorphism $\eta \co \VN(G) \to \cal{M} \otvn \cal{M}$ defined in \eqref{autre-eta} and a similar reasoning using its associated canonical trace preserving faithful conditional expectation $\E \co \cal{M} \otvn \cal{M} \to \VN(G)$, which satisfies $\E\eta=\Id$, we see that
\begin{align*}
\MoveEqLeft
\norm{\sum \phi(s) \lambda_s}_{\L^{p}(\VN(G))}
\ov{\eqref{norm-equality-QG}}{=}\norm{M_\phi}_{\L^1(\VN(G)) \to \L^p(\VN(G))}         
=\norm{\eta M_\phi}_{\L^1(\VN(G)) \to \L^p(\cal{M} \otvn \cal{M})}      \\
&\ov{\eqref{commute-bis}}{=} \norm{(\Id \ot R_\phi)\eta}_{\L^1(\VN(G)) \to \L^p(\cal{M} \otvn \cal{M})} \\
&\leq \norm{\Id_{} \ot R_\phi}_{\L^p(\cal{M},\L^1(\cal{M})) \to \L^p(\cal{M} \otvn \cal{M})} \norm{\eta}_{\L^1(\VN(G)) \to \L^p(\cal{M},\L^1(\cal{M}))} \\
&\ov{\eqref{ine-tensorisation-os}}{\leq}  \norm{R_\phi}_{\cb, \L^1(\cal{M}) \to \L^p(\cal{M})},
\end{align*}
where we used Proposition \ref{Prop-eta} in the last inequality.
\end{proof}

\subsection{Completely bounded minimal output entropy of the multipliers}
\label{entropy-cb-radial-multipliers}

We start with a useful lemma.

\begin{lemma}
\label{lemma-useful-678}
Let $G$ be an amenable discrete group. A function $\varphi \co G \to \mathbb{C}$ is positive definite if and only if
$\sum_{s \in G} \varphi(s)\lambda_s$ (symbolic notation, use an approximate unit) is a positive operator and we have 
\begin{equation}
\label{norm=1-789}
\norm{\sum_{s \in G} \varphi(s)\lambda_s}_{\L^1(\VN(G))}
=\varphi(e).
\end{equation}
\end{lemma}

\begin{proof}
For any function $\xi \co G \to \mathbb{C}$, we have
\begin{align*}
\MoveEqLeft
\bigg\la\sum_{s \in G} \varphi(s)\lambda_s \xi,\xi \bigg\ra_{\L^2(G)}
=\sum_{s \in G} \varphi(s) \la\lambda_s \xi,\xi \ra_{\L^2(G)}
=\sum_{s,t \in G} \varphi(s)\xi(s^{-1}t) \ovl{\xi(t)} \\
&=\sum_{s,t \in G} \varphi(ts^{-1})  \xi(s)\ovl{\xi(t)}
=\sum_{s,t \in G} \varphi(t^{-1}s)  \ovl{\check{\xi}(t)}\check{\xi}(s).
\end{align*}
Moreover, we have
$$
\norm{\sum_{s \in G} \varphi(s)\lambda_s}_{\L^1(\VN(G))}=\tau\bigg(\sum_{s \in G} \varphi(s)\lambda_s\bigg)
=\sum_{s \in G} \varphi(s)\tau(\lambda_s)
=\varphi(e).
$$
\end{proof}

Now, we can compute the completely bounded minimal output entropy of any radial multiplier which defines a quantum channel. By remark \ref{remark-trace-preserving} and Proposition \ref{Prop-cp}, this assumption means that $\phi(0)=1$ and that the complex function $\phi \co G \to \mathbb{C}$ is definite positive.

\begin{cor}
\label{Cor-ent-cb}
Consider a function $\phi \co G \to \mathbb{C}$. Suppose that the radial multiplier $R_\phi \co \cal{M} \to \cal{M}$ is a quantum channel. Assume that $\sum \phi(s) \lambda_s$ belong to $\L^p(\VN(G))$ and that the Segal entropy $\H_{\tau_{\VN(G)}}(f_\phi)$ of $\sum \phi(s) \lambda_s$ is finite. We have
\begin{equation}
\label{entropy-cb}
\H_{\cb,\min,\tau}(R_\phi)
=\H\big(\sum \phi(s) \lambda_s\big),
\end{equation}
where $\tau$ is the normalized trace defined in \eqref{trace-fer}.
\end{cor}

\begin{proof}
First note that the function $\phi$ is positive definite by Proposition \ref{Prop-cp} and consequently by Lemma \ref{lemma-useful-678}
\begin{align}
\MoveEqLeft
\label{norm-hg}
\norm{\sum \phi(s) \lambda_s}_{\L^1(\VN(G))}
\ov{\eqref{norm=1-789}}{=} 1.
\end{align}
Using Theorem \ref{th-norm-cb-rad} in the second equality and the equality $\norm{f_\phi}_{\L^1(G)}=1$ in the third equality, we obtain
\begin{align*}
\MoveEqLeft
\H_{\cb,\min,\tau}(R_\phi)
\ov{\eqref{Def-intro-Scb-min}}{=} -\frac{1}{\log 2}\frac{\d}{\d p} \big[\norm{R_\phi}_{\cb,\L^1(\cal{M}) \to \L^p(\cal{M})}\big]|_{p=1} \\
&\ov{\eqref{norm-cb}}{=} -\frac{1}{\log 2}\frac{\d}{\d p} \bigg[\norm{\sum \phi(s) \lambda_s}_{\L^p(\VN(G))}\bigg] |_{p=1} 
\ov{\eqref{deriv-norm-p}}{=} \H(f_\phi).
\end{align*}
\end{proof}

\begin{remark} \normalfont
\label{rem-cbmin-est-max}
We warn the reader that the trace used on the algebra $\cal{R}_n$ in the previous result for defining the completely bounded minimal output entropy is the \textit{normalized} one, defined in \eqref{trace-fer}. In particular, we have the inequality $\H_{\cb,\min,\tau}(R_\phi) \leq 0$ by \eqref{cbminless0}. 
\end{remark}


\begin{example} \normalfont
By the equivalence \eqref{carac-entropy-max} and the previous equality \eqref{entropy-cb}, the completely bounded minimal output entropy is maximal, i.e.~$\H_{\cb,\min,\tau}(R_\phi)=0$, if and only if $f_\phi=1$. This is equivalent to $\phi=\delta_0$, i.e.~$\phi(0)=1$ and $\phi(k)=0$ if $k \geq 1$. In this case, the associated quantum channel is the completely noisy channel.
\end{example}

\begin{remark} \normalfont
With \cite[Proposition 2.3]{LoW22}, we can express the entropy $\H(f_\phi)$ with the help of the relative entropy. We obtain the formula
$$
\H_{\cb,\min,\tau}(R_\phi)
=-\H(f_\phi\mu \| \mu),
$$  
where $f_\phi\mu$ is the probability measure on the abelian group $\Omega_n$ defined by the density $f_\phi$.
\end{remark}


%

\subsection{Completely $p$-summing multipliers and entanglement-assisted classical capacity}
\label{summing-radial-multipliers}

\paragraph{Completely $p$-summing maps}
We start by providing some background on the class of completely $p$-summing maps, where $1 \leq p <\infty$.

We begin by recalling the classical concept of $p$-summing operator, which is useful for understanding the notion of completely $p$-summing operator. Following the book \cite[p.~197]{DJT95}, we say that a bounded operator $T \co X \to Y$ is $p$-summing if there exists a constant $C \geq 0$ such that
\begin{equation}
\label{pq-summing}
\bigg( \sum_{j=1}^n \norm{T(x_j)}_Y^p \bigg)^\frac1p 
\leq C \sup_{y \in B_{X^*}} \bigg( \sum_{j=1}^n |\la y,x_j\ra_{X^*,X}|^p \bigg)^\frac1p
\end{equation}
for any integer $n \geq 1$ and any $x_1, \ldots, x_n \in X$. The $p$-summing norm of $T$, defined as the infimum of the constants $C$ as in \eqref{pq-summing}, is denoted by $\norm{T}_{\pi_{p},X \to Y}$.

Recall that the class of $p$-summing operators on Banach spaces is a Banach operator ideal by \cite[p.~128]{DeF93}. In particular, we have the inequality 
$$
\norm{R T S}_{\pi_{p}, W \to Z} \leq \norm{R}_{Y \to Z} \norm{T}_{\pi_{p},X \to Y} \norm{S}_{W \to X},
$$ 
with obvious notations. If $E$ and $F$ are operator spaces, recall that a linear map $T \co E \to F$ is said to be completely $p$-summing \cite[p.~51]{Pis98} if it induces a bounded map $\Id_{S^p} \ot T \co S^p \ot_{\min} E \to S^p(F)$. In this case, the completely $p$-summing norm is defined by
\begin{equation}
\label{def-norrm-cb-p-summing}
\norm{T}_{\pi_{p}^\circ,E \to F}
\ov{\mathrm{def}}{=} \norm{\Id_{S^p} \ot T}_{S^p \ot_{\min} E \to S^p(F)}.
\end{equation}
The class of completely $p$-summing maps satisfies a similar property where the operator norm is replaced by the completely bounded norm as the following easy lemma shows. This result \cite[(5.1) p.~51]{Pis98} is in the spirit of the definition of an operator space mapping ideal of \cite[p.~210]{EfR00}.

\begin{lemma}
\label{Lemma-is-bounded}
Suppose that $1 \leq p <\infty$. Let $E,F,G,H$ be operator spaces. Let $T \co E \to F$ be a completely $p$-summing map. If the linear maps $R \co F \to G$ and $S \co H \to E$ are completely bounded, then the linear map $R T S \co H \to G$ is completely $p$-summing and we have
\begin{equation}
\label{Ideal-ellp}
\norm{R T S}_{\pi_{p}^\circ, H \to G}
\leq \norm{R}_{\cb, F \to G} \norm{T}_{\pi_{p}^\circ, E \to F} \norm{S}_{\cb,H \to E}.
\end{equation}
\end{lemma}

The following is a simple useful observation is a noncommutative analogue of the classical result \cite[Example 2.9 (d) p.~40]{DJT95}, which says that if $\Omega$ is a \textit{finite} measure space then the canonical inclusion $i_p \co \L^\infty(\Omega) \hookrightarrow  \L^p(\Omega)$ is $p$-summing with 
$$
\norm{i_p}_{\pi_p,\L^\infty(\Omega) \to \L^p(\Omega)}=\mu(\Omega)^{\frac{1}{p}}.
$$

\begin{prop}
\label{Prop-inj-finite-avn}
Suppose that $1 \leq p < \infty$. If $\cal{M}$ is an approximately finite-dimensional finite von Neumann algebra equipped with a normal finite faithful trace $\tau$, then the canonical inclusion $i_p \co \cal{M} \hookrightarrow \L^p(\cal{M})$ is completely $p$-summing and
\begin{equation}
\label{norm-ip-bis}
\norm{i_p}_{\pi_{p}^\circ,\cal{M} \to \L^p(\cal{M})}
=(\tau(1))^{\frac{1}{p}}.
\end{equation}
\end{prop}

Let $\cal{M}$ be a von Neumann algebra equipped with a normal finite faithful trace. If $a,b \in \cal{M}$, we denote by $\Mult_{a,b} \co \cal{M} \to \L^p(\cal{M})$, $x \mapsto axb$ the two-sided multiplication map.
The following is a slight variation of \cite[Theorem 2.2 p.~360]{JuP15}, \cite[Theorem 2.4 p.~366]{JuP15} and \cite[Remark 2.2 p.~363]{JuP15}. It provides a nice characterization of the completely $p$-summing norm of linear maps acting on finite-dimensional von Neumann algebras. The proof is similar and we skip the details.
\begin{thm}
\label{thm-factorization-positivity}
Suppose that $1 \leq p < \infty$. Let $\cal{M}$ and $\cal{N}$ be finite-dimensional von Neumann algebras such that $\cal{M}$ is equipped with a faithful trace. Let $T \co \cal{M} \to \cal{N}$ be a linear map. Then, the following assertions are equivalent.
\begin{enumerate}
\item $\norm{T}_{\pi_{p}^\circ,\cal{M} \to \cal{N}} \leq C$.


\item There exist elements $a,b \in \cal{M}$ satisfying $\norm{a}_{\L^{2p}(\cal{M})} \leq 1$, $\norm{b}_{\L^{2p}(\cal{M})} \leq 1$ and a linear map $\tilde{T} \co \L^p(\cal{M}) \to \cal{N}$ such that
\begin{equation}
\label{sans-fin-33}
T
=\tilde{T} \circ \Mult_{a,b}
\quad \text{and} \quad 
\norm{\tilde{T}}_{\cb, \L^p(\cal{M}) \to \cal{N}} 
\leq C.
\end{equation}
\end{enumerate}
Furthermore, $
\norm{T}_{\pi_{p}^\circ,\cal{M} \to \cal{N}}=\inf\big\{C : C \text{ satisfies the previous conditions}\big\}$.
\end{thm}


We will use the two following results.

\begin{lemma}
\label{magic-434}
For any elements $c$ and $d$ belonging to the unit ball of the Banach space $\L^4(\cal{M})$, the linear map
\begin{equation}
\label{Theta-magic}
\L^2(\cal{M}) \to \L^2(\VN(G),\L^2(\cal{M})),\,  
x \mapsto (1 \ot c) \beta(x) (1 \ot d)
\end{equation}
is a (complete) contraction, where the map $\beta \co \cal{M} \to \VN(G) \otvn \cal{M}$ is defined in \eqref{actions-alpha-beta}.
\end{lemma}

\begin{proof}
Observe that by \cite[p.~139]{Pis03}, the operator spaces $\L^2(\cal{M})$ and $\L^2(\VN(G),\L^2(\cal{M}))$ are operator Hilbert spaces. Therefore, by \cite[Proposition 7.2 (iii) p.~127]{Pis03}, to establish complete contractivity, it is sufficient to prove that the map defined in \eqref{Theta-magic} is contractive.

Let $x \in \L^2(\cal{M})$. We will use the $*$-homomorphism $\alpha \co \cal{M} \to \cal{M} \otvn \VN(G)$ of \eqref{actions-alpha-beta} which is an ergodic action. Consequently, by Proposition \ref{Prop-transfer-1}, we have a contraction $\alpha \co \L^1(\cal{M}) \to \L^\infty(\cal{M},\L^1(\VN(G)))$. By interpolation with the compatible contractive map $\alpha \co \L^\infty(\cal{M}) \to \L^\infty(\cal{M}) \otvn \VN(G)$, we see that $\alpha$ induces a contraction $\alpha \co \L^2(\cal{M}) \to \L^\infty(\cal{M},\L^2(\VN(G)))$. By \eqref{LinftyLq-norms} applied with $p=2$, we deduce that
\begin{equation}
\label{}
\bnorm{(c \ot 1)\alpha(x)(d \ot 1)}_{\L^2(\cal{M},\L^2(\VN(G)))}   
\leq \norm{\alpha(x)}_{\L^\infty(\cal{M},\L^2(\VN(G)))}
\leq \norm{x}_{\L^2(\cal{M})}.
\end{equation}
This means that the map
\begin{equation*}
\L^2(\cal{M}) \to \L^2(\cal{M},\L^2(\VN(G))),\,  
x \mapsto (c \ot 1) \alpha(x) (d \ot 1)
\end{equation*}
is a well-defined contraction. Since $\beta=\flip \circ \alpha$, the map defined in \eqref{Theta-magic} is a contraction by Fubini's theorem, described in \eqref{Fubini}.  
\end{proof}

\begin{prop}
Suppose that $1 \leq p \leq \infty$. Assume that $a$ and $b$ are positive elements in the unit ball of the Banach space $\L^{2p}(\cal{M})$. Then the map
\begin{equation}
\label{Theta-bbis}
\Theta \co \L^p(\cal{M}) \to \L^p(\VN(G),\L^p(\cal{M})),\,  
x \mapsto (1 \ot a) \beta(x) (1 \ot b)
\end{equation}
is a well-defined complete contraction, where the $*$-homomorphism $\beta \co \cal{M} \to \VN(G) \otvn \cal{M}$ is defined in \eqref{actions-alpha-beta}.
\end{prop} 

\begin{proof}
Suppose that $2 \leq p \leq \infty$. For any complex number $z$ in the closed strip $\ovl{S} \ov{\mathrm{def}}{=} \{z \in \mathbb{C} : 0 \leq \Re z \leq 1 \}$, we consider the linear map
\begin{equation}
\label{def-Tz}
T_z \co \cal{M} \to \L^1(\VN(G),\L^1(\cal{M})),\quad x \mapsto \big(1 \ot a^{\frac{pz}{2}}\big) \beta(x) \big(1 \ot b^{\frac{pz}{2}}\big).
\end{equation}
Since $a$ and $b$ are positive elements in the unit ball of $\L^{2p}(\cal{M})$, the elements $a^{\frac{p}{2}}$ and $b^{\frac{p}{2}}$ belong to the unit ball of the Banach space $\L^4(\cal{M})$. Note that
\begin{equation}
\label{unitaries}
\bnorm{1 \ot a^{\frac{\i t p}{2}}}_{\VN(G) \otvn \cal{M}}
=\norm{1}_{\VN(G)} \bnorm{a^{\frac{\i t p}{2}}}_{\cal{M}}
=1
\end{equation}
and similarly
\begin{equation}
\label{unitaries-2}
\bnorm{1 \ot b^{\frac{\i t p}{2}}}_{\VN(G) \otvn \cal{M}}
=1.
\end{equation}
For any $t \in \R$ and any $x \in \cal{M}$, we deduce with H\"older's inequality and Lemma \ref{magic-434} that
\begin{align*}
\MoveEqLeft
\norm{T_{1+\i t}(x)}_{\L^2(\VN(G),\L^2(\cal{M}))}            
\ov{\eqref{def-Tz}}{=} \bnorm{\big(1 \ot a^{\frac{p+\i t p}{2}}\big) \beta(x) \big(1 \ot b^{\frac{p+\i t p}{2}}\big)}_{\L^2(\VN(G),\L^2(\cal{M}))} \\
&\ov{\eqref{Holder}}{\leq} \bnorm{1 \ot a^{\frac{\i t p}{2}}}_{\L^\infty(\L^\infty)} \bnorm{\big(1 \ot a^{\frac{p}{2}}\big) \beta(x) \big(1 \ot b^{\frac{p}{2}}\big)}_{\L^2(\VN(G),\L^2(\cal{M}))} \bnorm{1 \ot b^{\frac{\i t p}{2}}}_{\L^\infty(\L^\infty)} \\
&\ov{\eqref{unitaries}\eqref{unitaries-2}}{=} \bnorm{\big(1 \ot a^{\frac{p}{2}}\big) \beta(x) \big(1 \ot b^{\frac{p}{2}}\big)}_{\L^2(\VN(G),\L^2(\cal{M}))}
\ov{\eqref{Theta-magic}}{\leq} \norm{x}_{\L^2(\cal{M})}.
\end{align*}
We infer that the map $T_{1+\i t} \co \L^2(\cal{M}) \to \L^2(\VN(G),\L^2(\cal{M}))$ is contractive. Note that by \cite[p.~139]{Pis03}, the operator space $\L^2(\cal{M})$ is an operator Hilbert space. We conclude that this map is even completely contractive by \cite[Proposition 7.2 (iii) p.~127]{Pis03}.

For any $t \in \R$ and any $x \in \cal{M}$, we have
\begin{align*}
\MoveEqLeft
T_{\i t}(x)         
\ov{\eqref{def-Tz}}{=} \big(1 \ot a^{\frac{\i t p}{2}}\big) \beta(x) \big(1 \ot b^{\frac{\i t p}{2}}\big).
\end{align*} 
Note that the $*$-homomorphism $\beta \co \cal{M} \to \VN(G) \otvn \cal{M}$ is completely contractive by \cite[Proposition 1.2.4 p.~5]{BLM04}. Moreover, it is elementary with the equalities \eqref{unitaries} and \eqref{unitaries-2} to check that the two-sided multiplication map 
$$
\VN(G) \otvn \cal{M} \to \VN(G) \otvn \cal{M},\, y \mapsto (1 \ot a^{\frac{\i t p}{2}})y(1 \ot b^{\frac{\i t p}{2}})
$$ 
is also completely contractive (and even contractively decomposable by \cite[Exercise 12.1 p.~251]{Pis03}).
 As a composition of the completely contracting maps, the operator $T_{\i t} \co \L^\infty(\cal{M}) \to \VN(G) \otvn \cal{M}$ is completely contractive. 

Moreover, for any $x \in \cal{M}$, the function $\ovl{S} \to \L^1(\VN(G),\L^1(\cal{M}))$, $z \mapsto T_z(x)$ is continuous and bounded 
on the closed strip $\ovl{S}$ and analytic on the open strip $S$. 
Observe that by the reiteration theorem \cite[Theorem 4.6.1 p.~101]{BeL76} we have
\begin{equation}
\label{ident5}
\L^p(\cal{M})
=(\L^\infty(\cal{M}), \L^2(\cal{M}))_{\frac{2}{p}}
\end{equation}
and
\begin{equation}
\label{ident4}
\L^p(\VN(G),\L^p(\cal{M}))
=(\VN(G) \otvn \cal{M},\L^2(\VN(G),\L^2(\cal{M})))_{\frac{2}{p}}.
\end{equation}
By Stein's interpolation theorem (Theorem \ref{thm-Stein}), we conclude by taking $z=\frac{2}{p}$ that the linear map defined in \eqref{Theta-bbis} is a contraction. Actually, it is a complete contraction if we combine Stein's interpolation theorem with \eqref{defnormecb} and replacing the interpolation formulas \eqref{ident5} and \eqref{ident4} by the interpolation formulas 
$$
S^p(\L^p(\cal{M}))
=(S^\infty(\L^\infty(\cal{M})), S^2(\L^2(\cal{M})))_{\frac{2}{p}}
$$
and
$$
S^p(\L^p(\VN(G),\L^p(\cal{M})))
=(S^\infty(\VN(G) \otvn \cal{M}),S^2(\L^2(\VN(G),\L^2(\cal{M}))))_{\frac{2}{p}}.
$$

The case $1 \leq p \leq 2$ is similar. 
\end{proof}

\paragraph{Description of the completely $p$-summing norm of multipliers}
Here, we prove that there exists a simple expression for the completely $p$-summing norm of any multiplier $R_\phi \co \cal{M} \to \cal{M}$ viewed as an operator acting on the von Neumann algebra $\cal{M}$. It is equal to the completely bounded norm from $\L^p$ into $\L^\infty$. In the next result, we use the \textit{normalized} trace on the finite-dimensional von Neumann algebra $\cal{M}$.

\begin{thm}
\label{Th-p-sum}
Suppose that $1 \leq p < \infty$. Consider a function $\phi \co G \to \mathbb{C}$. We have
\begin{equation}
\label{norm-p-sum}
\norm{R_\phi}_{\pi_p^\circ,\cal{M} \to \cal{M}}
=\norm{R_\phi}_{\cb, \L^p(\cal{M}) \to \cal{M}}.
\end{equation}
\end{thm}

\begin{proof}
If we consider the canonical inclusion $i_p \co \cal{M} \hookrightarrow \L^p(\cal{M})$ then by Proposition \ref{Prop-inj-finite-avn}, 
we have the equality
\begin{equation}
\label{norm-ip}
\norm{i_p}_{\pi_{p}^\circ,\cal{M} \to \L^p(\cal{M})}
\ov{\eqref{norm-ip-bis}}{=} 1.
\end{equation} 
By Lemma \ref{Lemma-is-bounded}, we obtain by composition the estimate 
\begin{align*}
\MoveEqLeft
\norm{R_\phi}_{\pi_p^\circ, \cal{M} \to \cal{M}}
\ov{\eqref{Ideal-ellp}}{\leq}
\norm{R_\phi}_{\cb, \L^p(\cal{M}) \to \cal{M}} \norm{i_p}_{\pi_{p}^\circ,\cal{M} \to \L^p(\cal{M})} \\  
&\ov{\eqref{norm-ip}}{=} \norm{R_\phi}_{\cb, \L^p(\cal{M}) \to \cal{M}}.
\end{align*}

Now we prove the reverse inequality, which is more involved. Let $\epsi > 0$. 
By Theorem \ref{thm-factorization-positivity}, there exist elements $a,b \in \L^{2p}(\cal{M})$ satisfying $\norm{a}_{\L^{2p}(\cal{M})} \leq 1$ and $\norm{b}_{\L^{2p}(\cal{M})} \leq 1$ and a linear map $\tilde{R}_{\phi} \co \L^p(\cal{M}) \to \cal{M}$ such that 
\begin{equation}
\label{Facto-Fourier-multiplier}
R_{\phi}
\ov{\eqref{sans-fin-33}}{=} \tilde{R}_{\phi} \circ \Mult_{a,b}
\quad \text{and} \quad 
\norm{\tilde{R}_{\phi}}_{\cb,\L^p(\cal{M}) \to \cal{M}} 
\leq \norm{R_\phi}_{\pi_p^\circ, \cal{M} \to \cal{M}} + \epsi,
\end{equation}
where $\Mult_{a,b} \co \cal{M} \to \L^p(\cal{M})$ is the two-sided multiplication map. Replacing the map $\tilde{R}_{\phi}$ by the map $x \mapsto\tilde{R}_{\phi}(u x v^*)$ where $a=u|a|$ and $b^*=v|b^*|$ are the polar decompositions of the elements $a$ and $b^*$, we can suppose that $a$ and $b$ are positive elements. Indeed, firstly for any $x \in \L^p(\cal{M})$ we have 
$$
\tilde{R}_{\phi}\big(u \Mult_{|a|,|b^*|}(x) v^*\big)
=\tilde{R}_{\phi}\big(u |a|\,  x \, |b^*|v^*\big)
=\tilde{R}_{\phi}(a x b)
\ov{\eqref{Facto-Fourier-multiplier}}{=} R_\phi(x).
$$ 
Secondly, by \cite[Lemma 5.1]{Arh24a}, the two-sided multiplication map $\Mult_{u,v} \co \L^p(\cal{M}) \to \L^p(\cal{M})$, $x \mapsto u x v$ is decomposable with decomposable norm 
$$
\norm{\Mult_{u,v}}_{\dec,\L^p(\cal{M}) \to \L^p(\cal{M})} 
\leq \norm{u}_{\L^\infty(\cal{M})} \norm{v}_{\L^\infty(\cal{M})} 
=1,
$$ 
hence completely bounded by \cite[Proposition 3.30 p.~50]{ArK23} with completely bounded norm 
$$
\norm{\Mult_{u,v}}_{\cb,\L^p(\cal{M}) \to \L^p(\cal{M})}
\leq \norm{\Mult_{u,v}}_{\dec,\L^p(\cal{M}) \to \L^p(\cal{M})}
\leq 1
$$
(an elementary proof of this estimate is also possible). Since the linear map $\tilde{R}_{\phi} \co \L^p(\cal{M}) \to \cal{M}$ is completely bounded, we deduce the estimate 
\begin{align}
\MoveEqLeft
\label{Equa-138676}
\norm{\Id_{\L^p(G_n)} \ot \tilde{R}_{\phi}}_{\cb,\L^p(\VN(G),\L^p(\cal{M})) \to \L^p(\VN(G),\cal{M})}            
\ov{\eqref{ine-tensorisation-os}}{\leq} \norm{\tilde{R}_{\phi}}_{\cb,\L^p(\cal{M}) \to \cal{M}} \\
&\ov{\eqref{Facto-Fourier-multiplier}}{\leq}  \norm{R_{\phi}}_{\pi_p^\circ, \cal{M} \to \cal{M}} +\epsi. \nonumber
\end{align}

Let $\E \co \VN(G) \otvn \cal{M} \to \cal{M}$ be the canonical trace preserving  faithful conditional expectation associated with the trace preserving unital injective $*$-homomorphism $\beta \co \cal{M} \to \VN(G) \otvn \cal{M}$ defined in \eqref{actions-alpha-beta}, provided by Proposition \ref{prop-existence-conditional-expectation}. By Proposition \ref{prop-autre-extension}, we have a complete contraction $\beta \co \L^1(\cal{M}) \to \L^{p^*}(\VN(G),\L^1(\cal{M}))$. Using \eqref{Esp-dual}, we see that the adjoint map of this map identifies to the linear map $\E \co \L^p(\VN(G),\cal{M}) \to \cal{M}$, which is therefore a complete contraction. Now, we have
\begin{align*}
\MoveEqLeft
\beta R_{\phi}
\ov{\eqref{Equation-Delta-66}}{=} (\Id \ot R_{\phi}) \beta
\ov{\eqref{Facto-Fourier-multiplier}}{=}(\Id \ot \tilde{R}_{\phi} \Mult_{a,b}) \beta \\
&=(\Id \ot \tilde{R}_{\phi}) (\Id \ot \Mult_{a,b}) \beta
\ov{\eqref{Theta-bbis}}{=} (\Id \ot \tilde{R}_{\phi}) \Theta.           
\end{align*} 
Since $\E \beta=\Id_{\cal{M}}$, we deduce the equality 
$$
R_{\phi}=\E \beta R_{\phi}=\E (\Id \ot \tilde{R}_{\phi}) \Theta.
$$ 
That means that we have the following commutative diagram.
$$
\xymatrix @R=1.5cm @C=2cm{
\L^p(\VN(G),\L^p(\cal{M})) \ar[r]^{\Id_{\L^p(\VN(G))} \ot \tilde{R}_{\phi}}   & \ar[d]^{\E} \L^p(\VN(G),\cal{M}) \\
\L^p(\cal{M}) \ar[r]_{R_\phi} \ar[u]^{\Theta}   & \cal{M} \\
  }
$$
We infer that we have the estimate
\begin{align*}
\MoveEqLeft
\norm{R_{\phi}}_{\cb,\L^p(\cal{M}) \to \cal{M}}
=\norm{\Theta(\Id_{\L^p} \ot \tilde{R}_{\phi})\E}_{\cb, \L^p(\cal{M}) \to \cal{M}} \\
&\leq \norm{\Theta}_{\cb,\L^p \to \L^p(\L^p)} \norm{\Id_{\L^p} \ot \tilde{R}_{\phi}}_{\cb, \L^p(\L^p) \to \L^p(\L^\infty)} \norm{\E}_{\cb,\L^p(\L^\infty) \to \cal{M}} \\
&\ov{\eqref{Equa-138676}}{\leq} \norm{R_\phi}_{\pi_p^\circ, \cal{M} \to \cal{M}}+\epsi.
\end{align*}
Since $\epsi >0$ is arbitrary, we obtain the desired estimate 
$$
\norm{R_{\phi}}_{\cb,\L^p(\cal{M}) \to \cal{M}} 
\leq \norm{R_\phi}_{\pi_p^\circ, \cal{M} \to \cal{M}}.
$$
\end{proof}

If $\cal{M}=\M_n$ then Theorem \ref{Th-p-sum} combined with Theorem \ref{th-norm-cb-rad} allows us to use formula  \cite[Theorem 1.1 p.~355]{JuP15} of Junge and Palazuelos
\begin{equation}
\label{CEA-formula}
\C_{\EA}(T)
=\frac{1}{\log 2} \frac{\d}{\d p}\big[\norm{T^*}_{\pi_{p^*}^\circ, \M_n \to \M_n}\big]|_{p=1},
\end{equation}
which describes the entanglement-assisted classical capacity as a derivative of completely $p$-summing norms of the adjoint map $T^* \co \M_n \to \M_n$ of the quantum channel $T \co S^1_n \to S^1_n$. Here $p^* \ov{\mathrm{def}}{=} \frac{p}{p-1}$. \eqref{CEA-formula} to obtain the exact value of the entanglement-assisted classical capacity of these multipliers. 

\begin{thm}
\label{th-capacity-assisted}
Suppose $\cal{M}=\M_n$. Consider a function $\phi \co G \to \mathbb{C}$. Suppose that the multiplier $R_\phi \co \cal{M} \to \cal{M}$ is a quantum channel. We have
$$
\C_{\EA}(R_\phi)
=-\H(\sum_s \phi(s) \lambda_s), 
$$ 
where $\H_{\tau_{\VN(G)}}(\sum_s \phi(s) \lambda_s)$ is the Segal entropy of $(\sum_s \phi(s) \lambda_s$ defined in \eqref{Segal-entropy} with the normalized trace.
\end{thm}



\section{Examples}
\subsection{Example: projective representations of abelian finite groups}
\label{sec:mixed}

Let $G$ be a locally compact group equipped with a left Haar measure $\mu_G$.  Consider a $n$-dimensional continuous projective unitary representation $u \co G \to \M_n$ of $G$. Combining Example \ref{Ex-classical-action} and Example \ref{rep-action}, we have a right action $\alpha \co \M_n \to \M_n \otvn \L^\infty(G)=\L^\infty(G,\M_n)$ such that
\begin{equation}
\label{}
\alpha(x)(s)
=u(s)xu(s)^*, \quad s \in G, x \in \M_n.
\end{equation}
In this case, if $f \in \L^1(G)$, the map $T_{f,\alpha} \co \M_n \to \M_n$ of \eqref{convol-def-bis} is defined by 
\begin{equation}
\label{channel-commutatif}
T_{f,\alpha}(x)
=\int_G u(s)xu(s)^* f(s) \d \mu_G(s),\quad x \in \M_n.
\end{equation}
These maps are considered for finite groups in \cite[p.~9]{GJL18b} and \cite[p.~85]{GJL18a} under the notation <<$\cal{N}_f$>> and in \cite{CrN13} under the notation <<$\Theta(\mu)$>> where $\mu=f\mu_G$ and only for the right regular representation. Moreover, these maps are particular cases (for only representations) of the slightly more general maps considered in \cite{SmS05} (see also \cite[Lemma 4.4]{Sto80}). The descriptions of \cite[Example 3.2 and Example 3.3]{CrN13} show that the $n$-qubit flip channels and the bit-flip swap channels are particular cases of \eqref{channel-commutatif}. 

For a quantum channel $T \co S^1_n \to S^1_n$ which is <<covariant>> with respect to a compact group $G$ and some irreducible representation of $G$, the entanglement-assisted classical capacity $\C_{\EA}(T)$ is related to the completely bounded minimum output entropy by the formula 
$$
\C_{\EA}(T)
=\log_2 n-\H_{\cb,\min,\tr}(T)
$$ 
of \cite[p.~385]{JuP15}. Our approach allows us to give the exact values $\H_{\cb,\min,\tr}(T)$ and of $\C_{\EA}(T)$ in the case of a channel as in \eqref{channel-commutatif}, defined with an irreducible projective unitary representation of an abelian finite group.

The complex ordinary representations of finite abelian groups are relatively straightforward to understand. For instance, all irreducible uitary representations are one-dimensional. However, this simplicity does not extend to their projective representations, which present more complexity. This problem has been explored by several authors
in \cite{Mor73}, \cite{MST87} and \cite{Hig01}. 
In particular,  Higgs \cite{Hig01} constructed an irreducible representation for each 2-cocycle $\sigma \co G \times G \to \T$ of an elementary finite abelian group $G=(\mathbb{Z}/n\mathbb{Z})^k$ and any integers $k,n \geq 1$. 
However, analogous results for $(\mathbb{Z}/n^r\mathbb{Z})^k$ with $r > 1$ are still unknown. It is known \cite[Proposition 2.2 (3)]{Che15} that all finite-dimensional irreducible projective representations of a finite abelian group $G$ with 2-cocycle $\sigma$ have the same degree and \cite[Theorem 3.5]{Che15} that the degrees of the irreducible projective representations of $G$ divide the order of $G$.


\begin{example} \normalfont
Let $n \geq 2$ be an integer. By \cite[Proposition 1.3 p.~133]{Kar85}, the second cohomology group of the product group $G=\Z/n\Z \times \Z/n\Z$ is $\H^2(\Z/n\Z \times \Z/n\Z,\T)=\Z/n\Z$. Moreover, a direct consequence of \cite[Lemma 1.1 p.~359]{Kar94} and \cite[Lemma 2.23 p.~382]{Kar94} is that its elements may be represented by 
\begin{equation}
\label{sigma-k-cocycle-4}
\sigma_l((j,k),(j',k'))
=\e^{\frac{2\pi\i l}{n}k j'},
\end{equation}
where $0 \leq l \leq n-1$. Some projective unitary representations of this group are described in \cite[p.~120]{Kar85}. For example, consider the matrices 
\begin{equation}
\label{def-X-Z}
X
=X_n 
\ov{\mathrm{def}}{=}
\begin{bmatrix}
\e^{\frac{2\pi \i }{n}} & 0 & 0 & \cdots & 0 \\
0 & \e^{\frac{2\pi \i }{n}} & 0 & \cdots & 0 \\
0 & 0 & \e^{\frac{2\pi \i }{n}} & \cdots & 0 \\
\vdots & \vdots & \vdots & \ddots & \vdots \\
0 & 0 & 0 & \cdots & \e^{\frac{2\pi \i }{n}} \\
\end{bmatrix}
\quad \text{and} \quad
Z
=Z_n
\ov{\mathrm{def}}{=}
\begin{bmatrix}
0 & 1 & 0 & \cdots & 0 \\
0 & 0 & 1 & \cdots & 0 \\
\vdots & \vdots & \vdots & \ddots & \vdots \\
0 & 0 & 0 & \cdots & 1 \\
1 & 0 & 0 & \cdots & 0 \\
\end{bmatrix}.
\end{equation}
in the algebra $\M_n$. These unitary matrices are sometimes referred to as the discrete Weyl operators. Note that the commutation rule $XZ=\e^{\frac{2\pi \i}{n}}ZX$. The map $u \co \Z/n\Z \times \Z/n\Z \to \M_n$ defined by
$$
u(j,k)
\ov{\mathrm{def}}{=} X^j Z^k, \quad 0 \leq j,k \leq d-1
$$
is a projective unitary representation associated to te 2-cocycle $\sigma_1$, i.e.~we have $u(j,k)u(j',k')=\e^{\frac{2\pi\i}{n}k j'}u(j+j',k+k')$. Given a probability distribution $f \co \Z/n\Z \times \Z/n\Z \to \R$, the quantum channel of \eqref{channel-commutatif} admits the Kraus representation defined by
$$
T_f(x)
=\sum_{j,k=0}^{n-1} f(j,k)X^j Z^k x (X^jZ^k)^*, \quad x \in \M_n.
$$
These channels are the Weyl-covariant maps, see \cite[p.~217]{Wat18}.

For any $0 \leq t \leq 1$, the $d$-dimensional depolarizing channel $D_t \co \M_n \to \M_n$, $x \to t x+(1-t)\frac{\I}{n}$ can be written under this form. Indeed, we have by \cite[pp.~120-121]{Hol19}
$$
D_t(x)
=t x +\frac{1-t}{n^2} \sum_{j,k} X^j Z^k \rho(X^jZ^k)^*, \quad x \in \M_n.
$$
In particular, with \eqref{entropy-cb} we recover the formula 
of \cite[(5.5) p.~53]{DJKRB06} and \cite[Section 4.1]{GJL18a}.  
\end{example}

\begin{example} \normalfont
Consider the finite group $G=\Z_2 \times \Z_2 \times \Z_4 \times \Z_4= \la x_1 \ra \times \la x_2\ra \times \la x_3 \ra \times \la x_4 \ra$. By \cite[p.~775]{Hig01}, an irreducible projective representation $u \co G \to \M_8$ of $G$ is defined by
$$
u(s_1)
\ov{\mathrm{def}}{=} \begin{bmatrix}
   -1  & 0  \\
   0  &  1 \\
\end{bmatrix} \ot \I_4, \quad 
u(s_2)
\ov{\mathrm{def}}{=} \begin{bmatrix}
   0  &  1 \\
   1  &  0 \\
\end{bmatrix} \ot \I_4
$$
and
$$
u(s_3) \ov{\mathrm{def}}{=} \I_2 \ot \begin{bmatrix}
  \i   & 0 & 0 & 0 \\
    0 & -1 & 0 & 0 \\
    0 & 0 & -\i & 0 \\
    0 & 0 & 0 & 1 \\
\end{bmatrix}, \quad u(s_4) \ov{\mathrm{def}}{=} \I_2 \ot \begin{bmatrix}
    0 & 0 & 0 & 1 \\
   1  & 0 & 0 & 0 \\
   0  & 1 & 0 & 0 \\
    0 & 0 & 1 & 0 \\
\end{bmatrix}.
$$
It is associated to the the 2-cocycle $\sigma \co G \times G\to \T$ defined by $\sigma(s_i,s_j)
=\begin{cases}
-1&\text{ if } (i,j)=(2,1)\\
-\i & \text{ if } (i,j)=(4,3) \\
1& \text{ otherwise}
\end{cases}$.
See \cite[p.~359]{Kar94} for the other values. We refer to \cite{Hig01} for more general examples.
\end{example}

\subsection{Example: representations of quantum groups}
\label{Sec-Transference-Junge-channels}

Her, we connect our convolution operators with the quantum channels investigated in the paper \cite{GJL18b}.

If $u \in \M_{n} \ot \L^\infty(\QG)$ is a unitary representation of a compact quantum group $\QG$ (see \eqref{Def-rep}), then it is folklore \cite[(2.10)]{Wan99} \cite[Lemma 2.1]{Izu02} that the map $\alpha \co \M_{n} \to \M_{n} \ot \L^\infty(\QG)$ defined by
\begin{equation}
\label{canonical-implementation}
\alpha(x)
=u(x \ot 1)u^*, \quad x \in \M_{n}.
\end{equation}
is a right action of $\QG$ on the matrix algbra $\M_{n}$ in the sense of \eqref{Def-corep}. It is easy to check that $\alpha$ is ergodic if and only if $u$ is irreducible. The next proposition allows us to recover the quantum channel denoted by <<$N_f$>> in \cite[(9) p.~10]{GJL18b}. Here, we identify $\L^1(\QG)$ with $\L^\infty(\QG)_*$.

\begin{prop}
Let $\QG$ be finite quantum group and $u$ be a unitary representation. Let $f \in \L^1(\QG)$ be a state. The map $T_{f,\alpha} \co \M_{n} \to \M_{n}$ is given by 
\begin{equation}
\label{Quantum-group-channel}
T_{f,\alpha}(x)
=(\Id \ot \tau)\big(u(x \ot f)u^*\big), \quad x \in \M_{n}.
\end{equation}
\end{prop}

\begin{proof}
By finite-dimensionality, we can write $u=\sum_{i=1}^N x_i \ot a_i$ where $x_i \in \M_{n}$ and $a_i \in \L^\infty(\QG)$. For any $x \in \M_{n}$, we have using the preservation of the trace $\tau$ by $R$ in the last equality
\begin{align*}
\MoveEqLeft
T_{\tau(fR(\cdot)),\alpha}(x)            
\ov{\eqref{convol-def-left}}{=} ( (\Id \ot \tau(fR(\cdot))) \circ \alpha(x)
\ov{\eqref{canonical-implementation}}{=} ((\Id \ot \tau(fR(\cdot)))(u(x \ot 1)u^*)\\
&=((\Id \ot \tau(fR(\cdot)))\Bigg(\bigg(\sum_{i=1}^N x_i \ot a_i\bigg)(x \ot 1)\bigg(\sum_{j=1}^N x_j \ot a_j\bigg)^*\Bigg) \\
&=\sum_{i,j=1}^N ((\Id \ot \tau(fR(\cdot)))\big(x_i x x_j^* \ot a_i^*a_j)
=\sum_{i,j=1}^N (\Id \ot \tau(fR(\cdot)))\big(x_i x x_j^* \ot a_i^*a_j) \\
&=\sum_{i,j=1}^N x_i x x_j^* \ot \tau\big(fR(a_i^*a_j)\big)
=\sum_{i,j=1}^N x_i x x_j^* \ot \tau\big(f R(a_j)R(a_i)^*\big)\\
&=\sum_{i,j=1}^N x_i x x_j^* \ot \tau\big(R(a_i)^*f R(a_j)\big)
=(\Id \ot \tau(R(\cdot))\Bigg(\bigg(\sum_{i=1}^N x_i \ot a_i\bigg)(x \ot f)\bigg(\sum_{j=1}^N x_j \ot a_j\bigg)^*\Bigg)\\
&=(\Id \ot \tau(R(\cdot)))\big(u(x \ot f)u^*\big)
=(\Id \ot \tau)\big(u(x \ot f)u^*\big).
\end{align*} 
\end{proof}

\begin{remark} \normalfont
\label{Rem-classical-capacity-Tfalpha}
In the case of a unitary representation (i.e.~non projective) $u$, we recover with \eqref{Quantum-group-channel} the channel defined by \eqref{channel-commutatif}. If $k \geq 1$ is an integer then the tensor product $(T_{f,\alpha})^{\ot k} \co \M_{n}^{\ot k} \to \M_{n}^{\ot k}$ identifies to the channel $y \mapsto (\Id \ot \tau^{\ot k})\big(u^{\ot k}(y \ot f^{\ot k})(u^{\ot k})^*\big)$ associated with the product group $G^k=G \times \cdots \times G$. Moreover, $u^{\ot k}$ identifies to the outer representation defined in \cite[Definition 1.43]{KaT13}. If $G$ is compact and if $u$ is irreducible, it is well-known that this representation is irreducible. So we obtain
\begin{align*}
\MoveEqLeft
\C(T_{f,\alpha})            
\ov{\eqref{Classical-capacity}}{=} \lim_{k \to +\infty} \frac{\chi(T_{f,\alpha}^{\ot k})}{k} 
=\lim_{k \to +\infty} \frac{\chi(T_{f^{\ot k},\alpha^{\ot k}})}{k} 
\ov{\eqref{belles-estimations}}{\leq} \lim_{k \to +\infty} \frac{-\H(f^{\ot k})}{k} 
\ov{\eqref{entropy-tensor-product}}{=} -\H(f).
\end{align*} 
For the quantum group case, we could use a similar trick. 
\end{remark}

\subsection{Example: twisted group von Neumann algebras}

Now, consider the case of a classical locally compact group $G$ equipped with a measurable $2$-cocycle $\sigma$ and a left Haar measure. 
Consider the $\sigma$-twisted left regular representation of $G$ on $\lambda_\sigma\co G \to \B(\L^2(G))$ defined by
$$
(\lambda_{\sigma,s}f)(t)
\ov{\mathrm{def}}{=} \sigma(s,s^{-1}t)f(s^{-1}t), \quad f \in \L^2(G), s,t \in G.
$$
It satisfies
\begin{equation}
\label{product-adjoint-twisted}
\lambda_{\sigma,s} \lambda_{\sigma,t} 
= \sigma(s,t) \lambda_{\sigma,st}, 
\qquad 
\big(\lambda_{\sigma,s}\big)^* 
= \ovl{\sigma(s,s^{-1})} \lambda_{\sigma,s^{-1}}, \quad s,t \in G.	
\end{equation}
We define the twisted group von Neumann algebra $\VN(G,\sigma)$ as the von Neumann subalgebra of $\B(\L^2(G))$ generated by the $*$-algebra
$$
\C(G,\sigma) 
\ov{\mathrm{def}}{=} \mathrm{span}\big\{\lambda_{\sigma,s} \ : \ s \in G\big\}.
$$

\begin{example} \normalfont
If $\sigma=1$, we recover the left regular representation $\lambda \co G \to \B(\ell^2_G)$ and the von Neumann algebra $\VN(G)$ of the locally compact group $G$. 
\end{example}

If the group $G$ is discrete then the von Neumann algebra $\VN(G,\sigma)$ is a finite algebra. If $(\epsi_t)_{t \in G}$ is the canonical basis of the complex Hilbert space $\ell^2_G$ then a normalized normal finite faithful trace is given by $\tau_{G,\sigma}(x)=\big\langle\epsi_{e},x(\epsi_{e})\big\rangle_{\ell^2_G}$, where $x \in \VN(G,\sigma)$. In particular, we have $\tau_{G,\sigma}(\lambda_{\sigma,s}) = \delta_{s,e}$ for any $s \in G$. Moreover, we have
\begin{equation}
\label{}
\tau_{G,\sigma}\big(\lambda_{\sigma,s}\lambda_{\sigma,t}\big) 
= \sigma(s,t) \delta_{s,t^{-1}},\quad s,t \in G.
\end{equation}
We refer to \cite{Kle62} and \cite{BeC09} and references therein for more information for this particular case. 

We have the following result, which is proved in \cite[Proposition 4.3 p.~168]{KaS82} for a discrete group $G$.  

\begin{prop}
Let $G$ be a second countable locally compact group or a discrete group. The map $\alpha_{\sigma} \co \VN(G,\sigma) \to \VN(G,\sigma) \otvn \VN(G)$, $\lambda_{\sigma,s} \mapsto \lambda_{\sigma,s} \ot \lambda_s$ defines a well-defined co-$G$-system.
\end{prop}

\begin{proof}
As observed\footnote{\thefootnote. The coaction is a $\hat{G}$-Galois object, hence ergodic.} in \cite[p.~8]{BGNT21}, the map $\alpha_{\sigma}$ is ergodic. For any $s \in G$ and any $f \in \mathrm{A}(G)$, we have
$$
(\Id \ot \la f, \cdot \ra_{\mathrm{A}(G),\VN(G)}) \circ \alpha_{\sigma}(\lambda_{\sigma,s})
=(\Id \ot \la f, \cdot \ra_{\mathrm{A}(G),\VN(G)}) (\lambda_{\sigma,s} \ot \lambda_s)
=f(s)\lambda_{\sigma,s}.
$$
We deduce that $s$ belongs to $\Sp(\alpha_{\sigma})$. Hence $\alpha_{\sigma}$ is a co-$G$-system.
\end{proof}

In this case the operator $T_{\tau(f \cdot),\alpha_{\sigma}} \co \VN(G,\sigma) \to \VN(G,\sigma)$ defined in \eqref{convol-def-bis} identifies to a Fourier multiplier, i.e. $T_{\tau(f \cdot),\alpha_{\sigma}}(\lambda_{\sigma,s})=f(s)\lambda_{\sigma,s}$ for any $s \in G$.

\begin{remark} \normalfont
It is known \cite[Proposition 2.14 p.~309]{Oml14} that the von Neumann algebra $\VN(G,\sigma)$ is injective if and only if the discrete group $G$ is amenable.
\end{remark}

Let $\sigma \co G \times G \to \T$ be a two-cocycle of a countable discrete group $G$. An element $r$ in $G$ is called $\sigma$-regular if for any element $s \in G$ commuting with $r$ we have $\sigma(r,s) = \sigma(s,r)$. If $r$ is $\sigma$-regular, then by \cite[Lemma 3 p.~559]{Kle62}  every conjugate $t r t^{-1}$ (where $t \in G$) of $r$ is also $\sigma$-regular. We say that a conjugacy class of $G$ is $\sigma$-regular if it contains a $\sigma$-regular element. Then a particular case of \cite[Lemma 4 p.~560]{Kle62} says that the von Neumann algebra $\VN(G,\sigma)$ is a factor if and only if every non-trivial $\sigma$-regular conjugacy class of $G$ is infinite. In this case, we say that $(G,\sigma)$ satisfies the condition of Kleppner. In the case where the 2-cocycle $\sigma$ is trivial, we recover the classical criterion \cite[Proposition 7.9 p.~367]{Tak02} (or \cite[Theorem 3.2.1 (ii) p.~20]{SiS08}) on group von Neumann algebras.

\begin{example} \normalfont 
Let $n \geq 2$ be an integer. Consider the product group $G=\Z/n\Z \times \Z/n\Z$. We have seen in \eqref{sigma-k-cocycle-4} that the 2-cocycles may be represented by 
\begin{equation*}
\label{}
\sigma_k((a_1,a_2),(b_1,b_2))
=\e^{\frac{2\pi\i k}{n}a_2 b_1},
\end{equation*}
where $0 \leq k \leq n-1$. An element $(a_1,a_2)$ in $\Z/n\Z \times \Z/n\Z$ is $\sigma_k$-regular if and only if $k a_1$ and $k a_2$ belong to $n\Z$. Hence $(\Z/n\Z \times \Z/n\Z,\sigma_k)$ satisfies the condition of Kleppner if and only if $k$ and $n$ are relatively prime. In this case, we have a $*$-isomorphism
\begin{equation}
\label{}
\VN(\Z/n\Z \times \Z/n\Z,\sigma_k)
= \M_n.
\end{equation}
\end{example}

\begin{example}[integer Heisenberg group] \normalfont
Let $G$ be the integer Heisenberg group, which is isomorphic to a semi-direct product $\Z^2 \rtimes_\gamma \Z$, where the homomorphism $\gamma_q \co \Z^2 \to \Z$ is given
$$
\gamma_q(m,n) 
= (m + n q, n), \quad m,n,q \in \Z.
$$
The multiplication on $G$ in this case is given by
$$
(m_1,n_1,q_1) \cdot (m_2,n_2,q_2) 
=(m+m_2+q n_2,n+n_2,q+q_2),\quad
(m_1,n_1,q_1), (m_2,n_2,q_2) \in G.
$$
The group can be represented as the subgroup of $\mathrm{SL}(3,\Z)$ of matrices $\begin{bmatrix}
   1  & q & m  \\
   0  & 1 & n	\\
   0  & 0 & 1   \\
\end{bmatrix}$ with $m,n,q \in \Z$. We have an isomorphism $\H^2(G,\T)=\T^2$. More precisely, by \cite[Proposition 1.1 p.~44]{Pac87} (see also \cite[Example 3.6 p.~356]{Pac08}), a complete parametrization of equivalences classes of 2-cocycles of $G$ is provided by
\begin{equation}
\label{}
\sigma_{(\lambda,\mu)}((m_1,n_1,q_1),(m_2,n_2,q_2))
=\lambda^{n_2m_1+q_1\frac{n_2(n_2-1)}{2}} \mu^{q_1 m_2+n_2\frac{q_1(q_1-1)}{2}},
\end{equation}
where $\lambda,\mu \in \T$. We refer to \cite{Oml15} for the more general case of the free nilpotent group $G(n)$ of class 2 and rank $n$. 
Note that the group $G(2)$ is isomorphic to the discrete Heisenberg group.  We also refer to \cite{LeP95} for a family of generalized Heisenberg groups.
\end{example}

\subsection{Example: noncommutative tori}

Let $d \geq 1$ be an integer. It is know \cite[Theorem 4.10 p.~210]{Bab70} \cite[Theorem 3.2 p.~283]{Bac70} (see also \cite[Example 4.10, p.~251]{OPT80}) that the second cohomology group of $\Z^d$ is given by $\H^2(\Z^d,\T)=\T^{\frac{d(d-1)}{2}}$. In particular, we have $\H^2(\Z,\T)=\{1\}$ and $\H^2(\Z^2,\T)=\T$. Moreover, each cohomology class may be represented by the following cocycle
\begin{equation}
\label{cocycle-Zn}
\sigma_\theta(m,n)
=\e^{2\pi\mathrm{i} \langle m, \theta n\rangle}
=\e^{2\pi \i \sum_{1 \leq i < j \leq n} m_i \theta_{ij} n_j}, \quad m,n \in \Z^d.
\end{equation}
for some real upper triangular matrix $\theta \in \M_d(\R)$ depending on $\frac{d(d-1)}{2}$ real parameters $\theta_{ij} \in [0,1[$. The resulting algebras $\L^\infty(\T^d_{\theta}) \ov{\mathrm{def}}{=} \VN(\Z^d,\sigma_\theta)$ are the so-called $d$-dimensional noncommutative tori. We refer to the seminal paper \cite{CXY13} for a study of harmonic analysis on this algebra. The $d$-dimensional noncommutative torus $\L^\infty(\T_{\theta}^d)$ is the von Neumann subalgebra of $\B(\ell^2_{\Z^d})$ generated by the $*$-algebra $
\mathcal{P}_{\theta}
\ov{\mathrm{def}}{=}\mathrm{span} \big\{ U^m \ : \ m \in \Z^d \big\}$. Recall that for any $m,n \in \Z^d$ we have
\begin{equation}
\label{product-adjoint-twisted}
U_m U_n 
= \sigma_\theta(m,n) U_{m+n}
\quad \text{and} \quad 
\big(U_m \big)^* 
= \ovl{\sigma_\theta(m,-m)} U_{-m}.	
\end{equation}
It is known \cite[Example 1.2 p.~303]{Oml14} that the pair $(\Z^2,\sigma_\theta)$ satisfies the condition of Kleppner if and only if $\theta_{12}$ is irrational. Finally, recall that we have a $*$-isomorphism $\VN(\Z)=\L^\infty(\T)$.

In this context, the $*$-homomorphism $\alpha \co \L^\infty(\T) \to \L^\infty(\T) \otvn \L^\infty(\T_\theta)$ of \eqref{actions-alpha-beta} identifies to the map $x \mapsto \tilde{x}$ of \cite[Corollary 2.2]{CXY13}. Actually, this map is a particular case of the twisted coproduct described in \cite[(4.1.5) p.~61]{ArK23}. Moreover, the map $\eta \co \L^\infty(\T) \to \L^\infty(\T^d_{\theta}) \otvn \L^\infty(\T^d_{\theta})^\op$ of \eqref{autre-eta} identifies to the map $\pi_\theta$ introduced in the proof of \cite[Theorem 7.1]{Arh24}.


\subsection{Example: noncommutative solenoids}

Consider an integer $N \geq 2$ and the discrete abelian group 
\begin{equation}
\label{NadicRationals}
\Z[\tfrac{1}{N}]
\ov{\mathrm{def}}{=} \left\{ \frac{q}{N^k} \in \mathbb{Q} : q \in \Z, k \in \N \right\}
\end{equation}
of $N$-adic rationals. 
By \cite[Theorem 10.7]{JHDS25}, its Pontryagin dual $\widehat{\Z[\tfrac{1}{N}]}$ identifies to the $N$-adic solenoid group
\[
\mathbb{S}_N 
= \left\{ (z_n)_{n \geq 0} \in \T^{\N} : \text{ for all integer $n$ we have } z_{n+1}^N = z_n \right\},
\]
endowed with the induced topology by the compact space $\T^{\N}$. The dual pairing between the groups $\Z[\tfrac{1}{N}]$ and $\mathbb{S}_N$ is given by
\[
\big\la \tfrac{q}{N^k}, (z_n)_{n \geq 0} \big\ra_{\Z[\tfrac{1}{N}],\mathbb{S}_N}
= z_k^q, \quad \tfrac{q}{N^k} \in \Z[\tfrac{1}{N}], (z_n)_{n \in \N} \in \mathbb{S}_N.
\]
The group $\mathbb{S}_N$ is compact, connected and indecomposable. We refer to \cite{JHDS25} for a readable account on this group. 


By \cite[Theorem 2.3 p.~161]{LaP18} and \cite[Theorem-Definition 1.2]{LaP17}, the second cohomology group $\H^2(\Z[\frac{1}{N}] \times \Z[\frac{1}{N}],\T)$ of the product group $\Z[\frac{1}{N}] \times \Z[\frac{1}{N}]$ is isomorphic to the $N$-adic solenoid group $\mathbb{S}_N$. 

More precisely, each 2-cocycle is cohomologous to a 2-cocycle $\sigma_\alpha \co \Z[\frac{1}{N}] \times \Z[\frac{1}{N}] \to \T$ defined by 
\begin{equation*}
\sigma_\alpha\left(\left(\frac{q_1}{N^{k_1}},\frac{q_2}{N^{k_2}}\right),\left(\frac{q_3}{N^{k_3}},\frac{q_4}{N^{k_4}}\right)\right) 
\ov{\mathrm{def}}{=} z_{k_1+k_4}^{q_1q_4}.
\end{equation*}
for a unique $z=(z_n)_{n \geq 0}$ in $\mathbb{S}_N$. Following \cite[Definition 3.1 p.~167]{LaP18} at the von Neumann level, we define the von Neumann algebra $\L^\infty(\mathbb{S}_{N,\alpha}) \ov{\mathrm{def}}{=} \VN(\Z[\tfrac{1}{N}] \times \Z[\tfrac{1}{N}],\sigma_\alpha)$. We say that it is the algebra of the noncommutative solenoid. We refer to \cite{LaP13}, \cite{LaP17}, \cite{LaP18} and \cite{FLLP24} for more information. 

\subsection{Example: Weyl systems}
\label{sec-Weyl}


\paragraph{Self-dualities of locally compact abelian groups} Let $G$ be a locally compact abelian group $G$. Following \cite[p.~2430]{PSV10}, we say that a group isomorphism $\nabla \co G \to \hat{G}$ is a self-duality. If in addition we have $\nabla(s)(s) = 0$ for all $s \in G$, then $\nabla$ is called a symplectic self-duality. Now, we describe the basic example.

\begin{example} \normalfont
\label{standard-duality}
For any locally compact abelian group $H$, consider the product group $G=\ov{\mathrm{def}}{=}H \times \hat{H}$. Then the map $\nabla \co G \to \hat{G}$, defined by
\begin{equation*}
[\nabla(s,\chi)](t,\lambda)
\ov{\mathrm{def}}{=} \chi(t) - \lambda(s), \quad s,t \in H, \chi,\lambda \in \hat{H},
\end{equation*}
is a symplectic self-duality. It is called the standard symplectic self-duality of $G$. 
\end{example}
We say that a symplectic self-duality $\nabla$ is a standard pair (or that $\nabla$ is standard) if it is isomorphic to some standard symplectic self-duality as in Example \eqref{standard-duality}. It is known \cite[Theorem 11.2 p.~2451]{PSV10} that there exist a locally compact abelian group with symplectic self-duality not isomorphic to $H \times \widehat{H}$ for any locally compact abelian group $H$.

\paragraph{Bicharacters}

Let $G$ be an abelian locally compact group. Following \cite[p.~12]{Kle65}, we say that a continuous map $b \co G \times G \to \T$ is a bicharacter if
\begin{equation}
\label{bicarac}
b(s_1s_2,t)
=b(s_1,t)b(s_2,t) 
\quad \text{and} \quad
b(t,s_1s_2)
=b(t,s_1)b(t,s_2), \quad s_1,s_2,t \in G.
\end{equation}
The set of bicharacters is a group under pointwise multiplication. For each bicharacter $b$ there exists a unique continuous homomorphism $\Phi \co G \to \hat{G}$ such that
$$
b(s,t) 
=\la \Phi(s),t  \ra_{\hat{G},G}, \quad s,t \in G.
$$
We say that the bicharacter $b$ is anti-symmetric if $b(s,t)=b(t,s)^{-1}$ for any $s,t \in G$. If in addition $\Phi$ is an isomorphism, $b$ is said to be symplectic.

\begin{example} \normalfont
Any bicharacter is a 2-cocycle. Moreover, if $\sigma$ is a 2-cocycle then the map $b_\sigma \co G \times G \to \T$, $(s,t)\mapsto \sigma(s,t) \sigma(t,s)^{-1}$ is a bicharacter. 
\end{example} 

\paragraph{Heisenberg 2-cocycles} Let $G$ be a locally compact abelian group and $\sigma$ be a 2-cocycle of $G$. Then following \cite[p.~537]{DiV04} we say that $(G,\sigma)$ is a Heisenberg group and that $\sigma$ is a Heisenberg 2-cocycle if the associated bicharacter $\tilde{\sigma}$ is a symplectic bicharacter of $G$. If $\sigma$ is a Heisenberg 2-cocycle for $G$ then by \cite[Theorem 2 p.~537]{DiV04} (see also \cite{Var08}) there exists a unique (up to unitary equivalence) measurable irreducible unitary projective representation $\W \co G \to \B(\cal{H})$ with respect to the 2-cocycle $\sigma$ on some complex Hilbert space $\cal{H}$. Moreover, every $\sigma$-projective unitary representation of $G$ is a direct sum of copies of $\W$. We refer to \cite[Section 3.5]{DiV04} for a concrete model of $\cal{H}$ and $\W$. We call $\W$ and $\W(s)$, $s \in G$, the Weyl representation and the Weyl operators following the standard terminology. Note that the Weyl operators satisfy the canonical commutation relations (CCR)
\begin{equation*} 
\label{eq-Weyl-CCR}
\W(s) \W(t) 
= b_\sigma(s,t)\W(t)\W(s),\quad  s , t \in G.
\end{equation*}
We refer also to the paper \cite{BaK73} for related information.

\begin{example} \normalfont
Consider the group $G = H \times \widehat{H}$ for some locally compact abelian group $H$. A such group admits a canonical 2-cocycle, $\sigma_{\mathrm{can}} \co H \times \widehat{H} \to \T$ given by
\begin{equation}
\label{eq-can-2-cocycle}
\sigma_{\mathrm{can}}((s,\chi), (s', \chi')) 
\ov{\mathrm{def}}{=} \chi(s'),\quad s,s'\in H, \chi, \chi' \in \hat{H}.
\end{equation}
It is straightforward to see that $\sigma_{\mathrm{can}}$ is a Heisenberg multiplier. 
In this case, we have a simple description of the unique irreducible projective $\sigma_{\mathrm{can}}$-representation $\W = \W_{\mathrm{can}}$. We first define the translation operator $T_s \co \L^2(H) \to \L^2(H)$ and the modulation operator $\Mod_\chi \co \L^2(H) \to \L^2(H)$ for $s \in H$ and $\chi \in \hat{H}$ acting on the complex Hilbert space $\cal{H} \ov{\mathrm{def}}{=} \L^2(H)$ by
\begin{equation}
(T_s f)(t) \ov{\mathrm{def}}{=} f(t-s), \quad
(\Mod_\chi f)(t) 
\ov{\mathrm{def}}{=} \chi(t)f(t), \quad f \in \L^2(H),\quad t \in H.
\end{equation}
Then the projective representation $\W \co G \to \B(\L^2(H))$ is given by
\begin{equation*}
\W(s,\chi) 
\ov{\mathrm{def}}{=} T_s \Mod_\chi, \quad (s,\chi) \in G.
\end{equation*}
\end{example}

\begin{example}\normalfont 
\label{ex:quJit}
If $n \geq 3$ is an odd integer then consider the group $G=(\Z/n\Z)^d \times \widehat{(\Z/n\Z)^d}$. 
We have the $\widehat{(\Z/n\Z)^d} \cong (\Z/n\Z)^d$ with duality bracket
\begin{equation}
\label{e:char1}
\la y,x \ra
= \e^{\frac{2\pi \i}{n} \la x, y \ra}, \quad x,y \in (\Z/n\Z)^d.
\end{equation}
Under the canonical isometric identification $     
\ell^2((\Z/n\Z)^d)
=\ell^2(\Z_n) \ot \cdots \ot \ell^2(\Z_n)$, 
the corresponding multiplication operators $M_y \ov{\mathrm{def}}{=} M_{\gamma_y}$ satisfy
\begin{equation}
M_y
= X^{y_1} \ot \cdots \ot X^{y_n}, \quad y=(y_1,\ldots,y_d) \in (\Z/n\Z)^d,
\end{equation}
where the matrix $X$ is defined in \eqref{def-X-Z}. The translation operators are given by
\begin{equation}
T_x
= Z^{x_1} \ot \cdots \ot Z^{x_n}, \quad x=(x_1,\ldots,x_d) \in (\Z/n\Z)^d,    
\end{equation}
where the matrix $Z$ is defined in \eqref{def-X-Z}. In this case, the Weyl representation $\W \co (\Z/n\Z)^d \times (\Z/n\Z)^d \to \B((\ell^2_n)^{\ot d})$ is
\begin{equation}
\W(x,y)
= \e^{\frac{(n+1)\pi \i}{n}\la x,y\ra} Z^{x_1}X^{y_1} \ot \cdots \ot Z^{x_n}X^{y_n}, \quad x,y \in (\Z/n\Z)^d.
\end{equation}
We refer to \cite{Wer16} for more information on these systems, called Qudit systems.
\end{example}

\subsection{Example: spin systems} 
\label{sec-fermions}

Let $n \geq 1$ be an integer. We consider the finite abelian group $G=(\Z/2\Z)^n \times \widehat{(\Z/2\Z)^n} \cong (\Z/2\Z)^n \times (\Z/2\Z)^n=(\Z/2\Z)^{2n}$. 
We consider the matrix $A={ 
\begin{bmatrix} 
0 &\cdots&\cdots&\cdots&0\\ 
1& 0&&\\1&1&0&&\vdots\\ 
\vdots& \vdots& \ddots & \ddots & \\
1 & 1 & \cdots & 1 & 0 
\end{bmatrix}}$ in $\M_{2n}(\Z/2\Z)$. Following \cite{BCLPY22}, we consider the 2-cocycle $\sigma_{\mathrm{spin}} \co G \times G \to \mathbb{C}$ defined by
\begin{equation}
\label{eq-fer-2-cocycle}
\sigma_{\mathrm{spin}}(a,b) 
\ov{\mathrm{def}}{=}  (-1)^{a^T A b},\quad a,b \in G.
\end{equation}
It is easy to check that $\sigma_{\mathrm{spin}}$ is a Heisenberg 2-cocycle. Now, we recall the standard matrix representations of spin factors. Let
\[
\sigma_1
\ov{\mathrm{def}}{=} \begin{bmatrix}
   1  & 0  \\
   0  & -1  \\
\end{bmatrix}
, \quad
\sigma_2 \ov{\mathrm{def}}{=} \begin{bmatrix} 
0&1\\ 
1&0
\end{bmatrix}
, \quad
\sigma_3
\ov{\mathrm{def}}{=} \begin{bmatrix} 
0&\i\\ 
-\i&0
\end{bmatrix}
\]
be the Pauli matrices. Define the matrices
$s_1 \ov{\mathrm{def}}{=} \sigma_1 \ot \I_2^{\ot n-1},  
s_2 \ov{\mathrm{def}}{=} \sigma_2 \ot \I_2^{\ot n-1},s_3 \ov{\mathrm{def}}{=} \sigma_3 \ot \sigma_1 \ot \I_2^{\ot n-2}, 
s_4 \ov{\mathrm{def}}{=} 
\sigma_3 \ot \sigma_2\ot \I_2^{\ot n-2},\ldots,
s_{2i-1} \ov{\mathrm{def}}{=} \sigma_3^{\ot i-1} \ot \sigma_1 \ot \I_2^{\ot n-i}, 
s_{2i} \ov{\mathrm{def}}{=} 
\sigma_3^{\ot i-1} \ot \sigma_2 \ot \I_2^{\ot n-i},
\ldots,
s_{2n-1} \ov{\mathrm{def}}{=} \sigma_3^{\ot n-1} \ot \sigma_1,
s_{2n} \ov{\mathrm{def}}{=} \sigma_3^{\ot n-1 } \ot \sigma_2$ in the algebra $\M_{2^n}$.  Note that it is known that $\{s_1,\ldots,s_k\}$ is a spin system, i.e.~consists of symmetries $\not=\pm \Id$ such that $s_is_j=-s_is_j$ if $i \not= j$. If $k \in \{2n-1,2n\}$ with $k \geq 2$, the real linear span of the set $\{1,s_1,\ldots,s_k\}$ is a $\JW$-factor which is isomorphic to a <<spin factor>> by \cite[pp.~140-141]{HOS84} and \cite[pp.~103-104]{AlS03}. 

The unique irreducible projective unitary representation $\W_{\mathrm{spin}} \co G \to \M_{2^n}$ associated to the 2-cocycle $\sigma_{\mathrm{spin}}$ is given by
\begin{equation*}
\label{eq-W-eps}
\W_{\mathrm{spin}}(a) 
\ov{\mathrm{def}}{=} s^{a_1}_1 \cdots s^{a_{2n}}_{2n},\quad  a  = (a_1,\ldots,a_{2n}) \in G.    
\end{equation*}

\subsection{Example: hardcore bosons}
\label{subsubsec-hard-core-boson}

Let $n \geq 1$ be an integer. We consider again the finite abelian group $=(\Z/2\Z)^n \times \widehat{(\Z/2\Z)^n} \cong (\Z/2\Z)^n \times (\Z/2\Z)^n=(\Z/2\Z)^{2n}$. Here, we consider the canonical 2-cocycle $\sigma_{\can}$ on $G$  introduced in \eqref{eq-can-2-cocycle}. If $X$, $Z$ are the standard $2 \times 2$ Pauli matrices consider the matrices $h_k$, where $1 \leq k \leq 2n$, defined by
\begin{align*}
h_{2j-1} 
\ov{\mathrm{def}}{=} \I \ot \cdots \ot \I \ot X \ot \I \ot \cdots \ot \I, \quad
h_{2j}
\ov{\mathrm{def}}{=} \I \ot \cdots \ot \I \ot Z \ot \I \ot \cdots \ot \I,
\end{align*}
where $X$ and $Z$ appear at the $j$-th tensor component for $1 \leq j \leq n$. These matrices are selfadjoint satisfy
$$
h_k h_l 
= - h_l h_k,\quad \text{if } (k,l) = (2j-1,2j)
\quad \text{or}\quad (2j,2j-1),\quad 1 \leq j \leq n
$$
and $h_kh_l = h_lh_k$ for other choices of $(k,l)$. The associated quantum system corresponds to <<hardcore bosons>> of $n$ degrees of freedom and we refer to \cite[Section II]{CaL93} for more information. 
The associated Weyl operators are given by
\begin{equation*}
\label{eq-W-hardcore}
\W_{\mathrm{fer}}(a,b)
= X^{a_1}Z^{b_1} \ot \cdots \ot X^{a_n}Z^{b_n} 
= h^{a_1}_1 h^{b_1}_2 h^{a_2}_3 h^{b_2}_4 \cdots h^{a_n}_{2n-1} h^{b_n}_{2n}, \quad (a,b) \in G.
\end{equation*}
\subsection{Example: noncommutative Vilenkin systems}
\label{sec:Vilenkin}


\paragraph{Vilenkin groups}
For any integer $\ell \geq 2$, we denote by $\ZZ_\ell$ the cyclic group $\Z/\ell\Z = \{0,1,\ldots,\ell-1\}$. Let $m=(m_1,m_2,\ldots)$ be a sequence of integers $m_k \geq 2$ and consider the product group $G=G(m) \ov{\mathrm{def}}{=} \prod_{k=1}^{\infty} \ZZ_{m_k}$, equipped with the product topology and the normalized Haar measure. This group is an example of compact Vilenkin group, i.e.~an abelian totally disconnected second countable compact topological group. We refer to \cite[Section 3]{ArK25} for more information on more general locally compact Vilenkin groups.

The dual group is given by $
\hat{G} 
= \oplus_{k=1}^\infty \hat{\Z}_{m_k}$, where $\hat{\ZZ}_{m_k}$ is isomorphic to $\ZZ_{m_k}$. This group $\hat{G}$ can be identified with the collection of all sequences $l=(l_1,l_2,\ldots,)$ with $l_k \in \{0,1\ldots m_k-1\}$ for all $k$ and $l_k \not= 0$ for only finitely many values of $k$. The pairing between $G$ and $\hat{G}$ is given by 
$
\la x,l \ra_{G,\hat{G}}
=\psi_l(x),
$ 
where
\begin{equation*}
\label{}
\psi_l(x)
\ov{\mathrm{def}}{=} \prod_{k=1}^{\infty} \phi_k^{l_k}(x), \quad l=(l_1,l_2,\ldots) \in \hat{G}
\quad \text{with} \quad
\phi_k(x)
\ov{\mathrm{def}}{=} \e^{\frac{2\pi \i x_k}{p_k}}, \quad x=(x_1,x_2,\ldots) \in G.
\end{equation*}
The characters $(\psi_l)_{l \in \hat{G}}$ form a complete orthonormal system in the complex Hilbert space $\L^2(G)$, which is the Vilenkin system corresponding to $m=(m_1,m_2,\ldots)$. This system was introduced in the paper  \cite{Vil63}. Using a suitable ordering on $\hat{G}$, it is known \cite{Wat58} that this system becomes a Schauder basis of the Banach space $\L^p(G)$ for any $1 < p < \infty$. 

\begin{remark} \normalfont
It is worth noting that there exists a natural measure preserving identification between the group $G$ and the interval $[0,1]$, see \cite[p.~504]{SWS90}. So, we can see the Vilenkin system as a system of functions on $[0,1]$.
\end{remark}

If $m_k=2$ for all integer $k \geq 1$, the system is called the Paley-Walsh system. In this case, the group $G$ is the dyadic group and denoted by $\Omega$. The harmonic analysis of the dyadic group has been the focus of extensive research, with numerous results compiled in the comprehensive monograph \cite{SWS90}. See also \cite{SBSW15a} and \cite{SBSW15b}. For an in-depth exploration of harmonic analysis on compact Vilenkin systems, we direct the reader to \cite{Gos73}, \cite{PSTW22}, \cite{Tse24}, \cite{You76}, \cite{Wat58} and references therein. We additionally refer to the book \cite{AVDR81}, which seems to include valuable insights but is regrettably not available in English. 

If we view the elements of $G$ as sequences, we may also consider, for each integer $N \geq 1$, the truncation of these sequences. Let $
G_N 
\ov{\mathrm{def}}{=} \prod_{k=1}^{N} \ZZ_{m_k}$. We identify this group with the subgroup of truncated sequences of length $N$ in $G$. Similarly, we identify the dual group $\hat{G}_N$ with the subgroup of truncated sequences of length $N$ in $\hat{G}$.

\paragraph{Noncommutative Vilenkin systems} 
Let $\Phi \co G \to \Aut(\cal{M})$ be a faithful ergodic action of a Vilenkin group $G$ on a von Neumann algebra $\cal{M}$, with associated map $\alpha \co \cal{M} \to \cal{M} \otvn \L^\infty(G)$. We consider the unique $G$-invariant normal state $\varphi_{\cal{M}}$ on the von Neumann algebra $\cal{M}$ defined in \eqref{Def-ergodic}, which is a trace. Following essentially\footnote{\thefootnote. The authors use a trace preserving ergodic action $\alpha$.} \cite[Definition 1.1 p.~77]{DoS00}, a complete system $(v_l)_{ l \in \hat{G}}$ of eigenoperators of $\alpha$ is called a noncommutative Vilenkin system. It follows that any noncommutative Vilenkin system is a complete orthonormal system of the complex Hilbert space $\L^2(\cal{M})$. It is worth noting that, using a suitable ordering on $\hat{G}$, it is proved in \cite[Theorem 3.9 p.~423]{DFPS01} (generalizing the partial result \cite[Theorem 2.4 p.~83]{DoS00}) that this system becomes a Schauder basis of the Banach space $\L^p(\cal{M})$ for any $1 < p < \infty$. See also the survey paper \cite{DoS01}. We refer to \cite[Section 4]{DFPS01} for the classification of Vilenkin systems relying on the paper \cite{OPT80}.


Now, we give an example of concrete Vilenkin system.

\begin{example} \normalfont
Let $m=(m_1,m_2,\ldots)$ be a sequence of integers $m_k \geq 2$. Recall that the unique separable approximately finite-dimensional factor $\cal{R}$ of type $\II_1$ factor is $*$-isomorphic to the tensor product
\begin{equation*}
\label{}
(\cal{R},\tau) 
= \otvn_{k=1}^\infty (\M_{m_k},\tau_{m_k}),
\end{equation*}
where $\tau_{m_k}$ is the normalized trace on the matrix algebra $\M_{m_k}$. Note that $\tau$ is a normal finite faithful trace on the von Neumann algebra $\cal{R}$.   
For each integer $n \geq 1$, we may consider the finite truncation of this infinite tensor product by letting
\begin{equation*}
\label{def-de-cal-Rn}
(\cal{R}_n,\tau_n) 
\ov{\mathrm{def}}{=} \otvn_{k=1}^n (\M_{m_k}, \tau_{m_k}).
\end{equation*}


Let $2m = (2m_k)_{k=1}^\infty$ denote the doubled sequence given by $2m(2k-1) \ov{\mathrm{def}}{=} m_k$, and $2m(2k) \ov{\mathrm{def}}{=} m_k$, for each $k \geq 1$. Then we consider the Vilenkin group $G$ corresponding to the doubled sequence $2m$, with dual group $\hat{G}$. For each $l \in \hat{G}$, following \cite[p.~90]{DoS00} (see also \cite[p.~64]{ScS18}), we define
\begin{equation*}
\label{}
v_l
\ov{\mathrm{def}}{=} \bigotimes_{k=1}^{\infty} X_{m_k}^{l_{2k-1}}Z_{m_k}^{l_{2k}}
= X_{m_1}^{l_0} Z_{m_1}^{l_1} \ot X_{m_2}^{l_2} Z_{m_2}^{l_3} \ot \cdots,
\end{equation*}
where the matrices $X_{m_k}$ and $Z_{m_k}$ are defined in \eqref{def-X-Z}. It is known that $(v_l)_{ l \in \hat{G}}$ is a Vilenkin system associated to the group $G$ in the von Neumann algebra $\cal{R}$. In particular, the Vilenkin system $(v_l)_{ l \in \hat{G}}$ is an orthonormal basis for the complex Hilbert space $\L^2(\cal{R})$.

Let $\hat{G}_{2N}$ denote the subgroup of sequences of length $2N$ in $\hat{G}$. For each integer $N \geq 1$, the restriction $(v_l)_{ l \in \hat{G}_{2N}}$ of the Vilenkin system is a basis of $\cal{R}_N$ and an orthonormal basis of the complex Hilbert space $\L^2(\cal{R}_N)$. It is interesting to observe that the <<transference map $\pi$ >> considered in \cite[Lemma 3.6 p.~1162]{TiD24} is a very particular case of the second map introduced in \eqref{actions-alpha-beta}.
\end{example}

Another example is provided by the setting of the paper \cite{ArK3}. If $N \in \N$, we equip the abelian finite group $\{-1,1\}^N$ with its canonical Haar measure. If $1 \leq k \leq N$, we denote by $\epsi_k \co \{-1,1\}^N \to \{-1,1\}$ the $k$-th coordinate on $\{ -1 , 1 \}^N$. We can see the $\epsi_k$'s as independent Rademacher variables. We also consider a family $(s_i)_{1 \leq i \leq N}$ of selfadjoint operators on a Hilbert space $H$ satisfying the following canonical anticommutation relations
$$
s_is_j+s_js_i 
= 2\delta_{i=j}, \quad i,j \in \{1,\ldots,N\}.
$$
These operators share many properties with the Rademacher variables $\epsi_i$ although they are essentially noncommutative in nature. For a finite non-empty subset $A=\{k_1,\ldots,k_n\}$ of $\{1,\ldots,N\}$ with $k_1 < \cdots < k_n$ we let 
\begin{equation}
\label{def-sA}
s_A 
\ov{\mathrm{def}}{=} s_{k_1} \cdots s_{k_n} 
\quad \text{and} \quad 
\epsi_A 
\ov{\mathrm{def}}{=} \epsi_{k_1} \cdots \epsi_{k_n}. 
\end{equation}
We also set $s_\emptyset \ov{\mathrm{def}}{=} 1$ and $\epsi_\emptyset \ov{\mathrm{def}}{=} 1$. It is known that the involutive algebra generated by the $s_i$'s is equal to $\scr{C}=\mathrm{span}\{ s_A : A \subset{} \{1,\ldots,N\}\}$. The von Neumann algebra generated by this algebra is denoted $\scr{C}(H)$. There is a trace $\tau$ on this algebra such that $\tau(s_A) = 0$ if $A \not= \emptyset$. We refer to \cite{PlR94} and \cite[Section 9.B]{JMX06} for more information.


The following result is \cite[Proposition 4.5]{ArK3}.

\begin{prop}
\label{prop-ergodic}
Let $H$ be a finite-dimensional real Hilbert space with dimension $n \geq 2$. The map $\alpha \co \scr{C}(H) \to \scr{C}(H) \otvn \L^\infty(\Omega_n)$, $s_A\mapsto s_A \ot w_A$ defines a right action from the compact abelian group $\Omega_n$ on the fermion algebra $\scr{C}(H)$ which is ergodic. Finally, the unique $\Omega_n$-invariant normal state on the fermion algebra $\scr{C}(H)$ is the trace $\tau$ defined by \eqref{trace-fer}
\end{prop}

\begin{example} \normalfont
\label{Ex-multiplier-Fermions}
Consider the case $n=2$ and the Pauli matrices $s_1
\ov{\mathrm{def}}{=} \begin{bmatrix}
      0&1\\
      1&0
    \end{bmatrix} $ and $s_2 \ov{\mathrm{def}}{=}
    \begin{bmatrix}
      0& -\i \\
      \i&0
    \end{bmatrix} $.  
These matrices are selfadjoint unitaries and anticommute. We have $s_{\{1,2\}}\ov{\eqref{def-sA}}{=} s_1s_2=\begin{bmatrix}
      0&1\\
      1&0
    \end{bmatrix}    \begin{bmatrix}
      0& -\i \\
      \i&0
    \end{bmatrix} 
		=    \begin{bmatrix}
    \i  & 0 \\
    0  & -\i
    \end{bmatrix}=\i \sigma_3$ where $\sigma_3 \ov{\mathrm{def}}{=} \begin{bmatrix}
    1 & 0 \\
    0 & -1
\end{bmatrix}
$. 
Note that a 2x2 complex matrix $\begin{bmatrix}
    a & b \\
    c & d
\end{bmatrix}$ can be written 
\begin{equation}
\label{decompo}
\begin{bmatrix}
    a & b \\
    c & d
\end{bmatrix}
=\frac{a+d}{2}\I +\frac{b+c}{2}\sigma_1+\frac{\i(b-c)}{2}\sigma_{2}+\frac{a-d}{2\i}\i\sigma_3.
\end{equation}
Consider a complex function $\phi \co \{0,1,2\} \to \mathbb{C}$. The associated multiplier $R_\phi \co \M_2 \to \M_2$, $\alpha \I_2+\beta s_1+\gamma s_2+\zeta s_{\{1,2\}} \mapsto \phi(0)\alpha \I_2+\phi(1)\beta s_1+\phi(1)\gamma s_2+\phi(2)\zeta s_{\{1,2\}}$ of Definition \ref{def-multiplier-eigen} which are quantum channels identify to Pauli channels described in \cite[Example 2.12 p.~17]{Pet08}.
\end{example}

%
%
%

\subsection{Example: mixed commutation and anti-commutation relations}
\label{sec:mixed}



In \cite{Bia97}, Biane considered spin systems with mixed commutation and anti-commutation relations and used it to approximate some $q$-deformed systems, relying on Speicher's central limit Theorem \cite{Spe92}. Let $I=\{1,\ldots N\}$ for some integer $N$. Consider a function $\epsi \co I \times I \to \{-1,1\}$ satisfying $\epsi(i,j)=\epsi(j,i)$ and $\epsi(i,i)=-1$ for all $i,j \in I$. Let $\mathcal{A}(I,\epsi)$ be the free complex unital algebra with generators $(x_{i})_{i \in I}$ quotiented by the relations
\begin{equation}
\label{mix}
x_{i}x_{j}-\epsi(i,j) x_{j}x_{i}
=2 \delta_{i,j}, \quad i,j \in I.
\end{equation}
We can define an involution on the algebra $\mathcal{A}(I,\epsi)$ by letting $x_{i}^{*} \ov{\mathrm{def}}{=} x_{i}$. For any subset $A = \{i_{1},\dots ,i_{k}\}$ of $I$ with $i_{1}<\dots <i_{k}$ we define $x_{A} \ov{\mathrm{def}}{=} x_{i_{1}} \dots x_{i_{k}}$ with $x_{\emptyset} \ov{\mathrm{def}}{=}1$. Then the family $(x_{A})_{A\subset I}$ is a basis of the vector space $\mathcal{A}(I,\epsi)$. Let $\tau_{\epsi}$ be the trace defined by 
\begin{equation}
\label{tau-epsi}
\tau_{\epsi}(x_{A}) \ov{\mathrm{def}}{=} \delta_{A, \emptyset}
\end{equation}
for all subset $A$ in $I$. The family $(x_{A})_{A \subset I}$ is an orthonormal basis of the complex Hilbert space $\L^{2}(\mathcal{A}(I,\epsi))$. 


For each $i \in I$, define the following partial isometries $\beta_{i}^{*}$ and $\alpha_{i}^{*}$ acting on the complex Hilbert space $\L^{2}(\cal{A}(I,\epsi),\varphi^{\epsi})$ by
$$
\beta_{i}^{*}(x_{A})
\ov{\mathrm{def}}{=}
\begin{cases}
x_{i}x_{A} & \text{if}\quad i \not\in A\\
0 & \text{if} \quad i\in A
\end{cases} 
\quad \text{and} \quad
\alpha_{i}^{*}(x_{A}) 
\ov{\mathrm{def}}{=} 
\begin{cases}
x_{A}x_{i} & \text{if} \quad i \not\in A\\
0 & \text{if} \quad i\in A
\end{cases}.
$$
Note that their adjoints are given by
$$
\beta_{i}(x_{A})
=\begin{cases}
x_{i}x_{A} & \text{if}\quad i\in A\\
0 & \text{if}\quad i\not\in A
\end{cases}
\quad \text{and} \quad
\alpha_{i}(x_{A})
=\begin{cases}
x_{A}x_{i} &\text{if}\quad i\in A\\
0 & \text{if}\quad i\not\in A
\end{cases}.
$$
The operators $\beta_{i}^{*}$ and $\beta_{i}$ are called the left and right creation operators and the operators $\alpha_{i}^{*}$ and $\alpha_{i}$ are called left and right annihilation operators. 
%
It is known that the selfadjoint operators defined by $x_{i} \ov{\mathrm{def}}{=}\beta_{i}^{*}+\beta_{i}$ satisfy the  relations \eqref{mix}. 
It is possible \cite{Bia97} \cite[p.~18]{Mey95} \cite{Spe92} to find explicitly selfadjoint matrices $\beta_i$ such that $x_{i} \ov{\mathrm{def}}{=}\beta_{i}^{*}+\beta_{i}$ satisfy the relations \eqref{mix}.  
Indeed, by letting $\sigma_1 \ov{\mathrm{def}}{=} \begin{bmatrix}
  1   & 0  \\
  0   & 1 \\
\end{bmatrix}$, $\sigma_{-1} \ov{\mathrm{def}}{=} \begin{bmatrix}
 1    &  0 \\
  0   & -1  \\
\end{bmatrix}$, $b \ov{\mathrm{def}}{=} \begin{bmatrix}
  0   &  1 \\
  0   &  0 \\
\end{bmatrix}$, $b^* \ov{\mathrm{def}}{=} \begin{bmatrix}
  0   &  0 \\
  1   &  0 \\
\end{bmatrix}$. Then the matrices
$$
\beta_i
\ov{\mathrm{def}}{=} \sigma_{\epsi(1,i)} \ot \sigma_{\epsi(2,i)} \ot \cdots \ot \sigma_{\epsi(i-1,i)} \ot b \ot \sigma_1 \ot \cdots \ot \sigma_1
$$
in $\M_2^{\ot n}$ are convenient.

It is not difficult tor prove the following result which gives an ergodic action from the Walsch group $\Omega_n$. Recall that the product of two Walsh functions is given by $w_A w_B=w_{A\Delta B}$ and that \begin{equation}
\label{int-walsh}
\int_{\Omega_n} w_A=\delta_{A,\emptyset},
\end{equation}
where $\delta$ is the Kronecker symbol.

\begin{prop}
\label{prop-ergodic}
The linear map $\alpha \co \mathcal{A}(I,\epsi) \to \mathcal{A}(I,\epsi) \otvn \L^\infty(\Omega_n)$, 
\begin{equation}
\label{action-mixed}
\alpha(x_A)
\ov{\mathrm{def}}{=} x_A \ot w_A
\end{equation}
defines a right action from the compact abelian group $\Omega_n$ on the algebra $\mathcal{A}(I,\epsi)$ which is ergodic. Finally, the unique $\Omega_n$-invariant normal state on the algebra $\mathcal{A}(I,\epsi)$ is the trace $\tau_\epsi$ defined by \eqref{tau-epsi}.
\end{prop}

\begin{proof}
For any subsets $A$ and $B$ of $\{1,\ldots,n\}$, note that $x_A x_B=(-1)^{n(A,B)}x_{A \Delta B}$ and that  and $x_A^*=(-1)^{n_A} x_A$ for some integers $n(A,B)$ and $n_A$. Now, we have
\begin{align*}
\MoveEqLeft
\alpha(x_A x_B)
= (-1)^{n(A,B)} \alpha(x_{A \Delta B})
\ov{\eqref{action-mixed}}{=} (-1)^{n(A,B)} x_{A \Delta B} \ot w_{A \Delta B}
= \\
&= x_A x_B \ot w_A w_B
=(x_A \ot w_A)(s_B \ot w_B)
\ov{\eqref{action-mixed}}{=} \alpha(x_A)\alpha(x_B)
\end{align*}
and
$$
\alpha(x_A^*)
= (-1)^{n_A} \alpha(x_A)
= (-1)^{n_A} x_A \ot w_A
= (x_A \ot w_A)^*
\ov{\eqref{action-mixed}}{=} \alpha(x_A)^*.
$$
Consequently, the linear map $\alpha$ is a $*$-homomorphism, which is unital by construction. We also have
\begin{align*}
\MoveEqLeft
\bigg(\tau_{\epsi} \ot \int_{\Omega_n}\bigg)(\alpha(x_A))
\ov{\eqref{action-mixed}}{=} \bigg(\tau_{\epsi} \ot \int_{\Omega_n}\bigg)(x_A \ot w_A) 
=\tau_{\epsi}(x_A) \int_{\Omega_n} w_A
\ov{\eqref{tau-epsi} \eqref{int-walsh}}{=} 
\delta_{A,\emptyset}
\ov{\eqref{tau-epsi}}{=} \tau_{\epsi}(x_A).         
\end{align*}
Hence the map $\alpha$ is trace preserving. The injectivity is a consequence of \cite[Lemma 2.1 p.~2283]{Arh19} but can also be proved directly. For any finite subset $A$ of the set $\{1,\ldots,n\}$, we have
$$
(\alpha \ot \Id)\alpha(x_A)
\ov{\eqref{action-mixed}}{=} (\alpha \ot \Id)(x_A \ot w_A)
= \alpha(x_A) \ot w_A
\ov{\eqref{action-mixed}}{=} x_A \ot w_A \ot w_A
$$
and
$$
(\Id \ot \Delta)\alpha(x_A)
\ov{\eqref{action-mixed}}{=} (\Id \ot \Delta)(x_A \ot w_A)
= x_A \ot \Delta(w_A)
\ov{\eqref{delta-epsi}}{=} x_A \ot w_A \ot w_A.
$$
By linearity, we obtain the equality \eqref{Def-corep}. 

In order to prove the ergodicity of $\alpha$,  consider some element $y \in \mathcal{A}(I,\epsi)$ such that $\alpha(y) = y \ot 1$. If $\sum_{A} \lambda_A x_A$ is the <<Fourier series>> of $y$ in the space $\L^2(\mathcal{A}(I,\epsi))$, where each $\lambda_A$ is a complex number, we have $\alpha(\sum_{A} \lambda_A x_A) \ov{\eqref{action-mixed}}{=} \sum_{A} \lambda_A x_A \ot w_A$. The equality $\sum_{A} \lambda_A x_A \ot w_A=\big(\sum_{A} \lambda_A x_A\big) \ot 1$ is equivalent to $\sum_{A} \lambda_A x_A=1$. So the ergodicity is clear.

Finally, for any finite subset $A$ of the set $\{1,\ldots,n\}$ we have
\begin{align*}
\MoveEqLeft
\bigg(\Id \ot \int_{\Omega_n}\bigg) \circ \alpha(x_A)
\ov{\eqref{action-mixed}}{=} \bigg(\Id \ot \int_{\Omega_n}\bigg) (x_A \ot w_A)
= x_A \int_{\Omega_n} w_A \\
&\ov{\eqref{int-walsh}}{=} \delta_{A,\emptyset} x_A 
=\delta_{A,\emptyset} 1_{\mathcal{A}(I,\epsi)} 
\ov{\eqref{tau-epsi}}{=} \tau_\epsi(x_A) 1_{\mathcal{A}(I,\epsi)}.
\end{align*}
\end{proof}

\section{The non-ergodic case}
\subsection{Some general result}
We are aware that we can generalize Proposition \ref{Prop-transfer-1} to the non-ergodic case using the amalgamated $\L^p$-spaces of \cite{JuP10} and \cite{GJL20b}. Indeed, we can state the next result. However, we does not try to prove any application at the present state of the writing. We will continue in a next paper. Let $\cal{M}$ be a semifinite von Neumann algebra equipped with a normal semifinite faithful trace $\tr$. Let $\cal{N}$ be a subalgebra of a von Neumann algebra $\cal{M}$ endowed with a normal semifinite faithful trace $\tau$ such that $\tau|_{\cal{N}}$ is also semifinite. Recall that
\begin{equation}
\label{norm-cond}
\norm{x}_{\L^1_\infty(\cal{N} \subset \cal{M})}
=\inf_{x=yz}\norm{\E(yy^*)}_{\L^\infty(\cal{M})}^{\frac{1}{2}} \norm{\E(z^*z)}_{\L^\infty(\cal{M})}^{\frac{1}{2}}.
\end{equation}
It is known that 
\begin{equation}
\label{amal-useful}
\L_p^p(\cal{N} \subset \cal{M})=\L^p(\cal{M}) 
 \quad \text{and} \quad
\L_q^p(\cal{N} \subset \cal{M})
=(\L_{q_0}^{p_0}(\cal{N} \subset \cal{M}),\L_{q_1}^{p_1}(\cal{N}\subset\cal{M}))_\theta
\end{equation}
isometrically where $\frac{1-\theta}{p_0}+\frac{\theta}{p_1}=\frac{1}{p}$, $\frac{1-\theta}{q_0}+\frac{\theta}{q_1}=\frac{1}{q}$ and $(p_1-q_1)(p_2-q_2) \geq 0$.

It is worth noting that if $\E \ov{\mathrm{def}}{=} \Id \ot \tau_{\cal{N}} \co \cal{M} \otvn \cal{N} \to \cal{M}$ is the canonical trace preserving faithful conditional expectation, by essentially \cite[Proposition 2.6]{Arh25a} the space $\L^1_\infty(\cal{N} \subset \cal{M})$ is isometric to the space $\L^\infty(\cal{M},\L^1(\cal{N}))$. 

\begin{prop}
\label{Prop-transfer-non-ergodic}
Let $\alpha \co \cal{M} \to \cal{M} \otvn \L^\infty(\QG)$ be a trace preserving right action of a compact group of Kac type on a finite von Neumann algebra $\cal{M}$. Consider $1 \leq p < \infty$. The map $\alpha$ induces complete contractions $\alpha \co \L^1_p(\cal{M}^\alpha \subset \cal{M}) \to \L^p(\cal{M},\L^1(\QG))$, $\alpha \co \L^p_\infty(\cal{M}^\alpha \subset \cal{M}) \to \L^\infty(\cal{M},\L^p(\QG))$ where $\E \co \cal{M} \to \cal{M}$ is the trace preserving conditional expectation on the fixed point subalgebra $\cal{M}^\alpha$.
\end{prop}

\begin{proof}
On the one hand, we can write in the case of a \textit{positive} element $x$ in the space $\L^1(\cal{M})$
\begin{align}
\MoveEqLeft
\label{divers-5667890}
\norm{\alpha(x)}_{\L^\infty(\cal{M},\L^1(\QG))}              
\ov{\eqref{Norm-LpL1-pos}}{=} \bnorm{(\Id \ot h_\cal{\QG})(\alpha(x))}_{\L^\infty(\cal{M})} 
\ov{\eqref{cond-exp-fixed}}{=} \norm{\E(x)}_{\L^\infty(\cal{M})}. 
\end{align} 
Let $\epsi >0$. Now, consider an \textit{arbitrary} element $x$ of $\L^1_\infty(\cal{M}^\alpha \subset \cal{M})$. By \eqref{norm-cond}, there exist elements $y,z$ such that $x=yz$ and 
\begin{equation}
\label{sans-fin-987}
\norm{x}_{\L^1_\infty(\cal{M}^\alpha \subset \cal{M})}
\leq \norm{\E(yy^*)}_{\L^\infty(\cal{M})}^{\frac{1}{2}} \norm{\E(z^*z)}_{\L^\infty(\cal{M})}^{\frac{1}{2}}
\leq \norm{x}_{\L^1_\infty(\cal{M}^\alpha \subset \cal{M})}+\epsi.
\end{equation}
We have
\begin{align*}
\MoveEqLeft
\bnorm{(\Id \ot h_{\QG})(\alpha(y)\alpha(y)^*)}_{\L^\infty(\cal{M})}^{\frac{1}{2}} \bnorm{(\Id \ot h_{\QG})(\alpha(z)^*\alpha(z))}_{\L^\infty(\cal{M})}^{\frac{1}{2}} \\
&\ov{\eqref{Norm-LpL1-pos}}{=} \norm{\alpha(y)\alpha(y)^*}_{\L^\infty(\cal{M},\L^1(\QG))} ^{\frac{1}{2}} \norm{\alpha(z)^*\alpha(z)}_{\L^\infty(\cal{M},\L^1(\QG))} ^{\frac{1}{2}} \\
&\ov{\eqref{sans-fin-987}}{=} \norm{\alpha(yy^*)}_{\L^\infty(\cal{M},\L^1(\QG))} ^{\frac{1}{2}} \norm{\alpha(z^*z)}_{\L^\infty(\cal{M},\L^1(\QG))} ^{\frac{1}{2}} \\      
&\ov{\eqref{divers-5667890}}{=} \norm{\E(yy^*)}_{\L^\infty(\cal{M})}^{\frac{1}{2}}\norm{\E(z^*z)}_{\L^\infty(\cal{M})}^{\frac{1}{2}} 
\leq \norm{x}_{\L^1_\infty(\cal{M}^\alpha \subset \cal{M})}+\epsi.
\end{align*} 
Moreover, we have $\alpha(x)=\alpha(yz)=\alpha(y)\alpha(z)$. From the explanation following \eqref{norm-cond}, we deduce the inequality 
$$
\norm{\alpha(x)}_{\L^\infty(\cal{M},\L^1(\QG))} 
\leq \norm{x}_{\L^1_\infty(\cal{M}^\alpha \subset \cal{M})}+\epsi.
$$ 
We deduce a contraction $
\alpha \co \L^1_\infty(\cal{M}^\alpha \subset \cal{M}) \to \L^\infty(\cal{M},\L^1(\QG))$.  
We have complete contractions $\alpha \co \L^1_1(\cal{M}^\alpha \subset \cal{M})\ov{\eqref{amal-useful}}{=} \L^1(\cal{M}) \to \L^1(\cal{M},\L^1(\QG))$ and $\alpha \co \L^\infty_\infty(\cal{M}^\alpha \subset \cal{M}) \ov{\eqref{amal-useful}}{=} \L^\infty(\cal{M}) \to \L^\infty(\cal{M}) \otvn \L^\infty(\QG)$. We conclude by interpolation with the second equality described in \eqref{amal-useful}. The complete boundedness is similar.
\end{proof}

Recall that for finite von Neumann algebras $\cal{N} \subset \cal{M}$, it is proved in \cite[Theorem 3.9]{GJL20b} that the completely bounded relative entropy $\D_{\cb,p}(\cal{M}||\cal{N})$ satisfies
\begin{align*}
\D_{\cb,p}(\cal{M}||\cal{N})
:=p^*\log\norm{\Id}_{\cb, \L_\infty^{p^*}(\cal{N} \subset \cal{M}) \to \cal{M}}.
\end{align*}

\subsection{From Fourier multipliers to Herz-Schur multipliers:}
\label{Sec-Transference-Herz-Schur}


In this section, we use the classical notations of \cite{ArK23}. Here $G$ is discrete group. Recall that we say that a function $\varphi \co G \to \mathbb{C}$ induces a Fourier multiplier on the von Neumann algebra $\VN(G)$ if the linear map $\P_G \to \P_G$, $\lambda_s \mapsto \varphi(s)\lambda_s$ extends to a weak* continuous operator $M_\varphi \co \VN(G) \to \VN(G)$. In this case, we denote by $M_{\varphi}^{\HS} \co \B(\ell^2_G) \to \B(\ell^2_G)$, $x \mapsto [\varphi(st^{-1})x_{st}]$ the Herz-Schur multiplier on $\B(\ell^2_G)$ with associated matrix $C_\varphi \ov{\mathrm{def}}{=} [\varphi(st^{-1})]$. The operator $M_{\varphi}^{\HS}$ is denoted $\hat{\Theta}(\varphi)$ in the paper \cite{CrN13}. 

We start with a first transference principle. In the third point, we use the notations of \eqref{convol-def-bis}. Recall that the predual of the von Neumann algebra $\VN(G)$ identifies to the Fourier algebra $\mathrm{A}(G)$.

\begin{lemma}
\label{Lemma-Delta-Schur-1}
Let $G$ be a discrete group. 

\begin{enumerate}
\item The map $\alpha \co \B(\ell^2_G) \to \B(\ell^2_G) \otvn \VN(G)$, $e_{st} \mapsto e_{st} \ot \lambda_{st^{-1}}$ is a well-defined trace preserving injective unital normal $*$-homomorphism. 

\item If $\Delta$ is defined as in \eqref{coproduct-VNG}, we have $(\alpha \ot \Id) \circ \alpha=(\Id \ot \Delta) \circ \alpha$, i.e.~$\alpha$ is a right action of $\VN(G)$ on $\B(\ell^2_G)$, see \eqref{Def-corep}.

\item Let $\varphi \co G \to \mathbb{C}$ be a function belonging to $\mathrm{A}(G)$. If $\omega=\la\varphi, \cdot\ra$, we have $T_{\omega,\alpha} =M_{\varphi}^{\HS}$ and $T_{\omega}=M_{\varphi}$. In particular, we have
\begin{equation}
\label{entrelacement-transference-1}
\alpha \circ M_{\varphi}^{\HS}
=\big(\Id \ot M_{\varphi} \big) \circ \alpha.
\end{equation}

\item Suppose $1 \leq p <\infty$. The map $\alpha$ induces a complete isometry $\alpha \co S^p_G \to S^p_G(\L^p(\VN(G))$.

\item The fixed point subalgebra $\B(\ell^2_G)^\alpha$ is equal to the diagonal subalgabra of $\B(\ell^2_G)$.

\item Suppose $1 \leq p <\infty$. The map $\alpha$ induces a contraction $\alpha \co S^1_G \to S^p_G(\L^1(\VN(G)))$.

\end{enumerate}
\end{lemma}

\begin{proof}
1. If $(e_s)_{s \in G}$ is the canonical orthonormal basis of the Hilbert space $\ell^2_G$ then we can consider the fundamental unitary $W \co \ell^2_{G \times G} \to \ell^2_{G \times G}$ belonging to $\B(\ell^2_{G}) \otvn \VN(G)$ defined by
\begin{equation}
\label{W-discret}
W(e_t \ot e_r)
\ov{\mathrm{def}}{=} e_t \ot e_{tr}, \quad 
W^{-1}(e_t \ot e_r)
=e_t \ot e_{t^{-1}r}, \quad t,r \in G.
\end{equation}
For any $i,j,s,t \in G$, we have
\begin{align*}
\MoveEqLeft
W (e_{st} \ot 1)W^{-1}(e_i \ot e_j)            
\ov{\eqref{W-discret}}{=} W (e_{st} \ot 1)(e_i \ot e_{i^{-1}j}) 
=W\big(e_{st}e_i \ot e_{i^{-1}j}\big) \\
&= \delta_{t=i}W\big(e_s \ot e_{i^{-1}j}\big) 
\ov{\eqref{W-discret}}{=} \delta_{t=i} e_s \ot e_{st^{-1}j}.
\end{align*} 
Hence in $\B(\ell^2_G) \otvn \VN(G)$, we have
\begin{equation}
\label{calcul-890}
W (e_{st} \ot 1)W^{-1}
= e_{st} \ot \lambda_{st^{-1}}.
\end{equation}
Consequently, we can define $\alpha$ as the injective normal unital $*$-homomorphism $\alpha \co \B(\ell^2_G) \to \B(\ell^2_G) \otvn \VN(G)$, $x \mapsto W (x \ot 1)W^{-1}$. 
For any $i,j,r,s \in G$, we have
\begin{align*}
\MoveEqLeft
\big(\tr \ot \tau_G\big)\big(\alpha(e_{rs}e_{ij})\big)           
=\delta_{s=i}\big(\tr \ot \tau_G\big)\big(\alpha(e_{rj})\big)  
=\delta_{s=i}\big(\tr \ot \tau_G\big)(e_{rj} \ot \lambda_{rj^{-1}})\\
&=\delta_{s=i} \tr(e_{rj}) \tau_G(\lambda_{rj^{-1}})
=\delta_{s=i}\delta_{r=j}
=\tr(e_{rs}e_{ij})    
\end{align*}
We conclude with \cite[Theorem 6.2 p.~83]{Str81} that $(\tr \ot \tau_G) \circ \alpha=\tr$.


2. For any $s,t \in G$, we have
\begin{align*}
\MoveEqLeft
(\alpha \ot \Id) \circ \alpha(e_{st})
=(\alpha \ot \Id)(e_{st} \ot \lambda_{st^{-1}})         
=\alpha(e_{st}) \ot \lambda_{st^{-1}} 
=e_{st} \ot \lambda_{st^{-1}} \ot \lambda_{st^{-1}}.
\end{align*} 
and
\begin{align*}
\MoveEqLeft
(\Id \ot \Delta) \circ \alpha(e_{st})         
=(\Id \ot \Delta)(e_{st} \ot \lambda_{st^{-1}}) 
=e_{st} \ot \Delta(\lambda_{st^{-1}}) 
\ov{\eqref{coproduct-VNG}}{=} e_{st} \ot \lambda_{st^{-1}} \ot \lambda_{st^{-1}}.
\end{align*}

3. For any $s,t \in G$, we have
\begin{align*}
\MoveEqLeft
T_{\omega,\alpha}(e_{st})
\ov{\eqref{convol-def-bis}}{=} (\Id \ot \omega) \circ \alpha(e_{st})            
=(\Id \ot \varphi)(e_{st} \ot \lambda_{st^{-1}})   \\
&=e_{st} \ot \varphi(\lambda_{st^{-1}}) 
=\varphi(st^{-1}) e_{st}
=M_{\varphi}^{\HS}(e_{st})
\end{align*} 
and
\begin{align*}
\MoveEqLeft
T_{\omega}(\lambda_{s})
\ov{\eqref{convol-def-bis}}{=} (\Id \ot \omega) \circ \Delta(\lambda_{s})            
= (\Id \ot \omega)(\lambda_{s} \ot \lambda_{s}) 
=\varphi(s)\lambda_{s}
=M_{\varphi}(\lambda_{s}).
\end{align*} 

4. It suffices to use Lemma \ref{lemma-trace-preserving}.


5) If $x=\sum_{st} a_{st} e_{st}$ belongs to $\B(\ell^2_G)$, we have
$$
\alpha(x)
=\alpha\bigg(\sum_{s,t} a_{st} e_{st}\bigg)
=\sum_{s,t} a_{st} \alpha(e_{st})
=\sum_{s,t} a_{st} e_{st} \ot \lambda_{st^{-1}}.
$$
So $\alpha(x)=x \ot 1$ if and only if for any $s,t \in G$ we have $a_{st}\lambda_{st^{-1}}=a_{st}1$. This is equivalent to $a_{st}=0$ if $s \not=t$.

6. Let $\E \co \B(\ell^2_G) \to \B(\ell^2_G)$ be the trace preserving normal conditional expectation on the subalgebra of diagonal operators. This is a consequence of Proposition \ref{Prop-transfer-non-ergodic} using the contraction $S^1_G \subset S^p_G$.
\end{proof}

\begin{remark} \normalfont
The complete contractivity of the map $\alpha \co S^1_G \to S^p_G(\L^1(\VN(G)))$ is unclear. A positive answer implies the inequality $\leq$ of the equality \eqref{Hcbmin-HS}.
\end{remark}

Note Lemma \ref{lemma-useful-678}.

\begin{lemma}
\label{Lemma-entropy}
Let $G$ be a finite group. Let $\varphi \co G \to \mathbb{C}$ be a positive definite function with $\varphi(e)=1$. We have
\begin{equation}
\label{equality-ent}
\H\bigg(\sum_{s \in G} \varphi(s)\lambda_s\bigg)  
=\H\bigg(\frac{1}{|G|}C_\varphi\bigg)-\log_2 |G|.
\end{equation}
\end{lemma}

\begin{proof}
Note that we have a unital $*$-homomorphism $J \co \VN(G) \to \B(\ell^2_G)$, $\lambda_s \mapsto \sum_{t \in G} e_{st,t}$. Indeed, for any $s,r \in G$, we have
$$
J(\lambda_s)J(\lambda_r)
=\bigg(\sum_{t \in G} e_{st,t}\bigg)\bigg(\sum_{u \in G} e_{ru,u}\bigg)
=\sum_{t,u \in G} e_{st,t} e_{ru,u}
=\sum_{u \in G} e_{sru,u}
=J(\lambda_{sr})
$$
and
$$
J(\lambda_s)^*
=\bigg(\sum_{t \in G} e_{st,t}\bigg)^*
=\sum_{t \in G} e_{t,st}
=\sum_{t \in G} e_{s^{-1}t,t}
=J(\lambda_{s^{-1}})
=J(\lambda_s^*).
$$
Furthermore, for any $s \in G$, we have
$$
\tr J(\lambda_s)
=\tr\bigg(\sum_{t \in G} e_{st,t}\bigg)
=\sum_{t \in G} \tr(e_{st,t})
=\sum_{t \in G} \delta_{st=t}
=\sum_{t \in G} \delta_{s=e}
=|G| \tau_G(\lambda_s).
$$
Hence $\frac{\tr}{|G|} \circ J=\tau_G$, i.e.~$J$ preserves the \textit{normalized} traces. We infer that
\begin{align*}
\MoveEqLeft
\H\bigg(\sum_{s \in G} \varphi(s)\lambda_s\bigg)            
\ov{\eqref{Preservation-entropy}}{=}\H\bigg(J\bigg(\sum_{s \in G} \varphi(s)\lambda_s\bigg),\frac{\tr}{|G|}  \bigg)
=\H\bigg(\sum_{s \in G} \varphi_sJ(\lambda_s),\frac{\tr}{|G|}  \bigg)\\
&=\H\bigg(\sum_{s,t \in G} \varphi_se_{st,t},\frac{\tr}{|G|}  \bigg)
=\H\bigg(C_\varphi,\frac{\tr}{|G|}  \bigg)
 \ov{\eqref{changement-de-trace}}{=} \H\bigg(\frac{1}{|G|}C_\varphi\bigg)-\log_2 |G|.
\end{align*}
\end{proof}

If $\varphi \co G \to \mathbb{C}$ is a positive definite function with $\varphi(e)=1$, it is well-known that the Herz-Schur multiplier $M_\varphi^\HS \co \B(\ell^2_G) \to \B(\ell^2_G)$ is a quantum channel. It is stated in \cite[p.~114]{GJL18a} that 
\begin{equation}
\label{Hcbmin-HS}
-\H_{\cb,\min}\big(M_\varphi^\HS\big)
=\log_2 |G|-\H\bigg(\frac{1}{|G|}C_\varphi\bigg).
\end{equation}

Now, we have the following transference result.

\begin{prop}
\label{Prop-transference-1}
Let $G$ be a finite group. Suppose $1 < p \leq \infty$. Let $\varphi \co G \to \mathbb{C}$ be a complex function.  
We have
\begin{equation}
\label{First-equality}
\norm{M_\varphi^\HS}_{S^1_G \to S^p_G}
\geq |G|^{\frac{1}{p}-1}\norm{\sum_{s \in G} \varphi(s)\lambda_s}_{\L^{p}(\VN(G))}. 
\end{equation} 
\end{prop}

\begin{proof}
Using the $*$-homomorphism $J \co \VN(G) \to \B(\ell^2_G)$ of the proof of Lemma \ref{Lemma-entropy} 
we obtain
\begin{align*}
\MoveEqLeft
\norm{\sum_{s \in G} \varphi(s)\lambda_s}_{\L^{p}(\VN(G))}         
\ov{\eqref{norm-equality-QG-2}}{=} \norm{M_\varphi}_{\L^1(\VN(G)) \to \L^p(\VN(G))}
\leq |G|^{1-\frac{1}{p}}\norm{M_\varphi^\HS}_{S^1_{G} \to S^p_{G}}.
\end{align*}
\end{proof}

In particular, by differentiation we obtain the (optimal) inequality
\begin{equation}
\label{minmin}
-\H\bigg(\frac{1}{|G|}C_\varphi\bigg) 
\leq -\H_{\min}\big(M_\varphi^\HS\big).
\end{equation} 
For the entanglement-assisted classical capacity, we only obtain the following (sharp) inequality. 

\begin{thm}
\label{Th-CEA-Herz-Schur}
Let $G$ be a finite group. The entanglement-assisted classical capacity of a Herz-Schur multiplier $M_\varphi^\HS \co S^1_G \to S^1_G$ defined by a positive definite function $\varphi$ with $\varphi(e)=1$ satisfies
\begin{equation}
\label{}
\C_{\EA}(M_\varphi^\HS)
\leq 2\log_2|G|-\H\bigg(\frac{1}{|G|}C_\varphi\bigg).
\end{equation}
\end{thm}

\begin{proof}
Let $n=|G|$. Using \cite[Theorem 1.1]{JuP15} in the first equality, we have
\begin{align*}
\MoveEqLeft
\C_{\EA}(M_\varphi^\HS)         
=\frac{1}{\log(2)}\frac{\d}{\d p} \Big[\bnorm{\big(M_\varphi^\HS\big)^*}_{\pi_{p^*,\circ}, \B(\ell^2_G) \to \B(\ell^2_G)}\Big]|_{p=1} \\
&\ov{\eqref{elem-lim}}{=} \lim_{q \to \infty} q\log_2 \bnorm{\big(M_\varphi^\HS\big)^*}_{\pi_{q,\circ}, \B(\ell^2_G) \to \B(\ell^2_G)} 
\ov{\eqref{Summing-Schur}}{\leq} \lim_{q \to \infty} q\log_2 n^{\frac{1}{q}} \bnorm{\big(M_\varphi^\HS\big)^*}_{\cb,S^q_G \to \B(\ell^2_G)} \\
&=\log_2(n)+\lim_{q \to \infty} q\log_2 \bnorm{\big(M_\varphi^\HS\big)^*}_{\cb,S^{q}_G \to \B(\ell^2_G)} 
=\log_2(n)+\lim_{q \to \infty} q\log_2 \norm{M_\varphi^\HS}_{\cb,S^{1}_G \to S^{q^*}_G} \\
&\ov{\eqref{elem-lim}}{=} \log_2(n)+\frac{1}{\log(2)}\frac{\d}{\d p}  \Big[\norm{M_\varphi^\HS}_{\cb,S^{1}_G \to S^{p}_G}\Big]|_{p=1}
\ov{\eqref{Hcbmin-HS}}{=}2\log_2(n)-\H\bigg(\frac{1}{n}C_\varphi\bigg).
\end{align*}
%
\end{proof}

\begin{example} \normalfont
If $0 \leq t \leq 1$, consider the dephasing channel $T \co \M_2 \to \M_2$, $x \mapsto (1-t)x+tZxZ$ of \cite[p.~155]{Wil17} defined with the Pauli matrix $Z
\ov{\mathrm{def}}{=} \begin{bmatrix}
   1  &  0 \\
   0  &  -1 \\
\end{bmatrix}$. With $G=\Z/2\Z=\{\ovl{0},\ovl{1}\}$, it is elementary to check that this map identifies to the Herz-Schur multiplier of symbol $C_\varphi=\begin{bmatrix}
   1  &  1-2t \\
   1-2t  &  1 \\
\end{bmatrix}$ where $\varphi \co G \to \mathbb{C}$ is defined by $\varphi(\ovl{0})=1$ and $\varphi(\ovl{1})=1-2t$. The \textit{opposite} of the von Neumann entropy of the matrix $\frac{1}{|G|}C_\varphi=\begin{bmatrix}
   \frac{1}{2}  &  \frac{1}{2}-t \\
   \frac{1}{2}-t  &  \frac{1}{2} \\
\end{bmatrix}$ is\footnote{\thefootnote. The eigenvalues of a matrix $\begin{bmatrix}
   b  &  a \\
   a &  b \\
\end{bmatrix}$ are $b-a$ and $b+a$.} $t\log_2 t+\big( 1-t\big)\log_2\big(1-t\big)$. 
Our result for the entanglement assisted capacity gives the inequality
$$
\C_{\EA}(T)
\leq 2+t\log_2 t+(1-t)\log_2(1-t).
$$
Actually, by \cite[Exercise 21.6.2 p.~597]{Wil17} this inequality is an equality. Our first method of Section \ref{Sec-Dihedral} also gives this value (and the one of Section \ref{summing-radial-multipliers} too).
Recall that the quantum capacity is equal by \cite[Exercise 24.7.1 p.~666]{Wil17} to
\begin{equation}
\label{quantum-cap-dephasing}
\Q(T)
=1+t\log_2 t+(1-t)\log_2(1-t).
\end{equation}
Example \ref{ex-estimating-quantum} gives the estimate $\Q(T) \leq \log_2 \bgnorm{\begin{bmatrix}
   1  &  |1-2t| \\
   |1-2t|  &  1 \\
\end{bmatrix}}_{\M_2}
=\log_2(1  +  |1-2t|)$. In particular, if $t=\frac{1}{2}$, we recover that the quantum capacity is zero.
\end{example}

Now, we recover a sharp inequality on the classical capacity $\C(T)$ of the dephasing channel obtained in \cite[Example 8.10 p.~153]{Hol19} (or \cite[Example 1]{Amo22}). The computation of this reference relies on the formula $\C(T)=1-\H_{\min}(T)$ for a unital qubit channel $T \co S^1_2 \to S^1_2$ (see also \cite{Kin02}).

\begin{cor}
The classical capacity $\C(T)$ of the dephasing channel $T \co S^1_2 \to S^1_2$, $x \mapsto (1-t)x+tZxZ$ satisfies
$$
\C(T)
\geq 1+t\log_2 t+(1-t)\log_2(1-t).
$$
\end{cor}
%

\begin{remark} \normalfont
\label{theta-mu}
Let $\QG$ be a locally compact quantum group. Note that the fundamental unitary operator $W$ of $\B(\L^2(\QG) \ot_2 \L^2(\QG))$ induces a map $\tilde{\Delta} \co \B(\L^2(\QG)) \to \B(\L^2(\QG) \ot_2 \L^2(\QG))$, $x \mapsto W(x \ot \Id_{\L^2(\QG)})W^*$ such that $\tilde{\Delta}|\L^\infty(\QG)=\Delta$. 
Recall that each quantum measure $\mu \in \M_u(\QG)$ induces an operator $m^r_\mu \co \L^1(\QG) \to \L^1(\QG)$, $f \mapsto f*\mu$. We consider the operator $\Phi_\mu \ov{\mathrm{def}}{=} (m^r_\mu)^*$. It follows essentially of \cite[Theorem 4.5]{JNR09} that the map $\Phi_\mu \co \L^\infty(\QG) \to \L^\infty(\QG)$ has a unique weak* continuous unital completely positive extension $\Theta(\mu) \co \B(\L^2(\QG)) \to \B(\L^2(\QG))$ such that 
$$
\tilde{\Delta} \circ \Theta(\mu)
=(\Id \ot \Phi_\mu) \circ \tilde{\Delta}.
$$
See also \cite[(2.3)]{KNR14} and \cite[Section 4]{SaS17}. A \textit{part} of the arguments of this section applies more generally to the operator $\Theta(\mu)$. 
\end{remark}

\section{Entanglement breaking and PPT multipliers}
\subsection{Multiplicative domains of Markov maps}
\label{Sec-Multiplicative}


Let $A,B$ be $\mathrm{C}^*$-algebras. Let $T \co A \to B$ be a positive map. Recall that the left multiplicative domain \cite[p.~38]{Pau02} of $T$ is the set 
\begin{equation}
\label{Def-multiplicative-domain}
\LMD(T) 
\ov{\mathrm{def}}{=} \big\{ x \in A : T(xy)=T(x)T(y) \ \text{ for all } y \in A \big\}.
\end{equation}
By \cite[Theorem 3.18]{Pau02} \cite[Definition 2.1.4]{Sto13}, the right multiplicative domain $\RMD(T)$ of a completely positive map is\footnote{\thefootnote. The result remains true if $T$ is a Schwarz map.} a subalgebra of $A$. The right multiplicative domain\footnote{\thefootnote. We warn the reader that this notion is called multiplicative domain in \cite[Definition 2.1.4]{Sto13}.} is defined by 
\begin{equation}
\label{Def-right-multiplicative-domain}
\RMD(T) 
\ov{\mathrm{def}}{=} \big\{ x \in A : T(yx)=T(y)T(x) \ \text{ for all } y \in A \big\}
\end{equation}
We have $\RMD(T)^*=\LMD(T)$. The multiplicative domain is $\MD(T) \ov{\mathrm{def}}{=}\LMD(T) \cap \RMD(T)$. If $T \co A \to A$ is a map, we introduce the fixed point set
\begin{equation}
\label{Def-fixed}
\Fix(T) 
\ov{\mathrm{def}}{=}  \{ x \in A : T(x) = x\}.
\end{equation}

Recall that the definite set $\mathrm{D}_T$ \cite[Definition 2.1.4]{Sto13} of a positive map $T \co A \to A$ on a $\mathrm{C}^*$-algebra $A$ is defined by 
\begin{equation}
\label{Def-definite-set}
\mathrm{D}_T
\ov{\mathrm{def}}{=} \big\{x \in A_\sa : T(x^2)=T(x)^2\big\}.
\end{equation}

Let $(\cal{M},\phi)$ and $(\cal{N},\psi)$ be von Neumann algebras equipped with normal faithful states $\phi$ and $\psi$, respectively. A linear map $T \co \cal{M} \to \cal{N}$ is called a $(\phi,\psi)$-Markov map if
\begin{enumerate}
\item [$(1)$] $T$ is completely positive
\item [$(2)$] $T$ is unital
\item [$(3)$] $\psi \circ T=\phi$
\item [$(4)$] for any $t \in \R$ we have $T \circ \sigma_t^{\phi}=\sigma_t^\psi \circ T$ where $(\sigma_t^{\phi})_{t \in \R}$ and $(\sigma_t^{\psi})_{t \in \R}$ denote the automorphism groups of the states $\phi$ and $\psi$, respectively.
\end{enumerate}
When $(\cal{M},\phi)=(\cal{N},\psi)$, we say that $T$ is a $\phi$-Markov map. 
Note that a linear map $T \co \cal{M} \to \cal{N}$ satisfying previous conditions $(1)-(3)$ is automatically weak* continuous. If, moreover, condition $(4)$ is satisfied, then it is well-known, see \cite[p.~556]{HaM11} that there exists a unique completely positive, unital map $T^* \co \cal{N} \to \cal{M}$ such that
\begin{equation}
\label{Def-adjoint}
\phi(T^*(y)x) 
=\psi(yT(x)), \quad x \in \cal{M}, y \in \cal{N}. 
\end{equation}
It is easy to show that $T^*$ is a $(\psi,\phi)$-Markov map.
%
%

The following observation is a generalization of \cite{Rah17} (note that the result of \cite{Rah17} was known before, see \cite{CJK09}) proved with a rather similar argument. With the terminology of \cite[p~11]{Izu02} (see also \cite{Izu04}), it says that the multiplicative domain of $T$ is the Poisson boundary of the Markov operator $T^* \circ T \co \cal{M} \to \cal{M}$. 

\begin{lemma}
\label{Lemma-MDfix} 
Let $\cal{M}$ and $\cal{N}$ be von Neumann algebras equipped with normal faithful states $\phi$ and $\psi$. Let $T \co \cal{M} \to \cal{N}$ be a $(\phi,\psi)$-Markov map. Then 
$$
\LMD(T)
=\RMD(T)
=\MD(T)
= \Fix(T^* \circ T).
$$
\end{lemma}

\begin{proof} 
Let $x \in \LMD(T)$. Using the preservation of the states in the first equality, we have for any $y \in \cal{M}$
\begin{align*} 
\MoveEqLeft
\phi(xy) 
=\psi(T(xy)) 
\ov{\eqref{Def-multiplicative-domain}}{=} \psi\big(T(x)T(y)\big)
\ov{\eqref{Def-adjoint}}{=} \phi\big(T^*(T(x))y\big).
\end{align*}
Hence $\phi((x-T^*(T(x)))y)=0$. Taking $y=(x-T^*(T(x)))^*$, we obtain $x=(T^* \circ T)(x)$ by faithfulness of $\phi$. So $x$ belongs to $\Fix(T^* \circ T)$.

Conversely, let $x \in \cal{M}$ such that $T^*\circ T(x)=x$. Using the preservation of the states in the first equality and Schwarz inequality \cite[p~14]{Sto13}, we obtain 
\begin{align*}
\MoveEqLeft
\phi(x^*x)
=\psi(T(x^*x))
\geq \psi\big(T(x)^{*}T(x)\big) 
\ov{\eqref{Def-adjoint}}{=} \phi\big(T^*(T(x)^*) x\big) 
=\phi\big( (T^*(T(x)))^* x\big)
=\phi(x^* x).
\end{align*}
Since the extreme ends of the previous equations are identical, the inequality is an equality.
Hence we have 
$$
\psi(T(x^*x))
=\psi\big(T(x)^* T(x)\big).
$$
By the faithfulness of the state $\psi$ and Schwarz inequality, we infer that $T(x^*x)=T(x)^* T(x)$. By \cite[Theorem 3.18]{Pau02} \cite[Proposition 2.1.5]{Sto13}, we conclude that $x \in \LMD(T)$. The proof is similar for $\RMD(T)$. We also can use the equality $\RMD(T)^*=\LMD(T)$.
\end{proof}


We determine the adjoint of the multiplier $R_\phi \co \cal{M} \to \cal{M}$ defined in \eqref{def-rad-mult}.

\begin{prop}
For any complex function $\phi \co G \to \mathbb{C}$, we have $(R_\phi)^*=R_\phi$ for the bracket $\la x,y \ra= \tau(x y)$ and $(R_\phi)^*=R_{\ovl{\phi}}$ for the bracket $\la x, y\ra= \tau(x y^*)$.
\end{prop}

\begin{proof}
For any $s,t \in G$, we have
\begin{align*}
\MoveEqLeft
\tau(R_\phi(v_s) v_t)         
\ov{\eqref{def-rad-mult}}{=} \phi(s) \tau(v_s v_t) 
\ov{\eqref{calculus-Vilenkin-prod}}{=} \phi(s) \sigma(s,t) \tau(v_{s t}) 
\ov{\eqref{trace-fer}}{=} \phi(s) \sigma(s,t) \delta_{st,1}
=\phi(s) \sigma(s,s^{-1}) \delta_{t,s^{-1}}.
\end{align*}
By symmetry, this computation also shows that
\begin{align*}
\MoveEqLeft
\tau(v_s R_\phi(v_t))         
=\phi(\gamma) \sigma(t,t^{-1}) \delta_{s,t^{-1}}
=\phi(s^{-1}) \sigma(s^{-1},s) \delta_{t,s^{-1}}.
\end{align*}
We conclude since the 2-cocycle $\sigma$ satisfies \eqref{ss-1-2-cocycle}. The second assertion is a consequence of \cite[(2.3)]{JMX06}.
\end{proof}


The next result describes the multiplicative domain of multipliers.

\begin{prop}
Consider a function $\phi \co G \to \mathbb{C}$. An element $\sum_s x_s v_s$ of the von Neumann algebra $\cal{M}$ belongs to the multiplicative domain $\MD(R_\phi)$ of the multiplier $R_\phi \co \cal{M} \to \cal{M}$ if and only if for any $s \in G$ the condition $x_s \not=0$ implies $|\phi(s)|=1$. 
\end{prop}

\begin{proof}
For any $s \in G$, we have
\begin{equation*} 
\big((R_\phi)^* \circ R_{\phi}\big)(v_s)
=(R_{\phi} \circ R_{\phi})(v_s)
=\phi(s) R_{\phi}(v_s)
=|\phi(s)|^2 v_s.
\end{equation*}
Thus, the fixed points of the multiplier $R_\phi^* \circ R_\phi$ are elements of the form $x = \sum_s x_s v_s$, where the coefficients $x_s$ are non-zero only for elements $s \in G$ satisfying $|\phi(s)| = 1$.
\end{proof}

Let $G$ be a locally compact group. Suppose that $\varphi \co G \to \mathbb{C}$ is  a positive definite function with $\varphi(e)=1$. Note that the Fourier multiplier $M_\varphi \co \VN(G) \to \VN(G)$ is a $\tau$-Markov map where $\tau$ is the canonical normalized trace of the group von Neumann algebra $\VN(G)$. By \cite[Corollary 32.7]{HeR70} \cite[Corollary 1.4.22]{KaL18}, the set $G_\varphi \ov{\mathrm{def}}{=} \{ s \in G : |\varphi(s)| = 1 \}$ is a subgroup of $G$ and for any $s \in G_\varphi$ and any $t \in G$, we have
\begin{equation}
\label{}
\varphi(st) 
=\varphi(ts)
=\varphi(t) \varphi(s).
\end{equation}
In particular, $\varphi$ is a character of the group $G_\varphi$.

The following result describes the multiplicative domains of Fourier multipliers which are quantum channels.

\begin{prop}
\label{Prop-mult-Fourier-multipliers}
Let $G$ be a discrete group. Let $\varphi \co G \to \mathbb{C}$ be a positive definite function with $\varphi(e)=1$. An element $\sum_s x_s\lambda_s$ of the group von Neumann algebra $\VN(G)$ belongs to the multiplicative domain $\MD(M_\varphi)$ of the Fourier multiplier $M_\varphi \co \VN(G) \to \VN(G)$ if and only if for any $s \in G$ the condition $x_s \not=0$ implies $|\varphi(s)| = 1$. That is
\begin{equation}
\label{}
\MD(M_\varphi)
=\VN(G_\varphi). 
\end{equation}
\end{prop}

\begin{proof}
It is easy to check\footnote{\thefootnote. For any $s,t \in G$, we have $\tau(M_\varphi(\lambda_s)\lambda_t)=\varphi(s)\tau(\lambda_s\lambda_t)=\varphi(s)\delta_{s=t^{-1}}$ and $\tau(\lambda_s M_{\check{\varphi}}(\lambda_t))=\varphi(t^{-1})\tau(\lambda_s\lambda_t)=\varphi(t^{-1})\delta_{s=t^{-1}}$.} that $(M_\varphi)^*=M_{\check{\varphi}}$. By \cite[Theorem 32.4 (iii)]{HeR70}, we obtain that $(M_\varphi)^*=M_{\ovl{\varphi}}$. For any $s \in G$, we have
\begin{equation*} 
\big((M_\varphi)^* \circ M_{\varphi}\big)(\lambda_s)
=(M_{\ovl{\varphi}} \circ M_{\varphi})(\lambda_s)
=\varphi(s) M_{\ovl{\varphi}}(\lambda_s)
=|\varphi(s)|^2 \lambda_s.
\end{equation*}
Hence the fixed points of the Fourier multiplier $M_\varphi^* \circ M_{\varphi}$ are elements $x=\sum_s x_s\lambda_s$ with non-zero coefficients only at indices $s$ where $|\varphi(s)|=1$, that is $s \in G_\varphi$. We conclude with Lemma \ref{Lemma-MDfix}. The last formula is a consequence of \cite[Proposition 2.4.1]{KaL18}.
\end{proof}

\begin{remark} \normalfont
Note, that by \cite[Proposition 3.2.10]{ChL02} the von Neumann algebra $\VN(G_\varphi)$ is the Poisson boundary of $\varphi$ (or $M_\varphi$ with the terminology of \cite[p.~11]{Izu02}).
\end{remark}


\subsection{PPT and entangling breaking maps}
\label{Sec-PPT}

We refer to \cite{HSR03}, \cite{Kin03}, \cite{Kur18}, \cite{Rus03}, \cite{Sto13} and references therein for more information on entanglement breaking maps. The first part of the following definition is a generalization of \cite[Theorem 4 (E)]{HSR03}.

\begin{defi}
\label{Def-PPT}
Let $A,B$ be $\mathrm{C}^*$-algebras. Let $T \co A \to B$ be a linear map.

\begin{enumerate}
\item We say that $T$ is entanglement breaking if for any $\mathrm{C}^*$-algebra $C$ and any positive map $P \co B \to C$ the map $P \circ T \co A \to C$ is completely positive. 

\item We say that $T$ is $\PPT$ if the maps $T \co A \to B$ and $T \co A \to B^\op$ are completely positive.

\end{enumerate}
\end{defi}

\begin{remark} \normalfont
\label{Rem-ent-break-PPT}
An entanglement breaking map is $\PPT$ (in particular completely positive). It suffices to compose $T$ with the identity $\Id_{B}$ and the positive map $B \to B^\op$, $x \mapsto x^\op$.
\end{remark}


The following is a generalization of \cite[Lemma 3.1]{RJV18}. We give a simpler proof.

\begin{lemma}
\label{Lemma-commutative-range}
Let $A,B$ be unital $\mathrm{C}^*$-algebra. Let $T \co A \to B$ be a unital completely positive map such that the range $\Ran(T)$ is contained in an abelian $\mathrm{C}^*$-algebra $B_1$, then $T$ is entanglement breaking.
\end{lemma}

\begin{proof}
Let $P \co B \to C$ be a positive map where $C$ is a $\mathrm{C}^*$-algebra. Note that $1$ belongs to $\Ran(T)$, hence belongs to $B_1$. The restriction $P|B_1 \co B_1 \to C$ is positive hence completely positive by \cite[Theorem 5.15]{EfR00}. We conclude that the composition $P \circ T \co A \to C$ is completely positive.
\end{proof}

The following is a generalization of \cite[Lemma 3.1]{RJV18}. We replace the complicated computations of the proof by a simple conceptual argument which gives at the same time a more general result. Recall that a positive map $T$ is faithful \cite[p~10]{Sto13} if $x \geq 0$ and $T(x)=0$ imply $x=0$.  

\begin{thm}
\label{Th-abelian-PPT}
Let $\cal{M}$ be a von Neumann algebra equipped with a normal faithful state $\phi$. Let $T \co \cal{M} \to \cal{M}$ a unital $\PPT$ map. Then $T(\MD(T))$ is an abelian algebra contained in the center of the von Neumann algebra generated by the $\Ran T$. If $T$ is in addition faithful, then $\MD(T)$ is abelian.  
\end{thm}

\begin{proof}
Assume there exists a projection $p \in M_T$. We have $T(p)^2 \ov{\eqref{Def-multiplicative-domain}}{=} T(p)$ and $T(p)^*=T(p^*)=T(p)$. So $T(p)$ is a projection. Note that using $p \in \RMD(T)$ in the last equality
\begin{align*}
\MoveEqLeft
T(p)\cdot_{\op} T(p)
=T(p)T(p)           
\ov{\eqref{Def-right-multiplicative-domain}}{=} T(p^2).
\end{align*}
So by \eqref{Def-definite-set}, the selfadjoint element $p$ belongs to the definite set $\mathrm{D}_T$ of the map $T \co \cal{M} \to \cal{M}^\op$. By \cite[Proposition 2.1.5 p.~14]{Sto13}, we deduce that $p$ belongs to $\RMD(T \co \cal{M} \to \cal{M}^\op)$ and to $\LMD(T \co \cal{M} \to \cal{M}^\op)$ (since $p$ is selfadjoint). Using $p \in \LMD(T)$ in the first equality and $p \in \LMD(T \co \cal{M} \to \cal{M}^\op)$ in the second equality, we infer that for any $x \in \cal{M}$
\begin{equation}
\label{}
T(x)T(p)
=T(xp)
=T(x) \cdot_{\op} T(p)
=T(p)T(x).
\end{equation}
We let $q \ov{\mathrm{def}}{=} 1-p$. 
We deduce that for any $x \in \cal{M}$
\begin{align*}
\MoveEqLeft
T(x)
=T((p+q)x(p+q)) 
=T(pxp)+T(qxp)+T(pxq) +T(qxp)\\
&=T(pxp)+T(q)T(x)T(p)+T(p)T(x)T(q) +T(qxp)\\
&=T(p)T(x)T(p)+T(q)T(x)T(q)
\end{align*}
where we use $T(p)T(q)=0=T(q)T(p)$ in the last equality. We deduce easily that
$$
T(p)T(x)
=T(x)T(p).
$$
Hence each $T(p)$ belongs to the commutant of $\Ran T$ and obviously to $\Ran T$.
We conclude that $T(p)$ belongs to the center of the von Neumann generated by $\Ran T$. 

If $T$ is in addition faithful, restricting $T$ on $\MD(T)$ we obtain a $*$-isomorphism between $\MD(T)$ and $T(\MD(T))$. As $T(\MD(T))$ is abelian, it follows that $\MD(T)$ is also abelian.
\end{proof}

The following is a generalization of \cite[Theorem 3.3]{RJV18}. As before, we simplify the proof.

\begin{cor}
\label{Cor-esp-cond}
Let $\cal{M}$ be a von Neumann algebra equipped with a normal faithful state $\phi$. Let $\E \co \cal{M} \to \cal{M}$ be a state preserving normal conditional expectation. The following properties are equivalent.
\begin{enumerate}
	\item $\E$ is entangling breaking. 
	\item $\E$ is $\PPT$.
	\item the range $\Ran \E$ is abelian.
\end{enumerate}
\end{cor}

\begin{proof}
1. $\Rightarrow$ 2.: It is Remark \ref{Rem-ent-break-PPT}.

2. $\Rightarrow$ 3.: Such a conditional expectation is faithful. It suffices to use Theorem \ref{Th-abelian-PPT}. By \cite[p~.116]{Str81}, the range $\Ran \E$ of $\E$ is contained in $\MD(\E)$ which is abelian. 

3. $\Rightarrow$ 1.: It suffices to use Lemma \ref{Lemma-commutative-range}. 
\end{proof} 

\begin{example} \normalfont
We deduce that the identity $\Id_{\cal{M}} \co \cal{M} \to \cal{M}$ is entangling breaking if and only if it is is $\PPT$ if and only if $\cal{M}$ is abelian. We recover in particular \cite[Exercise 3.9 (vi)]{Pau02} for von Neumann algebras.
\end{example}

Now, we give a necessary condition for the property $\PPT$ of a Fourier multiplier. 

\begin{thm}
Let $G$ be a discrete group. Let $\varphi \co G \to \mathbb{C}$ be a function with $\varphi(e)=1$. If the Fourier multiplier $M_\varphi \co \VN(G) \to \VN(G)$ is a $\PPT$ quantum channel then the group $G_{\varphi}$ is abelian.
\end{thm}

\begin{proof}
By Theorem \ref{Th-abelian-PPT}, $M_\varphi(\MD(M_\varphi))$ is an abelian algebra. By Proposition \ref{Prop-mult-Fourier-multipliers}, we infer that $M_\varphi(\VN(G_\varphi))$ is abelian and it is easy to check that $M_\varphi(\VN(G_\varphi))=\VN(G_\varphi)$. We conclude that $\VN(G_\varphi)$ is abelian, hence $G_\varphi$ also.
\end{proof}

\begin{example} \normalfont
Let $G$ be discrete group and $H$ be a subgroup of $G$. We can identify the von Neumann algebra $\VN(H)$ as a subalgebra of $\VN(G)$. We can consider the canonical trace preserving normal conditional expectation $\E \co \VN(G) \to \VN(G)$, $\sum_{s \in G} a_s \lambda_s \mapsto \sum_{s \in H} a_s \lambda_s$ onto $\VN(H)$. It is obvious that the map $\E$ is a Fourier multiplier with symbol $1_H$. We have $G_{1_H}=\{ s \in G : |1_H(s)| = 1 \}=H$. We deduce with the results of this section that the quantum channel $\E$ is $\PPT$ if and only if $\E$ is entangling breaking if and only if $H$ is abelian.
\end{example}

\section{Appendices}
\label{Appendices}

\subsection{Appendix: output minimal entropy as derivative of operator norms}
\label{Appendix}

Here, we give a proof of \eqref{Smin-as-derivative-intro} in the appendix for finite-dimensional von Neumann algebras, see Theorem \ref{Th-Smin-as-derivative}. This part does not aim at originality. Indeed, the combination of \cite{ACN00} and \cite{Aud09} gives a proof of this result for algebras of matrices. So, our result is a mild generalization. 

Suppose $1 \leq p,q \leq \infty$. Let $\cal{M},\cal{N}$ be von Neumann algebras equipped with normal semifinite faithful traces. Consider a positive linear map $T \co \L^q(\cal{M}) \to \L^p(\cal{N})$. By an obvious generalization of the proof of \cite[Proposition 2.18]{ArK23}, such a map is bounded. We define the positive real number
\begin{equation}
\label{Def-T+qp}
\norm{T}_{+,\L^q(\cal{M}) \to \L^p(\cal{N})}
\ov{\mathrm{def}}{=}
\sup_{ x \geq 0, \norm{x}_{q}=1} \norm{T(x)}_p.
\end{equation}
We always have $\norm{T}_{+,\L^q(\cal{M}) \to \L^p(\cal{N})} \leq \norm{T}_{\L^q(\cal{M}) \to \L^p(\cal{N})}$. We start to show that if $T$ is in addition 2-positive (e.g. completely positive) then we have an equality. See \cite{Aud09} and \cite{Sza09} for a version of this result for maps acting on Schatten spaces.

\begin{prop}
\label{prop-norm-pos} 
Let $\cal{M},\cal{N}$ be von Neumann algebras equipped with normal semifinite faithful traces. Suppose $1 \leq p,q \leq \infty$. Let  $T \co \L^q(\cal{M}) \to \L^p(\cal{N})$ be a 2-positive map. We have
\begin{equation}
\label{qp=qp+}
\norm{T}_{\L^q(\cal{M}) \to \L^p(\cal{N})} 
=\norm{T}_{+,\L^q(\cal{M}) \to \L^p(\cal{N})}.
\end{equation}
\end{prop}

\begin{proof}
Suppose $\norm{T}_{+,q \to p} \leq 1$. Let $x \in \L^q(\cal{M})$ with $\norm{x}_q=1$. By \eqref{inverse-Holder} (and a suitable multiplication by a constant), there exist $y,z \in \L^{2q}(\cal{M})$ such that
\begin{equation}
\label{Ref-divers-9877} 
x=y^*z 
\quad \text{with} \quad
\norm{y}_{2q} 
=\norm{z}_{2q}
=1.
\end{equation}
We have
$
\begin{bmatrix} 
0   & 0
\\   y   & z 
\end{bmatrix}^*
\begin{bmatrix} 
0   & 0
\\   y   & z 
\end{bmatrix}
=
\begin{bmatrix} 
0   & y^*
\\   0   & z^* 
\end{bmatrix}
\begin{bmatrix} 0   & 0
\\   y   & z 
\end{bmatrix}
=\begin{bmatrix} 
y^*y   & y^*z
\\   z^*y   & z^*z 
\end{bmatrix}$. We infer that $
\begin{bmatrix} 
y^*y   & y^*z
\\   z^*y   & z^*z 
\end{bmatrix}$ is a positive element of $S^q_2(\L^q(\cal{M}))$. Since $T$ is 2-positive, the matrix 
$
(\Id_{\M_2} \ot T)\left(\begin{bmatrix} 
y^*y   & y^*z
\\   z^*y   & z^*z 
\end{bmatrix}\right)
=
\begin{bmatrix}
T(y^*y)    & T(y^*z) \\   
T(z^*y)    & T(z^*z) \end{bmatrix}
$ 
is positive. Moreover, we have $
\norm{T(y^*y)}_p 
\leq \norm{T}_{+,q \to p} \norm{y^*y}_{q}
\leq \norm{y^*y}_{q}
= \norm{y}_{2q}^2 
\ov{\eqref{Ref-divers-9877} }{=} 1$ and similarly $\norm{T(z^*z)}_p \leq 1$. Using \cite[Lemma 2.13]{ArK23} in the first inequality, we deduce that $
\norm{T(x)}_{p}
\ov{\eqref{Ref-divers-9877}}{=} \norm{T(y^*z)}_{p} \leq \sqrt{\norm{T(y^*y)}_p\norm{T(z^*z)}_p}
\leq 1$. Hence $\norm{T}_{q \to p} \leq 1$. By homogeneity, we conclude that $\norm{T}_{q \to p} \leq \norm{T}_{+,q \to p}$. The reverse inequality is obvious.
\end{proof}

The following proposition relying on Dini's theorem is elementary and left to the reader. The first part is already present in \cite[p.~625]{Seg60} without proof.

\begin{prop}
\label{prop-deriv-norm-p}
Let $\cal{M}$ be a finite von Neumann algebra equipped with a normal finite faithful trace $\tau$. For any positive element $x \in \cal{M}$ with $\norm{x}_{\L^1(\cal{M})}=1$ and $-\tau(x \log_2 x) > -\infty$, we have
\begin{equation}
\label{deriv-norm-p}
\frac{1}{\log 2}\frac{\d}{\d p} \norm{x}_{\L^p(\cal{M})}|_{p=1}
=\tau(x \log_2 x)
\ov{\eqref{Segal-entropy}}{=} -\H(x).
\end{equation}
Moreover, if $\cal{M}$ is finite-dimensionaal the monotonically decreasing family of functions $x \mapsto \frac{1-\norm{x}_{\L^p(\cal{M})}^p}{p-1}$ converge uniformly to the function $x \mapsto -\tau(x \log x)=\H(x)$ when $p \to 1^{+}$ on the subset of positive element $x \in \cal{M}$ with $\norm{x}_1=1$.
\end{prop}

\begin{proof}
Consider the spectral representation $x = \int_{0}^{\infty} \lambda e(\d\lambda)$ of $x$. Note that
\begin{align*}
\MoveEqLeft
\frac{\norm{x}_{\L^p(\cal{M})}^p -1}{p - 1}
=\frac{1}{p - 1}[\tau(x^p) -\tau(x)]
=\frac{1}{p - 1}\bigg[\int_{0}^{\infty} \lambda^p \tau(e(\d \lambda))-\int_{0}^{\infty}\lambda \tau(e(\d \lambda))\bigg] \\
&=\int_{0}^{\infty} \frac{\lambda^p - \lambda}{p - 1} \tau(e(\d \lambda)).
\end{align*}
Consider the net $(f_p)_{p > 1}$  of functions defined by $f_p(\lambda) \ov{\mathrm{def}}{=}  \frac{\lambda^p - \lambda}{p - 1} $. A computation shows that for any $\lambda \geq 0$, we have $f_p(\lambda) \leq f_q(\lambda)$ if $p < q$. Hence this net is increasing. Moreover, we have
$$
\lim_{p \to 1^+} f_p(\lambda) 
= \lambda \log \lambda.
$$
Lebesgue Monotone Convergence Theorem gives
$$
\lim_{p \to 1^+} \int_{0}^{\infty} f_p(\lambda) \tau(e(\d\lambda)) 
= \int_{0}^{\infty} \lambda \log \lambda \tau(e(\d\lambda)) 
= \H(x).
$$
Hence, we have shown that the function $p \mapsto \tau(x^p)$ is differentiable at $p = 1$, and the formula \eqref{deriv-norm-p}.
\end{proof}

\begin{thm}
\label{Th-Smin-as-derivative}
Let $\cal{M},\cal{N}$ be finite-dimensional von Neumann algebras equipped with finite faithful trace. Let $T \co \cal{M} \to \cal{N}$ be a trace preserving 2-positive map. We have
\begin{equation}
\label{}
\H_{\min}(T)
=-\frac{\d}{\d p} \norm{T}_{\L^1(\cal{M}) \to \L^p(\cal{N})}^p|_{p=1}. 
\end{equation}
\end{thm}

\begin{proof}
We can suppose that $\cal{M}$ and $\cal{N}$ are $\not=\{0\}$. Note that $\norm{T}_{\L^1(\cal{M}) \to \L^1(\cal{N})}\leq 1$ by a classical argument \cite[Theorem 5.1]{HJX10}. Since $\norm{T(1)}_1=\tau_{\cal{N}}(T(1))=\tau_{\cal{M}}(1)=\norm{1}_1$, we have $\norm{T}_{\L^1(\cal{M}) \to \L^1(\cal{N})}=1$. Using the uniform convergence of Proposition \ref{prop-deriv-norm-p} in the fifth equality and the fact that $T$ is trace preserving in the sixth equality, we obtain 
\begin{align*}
\MoveEqLeft
-\frac{\d}{\d p} \norm{T}_{\L^1(\cal{M}) \to \L^p(\cal{N})}^p|_{p=1}
=\lim_{p \to 1^{+}} \frac{1 - \norm{T}_{1 \to p}^p}{p-1}
\ov{\eqref{qp=qp+}}{=} \lim_{p \to 1^{+}} \frac{1 -\norm{T}_{+,1 \to p}^p}{p-1} \\
&\ov{\eqref{Def-T+qp}}{=} \lim_{p \to 1^{+}} \frac{1-\max_{x \geq 0, \norm{x}_1=1} \norm{T(x)}_{\L^p(\cal{N})}^p}{p-1} 
=\lim_{p \to 1^{+}} \min_{x \geq 0, \norm{x}_{1}=1} \frac{1-\norm{T(x)}_{\L^p(\cal{N})}^p}{p-1} \\
&\ov{\eqref{minimum-uniform-convergence}}{=} \min_{x \geq 0, \norm{x}_1=1} \lim_{p \to 1^{+}} \frac{1-\norm{T(x)}_{\L^p(\cal{N})}^p}{p-1} 
=\min_{ x \geq 0, \norm{x}_1=1} \frac{\d}{\d p} -\norm{T(x)}_{\L^p(\cal{N})}^p|_{p=1} \\
&\ov{\eqref{deriv-norm-p}}{=} \min_{ x \geq 0, \norm{x}_1=1} \H(T(x))
\ov{\eqref{Def-minimum-output-entropy}}{=} \H_{\min}(T).
\end{align*}
\end{proof}

\begin{remark} \normalfont
\label{Remark-without-p}
We also have the formulas $\H_{\min}(T)
=-\frac{\d}{\d p} \norm{T}_{\L^1(\cal{M}) \to \L^p(\cal{N})}|_{p=1}$ and
\begin{equation}
\label{deriv-norm-p-bis}
\frac{\d}{\d p} \norm{x}_{\L^p(\cal{M})}|_{p=1}
=\tau(x \log x)
=-\H(x).
\end{equation}
for any positive element $x \in \cal{M}$ \textit{with} $\norm{x}_1=1$. Indeed, if $f$ is a differentiable strictly positive function with $f(1)=1$ and if $F(p)=f(p)^{\frac{1}{p}}$, an elementary computation\footnote{\thefootnote. We have $g'(p)=f(p)\big(-\frac{1}{p^2}\log f(p)+\frac{1}{p}\frac{f'(p)}{f(p)}\big)$.}  show that $F'(1)=f'(1)$.
\end{remark}

\begin{remark} \normalfont
With obvious notations, a standard argument gives
\begin{equation}
\label{subadditivity-Hmin}
\H_{\min}(T_1 \ot T_2)
\leq \H_{\min}(T_1)+\H_{\min}(T_2).
\end{equation}
\end{remark}

\subsection{Appendix: multiplicativity of some completely bounded operator norms}
\label{Appendix2}

In this appendix, we prove in Theorem \ref{thm:multqp} the multiplicativity of the completely bounded  norms $\norm{\cdot}_{\cb,\L^q \to \L^p}$ if $1 \leq q \leq p \leq \infty$ and we deduce the additivity of $\H_{\cb,\min}(T)$ in Theorem \ref{Additivity-Scb-min}. The multiplicativity result is stated in \cite[Theorem 11]{DJKRB06} for Schatten spaces (which imply the additivity result for matrix spaces). Unfortunately, the second part\footnote{\thefootnote. Note that the first part of the proof of \cite[Theorem 11]{DJKRB06} is also incorrect but can be corrected without additional ideas, see the proof of Theorem \ref{thm:multqp}.} of the proof of \cite[Theorem 11]{DJKRB06} is loosely false and consequently is not convincing. The only contribution of this section is to provide a correct proof. We also take this opportunity for generalize the result to the more general setting of hyperfinite von Neumann algebras. However, it does not remove the merit of the authors of \cite{DJKRB06} to introduce this kind of questions. See also \cite{Jen06} for a related paper. 

The new ingredient is the following result which is crucial for the proof of Theorem \ref{thm:multqp}. Note that this result is in the same spirit of shuffles of  \cite[Theorem 4.1 and Corollary 4.2]{EKR1}, \cite[Theorem 6.1]{EfR03}\footnote{\thefootnote. If $E_1,E_2,F_1,F_2$ are operator spaces, we have a complete contraction $(E_1 \ot_{\mathrm{eh}} F_1) \ot_{\mathrm{nuc}} (E_2 \ot_{\mathrm{eh}} F_2) \to (E_1 \ot_{\mathrm{nuc}} E_2) \ot_{\mathrm{eh}} (F_1 \ot_{\mathrm{nuc}} F_2)$.} and \cite{EKR93} which are well-known to the experts.

\begin{prop}
\label{Prop-magic-flip}
Let $\cal{M}_1$, $\cal{M}_2$, $\cal{M}_3$ and $\cal{M}_4$ be hyperfinite von Neumann algebras equipped with normal semifinite faithful trace. Suppose $1 \leq q \leq p \leq \infty$. The map
\begin{equation}
\label{magic-flip}
\Id \ot \mathcal{F} \ot \Id \co \L^p\big(\cal{M}_1,\L^q(\cal{M}_2)\big) \times \L^p(\cal{M}_3,\L^q(\cal{M}_4)\big) \to \L^p\big(\cal{M}_1 \otvn \cal{M}_3,\L^q(\cal{M}_2 \otvn \cal{M}_4)\big)
 \end{equation}
which maps $(x \ot y, z \ot t)$ on $x \ot z \ot y \ot t$ is a well-defined contraction.
\end{prop}

\begin{proof}
Recall that by \eqref{belle-injection-2} we have an isometry $\cal{M}_1 \ot_{\min} \L^1(\cal{M}_2) \xhookrightarrow{} \CB_{\w^*}(\cal{M}_2,\cal{M}_1)$, $a \ot b \mapsto (z \mapsto  \la b,z\ra a)$ where the subscript $\w^*$ means weak* continuous. Similarly, we have an isometry $\cal{M}_3 \ot_{\min} \L^1(\cal{M}_4) \xhookrightarrow{} \CB_{\w^*}(\cal{M}_4,\cal{M}_3)$ and $(\cal{M}_1 \otvn \cal{M}_3) \ot_{\min} \L^1(\cal{M}_2 \otvn \cal{M}_4) \xhookrightarrow{} \CB_{\w^*}(\cal{M}_2 \otvn \cal{M}_4,\cal{M}_1 \otvn \cal{M}_3)$. If the tensors $x \in \cal{M}_1 \ot \L^1(\cal{M}_2)$ and $y \in \cal{M}_3 \ot \L^1(\cal{M}_4)$ identify to the maps $u \co \cal{M}_2 \to \cal{M}_1$ and $v \co \cal{M}_4 \to \cal{M}_3$ then it is easy to check that the element $(\Id \ot \mathcal{F} \ot \Id)(x \ot y)$ of $(\cal{M}_1 \otvn \cal{M}_3) \ot_{\min} \L^1(\cal{M}_2 \otvn \cal{M}_4)$ identify to the map $u \ot v \co \cal{M}_2 \otvn \cal{M}_4 \to \cal{M}_1 \otvn \cal{M}_3$. By \cite[p~40]{BLM04}, we have
\begin{align*}
\norm{u \ot v}_{\cb,\cal{M}_2 \otvn \cal{M}_4 \to \cal{M}_1 \otvn \cal{M}_3}
& \leq \norm{u}_{\cb,\cal{M}_2 \to \cal{M}_1} \norm{v}_{\cb,\cal{M}_4 \to \cal{M}_3}.
\end{align*}
We deduce that
$$ 
\bnorm{(\Id \ot \mathcal{F} \ot \Id)(x \ot y)}_{(\cal{M}_1 \otvn \cal{M}_3) \ot_{\min} \L^1(\cal{M}_2 \otvn \cal{M}_4))} 
\leq \norm{x}_{\cal{M}_1 \ot_{\min} \L^1(\cal{M}_2)}\norm{y}_{\cal{M}_3 \ot_{\min} \L^1(\cal{M}_4))}. 
$$
Hence by \eqref{Def-L0-infty} the map $\L^\infty(\cal{M}_1,\L^1(\cal{M}_2)) \times \L^\infty(\cal{M}_3,\L^1(\cal{M}_4)) \xrightarrow{} \L^\infty(\cal{M}_1 \otvn \cal{M}_3,\L^1(\cal{M}_2 \otvn \cal{M}_4))$ is contractive. It is not difficult to prove  that we have a contractive bilinear map 
$$ 
\L^\infty(\cal{M}_1,\cal{M}_2) \times \L^\infty(\cal{M}_3,\cal{M}_4)
\xrightarrow{}
\L^\infty(\cal{M}_1 \otvn \cal{M}_3,\cal{M}_2 \otvn \cal{M}_4).
$$
By interpolating these two maps with \cite[(3.5)]{Pis98}, we obtain a contractive bilinear map
$$ 
\L^\infty\big(\cal{M}_1,\L^q(\cal{M}_2)\big)\times \L^\infty\big(\cal{M}_3,\L^q(\cal{M}_4)\big) \xrightarrow{}
\L^\infty\big(\cal{M}_1 \otvn \cal{M}_3,\L^q(\cal{M}_2 \otvn \cal{M}_4)\big).
$$
Note that by \cite[(3.6)]{Pis98} we have a contractive bilinear map
$$ 
\L^q\big(\cal{M}_1,\L^q(\cal{M}_2)\big)\times \L^q\big(\cal{M}_3,\L^q(\cal{M}_4)\big) \xrightarrow{}
\L^q\big(\cal{M}_1 \otvn \cal{M}_3,\L^q(\cal{M}_2 \otvn \cal{M}_4)\big).
$$
Since $\frac{1}{p} = \frac{1-\frac{q}{p}}{\infty}+\frac{\frac{q}{p}}{q}$, we conclude by interpolation that the bilinear map \eqref{magic-flip} is contractive
\end{proof}

\begin{remark} \normalfont
The case $p \leq q$ seems to us true but we does not need it. So we does not give the proof.
\end{remark}

\begin{lemma}
\label{Lemma-ine-mult}
Let $\cal{M}_1, \cal{M}_2, \cal{N}_1, \cal{N}_2$ be hyperfinite von Neumann algebras. Suppose $1 \leq q \leq p \leq \infty$. Let $T_1 \co \L^q(\cal{M}_1) \to \L^p(\cal{N}_1)$ and $T_2 \co \L^q(\cal{M}_2) \to \L^p(\cal{N}_2)$ be linear maps. Then if $T_1 \ot  T_2 \co \L^q(\L^q) \to \L^p(\L^p)$ is completely bounded then $T_1$ and $T_2$ are completely bounded and we have
\begin{align*}
\norm{T_1 \ot  T_2}_{\cb,\L^q(\L^q) \to \L^p(\L^p)} 
\geq \norm{T_1}_{\cb,\L^q \to \L^p} \norm{T_2}_{\cb,\L^q \to \L^p}.
\end{align*}
\end{lemma}

\begin{proof}
Let $n,m \geq 1$ be some integers. Consider some element $x \in S^p_n(\L^q(\cal{M}_1))$ and $y \in S^p_m(\L^q(\cal{M}_2))$ such that
\begin{equation}
\label{Ine-divers-345}
\norm{x}_{S^p_n(\L^q(\cal{M}_1))} 
\leq 1 
\quad \text{and} \quad
\norm{y}_{S^p_m(\L^q(\cal{M}_2))} 
\leq 1. 
\end{equation}
Using the isometric flip $\cal{F}_p \co \L^p(\cal{M}_1 \otvn \M_m) \to \L^p(\M_m \otvn \cal{M}_1)$ and the flip $\cal{F} \co \L^q(\cal{M}_1,S^p_m) \to S^p_m(\L^q(\cal{M}_1))$ and Proposition \ref{Prop-magic-flip} in the third inequality, we have
\begin{align*}
\lefteqn{
\norm{(\Id_{S^p_n} \ot T_1)(x)}_{S^p_n(\L^p(\cal{M}_1))} \norm{(\Id_{S^p_m} \ot T_2)(y)}_{S^p_m(\L^p(\cal{M}_2))}} \\
&=\bnorm{\big((\Id_{S^p_n} \ot T_1)(x)\big) \ot \big((\Id_{S^p_m} \ot T_2)(y)\big)}_{\L^p(\M_n \otvn \cal{M}_1 \otvn \M_m \otvn \cal{M}_2)} \\
&=\bnorm{(\Id_{S^p_n} \ot T_1 \ot \Id_{S^p_m} \ot T_2)(x \ot y)}_{\L^p(\M_n \otvn \cal{M}_1 \otvn \M_m \otvn \cal{M}_2)} \\
&=\bnorm{(\Id_{} \ot \cal{F}_p \ot \Id_{})(\Id_{S^p_n} \ot T_1 \ot \Id_{S^p_m} \ot T_2)(x \ot y)}_{\L^p(\M_n \otvn \M_m \otvn \cal{M}_1 \otvn \cal{M}_2)}\\
&=\bnorm{(\Id_{S^p_n} \ot \Id_{S^p_m} \ot T_1 \ot T_2)\big((\Id_{S^p_n} \ot \cal{F} \ot \Id_{\L^q(\cal{M}_2)})(x \ot y)\big)}_{\L^p(\M_n \otvn \M_m \otvn \cal{M}_1 \otvn \cal{M}_2)} \\
&\leq\bnorm{\Id_{S^p_n} \ot \Id_{S^p_m} \ot T_1 \ot T_2}_{(p,p,q,q) \to(p,p,p,p)} \norm{(\Id_{S^p_n} \ot \cal{F} \ot \Id_{\L^q(\cal{M}_2)})(x \ot y)}_{(p,p,q,q)} \\
&\leq \norm {T_1 \ot T_2}_{\cb,\L^q(\L^q) \to \L^p(\L^p)} \norm{(\Id_{S^p_n} \ot \cal{F} \ot \Id_{\L^q(\cal{M}_2)})(x \ot y)}_{(p,p,q,q)} \\
&\leq \norm {T_1 \ot T_2}_{\cb,\L^q(\L^q) \to \L^p(\L^p)} \norm{x}_{S^p_n(\L^q(\cal{M}_1))} \norm{y}_{S^p_m(\L^q(\cal{M}_2))} \\
&\ov{\eqref{Ine-divers-345}}{\leq} \norm {T_1 \ot T_2}_{\cb,\L^q(\L^q) \to \to \L^p(\L^p)}.
\end{align*}
Taking the supremum two times, we obtain
\begin{align*}
\bnorm{\Id_{S^p_n} \ot T_1}_{S^p_n(\L^q)  \to S^p_n(\L^p)} \bnorm{\Id_{S^p_m} \ot T_2}_{S^p_m(\L^q) \to S^p_m(\L^p) }
\leq
\norm{T_1 \ot T_2}_{\cb,\L^q(\L^q) \to \L^p(\L^p)}. 
\end{align*}
Taking again the supremum two times, we conclude by \cite[Corollary 1.2]{Pis98} that
\begin{align*}
\norm{T_1}_{\cb,\L^q \to \L^p} \norm{T_2}_{\cb,\L^q \to \L^p}
\leq
\norm{T_1 \ot T_2}_{\cb,\L^q(\L^q) \to \L^p(\L^p)}. 
\end{align*}
\end{proof}

The main result of this section is the following theorem of multiplicativity.

\begin{thm}
\label{thm:multqp}
Let $\cal{M}_1, \cal{M}_2, \cal{N}_1, \cal{N}_2$ be hyperfinite von Neumann algebras equipped with normal semifinite faithful trace. Suppose $1 \leq q \leq p \leq \infty$. Let $T_1 \co \L^q(\cal{M}_1) \to \L^p(\cal{N}_1)$ et $T_2 \co \L^q(\cal{M}_2) \to \L^p(\cal{N}_2)$ be completely bounded maps. Then the map $T_1 \ot T_2 \co \L^q(\cal{M}_1 \ot \cal{N}_2) \to \L^p(\cal{N}_1 \ot \cal{N}_2)$ is completely bounded and we have
\begin{align}
\label{multiplicativity-T}
\norm{T_1 \ot T_2}_{\cb,\L^q(\L^q) \to \L^p(\L^p)} 
=\norm{T_1}_{\cb,\L^q \to \L^p} \norm{T_2}_{\cb,\L^q \to \L^p}.
\end{align}
\end{thm}

\begin{proof}
Using the completely contractive flip $\cal{F} \co \L^q(\L^p) \to \L^p(\L^q)$, $x \ot y \mapsto y \ot x$, we have using \cite[(3.1)]{Pis98} in the last inequality
\begin{align*}
\MoveEqLeft
\norm{T_1 \ot T_2}_{\cb,\L^q(\L^q) \to \L^p(\L^p)}
= \bnorm{(\Id_{\L^p} \ot T_2) (T_1 \ot \Id_{\L^q})}_{\cb,\L^q(\L^q) \to \L^p(\L^p)} \\
&\leq \bnorm{\Id_{\L^p} \ot T_2}_{\cb,\L^p(\L^q) \to \L^p(\L^p)} \bnorm{T_1 \ot \Id_{\L^q}}_{ \cb,\L^q(\L^q) \to \L^p(\L^q) } \\
&= \bnorm{\Id_{\L^p} \ot T_2}_{\cb,\L^p(\L^q) \to \L^p(\L^p)} \bnorm{\cal{F}(\Id_{\L^q} \ot T_1)}_{\cb,\L^q(\L^q) \to \L^p(\L^q)}  \\
&= \bnorm{\Id_{\L^p} \ot T_2}_{\cb,\L^p(\L^q) \to \L^p(\L^p)} \bnorm{\cal{F}}_{\cb,\L^q(\L^p) \to \L^p(\L^q)} \bnorm{\Id_{\L^q} \ot T_1}_{\cb,\L^q(\L^q) \to \L^q(\L^p)} \\
&\ov{\eqref{flip}}{\leq} \bnorm{\Id_{\L^p} \ot T_2}_{\cb,\L^p(\L^q) \to \L^p(\L^p)} \bnorm{\Id_{\L^q} \ot T_1}_{\cb,\L^q(\L^q) \to \L^q(\L^p)}  \\
&\leq \norm{T_1}_{\cb,\L^q \to \L^p} \norm{T_2}_{\cb,\L^q \to \L^p}.
\end{align*}
The reverse inequality is Lemma \ref{Lemma-ine-mult}.
\end{proof}

We show how deduce the additivity of the completely bounded minimal entropy (to compare with \eqref{subadditivity-Hmin}). The proof follows the one of \cite{DJKRB06} for matrix spaces. 

\begin{thm}
\label{Additivity-Scb-min}
Let $\cal{M}_1, \cal{M}_2, \cal{N}_1, \cal{N}_2$ be finite-dimensional von Neumann algebras equipped with finite faithful trace. Let $T_1 \co \L^1(\cal{M}_1) \to \L^1(\cal{N}_1)$ et $T_2 \co \L^1(\cal{M}_2) \to \L^1(\cal{N}_2)$ be quantum channels. Then we have
\begin{equation}
\label{Scbmin-additive}
\H_{\cb,\min}(T_1 \ot T_2)  
=\H_{\cb,\min}(T_1) + \H_{\cb,\min}(T_2).
\end{equation}
\end{thm}

\begin{proof}
For any $p>1$, we have
\begin{align*}
\MoveEqLeft
\frac{1 -\norm{T_1 \ot T_2}_{\cb, \L^1(\L^1) \to \L^p(\L^p)}^p }{p-1} 
\ov{\eqref{multiplicativity-T}}{=} \frac{1 - \norm{T_1}_{\cb, \L^1 \to \L^p}^p \norm{T_2}_{\cb, \L^1 \to \L^p}^p }{p-1} \\
&=\frac{1 -\norm{T_1}_{\cb, \L^1 \to \L^p}^p + \norm{T_1}_{\cb, \L^1 \to \L^p}^p-\norm{T_1}_{\cb, \L^1 \to \L^p}^p \, \norm{T_2}_{\cb, \L^1 \to \L^p}^p }{p-1} \\
&= \frac{1 -\norm{T_1}_{\cb, \L^1 \to \L^p}^p }{p-1}+ \norm{T_1}_{\cb, \L^1 \to \L^p}^p \frac{1 -\norm{T_2}_{\cb, \L^1 \to \L^p}^p  }{p-1}. 
\end{align*}
We can suppose that the von Neumann algebras are $\not=\{0\}$. Note that similarly to the beginning of the proof of Theorem \ref{Th-Smin-as-derivative}, we have $\norm{T_1}_{\cb,\L^1 \to \L^1}= 1$ and $\norm{T_2}_{\cb,\L^1 \to \L^1}= 1$. Taking the limit when $p \to 1^+$ and using $ \lim_{p \to 1^+} \norm{T_1}_{\cb, \L^1 \to \L^p}^p =\norm{T_1}_{\cb, \L^1 \to \L^1}=1$, we obtain with \eqref{Def-intro-Scb-min} the additivity formula \eqref{Scbmin-additive}. 
\end{proof}

Note that the derivative of \eqref{Def-intro-Scb-min} exists and there is a concrete description of the completely bounded minimal entropy $\H_{\cb,\min}(T)$. The proof of the following result (which generalize \cite[Theorem 3.4]{JuP16}) is left to the reader. 


\begin{thm}
\label{thm-entropy-concrete}
Let $\cal{M}$ be a finite-dimensional von Neumann algebra equipped with a faithful trace. Let $T \co \L^1(\cal{M}) \to \L^1(\cal{M})$ be a quantum channel. The function $p \mapsto \norm{T}_{\cb,1 \to p}^p$ is differentiable at $p=1$ and the opposite of its derivative is given by
\begin{equation}
\label{Smincb-as-inf}
\H_{\cb,\min}(T)
=\inf_{\rho \in \L^1(\cal{M}) \ot \L^1(\cal{M})} \Big\{\H\big[(\Id \ot T)(\rho)\big]- \H\big[(\Id \ot \tau)(\rho)\big] \Big\}
\end{equation}
where the infimum is taken on all states $\rho$. We can restrict the infimum to the pure states.
\end{thm}

\begin{remark} \normalfont
It should be not particularly difficult to give a second proof of the additivity of the completely bounded minimal entropy with the strong subadditivity of Segal's entropy \cite{Pod21} and the explicit formula \eqref{Smincb-as-inf}.
\end{remark}

With Remark \ref{Remark-without-p}, we also have the formula
\begin{equation}
\label{Def-Scb-min-without-p}
\H_{\cb,\min(T)}
= -\frac{\d}{\d p} \norm{T}_{\cb,\L^1(\mathcal{M}) \to \L^p(\mathcal{M})}|_{p=1}. 
\end{equation}

\begin{prop}
Let $\cal{M}$ be a finite-dimensional von Neumann algebra equipped with a faithful trace. Let $T \co \L^1(\cal{M}) \to \L^1(\cal{M})$ be a quantum channel.
\begin{equation}
\label{compar-ent-min}
\H_{\cb,\min}(T) 
\leq \H_{\min}(T) 
\end{equation}
\end{prop}

\begin{proof}
We have
\begin{align*}
\MoveEqLeft
\H_{\cb,\min(T)} 
\ov{\eqref{Def-Scb-min-without-p}}{=} -\frac{\d}{\d p} \norm{T}_{\cb,\L^1(\mathcal{M}) \to \L^p(\mathcal{M})}|_{p=1}           
\ov{\eqref{elem-lim}}{=} -\lim_{q \to \infty} q\ln \norm{T}_{\cb,\L^1(\mathcal{M}) \to \L^{q^*}(\mathcal{M})}\\
&\leq -\lim_{q \to \infty} q\ln \norm{T}_{\L^1(\mathcal{M}) \to \L^{q^*}(\mathcal{M})}\ov{\eqref{elem-lim}}{=} -\frac{\d}{\d p} \norm{T}_{\L^1(\mathcal{M}) \to \L^p(\mathcal{M})} |_{p=1}
= \H_{\min}(T).
\end{align*}  
\end{proof}

\begin{remark} \normalfont
Suppose that $\tau$ is normalized. We can prove the inequality $\H_{\cb,\min}(T) \leq 0$ with the following reasoning. If $\rho$ is a state, we let $\sigma_{12} \ov{\mathrm{def}}{=} (\Id \ot T)(\rho)$. Using the subadditivity of the entropy \cite[Theorem 8]{Pod21} in the first inequality, we obtain
\begin{align*}
\MoveEqLeft
\H\big[(\Id \ot T)(\rho)\big]- \H\big[(\Id \ot \tau)(\rho)\big]          
=\H(\sigma_{12})- \H(\sigma_1) \leq \H(\sigma_2) \leq 0.
\end{align*}

Finally, if the trace $\tau$ is \textit{normalized}, we have
\begin{equation}
\label{cbminless0}
\H_{\cb,\min,\tau}(T) \leq 0
\end{equation}
for any quantum channel $T \co \L^1(\mathcal{M}) \to \L^1(\cal{M})$.
\end{remark}

\subsection{Appendix: some computer codes}
\label{Appendix3}

The following Maple codes are used in Section \ref{Sec-Examples}.

\begin{verbatim}
> restart:
> with(linalg):
> z:=exp(2*I*Pi/3):
> A:=matrix([[1,1,1,0,0,1],[1,1,z,0,0,1/z],[1,1,1/z,0,0,z],
                    [1,-1,0,1,1,0],[1,-1,0,1/z,z,0],[1,-1,0,z,1/z,0]]):
> B:=transpose(A):
> C:=inverse(B):
> L:=array(1..36):
> E:=array(1..6):
> E[1]:=1: E[2]:=2 :E[3]:=11: E[4]:=12: E[5]:=21: E[6]:=22:
> for i from 1 to 6 do
   print('Delta','e',E[i],'egal');
   for j from 1 to 6 do
    for l from 1 to 6 do 
     L[i]:=simplify(sum(C[k,i]*B[j,k]*B[l,k],k=1..6)):
     if not(L[i]=0) then
      print(L[i],`e`,E[j],`ot`,`e`,E[l]);
     fi;
    od;
   od;
  od;
\end{verbatim}

\begin{verbatim}
> restart:
> with(linalg):
> A:=matrix([[1,1,1,1,1,0,0,1],[1,1,1,1,-1,0,0,-1],[1,1,-1,-1,I,0,0,-I],
        [1,1,-1,-1,-I,0,0,I],[1,-1,1,-1,0,1,-1,0],[1,-1,1,-1,0,-1,1,0],
        [1,-1,-1,1,0,I,I,0],[1,-1,-1,1,0,-I,-I,0]]):
> B:=transpose(A):
> C:=inverse(B):
> L:=array(1..48):
> E:=array(1..8):
> E[1]:=1: E[2]:=2:  E[3]:=3: E[4]:=4: E[5]:=11: E[6]:=12: E[7]:=21: E[8]:=22:
> for i from 1 to 8 do
   print('Delta','e',E[i],'egal');
   for j from 1 to 8 do
    for l from 1 to 8 do 
     L[i]:=simplify(sum(C[k,i]*B[j,k]*B[l,k],k=1..8)):
     if not(L[i]=0) then
      print(L[i],`e`,E[j],`ot`,`e`,E[l]);
     fi;
    od;
   od;
  od;
\end{verbatim}

\begin{verbatim}
> restart:
> with(linalg):
> A:=matrix([[1,1,1,1,1,0,0,1],[1,-1,-1,1,0,1,1,0],
             [1,1,-1,-1,I,0,0,-I],[1,-1,1,-1,0,-I,I,0],
           						 [1,1,1,1,-1,0,0,-1],[1,-1,-1,1,0,-1,-1,0],
             [1,1,-1,-1,-I,0,0,I],[1,-1,1,-1,0,I,-I,0]]):
> B:=transpose(A):
> C:=inverse(B):
> L:=array(1..48):
> E:=array(1..8):
> E[1]:=1: E[2]:=2:  E[3]:=3: E[4]:=4: E[5]:=11: E[6]:=12: E[7]:=21: E[8]:=22:
> for i from 1 to 8 do
   print('Delta','e',E[i],'egal');
   for j from 1 to 8 do
    for l from 1 to 8 do 
     L[i]:=simplify(sum(C[k,i]*B[j,k]*B[l,k],k=1..8)):
     if not(L[i]=0) then
      print(L[i],`e`,E[j],`ot`,`e`,E[l]);
     fi;
    od;
   od;
  od;
\end{verbatim}

\section{Future directions and open questions}
\label{sec-future}
We plan to generalize certain results in a future version and to improve and expand the text.

\section{Acknowledgement}
\label{sec-acknowledgement}

The author acknowledges support by the grant ANR-18-CE40-0021 (project HASCON) of the French National Research Agency ANR. The author wishes to express his thanks to Vijay Kodiyalam, Debbie Leung, Zhengwei Liu, Luis Rodriguez-Piazza, Carlos Palazuelos, Hanna Podsedkowska, Piotr Soltan and Ami Viselter for very short discussions and Li Gao for a short but very instructive discussion on von Neumann entropy and capacities of a quantum channel. He also thanks Adrian Gonzalez P\'erez for an exciting discussion on the topic of Section \ref{completely-bounded-description} and the reference \cite{GJP17} and Ion Nechita for a very interesting discussion at the conference ``Entangling Non-commutative Functional Analysis and Geometry of Banach Spaces'' and to a Seminar in Toulouse and Adam Skalski for some questions. The author wishes to thank Sebastien Palcoux for a profound discussion on subfactor planar algebras and encouragement. The author would like to express gratitude to Franz Luef for his invaluable encouragement in writing Section \ref{sec-Werner} and to Stefan Wagner for some useful information on ergodic actions. The author wishes to thank Ryszard Nest for kindly answering several naive questions regarding the paper \cite{DMN22}. Finally, he thanks Christian Le Merdy for some old very instructive discussions on some results of the Appendices.



\small
\begin{thebibliography}{79}

\bibitem[AbA02]{AbA02}
Y. A. Abramovich and C. D. Aliprantis.
\newblock An invitation to operator theory.
\newblock Graduate Studies in Mathematics, 50. American Mathematical Society, Providence, RI, 2002.

\bibitem[AVDR81]{AVDR81}
G. N. Agaev, N. Y. Vilenkin, G. M. Dzhafarli, A. I. Rubinshtein.
\newblock Multiplicative systems of functions and harmonic analysis on zero-dimensional groups (Russian).
\newblock Elm, Baku, 1981.

\bibitem[AMR18]{AMR18}
R. Akylzhanov, S. Majid and M. Ruzhansky.
\newblock Smooth dense subalgebras and Fourier multipliers on compact quantum groups.
\newblock Comm. Math. Phys. 362 (2018), no. 3, 761--799.

\bibitem[AlC18]{AlC18}
F. Alajaji and P.-N. Chen.
\newblock An introduction to single-user information theory.
\newblock Springer Undergrad. Texts Math. Technol., Springer, Singapore, 2018.

\bibitem[ACP20]{ACP20}
V. V. Albert, J. P. Covey, and J. Preskill.
\newblock Robust Encoding of a Qubit in a Molecule.
\newblock Phys. Rev. X 10, 031050.

\bibitem[AlS03]{AlS03}
E. M. Alfsen and F. W. Shultz.
\newblock Geometry of state spaces of operator algebras. 
\newblock Mathematics: Theory \& Applications. Birkhauser Boston, Inc., Boston, MA, 2003.


\bibitem[ACN00]{ACN00}
G. G. Amosov, A. S. Holevo and R. F. Werner.
\newblock On some additivity problems in quantum information theory. 
\newblock math-ph/0003002, 2000.

\bibitem[ACN20]{ACN20}
M. Alaghmandan, J. Crann and M. Neufang.
\newblock Mapping ideals of quantum group multipliers. 
\newblock Adv. Math. 374 (2020), 107353, 52 pp.

\bibitem[AHK80]{AHK80}
S. Albeverio and R. Hoegh-Krohn.
\newblock Ergodic actions by compact groups on $C^*$-algebras. 
\newblock Math. Z. 174 (1980), no. 1, 1--17. 











\bibitem[Amo22]{Amo22}
G. G. Amosov.
\newblock On capacity of quantum channels generated by irreducible projective unitary representations of finite groups.
\newblock Quantum Inf. Process. 21 (2022), no. 2, Paper No. 81, 15 pp.

\bibitem[Are04]{Are04}
W. Arendt.
\newblock Semigroups and evolution equations: functional calculus, regularity and kernel estimates. 
\newblock Evolutionary equations. Vol. I, 1--85, Handb. Differ. Equ., North-Holland, Amsterdam, 2004. 

\bibitem[Arh11]{Arh11}
C. Arhancet.
\newblock Noncommutative Fig\`a-Talamanca-Herz algebras for Schur multipliers.
\newblock Integral Equations Operator Theory 70 (2011), no. 4, 485--510.



\bibitem[Arh24]{Arh24}
C. Arhancet.
\newblock Spectral triples, Coulhon-Varopoulos dimension and heat kernel estimates
\newblock Adv. Math. 451 (2024), Paper No. 109794, 58 pp..

\bibitem[Arh25a]{Arh25a}
C. Arhancet.
\newblock Entanglement-assisted classical capacities of some channels acting as radial multipliers on fermion algebras.
\newblock J. Funct. Anal. 288 (2025), no. 5, Paper No. 110790.

\bibitem[Arh25]{Arh25}
C. Arhancet.
\newblock A HSW theorem for finite von Neumann algebras. 
\newblock Work in progress.

\bibitem[Arh19]{Arh19}
C. Arhancet.
\newblock Dilations of semigroups on von Neumann algebras and noncommutative $\L^p$-spaces.
\newblock J. Funct. Anal. 276 (2019), no. 7, 2279--2314.

\bibitem[Arh24a]{Arh24a}
C. Arhancet.
\newblock Contractively decomposable projections on noncommutative $\L^p$-spaces.
\newblock J. Math. Anal. Appl. 533 (2024), no. 2, Paper No. 128017, 36 pp.


\bibitem[Arh24b]{Arh24b}
C. Arhancet.
\newblock Sobolev algebras on Lie groups and noncommutative geometry.
\newblock J. Noncommut. Geom. 18 (2024), no. 2, 451--500.



\bibitem[ArK23]{ArK23}
C. Arhancet and C. Kriegler.
\newblock Projections, multipliers and decomposable maps on noncommutative $\L^p$-spaces.
\newblock M\'em. Soc. Math. Fr. (N.S.) No. 177 (2023).

\bibitem[ArK22]{ArK22}
C. Arhancet and C. Kriegler.
\newblock Riesz transforms, Hodge-Dirac operators and functional calculus for multipliers.
\newblock Lecture Notes in Mathematics, 2304. Springer, Cham, 2022.

\bibitem[ArK25]{ArK25}
C. Arhancet and C. Kriegler.
\newblock Sharp functional calculus for the Taibleson operator on non-archimedean local fields.
\newblock Preprint, arXiv:2407.10508.


\bibitem[ArK3]{ArK3}
C. Arhancet and C. Kriegler.
\newblock Optimal.
\newblock Preprint.







\bibitem[Aud09]{Aud09}
K M. R. Audenaert.
\newblock A note on the $p \to q$ norms of 2-positive maps. 
\newblock Linear Algebra Appl. 430 (2009), no. 4, 1436--1440. 


\bibitem[BaB70]{Bab70}
N. B. Backhouse and C. J. Bradley.
\newblock Projective representations of space groups. I. Translation groups.
\newblock Quart. J. Math. Oxford Ser. (2) 21 (1970), 203--222.

\bibitem[Bac70]{Bac70}
N. B. Backhouse.
\newblock Projective representations of space groups. II. Factor systems.
\newblock Quart. J. Math. Oxford Ser. (2) 21 (1970), 277--295.

\bibitem[BaK73]{BaK73}
L. Baggett and A. Kleppner.
\newblock Multiplier representations of abelian groups.
\newblock J. Funct. Anal. 14 (1973), 299--324.




\bibitem[Ban17]{Ban17}
T. Banica.
\newblock Liberation theory for noncommutative homogeneous spaces. 
\newblock Ann. Fac. Sci. Toulouse Math. (6) 26 (2017), no. 1, 127--156. 

 
\bibitem[Ban20]{Ban20}
T. Banica.
\newblock Quantum isometries and noncommutative geometry. 
\newblock Preprint, 2020 \href{https://banica.u-cergy.fr/}{https://banica.u-cergy.fr/}. 

\bibitem[BaG10]{BaG10}
T. Banica and D. Goswami.
\newblock Quantum isometries and noncommutative spheres. 
\newblock Comm. Math. Phys. 298 (2010), no. 2, 343--356. 

\bibitem[BSS12]{BSS12}
T. Banica, A. Skalski and P. Soltan.
\newblock Noncommutative homogeneous spaces: the matrix case. 
\newblock J. Geom. Phys. 62 (2012), no. 6, 1451--1466. 
 

\bibitem[Bar54]{Bar54}
V. Bargmann.
\newblock On unitary ray representations of continuous groups.
\newblock Ann. of Math. (2) 59 (1954), 1--46.

\bibitem[Bat21]{Bat21}
B.-O. Battseren.
\newblock Groups with tame cuts. 
\newblock PhD Thesis, 2021. 

\bibitem[BeM73]{BeM73}
B. Beauzamy and B. Maurey.
\newblock Op\'erateurs de convolution $r$-sommants sur un groupe compact ab\'elien. 
\newblock C. R. Acad. Sci. Paris S\'er. A-B 277 (1973), A113--A115.

\bibitem[BeC09]{BeC09}
E. B\'edos and R. Conti.
\newblock On twisted Fourier analysis and convergence of Fourier series on discrete groups.
\newblock J. Fourier Anal. Appl. 15 (2009), no. 3, 336--365.

\bibitem[BeT03]{BeT03}
E. B\'edos and L. Tuset.
\newblock Amenability and co-amenability for locally compact quantum groups. 
\newblock Internat. J. Math. 14 (2003), no. 8, 865--884.

\bibitem[BeH20]{BeH20}
B. Bekka and P. de la Harpe.
\newblock Unitary representations of groups, duals, and characters.
\newblock Mathematical Surveys and Monographs, 250. American Mathematical Society, Providence, RI, 2020.

\bibitem[BCLPY22]{BCLPY22}
C. Beny, J. Crann, H. H. Lee, S.-J. Park and S.-G. Youn.
\newblock Gaussian quantum information over general quantum kinematical systems I: Gaussian states.
\newblock Preprint, arXiv:2204.08162.

\bibitem[BeL76]{BeL76}
J. Bergh and J. L\"ofstr\"om.
\newblock Interpolation spaces. An Introduction.
\newblock Springer-Verlag, Berlin, Heidelberg, New York, 1976.

\bibitem[BBF24]{BBF24}
E. Berge, S. M. Berge, R. Fulsche.
\newblock A quantum harmonic analysis approach to Segal algebras.
\newblock Integral Equations Operator Theory 96 (2024), no. 3, Paper No. 20, 39 pp.

\bibitem[BBLS22]{BBLS22}
E. Berge, S. M. Berge, F. Luef and E. Skrettingland.
\newblock Affine quantum harmonic analysis.
\newblock J. Funct. Anal. 282 (2022), no. 4, Paper No. 109327, 64 pp.


\bibitem[Bia97]{Bia97}
P. Biane.
\newblock Free hypercontractivity.
\newblock Comm. Math. Phys. 184 (1997), no. 2, 457--474.


\bibitem[BRV06]{BRV06}
J. Bichon, A. De Rijdt and S. Vaes.
\newblock Ergodic coactions with large multiplicity and monoidal equivalence of quantum groups. 
\newblock Comm. Math. Phys. 262 (2006), no. 3, 703--728. 

\bibitem[BGNT21]{BGNT21}
P. Bieliavsky, V. Gayral, S. Neshveyev and L. Tuset.
\newblock Quantization of subgroups of the affine group.
\newblock J. Funct. Anal. 280 (2021), no.4, Paper No. 108844.


\bibitem[BiJ00]{BiJ00}
D. Bisch and V. Jones.
\newblock Singly generated planar algebras of small dimension. 
\newblock Duke Math. J. 101 (2000), no. 1, 41--75. 

\bibitem[BDG21]{BDG21}
M. Bischoff, S. Del Vecchio and L. Giorgetti.
\newblock Compact hypergroups from discrete subfactors. 
\newblock J. Funct. Anal. 281 (2021), no. 1, Paper No. 109004, 78 pp.. 



\bibitem[BLM04]{BLM04}
D. Blecher and C. Le Merdy.
\newblock Operator algebras and their modules-an operator space approach.
\newblock London Mathematical Society Monographs. New Series, 30. Oxford Science Publications. The Clarendon Press, Oxford University Press, Oxford, 2004.

\bibitem[BlH95]{BlH95}
W. R. Bloom and H. Heyer.
\newblock Harmonic analysis of probability measures on hypergroups. 
\newblock De Gruyter Studies in Mathematics, 20. Walter de Gruyter \& Co., Berlin, 1995. 

\bibitem[Boc95]{Boc95}
F. P. Boca.
\newblock Ergodic actions of compact matrix pseudogroups on $C^*$-algebras. 
\newblock Recent advances in operator algebras (Orl\'eans, 1992). Ast\'erisque No. 232 (1995), 93--109. 


\bibitem[Bou04]{Bou04}
N. Bourbaki.
\newblock Integration. II. Chapters 7--9. Translated from the 1963 and 1969 French originals by Sterling K. Berberian. Elements of Mathematics (Berlin).
\newblock Springer-Verlag, Berlin, 2004.

\bibitem[BCLY20]{BCLY20}
M. Brannan, B. Collins, H. H. Lee and S.-G. Youn.
\newblock Temperley-Lieb quantum channels. 
\newblock Comm. Math. Phys. 376 (2020), no. 2, 795--839. 


\bibitem[Bra17]{Bra17}
M. Brannan.
\newblock Approximation properties for locally compact quantum groups. 
\newblock Topological quantum groups, 185--232, Banach Center Publ., 111, Polish Acad. Sci. Inst. Math., Warsaw, 2017. 

\bibitem[BrN06]{BrN06}
J. Brodzki and G. A. Niblo.
\newblock Approximation properties for discrete groups. 
\newblock C*-algebras and elliptic theory, 23--35, Trends Math., Birkh\"auser, Basel, 2006. 

\bibitem[BrO08]{BrO08}
N. P. Brown and N. Ozawa.
\newblock $C^*$-algebras and finite-dimensional approximations.
\newblock Graduate Studies in Mathematics, 88. American Mathematical Society, Providence, RI, 2008.



\bibitem[CaL93]{CaL93}
E. A. Carlen and E. H. Lieb.
\newblock Optimal hypercontractivity for Fermi fields and related noncommutative integration inequalities.
\newblock Comm. Math. Phys. 155 (1993), no. 1, 27--46.





\bibitem[Cas17]{Cas17}
M. Caspers.
\newblock Locally compact quantum groups. 
\newblock Topological quantum groups, 153--184, Banach Center Publ., 111, Polish Acad. Sci. Inst. Math., Warsaw, 2017. 





\bibitem[CJK09]{CJK09}
M-D. Choi, N. Johnston and D. W. Kribs.
\newblock The multiplicative domain in quantum error correction. 
\newblock J. Phys. A 42 (2009), no. 24, 245303, 15 pp.. 

\bibitem[CST08]{CST08}
T. Ceccherini-Silberstein, F. Scarabotti and F. Tolli.
\newblock Harmonic analysis on finite groups. Representation theory, Gelfand pairs and Markov chains. 
\newblock Cambridge Studies in Advanced Mathematics, 108. Cambridge University Press, Cambridge, 2008. 


\bibitem[ChV99]{ChV99}
Y. A. Chapovsky and L. I. Vainerman.
\newblock Compact quantum hypergroups. 
\newblock J. Operator Theory 41 (1999), no. 2, 261--289. 

\bibitem[Cha17]{Cha17}
I. Chatterji.
\newblock Introduction to the rapid decay property. 
\newblock Around Langlands correspondences, 53--72, Contemp. Math., 691, Amer. Math. Soc., Providence, RI, 2017. 


\bibitem[Che15]{Che15}
C. Cheng.
\newblock A character theory for projective representations of finite groups.
\newblock Linear Algebra Appl. 469 (2015), 230--242.

\bibitem[ChH1]{ChH1}
A. Chirvasitua and  S. O. Hoche.
\newblock Ergodic actions of the compact quantum group $O_{-1}(2)$. 
\newblock Preprint, arXiv:1708.06457.

\bibitem[ChL23]{ChL23}
M.-D. Choi and C.-K. Li.
\newblock On unital qubit channels. 
\newblock Preprint, arXiv:2301.01358.

\bibitem[CPSW00]{CPSW00}
P. Clement, B. de Pagter, F. A. Sukochev and H. Witvliet.
\newblock Schauder decomposition and multiplier theorems.
\newblock Studia Math. 138 (2000), no. 2, 135--163.

\bibitem[Con90]{Con90}
J. B. Conway.
\newblock A course in functional analysis. Second edition.
\newblock Graduate Texts in Mathematics, 96. Springer-Verlag, New York, 1990.

\bibitem[Coo10]{Coo10}
T. Cooney.
\newblock A Hausdorff-Young inequality for locally compact quantum groups.
\newblock Internat. J. Math. 21 (2010), no. 12, 1619--1632.

\bibitem[CoT06]{CoT06}
T. M. Cover and J. A. Thomas.
\newblock Elements of information theory. Second edition.
\newblock Wiley-Interscience [John Wiley \& Sons], Hoboken, NJ, 2006.

\bibitem[CrN13]{CrN13}
J. Crann and M. Neufang.
\newblock Quantum channels arising from abstract harmonic analysis. 
\newblock J. Phys. A 46 (2013), no. 4, 045308, 22 pp.. 

\bibitem[CrN22]{CrN22}
J. Crann and M. Neufang.
\newblock A non-commutative Fej\'er theorem for crossed products, the approximation property, and applications. 
\newblock Int. Math. Res. Not. IMRN 2022, no. 5, 3571--3601. 


\bibitem[Cra17]{Cra17}
J. Crann.
\newblock Amenability and covariant injectivity of locally compact quantum groups II. 
\newblock Canad. J. Math. 69 (2017), no. 5, 1064--1086.


%









%






\bibitem[CXY13]{CXY13}
Z. Chen, Q. Xu and Z. Yin.
\newblock Harmonic analysis on quantum tori.
\newblock Comm. Math. Phys. 322 (2013), no. 3, 755--805.

\bibitem[CJK09]{CJK09}
M.-D. Choi, N. Johnston and D. W. Kribs.
\newblock The multiplicative domain in quantum error correction. 
\newblock J. Phys. A 42 (2009), no. 24, 245303, 15 pp.

\bibitem[ChL02]{ChL02}
C.-H. Chu and A. T.-M. Lau.
\newblock Harmonic functions on groups and Fourier algebras. 
\newblock Lecture Notes in Mathematics, 1782. Springer-Verlag, Berlin, 2002. 


\bibitem[CDHPRRS20]{CDHPRRS20}
S. X. Cui, D. Ding, X. Han, G. Penington, D. Ranard, B. C. Rayhaun, Z. Shangnan.
\newblock Kitaev's quantum double model as an error correcting code. 
\newblock Quantum 4, 331 (2020). 

\bibitem[Cur99]{Cur99}
C. W. Curtis.
\newblock Pioneers of representation theory: Frobenius, Burnside, Schur, and Brauer.
\newblock Hist. Math., 15. American Mathematical Society, Providence, RI; London Mathematical Society, London, 1999.

\bibitem[CwJ84]{CwJ84}
M. Cwikel and S. Janson.
\newblock Interpolation of analytic families of operators.
\newblock Studia Math. 79 (1984), no. 1, 61--71.

\bibitem[DFW21]{DFW21}
B. Das, U. Franz and X. Wang.
\newblock Invariant Markov semigroups on quantum homogeneous spaces. 
\newblock J. Noncommut. Geom. 15 (2021), no. 2, 531--580. 

\bibitem[DHS04]{DHS04}
N. Datta, A. S. Holevo and Y. Suhov.
\newblock A quantum channel with additive minimum output entropy.
\newblock Preprint, arXiv:quant-ph/0403072.

\bibitem[DHS06]{DHS06}
N. Datta, A. S. Holevo and Y. Suhov.
\newblock Additivity for transpose depolarizing channels.
\newblock International Journal of Quantum Information 04 (2006), No. 01, pp. 85--98.




\bibitem[Daw10]{Daw10}
M. Daws.
\newblock Multipliers, self-induced and dual Banach algebras.
\newblock Dissertationes Math. 470 (2010), 62 pp.




\bibitem[Daw11]{Daw11}
M. Daws.
\newblock Multipliers of locally compact quantum groups via Hilbert $C^*$-modules.
\newblock J. Lond. Math. Soc. (2) 84 (2011), no. 2, 385--407.

\bibitem[Daw12]{Daw12}
M. Daws.
\newblock Completely positive multipliers of quantum groups.
\newblock Internat. J. Math. 23 (2012), no. 12, 1250132, 23 pp.

\bibitem[DFSW16]{DFSW16}
M. Daws, P. Fima, A. Skalski and S. White.
\newblock The Haagerup property for locally compact quantum groups.
\newblock J. Reine Angew. Math. 711 (2016), 189--229. 



\bibitem[DeC10]{DeC10}
K. De Commer.
\newblock On cocycle twisting of compact quantum groups. 
\newblock J. Funct. Anal. 258 (2010), no. 10, 3362--3375.

\bibitem[DeC11a]{DeC11a}
K. De Commer.
\newblock On projective representations for compact quantum groups. 
\newblock J. Funct. Anal. 260 (2011), no. 12, 3596--3644. 

\bibitem[DeC11b]{Dec11b}
K. De Commer.
\newblock Galois objects and cocycle twisting for locally compact quantum groups.
\newblock J. Operator Theory 66 (2011), no. 1, 59--106.

\bibitem[DeC17]{DeC17}
K. De Commer.
\newblock Actions of compact quantum groups. 
\newblock Topological quantum groups, 33--100, Banach Center Publ., 111, Polish Acad. Sci. Inst. Math., Warsaw, 2017.

\bibitem[DMN22]{DMN22}
K. De Commer, R. Martos and R. Nest.
\newblock Projective representation theory for compact quantum groups and the quantum Baum-Connes assembly map. 
\newblock Preprint, arXiv:2112.04365. 

\bibitem[DeS83]{DeS83}
D. De Schreye.
\newblock Integrable, ergodic actions of abelian groups on von Neumann algebras.
\newblock Math. Scand. 53 (1983), no. 2, 265--280.


\bibitem[DeF93]{DeF93}
A. Defant and K. Floret.
\newblock Tensor norms and operator ideals.
\newblock North-Holland Mathematics Studies, 176. North-Holland Publishing Co., Amsterdam, 1993.


\bibitem[DVD11a]{DVD11a}
L. Delvaux and A. Van Daele.
\newblock Algebraic quantum hypergroups. 
\newblock Adv. Math. 226 (2011), no. 2, 1134--1167. 
 
\bibitem[DVD11b]{DVD11b}
L. Delvaux and A. Van Daele.
\newblock Algebraic quantum hypergroups II. Constructions and examples. 
\newblock Internat. J. Math. 22 (2011), no. 3, 407--434. 



\bibitem[DRVV10]{DRVV10}
A. De Rijdt and N. Vander Vennet.
\newblock Actions of monoidally equivalent compact quantum groups and applications to probabilistic boundaries.
\newblock Ann. Inst. Fourier (Grenoble) 60 (2010), no. 1, 169--216.

\bibitem[DJKRB06]{DJKRB06}
I. Devetak, M. Junge, C. King and M. B. Ruskai.
\newblock Multiplicativity of completely bounded $p$-norms implies a new additivity result. 
\newblock Comm. Math. Phys. 266 (2006), no. 1, 37--63. 

\bibitem[Dieu68]{Dieu68}
J. Dieudonn\'e.
\newblock \'El\'ements d'analyse. Tome II: Chapitres XII \`a XV. (French)
\newblock Cahiers Scientifiques, Fasc. XXXI Gauthier-Villars, \'Editeur, Paris 1968 x+408 pp.

\bibitem[Dieu78]{Dieu78}
J. Dieudonn\'e.
\newblock Treatise on analysis. Vol. VI. Translated from the French by I. G. Macdonald
\newblock Pure and Applied Mathematics, 10--VI. Academic Press, Inc. [Harcourt Brace Jovanovich, Publishers], New York-London, 1978.


\bibitem[DiV04]{DiV04}
T. Digernes and V. S. Varadarajan.
\newblock Models for the irreducible representation of a Heisenberg group.
\newblock Infin. Dimens. Anal. Quantum Probab. Relat. Top. 7 (2004), no. 4, 527--546.

\bibitem[Dix63]{Dix3}
J. Dixmier.
\newblock Traces sur les $C^*$-alg\`ebres. (French)
\newblock Ann. Inst. Fourier (Grenoble) 13 (1963), no. fasc., fasc. 1, 219--262. 

\bibitem[Dix77]{Dix77}
J. Dixmier.
\newblock $C\sp*$-algebras. Translated from the French by Francis Jellett.
\newblock North-Holland Mathematical Library, Vol. 15. North-Holland Publishing Co., Amsterdam-New York-Oxford, 1977. 

\bibitem[Dix81]{Dix81}
J. Dixmier.
\newblock Von Neumann algebras. With a preface by E. C. Lance. Translated from the second French edition by F. Jellett.
\newblock North-Holland Mathematical Library, 27. North-Holland Publishing Co., Amsterdam-New York, 1981. 

\bibitem[DJT95]{DJT95}
J. Diestel, H. Jarchow and A. Tonge.
\newblock Absolutely summing operators.
\newblock Cambridge Studies in Advanced Mathematics, 43. Cambridge University Press, 1995.

\bibitem[DoS00]{DoS00}
P. G. Dodds and F. A. Sukochev.
\newblock Non-commutative bounded Vilenkin systems.
\newblock Math. Scand. 87 (2000), no. 1, 73--92.

\bibitem[DoS01]{DoS01}
P. G. Dodds and F. A. Sukochev.
\newblock Vilenkin bases in non-commutative $L^p$-spaces.
\newblock Geometric analysis and applications (Canberra, 2000), 30--41. Proc. Centre Math. Appl. Austral. Nat. Univ., 39
Australian National University, Centre for Mathematics and its Applications, Canberra, 2001.

\bibitem[DFPS01]{DFPS01}
P. G. Dodds and S. V. Ferleger, B. de Pagter and F. A. Sukochev.
\newblock Vilenkin systems and generalized triangular truncation operator.
\newblock Integral Equations Operator Theory 40 (2001), no. 4, 403--435.

\bibitem[DLS24]{DLS24}
M. D\"orfler, F. Luef, E. Skrettingland.
\newblock Local structure and effective dimensionality of time series data sets.
\newblock Appl. Comput. Harmon. Anal. 73 (2024), Paper No. 101692, 25 pp.




\bibitem[EKR1]{EKR1}
E. G. Effros, J. Kraus and Z.-J. Ruan.
\newblock On two quantized tensor products. 
\newblock Operator algebras, mathematical physics, and low-dimensional topology (Istanbul, 1991), 125--145, Res. Notes Math., 5, A K Peters, Wellesley, MA, 1993.

\bibitem[Edw55]{Edw55}
R. E Edwards.
\newblock On factor functions. 
\newblock Pacific J. Math. 5 (1955), 367--378.

\bibitem[EKR93]{EKR93}
E. G. Effros, J. Kraus and Z.-J. Ruan.
\newblock On two quantized tensor products. 
\newblock Operator algebras, mathematical physics, and low-dimensional topology (Istanbul, 1991), 125--145, Res. Notes Math., 5, A K Peters, Wellesley, MA, 1993.

\bibitem[EfR03]{EfR03}
E. G. Effros and Z.-J. Ruan.
\newblock Operator space tensor products and Hopf convolution algebras. 
\newblock J. Operator Theory 50 (2003), no. 1, 131--156.


\bibitem[EfR00]{EfR00}
E. Effros and Z.-J. Ruan.
\newblock Operator spaces.
\newblock Oxford University Press (2000).

\bibitem[EiW07]{EiW07}
J. Eisert and M. Wolf.
\newblock Gaussian quantum channels. 
\newblock Quantum information with continuous variables of atoms and light, 23--42, Imp. Coll. Press, London, 2007. 


\bibitem[Eno77]{Eno77}
M. Enock.
\newblock Produit crois\'e d'une alg\`ebre de von Neumann par une alg\`ebre de Kac. (French) 
\newblock J. Funct. Anal. 26 (1977), no. 1, 16--47. 

\bibitem[EnS92]{EnS92}
M. Enock and J.-M. Schwartz.
\newblock Kac algebras and duality of locally compact groups. With a preface by Alain Connes. With a postface by Adrian Ocneanu.
\newblock Springer-Verlag, Berlin, 1992. 

\bibitem[EnV96]{EnV96}
M. Enock and L. Vainerman.
\newblock Deformation of a Kac algebra by an abelian subgroup. 
\newblock Comm. Math. Phys. 178 (1996), no. 3, 571--596. 




\bibitem[FHMV04]{FHMV04}
M. Fannes, B. Haegeman, M. Mosnyi, D. Vanpeteghem.
\newblock Additivity of minimal entropy output for a class of covariant channels.
\newblock arXiv:quant-ph/0410195.

\bibitem[FLLP24]{FLLP24}
C. Farsi, T. Landry, N. S. Larsen and J. Packer.
\newblock Spectral triples for noncommutative solenoids and a Wiener's lemma.
\newblock J. Noncommut. Geom. 18 (2024), no. 4, 1415--1452.

\bibitem[Fol16]{Fol16}
G. B. Folland.
\newblock A course in abstract harmonic analysis. Second edition. 
\newblock Textbooks in Mathematics. CRC Press, Boca Raton, FL, 2016. 

\bibitem[FHS11]{FHS11}
B. E. Forrest, H. H. Lee and E. Samei.
\newblock Projectivity of modules over Fourier algebras. 
\newblock Proc. Lond. Math. Soc. (3) 102 (2011), no. 4, 697--730.

\bibitem[FSS17]{FSS17}
U. Franz, A. Skalski and P. M. Soltan.
\newblock Introduction to compact and discrete quantum groups. 
\newblock Topological quantum groups, 9--31, Banach Center Publ., 111, Polish Acad. Sci. Inst. Math., Warsaw, 2017. 

\bibitem[FrG06]{FrG06}
U. Franz and R. Gohm.
\newblock Random walks on finite quantum groups. 
\newblock Quantum independent increment processes. II, 1--32, Lecture Notes in Math., 1866, Springer, Berlin, 2006. 

\bibitem[FrS09]{FrS09}
U. Franz and A. Skalski.
\newblock On idempotent states on quantum groups. 
\newblock J. Algebra 322 (2009), no. 5, 1774--1802. 


\bibitem[FrTS23]{FrTS23}
A. Freslon, F. Taipe and S. Wang.
\newblock Tannaka-Krein reconstruction and ergodic actions of easy quantum groups. 
\newblock Comm. Math. Phys. 399 (2023), no. 1, 105--172. 

\bibitem[FuW07]{FuW07}
M. Fukuda and M. M. Wolf.
\newblock Simplifying additivity problems using direct sum constructions. 
\newblock J. Math. Phys. 48 (2007), no. 7, 072101, 7 pp.. 

\bibitem[Ful20]{Ful20}
R. Fulsche.
\newblock Correspondence theory on $p$-Fock spaces with applications to Toeplitz algebras.
\newblock J. Funct. Anal. 279 (2020), no. 7, 108661.


\bibitem[FulG23]{FulG23}
R. Fulsche and N. Galke.
\newblock Quantum Harmonic Analysis on locally compact abelian groups.
\newblock Preprint, arXiv:2308.02078.

\bibitem[FuR23]{FuR23}
R. Fulsche, M. A. Rodriguez Rodriguez.
\newblock Commutative $G$-invariant Toeplitz $C^*$ algebras on the Fock space and their Gelfand theory through Quantum Harmonic Analysis.
\newblock Preprint, arXiv:2307.15632.

\bibitem[FLW24]{FLW24}
R. Fulsche, F. Luef and R. F. Werner.
\newblock Wiener's Tauberian theorem in classical and quantum harmonic analysis.
\newblock Preprint, arXiv:2405.08678.

\bibitem[FuH24]{FuH24}
R. Fulsche, R. Hagger.
\newblock Quantum harmonic analysis for polyanalytic Fock spaces.
\newblock Preprint, arXiv:2308.11292.
 
\bibitem[GPLS09]{GPLS09}
R. Garcia-Patr\'{o}n, S. Pirandola, S. Lloyd and J. H. Shapiro.
\newblock Reverse coherent information. 
\newblock Physical review letters, 102(21):210501, 2009. 







\bibitem[GJL18a]{GJL18a}
L. Gao, M. Junge and N. LaRacuente.
\newblock Capacity estimates via comparison with TRO channels.
\newblock Comm. Math. Phys. 364 (2018), no. 1, 83--121.

\bibitem[GJL18b]{GJL18b}
L. Gao, M. Junge, Nicholas LaRacuente.
\newblock Capacity bounds via operator space methods.
\newblock J. Math. Phys. 59 (2018), no. 12, 122202, 17 pp.

\bibitem[GJL20]{GJL20}
L. Gao, M. Junge and N. LaRacuente.
\newblock Fisher Information and Logarithmic Sobolev Inequality for Matrix Valued Functions.
\newblock Ann. Henri Poincar\'e 21 (2020), no. 11, 3409--3478.

\bibitem[GJL20b]{GJL20b}
L. Gao, M. Junge and N. LaRacuente.
\newblock Relative entropy for von Neumann subalgebras. 
\newblock Internat. J. Math. 31 (2020), no. 6, 2050046, 35 pp.



\bibitem[Gir15]{Gir15}
S. M. Girvin.
\newblock In Quantum machines: measurement and control of engineered quantum systems,
edited by M. H. Devoret, B. Huard, R. J. Schoelkopf, and L. F. Cugliandolo (Oxford University
Press, Oxford, 2015) Chap. 3.
\newblock 


\bibitem[GJP17]{GJP17}
A. Gonz\'alez-P\'erez, M. Junge and J. Parcet.
\newblock Smooth Fourier multipliers in group algebras via Sobolev dimension. 
\newblock Ann. Sci. \'Ec. Norm. Sup\'er. (4) 50 (2017), no. 4, 879--925.

\bibitem[Gos73]{Gos73}
J. Gosselin.
\newblock Almost everywhere convergence of Vilenkin-Fourier series.
\newblock Trans. Amer. Math. Soc. 185 (1973), 345--370 (1974).



\bibitem[GMS15]{GMS15}
V. P. Gupta, P. Mandayam and V. S. Sunder.
\newblock The Functional Analysis of Quantum Information Theory. A collection of notes based on lectures by Gilles Pisier, K. R. Parthasarathy, Vern Paulsen and Andreas Winter. 
\newblock Lecture Notes in Physics, 902. Springer, Cham, 2015.
		
\bibitem[GIN18]{GIN18}
L. Gyongyosi, S. Imre and H. V. Nguyen.
\newblock A Survey on Quantum Channel Capacities. 
\newblock IEEE Communications Surveys and Tutorials 20 (2018), no. 2, 1149--1205. 	
	
\bibitem[GuW15]{GuW15}
M. K. Gupta and M. M. Wilde.
\newblock Multiplicativity of completely bounded $p$-norms implies a strong converse for entanglement-assisted capacity. 
\newblock Comm. Math. Phys. 334 (2015), no. 2, 867--887. 	
	



 
\bibitem[Hal13]{Hal13}
B. Hall.
\newblock Quantum Theory for Mathematicians.
\newblock Springer New York, 2013. 

\bibitem[Hal23]{Hal23}
S. Halvdansson.
\newblock Quantum harmonic analysis on locally compact groups.
\newblock J. Funct. Anal. 285 (2023), no. 8, Paper No. 110096, 49 pp.


\bibitem[HLS81]{HLS81}
R. Hoegh-Krohn, M. B. Landstad and E. Stormer.
\newblock Compact ergodic groups of automorphisms. 
\newblock Ann. of Math. (2) 114 (1981), no. 1, 75--86. 







%





\bibitem[Gra14]{Gra14}
L. Grafakos.
\newblock Classical Fourier analysis. Third edition. 
\newblock Graduate Texts in Mathematics, 249. Springer, New York, 2014.

\bibitem[HJX10]{HJX10}
U. Haagerup, M. Junge and Q. Xu.
\newblock A reduction method for noncommutative $L_p$-spaces and applications.
\newblock  Trans. Amer. Math. Soc. 362 (2010), no. 4, 2125--2165.

\bibitem[Haa78]{Haa78}
U. Haagerup.
\newblock An Example of a nonnuclear C*-Algebra, which has the metric approximation property.
\newblock Invent. Math. 50 (1978/79), no. 3, 279--293.



 

\bibitem[HaM11]{HaM11}
U. Haagerup and M. Musat.
\newblock Factorization and dilation problems for completely positive maps on von Neumann algebras.
\newblock Comm. Math. Phys. 303 (2011), no. 2, 555--594.



\bibitem[HOS84]{HOS84}
H. Hanche-Olsen and E. St\o rmer.
\newblock Jordan operator algebras.
\newblock Monographs and Studies in Mathematics, 21. Pitman (Advanced Publishing Program), Boston, MA, 1984.

\bibitem[Har99]{Har99}
A. Harcharras.
\newblock Fourier analysis, Schur multipliers on $S^p$ and non-commutative $\Lambda(p)$-sets.
\newblock Studia Math. 137 (1999), no. 3, 203--260.

\bibitem[Has09]{Has09}
M. B. Hastings.
\newblock A Counterexample to Additivity of Minimum Output Entropy. 
\newblock Nature Physics 5, 255 (2009).

\bibitem[HaS21]{HaS21}
S. Hatui and P. Singla.
\newblock On Schur multiplier and projective representations of Heisenberg groups.
\newblock J. Pure Appl. Algebra   225 (2021), no. 11, Paper No. 106742, 16 pp.



\bibitem[HeR79]{HeR79}
E. Hewitt and K. A. Ross.
\newblock Abstract harmonic analysis. Vol. I.
Structure of topological groups, integration theory, group
representations. Second edition.
\newblock Grundlehren der Mathematischen Wissenschaften, 115. Springer-Verlag, Berlin-New York, 1979.

\bibitem[HeR70]{HeR70}
E. Hewitt and K. A. Ross.
\newblock Abstract harmonic analysis. Vol. II: Structure and analysis for compact groups. Analysis on locally compact Abelian groups.
\newblock Die Grundlehren der mathematischen Wissenschaften, Band 152, Springer-Verlag, New York-Berlin, 1970.





\bibitem[Hig01]{Hig01}
R. J. Higgs.
\newblock Projective representations of abelian groups.
\newblock J. Algebra 242 (2001), no. 2, 769--781.

\bibitem[Hol19]{Hol19}
A. S. Holevo.
\newblock Quantum Systems, Channels, Information: A Mathematical Introduction. 2nd edition
\newblock Gruyter Studies in Mathematical Physics, 16. De Gruyter, Berlin, 2019.

\bibitem[Hol05]{Hol05}
A. S. Holevo.
\newblock Additivity conjecture and covariant channels.
\newblock International Journal of Quantum Information, Vol. 3, No. 1 (2005) 41--47.

\bibitem[Hol16]{Hol16}
A. S. Holevo.
\newblock On the constrained classical capacity of infinite-dimensional covariant quantum channels.
\newblock Journal of Mathematical Physics 57, 015203 (2016).

\bibitem[HoW01]{HoW01}
A. S. Holevo and R.F. Werner.
\newblock Evaluating capacities of bosonic gaussian
channels.
\newblock Phys. Rev. A 63 (2001), 032312.

\bibitem[Hor60]{Hor60}
L. H\"ormander.
\newblock Estimates for translation invariant operators in $\L^p$ spaces. 
\newblock Acta Math. 104 (1960), 93--140. 

\bibitem[HSR03]{HSR03}
M. Horodecki, P. W. Shor and M. B. Ruskai.
\newblock Entanglement breaking channels. 
\newblock Rev. Math. Phys. 15 (2003), no. 6, 629--641. 

\bibitem[HNR10]{HNR10}
Z. Hu, M. Neufang and Z.-J. Ruan.
\newblock Multipliers on a new class of Banach algebras, locally compact quantum groups, and topological centres. 
\newblock Proc. Lond. Math. Soc. (3) 100 (2010), no. 2, 429--458. 

\bibitem[HNR11]{HNR11}
Z. Hu, M. Neufang and Z.-J. Ruan.
\newblock Completely bounded multipliers over locally compact quantum groups. 
\newblock Proc. Lond. Math. Soc. (3) 103 (2011), no. 1, 1--39. 

\bibitem[Hua99]{Hua99}
S.-Z. Huang.
\newblock Completeness of eigenvectors of group representations of operators whose Arveson spectrum is scattered.
\newblock Proc. Amer. Math. Soc. 127 (1999), no. 5, 1473--1482.

\bibitem[HuT22]{HuT22}
L. Huang, Z. Liu and J. Wu.
\newblock Quantum convolution inequalities on Frobenius von Neumann algebras. 
\newblock Preprint, arXiv:2204.04401. 

\bibitem[HuT83]{HuT83}
T. Huruya and J. Tomiyama.
\newblock Completely bounded maps of $C^*$-algebras. 
\newblock J. Operator Theory 10 (1983), no. 1, 141--152. 

\bibitem[HvNVW16]{HvNVW16}
T. Hyt\"onen, J. van Neerven, M. Veraar and L. Weis.
\newblock Analysis in Banach spaces, Volume~I: Martingales and Littlewood-Paley theory.
\newblock Springer, 2016. 





\bibitem[Izu97]{Izu97}
H. Izumi.
\newblock Constructions of non-commutative $L^p$-spaces with a complex parameter arising from modular actions.
\newblock Internat. J. Math. 8 (1997), no. 8, 1029--1066.

\bibitem[Izu98]{Izu98}
H. Izumi.
\newblock Natural bilinear forms, natural sesquilinear forms and the associated duality on non-commutative $\L^p$-spaces.
\newblock Internat. J. Math. 9 (1998), no. 8, 975--1039.

\bibitem[Izu02]{Izu02}
M. Izumi.
\newblock Non-commutative Poisson boundaries and compact quantum group actions. 
\newblock Adv. Math. 169 (2002), no. 1, 1--57. 

\bibitem[Izu04]{Izu04}
M. Izumi.
\newblock Non-commutative Poisson boundaries. 
\newblock Discrete geometric analysis, 69--81, Contemp. Math., 347, Amer. Math. Soc., Providence, RI, 2004. 


\bibitem[JHDS25]{JHDS25}
N. James, J. Huang, G. S. Dhanoa, K. P. Sidde.
\newblock On the $p$-adic solenoid, its constructions and visualizations.
\newblock Available at \href{https://www.google.com/url?sa=t&source=web&rct=j&opi=89978449&url=http://simonrs.com/eulercircle/padic2022/nicholas-jinfei-gurasees-krishna-solenoid.pdf&ved=2ahUKEwi-oOP5xJiLAxUKU6QEHbloLokQFnoECBMQAQ&usg=AOvVaw0IPmkVvCPd9k3-BGqMPn3m}{Internet}.

\bibitem[Jen06]{Jen06}
A. Jencova.
\newblock A relation between completely bounded norms and conjugate channels. 
\newblock Comm. Math. Phys. 266 (2006), no. 1, 65--70. 

\bibitem[Jew75]{Jew75}
R. I. Jewett.
\newblock Spaces with an abstract convolution of measures. 
\newblock Advances in Math. 18 (1975), no. 1, 1--101.

\bibitem[JLJ16]{JLJ16}
C. Jiang, Z. Liu and J. Wu.
\newblock Noncommutative uncertainty principles. 
\newblock J. Funct. Anal. 270 (2016), no. 1, 264--311. 








\bibitem[Jon21]{Jon21}
V. F. R. Jones.
\newblock Planar algebras, I.
\newblock New Zealand J. Math. 52 (2021 [2021--2022]), 1--107.

\bibitem[Jon00]{Jon00}
V. F. R. Jones.
\newblock The planar algebra of a bipartite graph.
\newblock Knots in Hellas '98 (Delphi), 94--117, Ser. Knots Everything, 24, World Sci. Publ., River Edge, NJ, 2000.


\bibitem[Jun04]{Jun04}
M. Junge.
\newblock Fubini's theorem for ultraproducts of noncommmutative $L_p$-spaces.
\newblock Canad. J. Math. 56 (2004), no. 5, 983--1021.

\bibitem[Jun1]{Jun1}
M. Junge.
\newblock Fubini's theorem for ultraproducts of noncommmutative $L_p$-spaces II.
\newblock Preprint.

\bibitem[Jun99]{Jun99}
M. Junge.
\newblock Factorization theory for spaces of operators. 
\newblock Institut for Matematik og Datalogi, Odense Universitet, 1999.

\bibitem[JMX06]{JMX06}
M. Junge, C. Le Merdy and Q. Xu.
\newblock $H^\infty$ functional calculus and square functions on noncommutative $L^p$-spaces.
\newblock Ast\'erisque No. 305 (2006).

\bibitem[JNR09]{JNR09}
M. Junge, M. Neufang and Z.-J. Ruan.
\newblock A representation theorem for locally compact quantum groups. 
\newblock Internat. J. Math. 20 (2009), no. 3, 377--400. 


\bibitem[JuP15]{JuP15}
M. Junge and C. Palazuelos.
\newblock Channel capacities via $p$-summing norms.
\newblock Adv. Math. 272 (2015), 350--398.

\bibitem[JuP16]{JuP16}
M. Junge and C. Palazuelos.
\newblock CB-norm estimates for maps between noncommutative Lp-spaces and quantum channel theory.
\newblock Int. Math. Res. Not. IMRN 2016, no. 3, 875--925.

\bibitem[JuP10]{JuP10}
M. Junge and J. Parcet.
\newblock Mixed-norm inequalities and operator space $L_p$ embedding theory.
\newblock Mem. Amer. Math. Soc. 203 (2010), no. 953.



\bibitem[KaP66]{KaP66}
G. I. Kac and V. G. Paljutkin.
\newblock Finite ring groups. (Russian)
\newblock Trudy Moskov. Mat. Obsc. 15 (1966), 224--261. 

\bibitem[Kad04]{Kad04}
R. V. Kadison.
\newblock Non-commutative conditional expectations and their applications.
\newblock Operator algebras, quantization, and noncommutative geometry, 143--79. Contemp. Math., 365. American Mathematical Society, Providence, RI, 2004.

\bibitem[Kal13]{Kal13}
M. Kalantar.
\newblock Representation of left centralizers for actions of locally compact quantum groups.
\newblock Internat. J. Math. 24 (2013), no. 4, 1350025, 10 pp.

\bibitem[Kan69]{Kan69}
E. Kaniuth.
\newblock Der Typ der regul\"aren Darstellung diskreter Gruppen. (German).
\newblock Math. Ann. 182 (1969), 334--339.





\bibitem[KaR97]{KaR97}
R. V. Kadison and J. R. Ringrose.
\newblock Fundamentals of the theory of operator algebras. Vol. II. Advanced theory. Corrected reprint of the 1986 original.
\newblock Graduate Studies in Mathematics, 16. American Mathematical Society, Providence, RI, 1997.

\bibitem[KNR14]{KNR14}
M. Kalantar, M. Neufang and Z.-J. Ruan.
\newblock Realization of quantum group Poisson boundaries as crossed products. 
\newblock Bull. Lond. Math. Soc. 46 (2014), no. 6, 1267--1275. 

\bibitem[Kal01]{Kal01}
A. A. Kalyuzhnyi.
\newblock Conditional expectations on compact quantum groups and new examples of quantum hypergroups. 
\newblock Methods Funct. Anal. Topology 7 (2001), no. 4, 49--68. 


\bibitem[KaT13]{KaT13}
E. Kaniuth and K. F. Taylor.
\newblock Induced representations of locally compact groups. 
\newblock Cambridge Tracts in Mathematics, 197. Cambridge University Press, Cambridge, 2013. 

\bibitem[KaL18]{KaL18}
E. Kaniuth and A. T.-M. Lau.
\newblock Fourier and Fourier-Stieltjes algebras on locally compact groups. 
\newblock Mathematical Surveys and Monographs, 231. American Mathematical Society, Providence, RI, 2018. 

\bibitem[KaS82]{KaS82}
Y. Katayama and G. Song.
\newblock Ergodic co-actions of discrete groups.
\newblock Math. Japon. 27 (1982), no. 2, 159--175.


\bibitem[Kar85]{Kar85}
G. Karpilovsky.
\newblock Projective representations of finite groups. 
\newblock Monographs and Textbooks in Pure and Applied Mathematics, 94. Marcel Dekker, Inc., New York, 1985.

\bibitem[Kar94]{Kar94}
G. Karpilovsky.
\newblock Group representations. Vol. 3.
\newblock North-Holland Math. Stud., 180. North-Holland Publishing Co., Amsterdam, 1994. 



\bibitem[KaS14]{KaS14}
P. Kasprzak, P. M. Soltan.
\newblock Embeddable quantum homogeneous spaces. 
\newblock J. Math. Anal. Appl. 411 (2014), no. 2, 574--591. 

\bibitem[Key02]{Key02}
M. Keyl.
\newblock Fundamentals of Quantum Information Theory. 
\newblock Phys. Rep. 369, no. 5, 431--548 (2002).

\bibitem[KhW20]{KhW20}
S. Khatri, M. M. Wilde.
\newblock Principles of Quantum Communication Theory: A Modern Approach. 
\newblock Preprint, arXiv:2011.04672. 

\bibitem[Kin02]{Kin02}
C. King.
\newblock Additivity for unital qubit channels. 
\newblock J. Math. Phys. 43 (2002), no. 10, 4641--4653. 

\bibitem[Kin03]{Kin03}
C. King.
\newblock Maximal $p$-norms of entanglement breaking channels. 
\newblock Quantum Inf. Comput. 3 (2003), no. 2, 186--190. 
 
\bibitem[KiR01]{KiR01}
C. King and M. B. Ruskai.
\newblock Minimal entropy of states emerging from noisy quantum channels.
\newblock IEEE Trans. Inform. Theory 47 (2001), no. 1, 192--209.

\bibitem[Kit02]{Kit02}
A. Y. Kitaev.
\newblock Topological quantum codes and anyons. 
\newblock Quantum computation: a grand mathematical challenge for the twenty-first century and the millennium (Washington, DC, 2000), 267--272, Proc. Sympos. Appl. Math., 58, AMS Short Course Lecture Notes, Amer. Math. Soc., Providence, RI, 2002.

\bibitem[Kit03]{Kit03}
A. Y. Kitaev.
\newblock Fault-tolerant quantum computation by anyons. 
\newblock Ann. Physics 303 (2003), no. 1, 2--30. 

\bibitem[Kit06]{Kit06}
A. Y. Kitaev.
\newblock Anyons in an exactly solved model and beyond. 
\newblock Ann. Physics 321 (2006), no. 1, 2--111.

\bibitem[Kit10]{Kit10}
A. Y. Kitaev.
\newblock Topological phases and quantum computation. 
\newblock Notes written by C. Laumann. Exact methods in low-dimensional statistical physics and quantum computing, 101--125, Oxford Univ. Press, Oxford, 2010.

\bibitem[KlR78]{KlR78}
A. Klein and B. Russo.
\newblock Sharp inequalities for Weyl operators and Heisenberg groups. 
\newblock Math. Ann. 235 (1978), no. 2, 175--194. 

\bibitem[Kle62]{Kle62}
A. Kleppner.
\newblock The structure of some induced representations.
\newblock Duke Math. J. 29 (1962), 555--572.

\bibitem[Kle65]{Kle65}
A. Kleppner.
\newblock Multipliers on abelian groups.
\newblock Math. Ann. 158 (1965), 11--34.

\bibitem[Kle74]{Kle74}%
A. Kleppner.
\newblock Continuity and measurability of multiplier and projective representations.
\newblock J. Funct. Anal. 17 (1974), 214--226.


\bibitem[KSSS19]{KSSS19}
V. Kodiyalam, S. Sruthymurali, L. Sohan and V. S. Sunder.
\newblock On a presentation of the spin planar algebra.
\newblock Proc. Indian Acad. Sci. Math. Sci. 129 (2019), no. 2, Paper No. 27, 11 pp.



\bibitem[KMS20]{KMS20}
P. Kopszak, M. Mozrzymas, M. Studzinski.
\newblock Positive maps from irreducibly covariant operators. 
\newblock J. Phys. A 53 (2020), no. 39, 395306, 33 pp.

\bibitem[KPC10]{KPC10}
A. A. Kalyuzhnyi, G. B. Podkolzin, Y. A. Chapovsky.
\newblock Harmonic analysis on a locally compact hypergroup. 
\newblock Methods Funct. Anal. Topology 16 (2010), no. 4, 304--332. 

\bibitem[KrS16]{KrS16}
J. Krajczok and P. M. Soltan.
\newblock Center of the algebra of functions on the quantum group $\mathrm{SU}_q(2)$ and related topics. 
\newblock Comment. Math. 56 (2016), no. 2, 251--272.






 



\bibitem[KrS18]{KrS18}
J. Krajczok and P. M. Soltan.
\newblock Compact quantum groups with representations of bounded degree. 
\newblock J. Operator Theory 80 (2018), no. 2, 415--428.












\bibitem[Kur18]{Kur18}
Y. Kuramochi
\newblock Entanglement-breaking channels with general outcome operator algebras. 
\newblock J. Math. Phys. 59 (2018), no. 10, 102206, 15 pp.. 

\bibitem[Kus05]{Kus05}
J. Kustermans.
\newblock Locally compact quantum groups. 
\newblock Quantum independent increment processes. I, 99--180, Lecture Notes in Math., 1865, Springer, Berlin, 2005. 

\bibitem[KuV03]{KuV03}
J. Kustermans and S. Vaes.
\newblock Locally compact quantum groups in the von Neumann algebraic setting. 
\newblock Math. Scand. 92 (2003), no. 1, 68--92. 

\bibitem[KuV00]{KuV00}
J. Kustermans and S. Vaes.
\newblock The operator algebra approach to quantum groups. 
\newblock Proc. Natl. Acad. Sci. USA 97 (2000), no. 2, 547--552. 


\bibitem[LaP13]{LaP13}
F. Latr\'emoli\`ere and J. Packer.
\newblock Noncommutative solenoids and their projective modules.
\newblock Commutative and noncommutative harmonic analysis and applications, 35--53. Contemp. Math., 603. American Mathematical Society, Providence, RI, 2013.

\bibitem[LaP17]{LaP17}
F. Latr\'emoli\`ere and J. Packer.
\newblock Noncommutative solenoids and the Gromov-Hausdorff propinquity.
\newblock Proc. Amer. Math. Soc. 145 (2017), no. 5, 2043--2057.

\bibitem[LaP18]{LaP18}
F. Latr\'emoli\`ere and J. Packer.
\newblock Noncommutative solenoids.
\newblock New York J. Math. 24A (2018), 155--191.


\bibitem[LPW21]{LPW21}
Z. Liu, S. Palcoux and J. Wu.
\newblock Fusion bialgebras and Fourier analysis: analytic obstructions for unitary categorification. 
\newblock Adv. Math. 390 (2021), Paper No. 107905, 63 pp.. 

\bibitem[LeP95]{LeP95}
S. T. Lee and J. A. Packer.
\newblock Twisted group $C^*$-algebras for two-step nilpotent and generalized discrete Heisenberg groups.
\newblock J. Operator Theory 34 (1995), no. 1, 91--124.


\bibitem[LeY22]{LeY22}
H. H. Lee and S.-G. Youn.
\newblock Quantum channels with quantum group symmetry. 
\newblock  Comm. Math. Phys. 389 (2022), no. 3, 1303--1329.





\bibitem[LPW21]{LPW21}
Z. Liu, S. Palcoux and J. Wu.
\newblock Fusion bialgebras and Fourier analysis: analytic obstructions for unitary categorification. 
\newblock Adv. Math. 390 (2021), Paper No. 107905, 63 pp.. 





\bibitem[LeZ22]{LeZ22}
C. Le Merdy and S. Zadeh.
\newblock On factorization of separating maps on noncommutative $L^p$-spaces. 
\newblock Indiana Univ. Math. J. 71 (2022), no. 5, 1967--2000. 
 


\bibitem[LoW22]{LoW22}
R. Longo and E. Witten.
\newblock A note on continuous entropy.
\newblock Preprint, arXiv:2202.03357.

\bibitem[Los84]{Los84}
V. Losert.
\newblock On tensor products of Fourier algebras.
\newblock Arch. Math. (Basel) 43 (1984), no. 4, 370--372.

\bibitem[LuP17]{LuP17}
A. Luczak and H. Podsedkowska.
\newblock Properties of Segal's entropy for quantum systems.
\newblock Internat. J. Theoret. Phys. 56 (2017), no. 12, 3783--3793.

\bibitem[LuP19]{LuP19}
A. Luczak and H. Podsedkowska.
\newblock Mappings preserving Segal's entropy in von Neumann algebras. 
\newblock Ann. Acad. Sci. Fenn. Math. 44 (2019), no. 2, 769--789. 

\bibitem[LPS17]{LPS17}
A. Luczak, H. Podsedkowska and M. Seweryn.
\newblock Maximum entropy models for quantum systems.
\newblock Entropy 19 (2017), no. 1, Paper No. 1, 8 pp.

\bibitem[LuS18a]{LuS18a}
F. Luef and E. Skrettingland.
\newblock Convolutions for localization operators.
\newblock J. Math. Pures Appl. (9) 118 (2018), 288--316.

\bibitem[LuS19]{LuS19}
F. Luef and E. Skrettingland.
\newblock Mixed-state localization operators: Cohen's class and trace class operators.
\newblock J. Fourier Anal. Appl. 25 (2019), no. 4, 2064--2108.

\bibitem[LuS20]{LuS20}
F. Luef and E. Skrettingland.
\newblock On accumulated Cohen's class distributions and mixed-state localization operators.
\newblock Constr. Approx. 52 (2020), no. 1, 31--64.

\bibitem[LuS21]{LuS21}
F. Luef and E. Skrettingland.
\newblock A Wiener Tauberian theorem for operators and functions.
\newblock J. Funct. Anal. 280 (2021), no. 6, Paper No. 108883, 44 pp.

\bibitem[LuS18b]{LuS18b}
F. Luef and E. Skrettingland.
\newblock Convolutions for Berezin quantization and Berezin-Lieb inequalities.
\newblock J. Math. Phys. 59 (2018), no. 2, 023502, 11 pp.

\bibitem[Lun18]{Lun18}
A. Lunardi.
\newblock Interpolation theory. Third edition. 
\newblock Appunti. Scuola Normale Superiore di Pisa (Nuova Serie) [Lecture Notes. Scuola Normale Superiore di Pisa (New Series)], 16. Edizioni della Normale, Pisa, 2018. 

\bibitem[LWW17]{LWW17}
Z. Liu, S. Wang and J. Wu.
\newblock Young's inequality for locally compact quantum groups.
\newblock J. Operator Theory 77 (2017), no. 1, 109--131.


\bibitem[Mac58]{Mac58}%
G. W. Mackey.
\newblock Unitary representations of group extensions. I.
\newblock Acta Math. 99 (1958), 265--311.

\bibitem[MaD98]{MaD98}
A. Maes and A. Van Daele.
\newblock Notes on compact quantum groups. 
\newblock Nieuw Arch. Wisk. (4) 16 (1998), no. 1-2, 73--112. 

\bibitem[Maj06]{Maj06}
S. Majid.
\newblock What is a quantum group? 
\newblock Notices Amer. Math. Soc. 53 (2006), no. 1, 30--31. 

\bibitem[Mat21]{Mat21}
J. Matsuda.
\newblock Classification of Quantum Graphs on $M_2$ and their Quantum Automorphism Groups. 
\newblock Preprint, arXiv:2110.09085v2. 

\bibitem[Meg98]{Meg98}
R. E. Megginson.
\newblock An introduction to Banach space theory.
\newblock Graduate Texts in Mathematics, 183. Springer-Verlag, New York, 1998.

\bibitem[Mey95]{Mey95}
P. A. Meyer.
\newblock Quantum probability for probabilists. Second edition
\newblock Lecture Notes in Mathematics. Vol. 1538, Berlin-Heidelberg-New York: Springer, 1995.

\bibitem[Moo64]{Moo64}
C. C. Moore.
\newblock Extensions and low dimensional cohomology theory of locally compact groups. I.
\newblock Trans. Amer. Math. Soc. 113 (1964), 40--63.

\bibitem[Moo72]{Moo72}
C. C. Moore.
\newblock Groups with finite dimensional irreducible representations. 
\newblock Trans. Amer. Math. Soc. 166 (1972), 401--410. 

\bibitem[Moo76]{Moo76}
C. C. Moore.
\newblock Group extensions and cohomology for locally compact groups. III.
\newblock Trans. Amer. Math. Soc. 221 (1976), no. 1, 1--33.

\bibitem[Mor17]{Mor17}
V. Moretti.
\newblock Spectral theory and quantum mechanics. Mathematical foundations of quantum theories, symmetries and introduction to the algebraic formulation. 
\newblock Unitext, 110. La Matematica per il 3+2. Springer, Cham, 2017.

\bibitem[Mor62]{Mor62}
A. O. Morris.
\newblock The spin representation of the symmetric group.
\newblock Proc. London Math. Soc. (3) 12 (1962), 55--76.

\bibitem[Mor73]{Mor73}
A. O. Morris.
\newblock Projective representations of Abelian groups.
\newblock J. London Math. Soc. (2) 7 (1973), 235--238.

\bibitem[MST87]{MST87}
A. O. Morris, M. Saeed-Ul-Islam, E. Thomas.
\newblock Some projective representations of finite abelian groups.
\newblock Glasgow Math. J. 29 (1987), no. 2, 197--203.

\bibitem[MSD17]{MSD17}
M. Mozrzymas, M. Studzinski and N. Datta.
\newblock Structure of irreducibly covariant quantum channels for finite groups. 
\newblock J. Math. Phys. 58 (2017), no. 5, 052204, 34 pp.

\bibitem[Naa15]{Naa15}
P. Naaijkens.
\newblock Kitaev's quantum double model from a local quantum physics point of view. 
\newblock Advances in algebraic quantum field theory, 365--395, Math. Phys. Stud., Springer, Cham, 2015. 

\bibitem[NaU61]{NaU61}
M. Nakamura and H. Umegaki.
\newblock A note on the entropy for operator algebras. 
\newblock Proc. Japan Acad. 37 (1961), 149--154. 

\bibitem[NSSFS08]{NSSFS08}
C. Nayak, S. H. Simon, A. Stern, M. Freedman, S. Das Sarma.
\newblock Nonabelian anyons and topological quantum computation. 
\newblock Rev. Modern Phys. 80, 1083--1159 (2008). 
 
\bibitem[NaT79]{NaT79}
Y. Nakagami and M. Takesaki.
\newblock Duality for crossed products of von Neumann algebras. 
\newblock Lecture Notes in Mathematics, 731. Springer, Berlin, 1979. 

\bibitem[Naz88]{Naz88}
M. L Nazarov.
\newblock An orthogonal basis in irreducible projective representations of the symmetric group.
\newblock Funktsional. Anal. i Prilozhen. 22 (1988), no. 1, 77--78 (Russian); translation in Funct. Anal. Appl. 22 (1988), no. 1, 66--68.

\bibitem[Naz90]{Naz90}
M. L Nazarov.
\newblock Young's orthogonal form of irreducible projective representations of the symmetric group.
\newblock J. London Math. Soc. (2) 42 (1990), no. 3, 437--451.

\bibitem[NeT13]{NeT13}
S. Neshveyev and L. Tuset.
\newblock Compact quantum groups and their representation categories. 
\newblock Cours Sp\'ecialis\'es [Specialized Courses], 20. Soci\'et\'e Math\'ematique de France, Paris, 2013. 

\bibitem[NSSS21]{NSSS21}
M. Neufang, P. Salmi, A. Skalski and N. Spronk.
\newblock Fixed points and limits of convolution powers of contractive quantum measures.
\newblock Indiana Univ. Math. J. 70 (2021), no. 5, 1971--2009.

\bibitem[NeR11]{NeR11}
S. Neuwirth and \'E. Ricard.
\newblock Transfer of Fourier multipliers into Schur multipliers and sumsets in a discrete group.
\newblock Canad. J. Math. 63 (2011), no. 5, 1161--1187.  


\bibitem[NiC10]{NiC10}
M. A. Nielsen and I. L. Chuang.
\newblock Quantum computation and quantum information: 10th Anniversary Edition. 
\newblock Cambridge University Press, Cambridge, 2010. 





 
\bibitem[OcS78]{OcS78}
W. Ochs and H. Spohn.
\newblock A characterization of the Segal entropy. 
\newblock Rep. Math. Phys. 14 (1978), no. 1, 75--87.


\bibitem[Oml14]{Oml14}
T. A. Omland.
\newblock Primeness and primitivity conditions for twisted group $C^*$-algebras.
\newblock Math. Scand. 114 (2014), no. 2, 299--319.

\bibitem[Oml15]{Oml15}
T. A. Omland.
\newblock $C^*$-algebras generated by projective representations of free nilpotent group.
\newblock J. Operator Theory 73 (2015), no. 1, 3--25.

\bibitem[Osa91]{Osa91}
H. Osaka.
\newblock Completely bounded maps between the preduals of von Neumann algebras. 
\newblock Proc. Amer. Math. Soc. 111 (1991), no. 4, 961--965.




\bibitem[OhP93]{OhP93}
M. Ohya and D. Petz.
\newblock Quantum entropy and its use. 
\newblock Texts and Monographs in Physics. Springer-Verlag, Berlin, 1993. 


 

\bibitem[OPT80]{OPT80}
D. Olesen, G. K. Pedersen and M. Takesaki.
\newblock Ergodic actions of compact abelian groups. 
\newblock J. Operator Theory 3 (1980), no. 2, 237--269. 

\bibitem[Pac87]{Pac87}
J. A. Packer.
\newblock $C^*$-algebras generated by projective representations of the discrete Heisenberg group.
\newblock J. Operator Theory 18 (1987), no. 1, 41--66.

\bibitem[Pac08]{Pac08}
J. A. Packer.
\newblock Projective representations and the Mackey obstruction--a survey.
\newblock Group representations, ergodic theory, and mathematical physics: a tribute to George W. Mackey, 345--378. Contemp. Math., 449. American Mathematical Society, Providence, RI, 2008.

\bibitem[Pal01]{Pal01}
T. W. Palmer.
\newblock Banach algebras and the general theory of $*$-algebras. Vol. 2. *-algebras.
\newblock Encyclopedia of Mathematics and its Applications, 79. Cambridge University Press, Cambridge, 2001.

\bibitem[Par69]{Par69}
K. R. Parthasarathy.
\newblock Multipliers on locally compact groups.
\newblock Lecture Notes in Math., Vol. 93. Springer-Verlag, Berlin-New York, 1969.


\bibitem[Pau02]{Pau02}
V. Paulsen.
\newblock Completely bounded maps and operator algebras.
\newblock Cambridge Univ. Press (2002).


\bibitem[Ped18]{Ped18}
G. K. Pedersen.
\newblock $C^*$-algebras and their automorphism groups. Second edition. Edited and with a preface by S\o ren Eilers and Dorte Olesen.
\newblock Pure Appl. Math. (Amst.) Academic Press, London, 2018.

\bibitem[PSTW22]{PSTW22}
L.-E. Persson, F. Schipp, G. Tephnadze and F. Weisz.
\newblock An analogy of the Carleson-Hunt theorem with respect to Vilenkin systems.
\newblock J. Fourier Anal. Appl. 28 (2022), no. 3, Paper No. 48, 29 pp.


\bibitem[Pet08]{Pet08}
D. Petz.
\newblock Quantum information theory and quantum statistics. 
\newblock Theoretical and Mathematical Physics. Springer-Verlag, Berlin, 2008.

\bibitem[Pet01]{Pet01}
D. Petz.
\newblock Entropy, von Neumann and the von Neumann entropy. 
\newblock John von Neumann and the foundations of quantum physics (Budapest, 1999), 83--96, Vienna Circ. Inst. Yearb., 8, Kluwer Acad. Publ., Dordrecht, 2001. 





\bibitem[Pis98]{Pis98}
G. Pisier.
\newblock Non-commutative vector valued $L_p$-spaces and completely $p$-summing maps.
\newblock Ast\'erisque, 247, 1998.

\bibitem[Pis03]{Pis03}
G. Pisier.
\newblock Introduction to operator space theory.
\newblock Cambridge University Press, Cambridge, 2003.


%

\bibitem[Pis12]{Pis12}
G. Pisier.
\newblock Completely co-bounded Schur multipliers.
\newblock Oper. Matrices 6 (2012), no. 2, 263--270.

\bibitem[PiX03]{PiX03}
G. Pisier and Q. Xu.
\newblock Non-commutative $L^p$-spaces.
\newblock 1459--1517 in Handbook of the Geometry of Banach Spaces, Vol. II, edited by W.B. Johnson and J. Lindenstrauss, Elsevier (2003).

\bibitem[PlR94]{PlR94}
R. J. Plymen and P. L. Robinson.
\newblock Spinors in Hilbert space. 
\newblock Cambridge Tracts in Mathematics, 114. Cambridge University Press, Cambridge, 1994. 

\bibitem[Pod21]{Pod21}
H. Podsedkowska.
\newblock Strong subadditivity of quantum mechanical entropy for semifinite von Neumann algebras.
\newblock Studia Math. 257 (2021), no. 1, 71--85.

\bibitem[Pra11]{Pra11}
A. Prasad.
\newblock An easy proof of the Stone-von Neumann-Mackey theorem.
\newblock Expo. Math. 29 (2011), no. 1, 110--118.

\bibitem[PSV10]{PSV10}
A. Prasad, I. Shapiro and M. K. Vemuri.
\newblock Locally compact abelian groups with symplectic self-duality.
\newblock Adv. Math. 225 (2010), no. 5, 2429--2454.









\bibitem[RKSE10]{RKSE10}
P. Raynal, A. Kalev, J. Suzuki, and B.-G. Englert.
\newblock Encoding many qubits in a rotor.
\newblock Phys. Rev. A 81, 052327 (2010).

\bibitem[RJV18]{RJV18}
M. Rahaman, S. Jaques and V. I. Paulsen.
\newblock Eventually entanglement breaking maps. 
\newblock J. Math. Phys. 59 (2018), no. 6, 062201, 11 pp.. 

\bibitem[Rah17]{Rah17}
M. Rahaman.
\newblock Multiplicative properties of quantum channels. 
\newblock J. Phys. A 50 (2017), no. 34, 345302, 26 pp.. 

\bibitem[Ray03]{Ray03}
Y. Raynaud.
\newblock $L_p$-spaces associated with a von Neumann algebra without trace: a gentle introduction via complex interpolation.
\newblock  Trends in Banach spaces and operator theory (Memphis, TN, 2001), 245--273, Contemp. Math., 321, Amer. Math. Soc., Providence, RI, 2003. 











\bibitem[Rob91]{Rob91}
D. W. Robinson.
\newblock Elliptic operators and Lie groups. 
\newblock Oxford Mathematical Monographs. Oxford Science Publications. The Clarendon Press, Oxford University Press, New York, 1991. 





\bibitem[Rua96]{Rua96}
Z.-J. Ruan.
\newblock Amenability of Hopf von Neumann algebras and Kac algebras. 
\newblock J. Funct. Anal. 139 (1996), no. 2, 466--499.



\bibitem[Run08]{Run08}
V. Runde.
\newblock Characterizations of compact and discrete quantum groups through second duals. 
\newblock J. Operator Theory 60 (2008), no. 2, 415--428. 

\bibitem[Rus73]{Rus73}
M. B. Ruskai.
\newblock A generalization of entropy using traces on von Neumann algebras. 
\newblock Ann. Inst. H. Poincar\'e Sect. A (N.S.) 19 (1973), 357--373 (1974).

\bibitem[Rus03]{Rus03}
M. B. Ruskai.
\newblock Qubit entanglement breaking channels. 
\newblock Rev. Math. Phys. 15 (2003), no. 6, 643--662.

\bibitem[RSW02]{RSW02}
M. B. Ruskai, S. Szarek and E. Werner.
\newblock An analysis of completely positive trace-preserving maps on $M_2$.
\newblock Linear Algebra Appl. 347 (2002), 159--187.

\bibitem[Rya02]{Rya02}
R. A. Ryan.
\newblock Introduction to tensor products of Banach spaces.
\newblock Springer Monographs in Mathematics. Springer-Verlag London, Ltd., London, 2002.


 





\bibitem[Sak98]{Sak98}
S. Sakai.
\newblock $C^*$-algebras and $W^*$-algebras.
\newblock Reprint of the 1971 edition. Classics in Mathematics. Springer-Verlag, Berlin, 1998.

\bibitem[SaS17]{SaS17}
P. Salmi and A. Skalski.
\newblock Actions of locally compact (quantum) groups on ternary rings of operators, their crossed products, and generalized Poisson boundaries. 
\newblock Kyoto J. Math. 57 (2017), no. 3, 667--691. 



\bibitem[ScS18]{ScS18}
T. T. Scheckter and F. Sukochev.
\newblock Weak type estimates for the noncommutative Vilenkin-Fourier series.
\newblock Integral Equations Operator Theory 90 (2018), no. 6, Paper No. 64, 19 pp.

\bibitem[SWS90]{SWS90}
F. Schipp, W. R. Wade and P. Simon.
\newblock Walsh series. An introduction to dyadic harmonic analysis. 
\newblock Adam Hilger, Ltd., Bristol, 1990. 

\bibitem[Sch04]{Sch04}
J. Schur.
\newblock \"{U}ber die Darstellung der endlichen Gruppen durch gebrochene lineare Substitutionen.
\newblock J. Reine Angew. Math. 127 (1904), 20--50.

\bibitem[Sch07]{Sch07}
J. Schur.
\newblock Untersuchungen \"uber die Darstellung der endlichen Gruppen durch gebrochene lineare Substitutionen.
\newblock J. Reine Angew. Math. 132 (1907), 85--137.

\bibitem[Sch11]{Sch11}
J. Schur.
\newblock \"{U}ber die Darstellung der symmetrischen und der alternierenden Gruppe durch gebrochene lineare Substitutionen.
\newblock J. Reine Angew. Math. 139 (1911), 155--250.


\bibitem[Seg60]{Seg60}
I. E. Segal.
\newblock A note on the concept of entropy.
\newblock J. Math. Mech. 9 (1960) 623--629.

\bibitem[Sek96]{Sek96}
Y. Sekine.
\newblock An example of finite-dimensional Kac algebras of Kac-Paljutkin type. 
\newblock Proc. Amer. Math. Soc. 124 (1996), no. 4, 1139--1147. 



\bibitem[Sho95]{Sho95}
P. W. Shor.
\newblock Scheme for reducing decoherence in quantum computer memory. 
\newblock Physical Review A 52(4), R2493-R2496. 

\bibitem[SiS08]{SiS08}
A. Sinclair and R. Smith.
\newblock Finite von Neumann algebras and masas.
\newblock London Mathematical Society Lecture Note Series, 351. Cambridge University Press, Cambridge, 2008.

\bibitem[ShU19]{ShU1}
T. Shulman and O. Uuye.
\newblock Approximations of subhomogeneous algebras. 
\newblock Bull. Aust. Math. Soc. 100 (2019), no. 2, 328--337. 

\bibitem[SkV19]{SkV19}
A. Skalski and A. Viselter.
\newblock Convolution semigroups on locally compact quantum groups and noncommutative Dirichlet forms.
\newblock J. Math. Pures Appl. (9) 124 (2019), 59--105.

\bibitem[Skr20]{Skr20}
E. Skrettingland.
\newblock Quantum harmonic analysis on lattices and Gabor multipliers.
\newblock J. Fourier Anal. Appl. 26 (2020), no. 3, Paper No. 48, 37 pp.


\bibitem[Smi72]{Smi1}
M. Smith.
\newblock Regular representations of discrete groups.
\newblock J. Funct. Anal. 11 (1972), 401--406.

\bibitem[Smi83]{Smi83}
R. R. Smith.
\newblock Completely bounded maps between $C^*$-algebras. 
\newblock J. London Math. Soc. (2) 27 (1983), no. 1, 157--166. 

\bibitem[SmS05]{SmS05}
R. R. Smith and N. Spronk.
\newblock Representations of group algebras in spaces of completely bounded maps. 
\newblock Indiana Univ. Math. J. 54 (2005), no. 3, 873--896. 

\bibitem[Sol10]{Sol10}
P. M. Soltan.
\newblock Quantum $\SO(3)$ groups and quantum group actions on $M_2$. 
\newblock J. Noncommut. Geom. 4 (2010), no. 1, 1--28.


\bibitem[SBSW15a]{SBSW15a}
R. S. Stankovic, P. L. Butzer, F. Schipp and W. R. Wade.
\newblock Dyadic Walsh analysis from 1924 onwards--Walsh-Gibbs-Butzer dyadic differentiation in science. Vol. 1. Foundations. A monograph based on articles of the founding authors, reproduced in full. In collaboration with the co-authors: Weiyi Su, Yasushi Endow, Sandor Fridli, Boris I. Golubov, Franz Pichler and Kees Onneweer.
\newblock Atlantis Studies in Mathematics for Engineering and Science, 12. Atlantis Press, Paris, 2015.

\bibitem[SBSW15b]{SBSW15b}
R. S. Stankovic, P. L. Butzer, F. Schipp and W. R. Wade.
\newblock Dyadic Walsh analysis from 1924 onwards--Walsh-Gibbs-Butzer dyadic differentiation in science. Vol. 2. Extensions and generalizations. A monograph based on articles of the founding authors, reproduced in full. In collaboration with the co-authors: Weiyi Su, Yasushi Endow, Sandor Fridli, Boris I. Golubov, Franz Pichler and Kees Onneweer.
\newblock Atlantis Studies in Mathematics for Engineering and Science, 13. Atlantis Press, Paris, 2015.


\bibitem[Spe92]{Spe92}
R. Speicher.
\newblock A noncommutative central limit theorem.
\newblock Math. Z. 209 (1992), no. 1, 55--66.


\bibitem[Sto13]{Sto13}
E. St\o rmer.
\newblock Positive linear maps of operator algebras.
\newblock Springer Monographs in Mathematics. Springer, Heidelberg, 2013.
 
\bibitem[Sto80]{Sto80}
E. St\o rmer.
\newblock Regular abelian Banach algebras of linear maps of operator algebras.
\newblock J. Funct. Anal. 37 (1980), no. 3, 331--373.

\bibitem[Sto07]{Sto07}
E. St\o rmer.
\newblock A reduction theorem for capacity of positive maps.
\newblock Positivity 11 (2007), no. 1, 69--75.



%
%





%






\bibitem[Ste12]{Ste12}
B. Steinberg.
\newblock Representation theory of finite groups. An introductory approach.
\newblock Universitext. Springer, New York, 2012.


\bibitem[Sto74]{Sto74}
E. St\o rmer.
\newblock Spectra of ergodic transformations
\newblock J. Funct. Anal. 15 (1974), 202--215.

\bibitem[Str81]{Str81}
S. Stratila.
\newblock Modular theory in operator algebras.
\newblock Taylor and Francis, 1981.

\bibitem[Sza09]{Sza09}
S. J. Szarek.
\newblock Still more on norms of completely positive maps. 
\newblock Proceedings of IWOTA 2008, Operator Theory: Advances and Applications, Vol. 202, Birkhauser 2009, 535--538.



\bibitem[Tak02]{Tak02}
M. Takesaki.
\newblock Theory of operator algebras. I. Reprint of the first (1979) edition.
\newblock Encyclopaedia of Mathematical Sciences, 124. Operator Algebras and Non-commutative Geometry, 5. Springer-Verlag, Berlin, 2002.

\bibitem[Tak03]{Tak03}
M. Takesaki.
\newblock Theory of operator algebras. II.
\newblock Encyclopaedia of Mathematical Sciences, 125. Operator Algebras and Non-commutative Geometry, 6. Springer-Verlag, Berlin, 2003.



\bibitem[Ter81]{Ter81}
M. Terp.
\newblock $L^p$ spaces associated with von Neumann algebras.
\newblock Notes, Math. Institute, Copenhagen Univ., 1981.


\bibitem[Ter82]{Ter82}
M. Terp.
\newblock Interpolation spaces between a von Neumann algebra and its predual.
\newblock J. Operator Theory 8 (1982), no. 2, 327--360.

\bibitem[TiD24]{TiD24}
C. Tian and D. Zhou.
\newblock Weak type estimate for the partial sums of noncommutative Vilenkin-Fourier series.
\newblock Proc. Amer. Math. Soc. 152 (2024), no. 3, 1153--1167.

\bibitem[Tro16]{Tro16}
S. Trotter.
\newblock Involutive Algebras and Locally Compact Quantum Groups.
\newblock PhD dissertation, University of Leeds, 2016.

\bibitem[Tse24]{Tse24}
A. Tselishchev.
\newblock Littlewood-Paley-Rubio de Francia inequality for unbounded Vilenkin systems.
\newblock J. Approx. Theory 298 (2024), Paper No. 106006, 31 pp.

\bibitem[Tus22]{Tus22}
L. Tuset.
\newblock Analysis and quantum groups.
\newblock Springer, Cham, 2022.













\bibitem[Vae01]{Vae01}
S. Vaes.
\newblock The unitary implementation of a locally compact quantum group action. 
\newblock J. Funct. Anal. 180 (2001), no. 2, 426--480.

\bibitem[Vae02]{Vae02}
S. Vaes.
\newblock Locally compact quantum groups.
\newblock Ph. D. Thesis, available at \href{http://www.wis.kuleuven.ac.be/analyse/,
  2002}{http://www.wis.kuleuven.ac.be/analyse/,
  2002}.

\bibitem[Vae03]{Vae03}
S. Vaes.
\newblock Course. Locally compact quantum groups.
\newblock Available at \href{http://www.wis.kuleuven.ac.be/analyse/,
  2002}{http://www.wis.kuleuven.ac.be/analyse/,
  2002}.


\bibitem[VaV03]{VaV03}
S. Vaes and L. Vainerman.
\newblock Extensions of locally compact quantum groups and the bicrossed product construction. 
\newblock Adv. Math. 175 (2003), no. 1, 1--101.

\bibitem[VDa14]{VDa14}
A. Van Daele.
\newblock Locally compact quantum groups. A von Neumann algebra approach. 
\newblock SIGMA Symmetry Integrability Geom. Methods Appl. 10 (2014), Paper 082, 41 pp.

\bibitem[VDa97]{VDa97}
A. Van Daele.
\newblock The Haar measure on finite quantum groups. 
\newblock Proc. Amer. Math. Soc. 125 (1997), no. 12, 3489--3500.

\bibitem[Var85]{Var85}%
V. S. Varadarajan.
\newblock Geometry of quantum theory. Second edition.
\newblock Springer-Verlag, New York, 1985.

\bibitem[Var96]{Var96}%
V. S. Varadarajan.
\newblock Quantum kinematics and projective unitary representations of abelian groups.
\newblock Analysis, geometry and probability, 362--396. Texts Read. Math., 10. Hindustan Book Agency, Delhi; distributed by the, 1996.

\bibitem[Var08]{Var08}
V. S. Varadarajan.
\newblock George Mackey and his work on representation theory and foundations of physics.
\newblock Group representations, ergodic theory, and mathematical physics: a tribute to George W. Mackey, 417--446. Contemp. Math., 449. American Mathematical Society, Providence, RI, 2008.

\bibitem[Vil63]{Vil63}
N. J. Vilenkin.
\newblock On a class of complete orthonormal systems.
\newblock Amer. Math. Soc. Transl. (2) 28 (1963) 1--35.

\bibitem[Voi92]{Voi92}
J. Voigt.
\newblock Abstract Stein interpolation. 
\newblock Math. Nachr. 157 (1992), 197--199.

\bibitem[Wal79]{Wal79}
D. B. Wales.
\newblock Some projective representations of $S_n$.
\newblock J. Algebra 61 (1979), no. 1, 37--57.


 
\bibitem[Wan99]{Wan99}
S. Wang.
\newblock Ergodic actions of universal quantum groups on operator algebras. 
\newblock Comm. Math. Phys. 203 (1999), no. 2, 481--498. 

\bibitem[Wan13]{Wan13}
S. Wang.
\newblock More examples of algebraic quantum hypergroups. 
\newblock Algebr. Represent. Theory 16 (2013), no. 1, 205--228. 

\bibitem[Wat58]{Wat58}
C. Watari.
\newblock On generalized Walsh Fourier series.
\newblock Tohoku Math. J. (2) 10 (1958), 211--241.

\bibitem[Was88a]{Was88a}
A. Wassermann.
\newblock Ergodic actions of compact groups on operator algebras. II. Classification of full multiplicity ergodic actions. 
\newblock Canad. J. Math. 40 (1988), no. 6, 1482--1527. 

\bibitem[Was88b]{Was88b}
A. Wassermann.
\newblock Ergodic actions of compact groups on operator algebras. III. Classification for SU(2). 
\newblock Invent. Math. 93 (1988), no. 2, 309--354. 

\bibitem[Was89]{Was89}
A. Wassermann.
\newblock Ergodic actions of compact groups on operator algebras. I. General theory. 
\newblock Ann. of Math. (2) 130 (1989), no. 2, 273--319.

\bibitem[Wat88]{Wat88}
K. Watanabe.
\newblock Dual of noncommutative $L^p$-spaces with $0<p<1$.
\newblock Math. Proc. Cambridge Philos. Soc. 103 (1988), no. 3, 503--509.

\bibitem[Wat18]{Wat18}
J. Watrous.
\newblock The Theory of Quantum Information. 
\newblock Cambridge university press, 2018.

\bibitem[Wer16]{Wer16}
R. F. Werner.
\newblock Uncertainty relations for general phase spaces.
\newblock Front. Phys. 11 (2016), no. 5, 1--10.

\bibitem[Wer84]{Wer84}
R. Werner.
\newblock Quantum harmonic analysis on phase space.
\newblock J. Math. Phys. 25 (1984), no. 5, 1404--1411.


\bibitem[WeH02]{WeH02}
R. F. Werner and A. S. Holevo.
\newblock Counterexample to an additivity conjecture for output purity of quantum channels.
\newblock J. Math. Phys. 43 (2002), no. 9, 4353--4357.

\bibitem[Wig39]{Wig39}
E. Wigner.
\newblock On unitary representations of the inhomogeneous Lorentz group.
\newblock Ann. of Math. (2) 40 (1939), no. 1, 149--204.

\bibitem[Wil17]{Wil17}
M. M. Wilde.
\newblock Quantum information theory. Second edition.
\newblock Cambridge University Press, Cambridge, 2017.

\bibitem[WiY16]{WiY16}
A. Winter and D. Yang.
\newblock Potential capacities of quantum channels. 
\newblock IEEE Trans. Inform. Theory 62 (2016), no. 3, 1415--1424.




\bibitem[YHW19]{YHW19}
D. Yang, K. Horodecki and A. Winter.
\newblock Distributed Private Randomness Distillation. 
\newblock Phys. Rev. Lett. 123, 170501 - Published 22 October 2019.

\bibitem[You76]{You76}
W. S. Young.
\newblock Mean convergence of generalized Walsh-Fourier series.
\newblock Trans. Amer. Math. Soc. 218 (1976), 311--320.



\bibitem[Zha20]{Zha20}
H. Zhang.
\newblock Infinitely divisible states on finite quantum groups. 
\newblock Math. Z. 294 (2020), no. 1-2, 571--592.

\bibitem[ZM68]{ZM68}
G. Zeller-Meier.  
\newblock Produits crois\'es d'une $C^*$-alg\`ebre par un groupe d'automorphismes. (French).
\newblock J. Math. Pures Appl. (9) 47 (1968), 101--239.

\bibitem[ZhL21]{ZhL21}
Y. Zhang and S. Luo.
\newblock Quantifying decoherence of Gaussian noise channels. 
\newblock J. Stat. Phys. 183 (2021), no. 2, Paper No. 19, 18 pp.

\end{thebibliography}
\end{document}